\documentclass[preprint,11pt]{imsart}
\usepackage{xr}
\usepackage{amsmath,amssymb}
\usepackage{parskip}
\usepackage{marvosym}
\usepackage[hang,small,bf]{caption}
\usepackage{bm}
\usepackage[colorlinks]{hyperref}
\usepackage{sansmath}
\hypersetup{
	colorlinks,%
	citecolor=blue,%
	filecolor=black,%
	linkcolor=blue,%
	urlcolor=blue
}
\usepackage{enumerate}
\usepackage{multirow}
\RequirePackage[OT1]{fontenc}

\usepackage{amsthm}
\usepackage{amsmath}
\usepackage{amssymb}
\usepackage{thmtools} 
\usepackage{subcaption}
\usepackage{textcomp}

\usepackage{graphicx}
\usepackage{verbatim}
\usepackage{array, float}
\usepackage{fontenc}
\usepackage[toc,page]{appendix}
\usepackage[nottoc]{tocbibind}
\usepackage[english]{babel}
\usepackage{marvosym}

\usepackage[pagewise]{lineno}

\usepackage{mathtools}
\allowdisplaybreaks
\usepackage{bm}
\usepackage[noabbrev,capitalize]{cleveref}

\newtheoremstyle{exampstyle}
{8pt} 
{8pt} 
{\it} 
{} 
{\bfseries} 
{.} 
{.5em} 
{} 

\theoremstyle{exampstyle}
\newtheorem{theorem}{Theorem}
\newtheorem{lemma}{Lemma}

\newtheorem{prop}{Proposition}

\newtheoremstyle{remarkstyle}
{8pt} 
{8pt} 
{} 
{} 
{\bfseries} 
{.} 
{.5em} 
{} 

\theoremstyle{remarkstyle}
\newtheorem{example}{Example}
\newtheorem{remark}{Remark}
\newtheorem{assumption}{Assumption}[section]
\newtheorem{definition}{Definition}

\numberwithin{equation}{section}
\numberwithin{example}{section}
\numberwithin{theorem}{section}
\numberwithin{lemma}{section}
\numberwithin{corollary}{section}
\numberwithin{prop}{section}
\numberwithin{definition}{section}
\numberwithin{remark}{section}

\usepackage{tikz}

\startlocaldefs

\newcommand{\eat}[1]{}

\usepackage{enumitem}
\DeclareMathOperator*{\argmin}{\arg\!\min}

\newcommand{\mZ}{\mathcal{Z}}
\newcommand{\bJ}{\bm{J}}
\newcommand{\bZ}{\bm{Z}}

\renewcommand{\hat}[1]{\widehat{#1}}
\renewcommand{\tilde}[1]{\widetilde{#1}}
\newcommand{\E}{\mathbb{E}}
\renewcommand{\P}{\mathbb{P}}
\newcommand{\mX}{{\mathcal{X}}}
\newcommand{\bKe}{\mathsf{K}}
\newcommand{\cF}{\boldsymbol{\mathcal{F}}}

\newcommand{\bzr}{\boldsymbol{0}}
\newcommand{\mY}{{\mathcal{Y}}}
\newcommand{\bX}{\bm{X}}
\newcommand{\bY}{\bm{Y}}

\newcommand{\R}{\mathbb{R}} 
\newcommand{\Serd}{\Sigma_{\mathrm{ERD}}}

\newcommand{\be}{\bm{e}}
\newcommand{\bG}{\bm{G}}

\newcommand{\mz}{\bm{Z}}

\newcommand{\mH}{\mathcal{H}}

\newcommand{\Z}{\mathcal{Z}}

\newcommand{\bW}{\bm{W}}

\newcommand{\grsc}{\gamma_{m,n}^{\nu,\bJ}}
\newcommand{\grsr}{\gamma_{m,n}^{\nu,\bJ,\mathrm{or}}}
\newcommand{\trsr}{\gamma_{m,n}^{\mathrm{or}}}

\newcommand{\bD}{\bm \Delta}

\newcommand{\by}{\bm{y}}
\newcommand{\rh}{T^{\nu}_{m,n}}
\newcommand{\rhsc}{T^{\nu,\bJ}_{m,n}}
\newcommand{\prsc}{\phi^{\nu,\bJ}_{m,n}}
\newcommand{\prnk}{\phi^{\nu,\bJ}_{n}}
\newcommand{\ptsc}{\tilde{\phi}^{\nu,\bJ}_{m,n}}
\newcommand{\Rb}{\bm{R}_{\mathrm{H}_{\ell}}^{\nu}} 
\newcommand{\hbG}{\hat{\bm{S}}}
\newcommand{\bGa}{\bm{S}}
\newcommand{\ind}{\bm{1}}

\newcommand{\Pacd}{\mathcal{P}_{\mathrm{ac}}(\mathbb{R}^d)}
\newcommand{\Don}{\Delta^{(1)}}
\newcommand{\Dwo}{\Delta^{(2)}}
\newcommand{\Donor}{\Delta^{(1),\mathrm{\mathrm{or}}}}
\newcommand{\Dwor}{\Delta^{(2),\mathrm{\mathrm{or}}}}
\newcommand{\rhsr}{T^{\mathrm{rank},\mathrm{or}}_{m,n,\mathrm{sc}}} 
\newcommand{\bR}{\bm{R}}
\newcommand{\hbR}{\hat{\bm{R}}}
\newcommand{\Rl}{\bm{R}_{\mathrm{H}_1}^{\nu}}
\newcommand{\Rmu}{\bm{R}_{\mathrm{H}_0}^{\nu}}
\newcommand{\Q}{\bm{Q}}
\newcommand{\mF}{\bm{F}}

\@addtoreset{proofpart}{theorem}

\makeatletter
\newcommand*{\rom}[1]{\expandafter\@slowromancap\romannumeral #1@}
\makeatother

\def\argmin{\mathop{\rm argmin}}

\renewcommand{\hat}[1]{\widehat{#1}}
\renewcommand{\tilde}[1]{\widetilde{#1}}
\renewcommand{\P}{\mathbb{P}}
\newcommand{\Th}{\Theta}
\newcommand{\bt}{\boldsymbol{\theta}}
\newcommand{\bh}{\bm{h}}

\newcommand{\bhg}{\hat{\bm{S}} }
\newcommand{\pg}{\bm{S}} 

\newcommand{\tP}{\tilde{\Psi}}

\newcommand{\pl}{u_{\mathrm{H}_1}^{\nu}}
\newcommand{\nhg}{\mathsf{K}}

\newcommand{\bz}{\bm{z}}

\newcommand{\fnd}{\mathcal{F}_{\mathrm{ind}}}
\newcommand{\fel}{\mathcal{F}_{\mathrm{ell}}}
\newcommand{\fgen}{\mathcal{F}_{\mathrm{gen}}}

\newcommand{\Ezn}{\E_{\mathcal{Z}_n}}

\newcommand{\Rdcov}{\mathrm{RdCov}}

\newcommand{\hG}{\hat{\mathsf{K}}}

\newcommand{\ron}{\bar{\bJ_{1}}}
\newcommand{\rtw}{\bar{\bJ_2}}

\newcommand{\bo}{\boldsymbol{1}}

\newcommand{\pho}{\mathbb{P}_{\mathrm{H}_{1}}}
\newcommand{\zph}{\mathbb{P}_{\mathrm{H}_0}}

\newcommand{\ezph}{\mathbb{E}_{\mathrm{H}_0}}
\newcommand{\bH}{\bm{H}}

\newcommand{\mx}{\bm{x}}
\newcommand{\atr}{\mathrm{ARE}(T^{\nu,\bJ},T)}
\newcommand{\atrt}{\mathrm{ARE}(T^{\nu},T)}
\newcommand{\ati}{\mathrm{ARE}(R^{\nu,\bJ},R)}
\newcommand{\hds}{\mathrm{HD}^2(\P^{(N)},\mathbb{Q}^{(N)})}

\newcommand{\tv}{\mathrm{TV}(\P^{(N)},\mathbb{Q}^{(N)})}
\newcommand{\hd}{\mathrm{HD}(\P^{(N)},\mathbb{Q}^{(N)})}

\newcommand{\obx}{\bar{\bX}}
\newcommand{\oby}{\bar{\bY}}

\newcommand{\sn}{\Serd^{(1)}}
\newcommand{\st}{\Serd^{(2)}}
\newcommand{\rhnk}{R_{n}^{\nu,\bJ}}
\newcommand{\ornk}{R_{n}^{\nu,\bJ,\mathrm{or}}}
\newcommand{\jro}{\bar{\bJ_1}^{\mathrm{or}}}
\newcommand{\jrt}{\bar{\bJ_2}^{\mathrm{or}}}

\newcommand{\bsg}{\bm{\gamma}}

\newcommand{\my}{\bm{y}}
\newcommand{\myt}{\tilde{\bm{y}}}

\newcommand{\grun}{\gamma_{m,n}^{\mathrm{or,U}}}
\newcommand{\grln}{\gamma_{m,n,L}^{\mathrm{or,U}}}
\newcommand{\grnn}{\gamma_{m,n,L_N}^{\mathrm{or,U}}}

\allowdisplaybreaks

\newcommand{\bd}{\bm{d}}

\definecolor{LightCyan}{rgb}{0.88,1,1}
\definecolor{Gray}{gray}{0.9}
\definecolor{LightCyan}{rgb}{0.88,1,1}
\definecolor{Gray}{gray}{0.9}

\endlocaldefs

\usepackage{abstract}
\usepackage{xr}
\usepackage[numbers]{natbib}
\begin{document}
	
	\renewcommand{\abstractname}{}    
	\renewcommand{\absnamepos}{empty}
	
	\begin{frontmatter}
		\title{Pitman Efficiency Lower Bounds for Multivariate Distribution-Free Tests Based on Optimal Transport}
		
		\runtitle{Efficiency of Distribution-Free Hotelling-Type Tests}		
		
		\begin{aug}
			\author{\fnms{Nabarun} 
				\snm{Deb}, \ead[label=e1]{nd2560@columbia.edu}}\author{\fnms{Bhaswar B.} \snm{Bhattacharya}\thanksref{t1},
				\ead[label=e2]{bhaswar@wharton.upenn.edu}}
			\and
			\author{\fnms{Bodhisattva} \snm{Sen}\thanksref{t2}\ead[label=e3]{bodhi@stat.columbia.edu}}
			\affiliation{
				Columbia University\thanksmark{a1},
				University of Pennsylvania\thanksmark{a2}, and
				Columbia University\thanksmark{a3}
			}
			\thankstext{t1}{Supported by NSF CAREER grant DMS-2046393 and a Sloan Research Fellowship.}
			\thankstext{t2}{Supported by NSF grant DMS-2015376.}
			
			\runauthor{Deb, Bhattacharya, and Sen}
			
			\address{1255 Amsterdam Avenue \\
				New York, NY 10027\\
				\printead{e1} \\
			}
			\address{265 South 37th Street \\
				Philadelphia, PA 19104 \\
				\printead{e2} \\			
			}
			\address{1255 Amsterdam Avenue \\
				New York, NY 10027\\
				\printead{e3} \\
			}
		\end{aug}
		\vspace{0.2in}
		\begin{abstract}
			Distribution-free tests such as the Wilcoxon rank sum test are popular for testing the equality of two univariate 
			distributions. Among the important reasons for their
			popularity are the striking results of Hodges-Lehmann
			(1956) and Chernoff-Savage (1958), where the authors show that the asymptotic (Pitman) relative efficiency of Wilcoxon's test with respect to
			Student's $t$-test, under location-shift alternatives, never falls below
			$0.864$ (with the identity score) and $1$ (with the Gaussian score) respectively, despite the former being exactly distribution-free for all
			sample sizes. Motivated by these results, we propose and study a large family of exactly distribution-free multivariate rank-based two-sample tests by leveraging the theory of optimal transport. First, we propose distribution-free analogs of the Hotelling $T^2$
			test (the natural multidimensional counterpart of Student's $t$-test) and
			show that they satisfy Hodges-Lehmann and Chernoff-Savage-type efficiency lower
			bounds over natural sub-families of multivariate distributions, despite being entirely agnostic to the underlying data
			generating mechanism --- making them the first
			multivariate, nonparametric, exactly distribution-free
			tests that provably achieve such efficiency lower bounds. As these tests are derived from Hotelling $T^2$, naturally they are not universally consistent (same as Wilcoxon's test). To overcome this, we propose exactly distribution-free versions of the celebrated kernel maximum mean discrepancy test and the energy test. These tests are indeed universally consistent under no moment assumptions, exactly distribution-free for all sample sizes, and have non-trivial Pitman efficiency. We believe this trifecta of properties hasn't yet been proven for any existing test in the literature. Through extensive simulations, we demonstrate the favorable finite sample performance of our procedures, including robustness to outliers and contamination, compared to existing tests, in both low and high-dimensional settings. Finally, we demonstrate the broader scope of our methods by constructing
			exactly distribution-free multivariate tests for mutual independence using optimal transport,
			which suffer from no loss in asymptotic efficiency against the
			classical Wilks' likelihood ratio test.
		\end{abstract}

		
		\begin{keyword}
			\kwd{Asymptotic relative efficiency}
			\kwd{convergence of optimal transport maps}
			\kwd{elliptically symmetric distributions} 
			\kwd{Konijn alternatives}
			\kwd{local contiguous alternatives}
			\kwd{score functions}
		\end{keyword}
	\end{frontmatter}

	\section{Introduction}
	Given two probability measures $\mu_1$ and $\mu_2$ on $\mathbb{R}^d$, $d\geq 1$, and $\bX_1,\ldots ,\bX_m\overset{i.i.d.}{\sim}\mu_1$ and $\bY_1,\ldots ,\bY_n\overset{i.i.d.}{\sim}\mu_2$, our goal is to test the hypothesis
	\begin{equation}\label{eq:twosam}
		\mathrm{H}_0:\mu_1=\mu_2  \qquad \mathrm{versus} \qquad \mathrm{H}_1:\mu_1\neq\mu_2.
	\end{equation}
	This is the {\it two-sample equality of distributions} testing problem which has been studied extensively over the last hundred years (see, for example,~\cite{Bickel1965,hollander2013nonparametric,Jureckova2012,Oja2010,student1908probable, Thas2010,weiss1960two} and the references therein). It has numerous applications, such as  pharmaceutical studies~\cite{farris1999between,rudolph2018joint}, causal inference~\cite{folkes1987field,goldman2018comparing}, remote sensing~\cite{conradsen2003test,liu2018arbitrary}, econometrics~\cite{mayer1975selecting,schepsmeier2019goodness}, among others. When $d=1$, some popular testing procedures for problem~\eqref{eq:twosam} include Student's $t$-test~\cite{student1908probable}, Kolmogorov-Smirnov test~\cite{kolmogorov1933,Smirnov1948}, Cram\'{e}r-von Mises test~\cite{AndersonCVM1962}, Wald-Wolfowitz runs test~\cite{wald1940}, and  Wilcoxon rank sum test/Mann-Whitney $U$-test~\cite{mann1947,Wilcoxon1947}. When $d>1$, a plethora of procedures have been proposed for testing~\eqref{eq:twosam} that include the Hotelling $T^2$ test~\cite{hotelling1992generalization}, multivariate Cram\'{e}r-von Mises test~\cite{cotterill1982limiting}, tests based on geometric graphs~\cite{chen2017,friedman1979,bbbm2019,rosenbaum2005,schilling1986}, data depth-based tests~\cite{Liu1990,liu1993}, kernel maximum mean discrepancy (MMD)/energy distance tests~\cite{Baringhaus2004,Gretton2012,gretton2008,Szekely2013}, among others.
	
	In the case when $d=1$, a special feature of some of the tests cited above such as the Wilcoxon rank sum (among others) is that it is exactly {\it distribution-free} under the null (as long as $\mu_1$ is absolutely continuous), that is, under $\mathrm{H}_0$, the distribution of the Wilcoxon rank sum test statistic is free of the underlying (unknown) data generating distributions, for all sample sizes. Thus, the critical value of this test is universal and one does not need to resort to permutation distributions/asymptotic approximations to carry out the test; thereby yielding \emph{uniform} level $\alpha$ tests. This rather attractive property arises from the fact that the Wilcoxon rank sum test is based on the ranks of the pooled sample $\mX_m\cup\mY_n$, where $\mX_m:=\{\bX_1,\ldots ,\bX_m\}$ and $\mY_n:=\{\bY_1,\ldots ,\bY_n\}$, instead of the exact values of the individual observations.

	In the $d$-dimensional Euclidean space, for $d \ge 2$, due to the absence of a canonical ordering, the existing extensions of concepts of ranks, such as component-wise ranks~\cite{Bickel1965,Puri1965}, spatial ranks~\cite{Chaudhuri1996,Marden1999}, depth-based ranks~\cite{liu1993,Zuo2000} and Mahalanobis ranks and interdirections~\cite{hallin2002optimal,Peters1990,Randles1989}, and the corresponding rank-based tests no longer possess exact distribution-freeness. This raises a fundamental question: {\it How do we define multivariate ranks that can lead to distribution-free testing procedures?} A major breakthrough in this regard was made very recently in the pioneering work of Marc Hallin and co-authors (\cite{delbarrio2019}) where they propose a notion of multivariate ranks, using optimal  transport, that 
	possesses many of the desirable properties present in their one-dimensional counterparts. 
	Recently, using this notion of multivariate ranks (defined via optimal transport),~\cite{Deb19} proposed a general framework for multivariate distribution-free rank-based testing for problem~\eqref{eq:twosam}. \par 
	
	However, distribution-freeness by itself is \emph{not enough}. For $d=1$, another important reason for the popularity of rank-based tests is their validity under weak assumptions and their high efficiency. To quote from E.L. Lehmann's book~\cite{lehmann1975nonparametrics} ``... it turned out, rather surprisingly, that the efficiency of the Wilcoxon tests and other nonparametric procedures hold up quite well under the classical assumption of normality and that these procedures may have considerable advantages in efficiency (as well as validity) when the assumption of normality is not satisfied". Two striking results in this direction are: 
	\begin{enumerate} 
		\item[$(1)$] the celebrated ``0.864 result" by Hodges and Lehmann~\cite{Hodges1956} 
		and 
		\item[$(2)$] the efficiency of the Gaussian (van der Waerden) score (see~\eqref{eq:vanscore}) transformed Wilcoxon's test by Chernoff and Savage~\cite[Theorem 3]{Chernoff1958}. 
	\end{enumerate} 
	In~\cite{Hodges1956} the authors showed that the ARE of the Wilcoxon rank sum test relative to Student's $t$-test could \emph{never fall below $108/125 \approx 0.864$} (when working with contiguous location-shift  alternatives); whereas the efficiency can be arbitrarily large (tends to $+\infty$) for heavy tailed distributions. In~\cite{Chernoff1958}, the authors showed that surprisingly the ARE of the Gaussian score transformed Wilcoxon's test, relative to Student's $t$-test, \emph{never falls below $1$}. This result shows in particular that there are univariate nonparametric rank-based tests which, in addition to being more robust and consistent beyond location-shift alternatives, unlike the $t$-test, actually do not suffer from any loss in asymptotic efficiency in the Pitman sense relative to the $t$-test. At this point, it is tempting to ask: 
	
	\begin{changemargin}{0.5cm}{0.5cm} 
		\emph{Can we design multivariate nonparametric distribution-free tests for problem~\eqref{eq:twosam} that enjoy similar AREs when compared to the classical Hotelling $T^2$ test (the natural multivariate counterpart of Student's $t$-test)?}
	\end{changemargin}
	
	In this paper we answer the above question in the affirmative. We spell out a general principle for constructing multivariate distribution-free tests.  We then define a class of distribution-free analogs of the classical Hotelling $T^2$ test~\cite{hotelling1992generalization}, using the aforementioned notion of multivariate ranks based on optimal transport. These tests reduce to the Wilcoxon rank sum test, with appropriate score functions, when $d=1$. We obtain numerous ARE lower bounds of our proposed  tests with respect to the Hotelling $T^2$ for multiple sub-families of multivariate distributions. Among other things, we prove a multivariate analog of the ``$0.864$" ARE lower bound result in Hodges-Lehmann~\cite{Hodges1956}. Our \emph{most interesting} observation here is that it is possible to get exactly distribution-free nonparametric tests for  problem~\eqref{eq:twosam}, when $d\geq 1$, that suffer \emph{no loss in ARE} compared to the Hotelling $T^2$ test across multiple sub-families of multivariate distributions, despite being completely agnostic to the underlying sub-family. This can be viewed as a direct multivariate analog of the ARE lower bound in Chernoff-Savage~\cite{Chernoff1958}. To the best of our knowledge, this is the first time that such lower bounds on the ARE are being established for multivariate rank tests based on optimal transport.
	
	The class of tests described above, being distribution-free analogs of Hotelling $T^2$ are naturally not universally consistent (like Wilcoxon's test in $d=1$) against fixed alternatives. This brings us to our second main contribution of the paper. In particular, we study distribution-free analogs of the popular kernel maximum mean discrepancy (MMD);  see~\cite{Gretton2012}, and the energy statistic; see~\cite{Szekely2013}. For $d=1$, a member of this class of tests has been shown to be equivalent to the Cram\'{e}r-von Mises test~\cite{AndersonCVM1962}; see~\cite[Lemma 2.2]{Deb19}. We prove three key desirable properties of these tests: (a) finite sample distribution-freeness and uniform type I error control, (b) universal consistency against all fixed alternatives, and (c) non-trivial ARE versus non-distribution-free counterparts. To our understanding, no other existing test in the literature, is known to satisfy the above trifecta of properties. Extensive numerical experiments bear out our theoretical results and also show favorable finite sample performance of the proposed tests against existing tests in the literature, for many natural classes of alternatives, in both low and high-dimensional problems.

	\subsection{Summary of our Contributions}\label{sec:Rank-Hotelling}

	In~\cref{sec:multi_ranks} below, we recall the notion of multivariate ranks based on optimal transport. We then use it to state a general principle for constructing distribution-free tests, with examples, and describe our main results related to problem~\eqref{eq:twosam}; see~\cref{sec:hotelling_dist_free}. The broader impact of our results in another testing problem, namely that of testing for mutual independence, is briefly discussed in~\cref{sec:bscope}.
	
	\subsubsection{Multivariate Ranks Defined via Optimal Transport} 
	\label{sec:multi_ranks}
	
	Set $N:=m+n$ and denote by $\mZ_N:=\mX_m\cup\mY_n$ the pooled sample, which we enumerate as $\mZ_N=\{\bZ_1,\ldots ,\bZ_N\}$. Let $\mH_N^d:=\{\bh_1^d,\ldots ,\bh_N^d\}$ denote the set of multivariate ranks --- a set of $N$ fixed vectors in $\R^d$ that can be thought of as a ``natural'' discretization of a prespecified probability distribution $\nu$ on $\R^d$,  that is, we assume that the empirical measure on $\mH_N^d$ converges weakly to $\nu$. For example, when $d=1$, $\mH_N^d$ is usually chosen as $\{1/N,\ldots, N/N\}$ and the empirical measure on $\mH_N^1$ converges weakly to $\mathrm{Unif}[0,1]$, the uniform distribution on $[0,1]$. Hereafter, $\nu$ will be called the \emph{reference distribution}. Next, let $S_N$ be the set of all permutations of $[N]:=\{1,2,\ldots ,N\}$ and consider the following optimization problem:
	\begin{equation}\label{eq:empopt}
		\hat{\sigma}:=\argmin_{\sigma=(\sigma_1,\ldots ,\sigma_N)\in S_N} \sum_{i=1}^N \lVert \bZ_i-\bh_{\sigma_i}^d\rVert^2.
	\end{equation}
	Define the pooled {\it multivariate ranks} as
	\begin{equation}\label{eq:defemprank}
		\hat{\bR}_{m,n}(\bZ_i):=\bh_{\hat{\sigma}_i}^d.
	\end{equation}
	Here $\lVert \cdot\rVert$ denotes the standard Euclidean norm. The optimization problem in~\eqref{eq:empopt} can be viewed as an example of the \textit{assignment problem}, which in turn can be solved using a linear program, for which algorithms of worst case time complexity $O(N^3)$ are available in the literature (see~\cite{bertsekas1988,Edmonds1970,jonker1987,munkres1957}). We also refer the interested reader to~\cite{Gabow1989,Agarwal2014,Sharathkumar2012} and the references therein, for a review of faster approximate algorithms addressing~\eqref{eq:empopt}; also see~\cref{sec:compasgn}. To get a better intuition for~\eqref{eq:empopt}~and~\eqref{eq:defemprank}, when $\mH_N^1=\{1/N, 2/N, \ldots, N/N \}$ for $d=1$, then~\eqref{eq:empopt}~and~\eqref{eq:defemprank} reduce to the standard univariate ranks by an application of the rearrangement inequality (see~\cite[Theorem 368]{Hardy1952}). In fact, even in multidimension, these empirical ranks preserve a notion of direction, in the sense that the extreme data points get mapped to the corresponding extreme points of the fixed grid $\mH_N^d$ (see e.g.,~\cite[Theorem 5.1]{figalli2013holder}~and~\cite[Figure 1]{Deb19}). \medskip

	\subsubsection{Distribution-Free Tests Based on Multivariate Ranks} 
	\label{sec:hotelling_dist_free}
	Based on~\cite[Proposition 2.5]{delbarrio2019}~and~\cite[Proposition 2.2]{Deb19}, it is easy to conclude that if $\mu_1=\mu_2$ is continuous, then $(\hat{\bR}_{m,n}(\bZ_1),\ldots ,\hat{\bR}_{m,n}(\bZ_N))$ are uniformly distributed over all permutations of the set $\mH_N^d$. This observation leads to a general principle for constructing distribution-free tests as follows. Take any test statistic for problem~\eqref{eq:twosam} which is \emph{not distribution-free}, say
	$T_{m,n}\equiv T_{m,n}(\bX_1,\ldots ,\bX_m,\bY_1,\ldots ,\bY_n).$
	We can then construct its multivariate rank-based \emph{distribution-free} version, by simply constructing a new statistic with the $\bX_i$'s and $\bY_j$'s replaced by $\hat{\bR}_{m,n}(\bX_i)$'s and $\hat{\bR}_{m,n}(\bY_j)$'s, i.e.,
	\begin{equation}
		\label{eq:genprin}T_{m,n}^{\nu}\equiv T_{m,n}(\hat{\bR}_{m,n}(\bX_1),\ldots ,\hat{\bR}_{m,n}(\bX_m),\hat{\bR}_{m,n}(\bY_1),\ldots ,\hat{\bR}_{m,n}(\bY_n)).
	\end{equation}
	As is evident, this is a very general strategy to construct a broad variety of distribution-free testing procedures. We focus on two examples of this approach in this paper.\medskip 
	
	\emph{\bf Rank Hotelling $T^2$}: One of the most celebrated and useful multivariate two-sample tests is the Hotelling $T^2$ statistic~\cite{hotelling1992generalization} (the multivariate analog of Student's $t$-test), and is given by:
	\begin{equation}\label{eq:hotelling}
		T_{m,n}:=\frac{mn}{m+n}\left(\bar{\bX}-\bar{\bY}\right)^{\top}S_{m,n}^{-1}\left(\bar{\bX}-\bar{\bY}\right),
	\end{equation}
	where $\bar{\bX}:=\frac{1}{m} \sum_{i=1}^m \bX_i$, $\bar{\bY}:= \frac{1}{n} \sum_{j=1}^n \bY_j$, and $S_{m,n}$ is the usual pooled sample covariance matrix. 
	We reject $\mathrm{H}_0$ in~\eqref{eq:twosam} when $T_{m,n}$ exceeds the $(1-\alpha)$-th quantile of the $\chi^2_d$ distribution. 
	Although the Hotelling $T^2$ statistic is quite classical, it still remains popular due to its simplicity and has found many statistical applications in the recent years (see e.g.,~\cite{brereton2016hotelling,herve2018multivariate,zhao2018generalized}). Using the general principle outlined above, the multivariate rank-version of $T_{m,n}$ in~\eqref{eq:hotelling} can be shown to be equivalent to the following:
	\begin{equation}\label{eq:rankhotelling}
		\rh:=\frac{mn}{m+n}\Bigg\lVert \frac{1}{m}\sum_{i=1}^m \hat{\bR}_{m,n}(\bX_i)-\frac{1}{n}\sum_{j=1}^n \hat{\bR}_{m,n}(\bY_j)\Bigg\rVert^2.
	\end{equation}
	Note that when  $d=1$, the statistic \eqref{eq:rankhotelling} is equivalent to the two-sided Wilcoxon rank sum statistic (since the sum of the pooled sample ranks is a constant). Hence, $\rh$ can also be thought of as a multivariate analog of the celebrated Wilcoxon rank sum test. Also note that, unlike the Hotelling $T^2$ statistic, the construction of $\rh$ above does not require any covariance matrix estimation. In Section \ref{sec:ranktsq} we will consider a class of test statistics which generalizes \eqref{eq:rankhotelling} by incorporating score functions (see~\cite[Chapter 13]{Van1998}). \medskip
	
	\emph{\bf Rank kernel MMD}: As mentioned before, the rank Hotelling test statisticin~\eqref{eq:rankhotelling}, despite its attractive properties, does not guarantee universal consistency. The main reason is that the Hotelling $T^2$ statistic in~\eqref{eq:hotelling} is itself not consistent beyond location-shift type alternatives. This leads us to considering a new class of distribution-free tests. Here we choose $T_{m,n}$ to be the unbiased kernel MMD statistic which has attracted a lot of attention in machine learning over the years (see~\cite{Gretton2012,Sejdinovic2013,gretton2009fast}). By leveraging the general principle outlined above (see~\eqref{eq:genprin}), we construct the multivariate rank version of kernel MMD, which is now both distribution-free in finite samples and has universal consistency (see~\cref{sec:rankmmd} for details).
	
	\subsubsection{Main results}\label{sec:oures}
	With the above classes of distribution-free tests in mind, let us briefly outline our main results.
	
	\begin{itemize}
		\item \textit{Convergence of empirical ranks}: In~\cref{theo:rankmapcon}, we prove a new convergence result for functionals of the empirical ranks $\hat{\bR}_{m,n}(\bX_i)$'s and $\hat{\bR}_{m,n}(\bY_j)$'s to the appropriate population versions, which is of independent interest. This result yields convergence of all the test statistics considered in this paper, both under the null, and the alternative. In fact, it also helps address an open problem in~\cite{shi2020rate}; see~\cref{sec:bscope} for more details.
		\item \textit{General reference distributions and scores}: Thanks to the Chernoff-Savage~\cite{Chernoff1958} paper, it has become a staple in nonparametric rank-based inference to incorporate score functions (see~\cite[Chapter 13]{Van1998}) to enhance power properties. In~\cref{sec:ranktsq} we follow this same principle and further generalize~\eqref{eq:rankhotelling} by replacing $\hbR_{m,n}(\bX_i)$ and $\hbR_{m,n}(\bY_j)$ with $\bJ(\hbR_{m,n}(\bX_i))$ and $\bJ(\hbR_{m,n}(\bY_j))$ for some continuous and invertible score function $\bJ(\cdot)$ taking values in $\R^d$ (see~\eqref{eq:rankschotelling}). This comes in addition to a flexible choice of reference distribution $\nu$. We prove that this generalization leads to tests with better ARE properties; see Theorems~\ref{prop:areind}~and~\ref{prop:areell}, \cref{rem:Gausscore}, and~\cref{sec:sim}. 	
		\item \textit{Distribution-freeness}: We prove that both rank Hotelling $T^2$ and rank kernel MMD are distribution-free (see Propositions \ref{prop:dfree} and \ref{prop:dfreeker} respectively) and consequently yield uniformly level $\alpha$ tests (see \eqref{eq:uniflevelt} and \eqref{eq:uniflevelker} respectively).
		\item \textit{Consistency}: The rank Hotelling $T^2$ test is shown to be consistent against a large class of alternatives (see~\cref{theo:rhconsis}), which contains the class of location-shift alternatives for $d\geq 1$ (see~\cref{prop:conloc}) and contamination alternatives (see~\cref{prop:conlocontam} in the Appendix). In fact, numerical experiments (see~\cref{sec:beyondloc})  show that this class also contains some scale families where the standard Hotelling $T^2$ is clearly inconsistent. On the other hand, we prove the \emph{universal consistency} of the rank kernel MMD (see~\cref{theo:rhconsis}).
		\item \textit{Asymptotic distribution under null and contiguous alternatives}: We obtain the asymptotic null distributions of rank Hotelling $T^2$ and rank kernel MMD (see Theorems \ref{theo:nullrankhotelling} and \ref{theo:nullker}) both of which are naturally free of the data distribution. These limits allow the practitioner to use the corresponding asymptotic level $\alpha$ cutoffs without observing the data, for larger sample sizes. We also obtain asymptotic distributions under contiguous alternatives in Theorems \ref{theo:locpowermain1} and \ref{theo:Piteffrank} which help us draw conclusions about the relevant AREs which we discuss next.
		\item \textit{Asymptotic relative efficiency}: Given two level $\alpha$  tests $T_1$ and $T_2$, the ARE can informally be described as follows (see~\cref{def:asympeff} in the Appendix for a formal definition): 
		\begin{changemargin}{0.5cm}{0.5cm} 
			\textit{The ARE of $T_1$ relative to $T_2$ is the limiting ratio of the number of samples needed to attain a power of $\beta\in (\alpha,1)$ when using the test $T_2$ compared to the same for test $T_1$, where the limit is taken as $\mu_2$ ``converges" to $\mu_1$.}
		\end{changemargin} 
		\noindent For example, if the ARE of $T_1$ with respect to $T_2$ is $0.9$, intuitively it means that $T_2$ takes $10\%$ fewer samples than $T_1$ to attain the power level $\beta$. This yields a simple, interpretable comparison between two tests and is widely used for comparisons in the hypothesis testing literature (see \cite{bbb2019,Chernoff1958,Chikkagoudar2014,Hallin2002,hallin2002optimal,Hodges1956,kim2018robust,hallin1994pitman,hallin2000efficiency} and the references therein). 
		\begin{itemize}
			\item[(a)] Rank Hotelling $T^2$: We obtain bounds/expressions for the ARE of rank Hotelling $T^2$ (in~\eqref{eq:rankhotelling}) versus the usual Hotelling $T^2$ (see \cref{prop:Gaussare}, Theorems \ref{prop:areind}, \ref{prop:areell}, and \ref{prop:areica}; the last one is in the Appendix), for different ``effective" reference distributions/ERDs (see~\cref{def:erd}) which takes into account both the score function $\bJ(\cdot)$ and the reference distribution $\nu$. With a notational abuse, let us denote this efficiency by $\atrt$ temporarily. We give some highlights here. Let $\fnd$ denote the ``smooth" location family of product measures. In~\cref{prop:areind}, we show that, by using $\mathrm{Unif}[0,1]^d$ ERD, $\atrt$ satisfies $\inf_{\fnd}\atrt= 0.864$ (same as the univariate Hodges-Lehmann~\cite{Hodges1956} result) whereas with a standard Gaussian ERD, it satisfies $\inf_{\fnd}\atrt= 1$ (same as the univariate Chernoff-Savage~\cite{Chernoff1958} result). Note that both these lower bounds neither depend on the dimension $d$, nor on the level $\alpha$ or the power $\beta$. The lower bound of $1$ with Gaussian ERD is naturally attractive and we prove in Theorems~\ref{prop:areell} and \ref{prop:areica} (see the Appendix) that the same lower bound of $1$ holds if $\fnd$ is replaced by two other popular families: (a) elliptically symmetric distributions (see~\cite{Chmielewski1981,frahm2004generalized}) and (b) the model for blind source separation (see~\cite{Samarov2004,shlens2014tutorial}); see~\cref{sec:bsep}. It should be noted that none of these efficiency lower bound results or the consistency results mentioned above require any moment assumptions on the data generating distribution. Consequently $\atrt$ can be arbitrarily large by choosing heavy-tailed data distributions (also see~\cite[Page 4]{Hodges1956}). These observations provide strong theoretical evidence in favor of using the Gaussian ERD.
			\item[(b)] Rank kernel MMD: The ARE of rank kernel MMD against usual kernel MMD actually depends on the level $\alpha$ and the power $\beta$, which makes it difficult to obtain such succinct lower bounds as above. Nevertheless~\cref{theo:Piteffrank} shows that rank kernel MMD has a non-trivial ARE under local alternatives (also see~\eqref{eq:powk}). This property is itself unique because no other test in the literature has universal consistency, exact distribution-freeness, and a non-trivial efficiency (see~\cref{sec:compot}).
		\end{itemize}
		\item \textit{Simulation experiments}: We carry out extensive simulations in~\cref{sec:sim}. Two important and recurrent observations should be noted: (a) the use of Gaussian ERD consistently leads to more powerful tests compared to non distribution-free counterparts, thereby providing empirical evidence supporting our theoretical findings outlined above; (b) in heavy-tailed settings, rank-based tests generally are more robust and generally outperform existing tests in the literature. In addition, we found, somewhat surprisingly, that in high-dimensional problems too, the proposed rank-based tests are very competitive with popular two-sample tests (see~\cref{sec:highdpower}). 
	\end{itemize}	
	
	\subsubsection{Broader Scope}\label{sec:bscope}
	
	Our general strategy in~\eqref{eq:genprin} is not just confined to the two-sample problem in~\eqref{eq:twosam}, but is useful in other multivariate nonparametric testing problems for obtaining distribution-free procedures. We illustrate this in~\cref{sec:indtest} where we construct a class of multivariate analogs of the classical Spearman's rank correlation (see~\cite{spearman1904proof}) and use it for testing multivariate independence. In~\cref{prop:indepeff} in the Appendix, we show that, once again, under the Gaussian reference distribution, our proposed test of independence suffers from no loss in efficiency over different classes of multivariate probability distributions compared to the Wilks' likelihood ratio test~\cite{wilks1935independence} (the natural multivariate analog of Pearson's correlation~\cite{pearson1920notes}); see~\cite{shi2021center,shi2020rate} for related results. 
	
	In~\cref{sec:comparelit}, we demonstrate that our techniques are useful towards proving consistencies of other nonparametric testing procedures based on optimal transport. One of the main technical tools in this regard is~\cref{theo:rankmapcon}, which we use to answer an open question regarding the consistency of the Gaussian score transformed rank distance covariance test for independence, as laid out in~\cite{shi2020rate} (see~\cref{prop:rdcovcon} in the Appendix). In fact, the same theorem can be used to prove consistencies of other tests in~\cite{Deb19,hallin2020fully,hallin2020center,shi2020distribution}. Very recently, a statistic closely related to \eqref{eq:rankhotelling} was presented in passing in~\cite[Page 25]{hallin2020fully} for the special case when the reference distribution is spherical uniform and for a specific choice of the set $\{\bh_1^d,\ldots ,\bh_N^d\}$. However, none of its theoretical properties, pertaining to consistency or ARE, were derived. In Section \ref{sec:comparelit} we show that our results readily imply its consistency and ARE lower bounds (see~\cref{prop:hallinlb} in the Appendix for details).
	
	\subsection{Related Work}\label{sec:litrev}
	
	Nonparametric multivariate two-sample tests based on rank and data-depth based methods have mostly been restricted to testing against location-scale alternatives~\cite{choi1997approach,hettmansperger1998,randles1990,motto1995,oja1999affine}. Asymptotically distribution-free depth-based tests which are consistent if restricted to the above class of alternatives are discussed in~\cite{liu2010versatile,rousson2002}. Multivariate generalizations of the Wilcoxon rank sum test based on data depth were studied in~\cite{Liu1990,liu1993,zuo2006limiting}, which are also asymptotically distribution-free. However, these tests are not exactly distribution-free in finite samples and are difficult to compute  when  the dimension is large because computation of depth-functions generally require time that scales exponentially with the dimension. An alternative route for testing is through geometric graphs. This includes the celebrated Friedman-Rafsky test based on the minimum spanning tree (MST) \cite{friedman1979}, the tests based on nearest-neighbor graphs~\cite{henze1988,schilling1986,chen2017}, and Rosenbaum's cross-match test \cite{rosenbaum2005} based on minimum non-bipartite matching. These tests are asymptotically distribution-free (apart from the cross-match test, which is exactly distribution-free in finite samples), computationally feasible, and universally consistent, but have no power against $O(1/\sqrt{N})$  alternatives, that is, they have zero asymptotic Pitman efficiency~\cite{bbb2019}. Another approach is to compare pairwise-distances between and within the samples. This is the celebrated energy distance test~\cite{szekely2003statistics,Szekely2013,szekely2004testing,Baringhaus2004} which is a special case of the kernel MMD~\cite{Gretton2012,gretton2009fast}. These tests are consistent against general fixed alternatives, have non-trivial power against local contiguous alternatives, but are not distribution-free, even asymptotically. We will carry out elaborate numerical comparisons with these tests in~\cref{sec:sim}; also see~\cref{sec:compot}. Finally, a new line of work which uses optimal transport based methods for distribution-free testing has found applications in the two-sample testing problem~\cite{boeckel2018multivariate,Deb19}, independence testing problem~\cite{Deb19,shi2020distribution,shi2020rate}, linear regression~\cite{hallin2020fully,hallin2020center}), etc. We will compare our work with these other papers in~\cref{sec:compot}.

	\subsection{Organization}\label{sec:orgz}
	
	The rest of the paper is organized as follows: In Section \ref{sec:bg} we provide some background on population multivariate ranks defined via optimal transport and present a general convergence result. The family of rank Hotelling $T^2$ test statistics and its asymptotic properties (consistency, null distribution, power against local alternatives, and ARE computations) are described in Section \ref{sec:ranktsq}. Similar set of results for the class of rank kernel MMD tests are provided in~\cref{sec:rankmmd}. Implications of the results obtained in this paper to other nonparametric testing problems (e.g., in testing for mutual independence) are discussed in~\cref{sec:compare}. While most of the main paper focuses on location-shift type alternatives, we provide extensive asymptotic analysis for contamination alternatives too, which are deferred to~\cref{sec:auxdet}. Proofs of our results, numerical experiments depicting the finite sample performance of our proposed tests, and other additional technical details are given in Sections \ref{sec:pfmain} and \ref{sec:sim} in the Appendix. 	
	
	\section{Background}\label{sec:bg} 
	In this section, we introduce the notion of the population multivariate rank map define via optimal transport (see \cref{sec:poprank}) and then  present a general convergence result which gives conditions under which the empirical ranks (see~\eqref{eq:defemprank}) and population ranks are asymptotically ``close" (Section \ref{sec:rank_pqr}). 
	
	
	\subsection{Population Rank Map}\label{sec:poprank}
	
	Let $\mathcal{P}(\R^d)$ and $\mathcal{P}_{ac}(\R^d)$ denote the space of probability measures and the space of Lebesgue absolutely continuous probability measures on $\R^d$, respectively. Given measures $\mu,\nu\in\mathcal{P}(\R^d)$, consider the following optimization problem:
	\begin{align}\label{eq:Mongeproblem}
		\inf_{\mF:\R^d\to\R^d} \int  \lVert \bm{x}-\mF(\bm{x})\rVert^2 \,\mathrm d\mu( \bm{x})\qquad \mbox{subject to}\quad \mF\#\mu=\nu;
	\end{align}
	where $\mF\#\mu=\nu$ means that $\mF(\bm{X}) \sim \nu$, where $\bm{X} \sim \mu$. This optimization problem is often referred to as {\it Monge's problem} (see~\cite{monge1781memoire}) and a minimizer of~\eqref{eq:Mongeproblem}, if it exists, is referred to as an {\it optimal transport map}. 
	One of the most powerful results in this field was proved by Robert McCann in 1995, where he took a geometric approach to~\eqref{eq:Mongeproblem}. We now state \textit{McCann's theorem} in a form which will be useful to us (see~\cite[Theorem 2.12 and Corollary 2.30]{Villani2003}) and then use it to define the population multivariate rank map.
	\begin{prop}[McCann's theorem~\cite{Mccann1995}]\label{prop:Mccan}
		Suppose that $\mu, \nu\in\mathcal{P}_{ac}(\R^d)$. Then there exists functions $\bR(\cdot)$ and $\Q(\cdot)$, both of which are gradients of (extended) real-valued $d$-variate convex functions, such that $\bR\# \mu=\nu$, $\Q\# \nu=\mu$, $\bR$ and $\Q$ are unique ($\mu$ and $\nu$ a.e., respectively), $\bR\circ \Q (\bm{y})=\bm{y}$ ($\nu$ a.e.) and $\Q\circ \bR(\bm{x})=\bm{x}$ ($\mu$ a.e.). Moreover, if $\mu$ and $\nu$ have finite second moments, $\bR(\cdot)$ is also the solution to the problem in~\eqref{eq:Mongeproblem}.
	\end{prop}
	
	\begin{definition}[Population multivariate rank map]\label{def:popquanrank}
		Given a pre-specified reference distribution $\nu\in\mathcal{P}_{\mathrm{ac}}(\R^d)$ and a measure $\mu\in\mathcal{P}_{ac}(\mathbb{R}^d)$, the {\it population rank map} for the measure $\mu$ with reference distribution $\nu$ is the function $\bR(\cdot)$ as in~\cref{prop:Mccan}. Note that $\bR(\cdot)$ is unique up to measure zero sets with respect to $\mu$. 
	\end{definition}
	When $d=1$, a natural choice for the reference distribution $\nu$ is $\mathrm{Unif}[0,1]$. In this case, by~\cref{prop:Mccan}, it is easy to check that the population rank map is the cumulative distribution function associated with the probability measure $\mu$.
	Hereafter, we will fix a pre-specified reference distribution $\nu$ on $\R^d$ and assume that $\mu_1,\mu_2,\nu\in\mathcal{P}_{\mathrm{ac}}(\R^d)$. The densities of $\mu_1$ and $\mu_2$ with respect to the Lebesgue measure on $\R^d$ will be denoted by $f_1(\cdot)$ and $f_2(\cdot)$. The hypothesis testing problem~\eqref{eq:twosam} can then be reformulated as:
	\begin{equation}\label{eq:twosamden}
		\mathrm{H}_0:f_{1}=f_{2}  \qquad \mathrm{versus} \qquad \mathrm{H}_1:f_{1}\neq f_{2}.
	\end{equation}
	We will also denote by $\Rl(\cdot)$ the population rank map, associated with the measure $\lambda \mu_1+(1-\lambda)\mu_2$ and reference distribution $\nu$ (as in~\cref{def:popquanrank}). Similarly, $\Rmu(\cdot)$ will denote the rank map associated with $\mu_1$, that is, $\Rmu\#\mu_1=\nu$. 	
	
	\subsection{Convergence of Empirical Rank Maps}
	\label{sec:rank_pqr}
	
	In this section we will address the question as to how the empirical rank map defined in~\eqref{eq:empopt}~and~\eqref{eq:defemprank} estimates its population counterpart (see Definition~\ref{def:popquanrank} above). 
	To this end, define $[N]:=\{1,2,\ldots ,N\}$, for any $N\geq 1$. 
	Also, recall that $\mZ_N=\{\bZ_1,\ldots ,\bZ_N\}$ denotes the pooled sample $\mX_m\cup\mY_n$ and $\boldsymbol{\mathcal{H}}_N^d=\{\bm{h}_1^d,\ldots ,\bm{h}_N^d\}$.

	\begin{theorem}\label{theo:rankmapcon}
		Suppose $\mu_1,\mu_2,\nu\in\Pacd$ and
		\begin{equation}\label{eq:empgrid} 
			\frac{1}{N}\sum_{i=1}^N \delta_{\bh_i^d}\overset{w}{\longrightarrow}\nu,
		\end{equation}
		where $\overset{w}{\longrightarrow}$ denotes weak convergence.
		Fix $p,q\in\mathbb{N}$ and assume that $\cF(\cdot):(\R^d)^p\to\R^q$, $\bJ:\R^d\to\R^d$ are continuous Lebesgue a.e. Suppose that for all $1\le r\le p$,
		\begin{align}\label{eq:rankmapconew}
			\limsup_{N\to\infty}\frac{1}{N^p}&\E\Bigg[\sum_{(i_1,\ldots ,i_p)\in [N]^p} \Bigg(\Big\lVert\cF\Big(\bJ(\hbR_{m,n}(\mz_{i_1})),\ldots ,\bJ(\hbR_{m,n}(\mz_{i_r})),\bJ(\Rl(\mz_{i_{r+1}})),\nonumber \\ & \ldots ,\bJ(\Rl(\mz_{i_p}))\Big)\Big\rVert\Bigg)\Bigg]\leq \int \lVert\cF(\bJ(\bm{z}_1),\ldots ,\bJ(\bm{z}_p))\rVert \,\mathrm d\nu(\bm{z}_1)\ldots \, \mathrm d\nu(\bm{z}_p)<\infty.
		\end{align}
		Then the following conclusion holds:
		\begin{align*}
			\frac{1}{N^p}&\sum_{(i_1,\ldots ,i_p)\in [N]^p} \Big\lVert\cF\Big(\bJ(\hbR_{m,n}(\mz_{i_1})),\ldots ,\bJ(\hbR_{m,n}(\mz_{i_r})),\bJ(\Rl(\mz_{i_{r+1}})),\ldots ,\bJ(\Rl(\mz_{i_p}))\Big)\\ &~~~~~~~~~~~~~~~~~~~~~~~ - \cF\Big(\bJ(\Rl(\mz_{i_1})),\ldots ,\bJ(\Rl(\mz_{i_p}))\Big)\Big\rVert\overset{P}{\longrightarrow} 0, 
		\end{align*}
		conditionally on the set $\boldsymbol{\mathcal{H}}_N^d$. Further, if $\cF(\cdot)$ and $\bJ(\cdot)$ are Lipschitz, $\mbox{supp}(\nu)$ (support of the distribution $\nu$) 
		is compact, and $\{\bh_1^d, \bh_2^d, \ldots, \bh_N^d\} \subseteq  \mbox{supp}(\nu)$, then the above convergence holds a.s. If $\mu_1=\mu_2$, the same conclusions also hold if $\Rl(\cdot)$ is replaced with $\Rmu(\cdot)$.
	\end{theorem}
	\cref{theo:rankmapcon} (see~\cref{sec:pfmain} for its proof) shows how a function $\cF(\cdot):(\R^d)^p\to\R^q$, where $1 \leq r \leq p$ of its arguments are evaluated at the empirical rank map and the remaining $p-r$ coordinates are evaluated at the population rank map, can be approximated by its population counterpart in the asymptotic limit. In particular, by choosing $r=p=1$, $q=d$, $\cF(\mx)=\mx$ and $\bJ(\mx)=\mx$ in Theorem \ref{theo:rankmapcon} gives, 
	$$\frac{1}{N} \sum_{i=1}^N \lVert \hbR_{m,n}(\mz_i)-\Rl(\mz_i)\rVert\overset{P}{\longrightarrow}0,$$
	whenever~\eqref{eq:empgrid} and \eqref{eq:rankmapconew} hold. Note that the above conclusion does not require $\nu$ to be compactly supported. On the other hand, it yields convergence in probability instead of almost sure convergence, which in turn requires more stringent assumptions on $\cF(\cdot)$, $\bJ(\cdot)$ and $\nu$. We expect Theorem \ref{theo:rankmapcon} to be of independent interest, because it can be more widely applicable to other nonparametric two-sample and independence tests based on multivariate ranks (see~\cref{sec:compare}). Under stronger assumptions, it is also possible to obtain rates of convergence in~\cref{theo:rankmapcon} using techniques similar to~\cite{deb2021rates}.

	
	\begin{remark}[Verifying~\eqref{eq:rankmapconew}]
		A natural way to verify~\eqref{eq:rankmapconew} is by showing that 
		\begin{align*}
			\limsup_{N\to\infty}\frac{1}{N^p}&\E\Bigg[\sum_{(i_1,\ldots ,i_p)\in [N]^p} \Bigg(\Big\lVert\cF\Big(\bJ(\hbR_{m,n}(\mz_{i_1})),\ldots ,\bJ(\hbR_{m,n}(\mz_{i_r})),\bJ(\Rl(\mz_{i_{r+1}})),\nonumber \\ & \ldots ,\bJ(\Rl(\mz_{i_p}))\Big)\Big\rVert^{1+\delta}\Bigg)\Bigg]< \infty
		\end{align*}
		for some $\delta>0$. The above condition can be easily verified for many  natural choices of $\cF(\cdot)$, $\bJ(\cdot)$, and $\nu$ (see~\cref{sec:verbd} for some  examples). In fact, the above condition will imply that~\eqref{eq:rankmapconew} is satisfied with $=$ instead of $\leq$.
	\end{remark}
	
	\begin{remark}[Choice of $\boldsymbol{\mathcal{H}}_N^d$]\label{rem:chooserank}
		Note that~\cref{theo:rankmapcon} is flexible on the choice of the elements of $\boldsymbol{\mathcal{H}}_N^d$. One choice includes drawing a random sample of size $N$ from $\nu$ and then fixing the obtained sequence. Otherwise one can also choose deterministic sequences that approximate $\nu$. When $\nu=\mathrm{Unif}[0,1]^d$, some popular examples of such sequences can be found in the quasi-Monte Carlo literature (see~\cite{Rabinowitz1987} and the references therein). The use of deterministic sequences to define multivariate ranks has been advocated in~\cite{Deb19,shi2020rate}. If $\bh_1^d, \bh_2^d, \ldots, \bh_N^d$ are chosen using a random sample from $\nu$, then assumption~\eqref{eq:empgrid} is to be interpreted as weak convergence a.s. (which follows using the Varadarajan theorem, see~\cite{Varadarajan1958}). In this case, crucially, the convergence results hold \emph{conditionally} on $\{\bh_1^d,\ldots ,\bh_N^d\}$.		
	\end{remark}
	
	
	\section{Score Transformed Hotelling-Type Tests Based on Optimal Transport}\label{sec:ranktsq}
	
	In this section we generalize the basic version the multivariate distribution-free Hotelling $T^2$ test statistic already presented in~\eqref{eq:rankhotelling} by incorporating \textit{score functions}, which are injective functions $\bJ(\cdot):\R^d\to\R^d$ that are \emph{continuous Lebesgue a.e. in $\R^d$}. A natural example of a score function is $\bJ(\cdot):=(F_1^{-1}(\cdot),\ldots ,F_d^{-1}(\cdot))$, where $F_1(\cdot), F_2(\cdot), \ldots, F_d(\cdot)$ are univariate distribution functions. When $d=1$ and $F(\cdot)=\Phi(\cdot)$ is the standard Gaussian distribution function, then the corresponding score function $\bJ(\cdot)$ is called the \emph{van der Waerden score function} (see~\cite{vander1952}; also see~\eqref{eq:vanscore}). A very useful notion for the sequel is that of an effective reference distribution which is obtained by combining a score function with a reference measure as follows: 
	
	\begin{definition}\label{def:erd}
		Given a reference distribution $\nu$ and a score function $\bJ(\cdot):\R^d\to\R^d$, the \textit{effective reference distribution}, hereby abbreviated as ERD, is the push-forward measure $\bJ\#\nu$. For example, in the classical univariate Chernoff-Savage framework (see~\cite{Chernoff1958,Gastwirth1968}), the ERD is the $\mathcal{N}(0,1)$ distribution, which can be obtained by choosing the reference distribution $\nu=\mathrm{Unif}[0,1]$ and the score function $\bJ(\cdot)=\Phi^{-1}(\cdot)$, the standard Gaussian quantile function. 
	\end{definition} 
	
	\begin{assumption}\label{assumption3} 
		The effective reference distribution (ERD) is non-degenerate in the sense that it has a well-defined finite and positive definite covariance matrix $\Serd$.
	\end{assumption}
	
	Under the above assumption, we define the {\it rank Hotelling $T^2$ statistic with reference distribution $\nu$ and score function $\bJ$} as:
	\begin{align}\label{eq:rankschotelling}
		\rhsc&:=\tfrac{mn}{m+n} (\bm{\Delta}_{m, n}^{\nu,\bJ})^{\top} \Serd^{-1} \bm{\Delta}_{m, n}^{\nu,\bJ}, 
	\end{align}
	where 
	\begin{align}\label{eq:delta}
		\bm{\Delta}_{m, n}^{\nu,\bJ} := \frac{1}{m}\sum_{i=1}^m \bJ\left(\hat{\bR}_{m,n}(\bX_i)\right)-\frac{1}{n}\sum_{j=1}^n \bJ\left(\hat{\bR}_{m,n}(\bY_j) \right) . 
	\end{align}
	is the difference of the means of the score transformed pooled ranks of the two samples. 
	
	Note that $\rhsc$ is equivalent to statistic $\rh$ introduced~\eqref{eq:rankhotelling} when the score function is $\bJ(\mx)=\mx$ and the reference distribution $\nu$ has uncorrelated components with the same marginal distributions, each of which have finite second moments. This is because, in this case, $\Serd=\sigma^2 \bm I_d$, where $\sigma>0$. 
	The class of statistics in \eqref{eq:rankschotelling} leads to a new family of two-sample tests that are distribution-free and consistent for a general collection of alternatives, which we discuss in the sequel.

	\begin{prop}[Distribution-freeness]\label{prop:dfree}
		Assume that $\mathrm{H}_0$ is true and $\mu_1=\mu_2\in \mathcal{P}_{\mathrm{ac}}(\R^d)$. Then the distribution of $\rhsc$ is universal, that is, it is free of $\mu_1=\mu_2$, for all $m,n\geq 1$.
	\end{prop}
	
	Using the above result we can readily obtain a finite sample distribution-free two-sample test that uniformly controls the Type I error. To this end, fix a level $\alpha\in (0,1)$ and let $c_{m,n}$ denote the upper $\alpha$ quantile of the universal distribution in~\cref{prop:dfree}. Consider the test function:
	\begin{equation}\label{eq:testtrank}
		\prsc:=\ind\left(\rhsc\geq c_{m,n}\right).
	\end{equation}
	This test is exactly distribution-free for all $m,n\geq 1$ and uniformly level\footnote{Strictly speaking, to guarantee exact level $\alpha$, we have to randomize $\prsc$, as the exact distribution of $\rhsc$ is discrete. But unless $m,n$ are extremely small, this makes no practical difference.} $\alpha$ under $\mathrm{H}_0$, that is, 
	\begin{equation}\label{eq:uniflevelt}
		\sup_{\mu_1=\mu_2\in \mathcal{P}_{\mathrm{ac}}(\R^d)} \E\left[\prsc\right]=\alpha.
	\end{equation}

	\begin{remark} In addition to being distribution-free, the statistic~\eqref{eq:rankschotelling} has the advantage that the matrix $\Serd$ is deterministic. In fact, it can be computed once the reference distribution $\nu$ and score function $\bJ$ are specified, and does not depend on the data. This makes the implementation of the test \eqref{eq:testtrank} particularly convenient, because no covariance matrix estimation is required for computing $\rhsc$, unlike other nonparametric tests of location such as the interdirections-based tests in~\cite{hallin2002optimal,Peters1990,Randles1989} and the tests based on data-depth in~\cite{liu1993}.
	\end{remark} 
	
	\subsection{Asymptotic Null Distribution and Consistency}\label{sec:condist}
	
	In this subsection we discuss the consistency of the test $\prsc$ and the asymptotic null distribution of the test statistic $\rhsc$. Throughout we will assume that $\bX \sim \mu_1$ and $\bY \sim  \mu_2$. We first describe the asymptotic null distribution of the statistic $\rhsc$. This is formalized in the following theorem, which is proved in\cref{sec:pfnulldistribution}.
	
	\begin{theorem}[Limiting null distribution]\label{theo:nullrankhotelling}
		Suppose the condition in \eqref{eq:empgrid} and assumption {\em\ref{assumption3}} hold. Further, assume that
		\begin{equation}\label{eq:nullrhas}
			\limsup\limits_{N\to\infty} \frac{1}{N}\sum_{i=1}^N \lVert \bJ(\bh_i^d)\rVert^2\leq \int \lVert \bJ(\bz)\rVert^2\,\mathrm d\nu(\bz).
		\end{equation}
		We also consider the asymptotic regime where $N = m+n \rightarrow \infty$ such that 
		\begin{equation}\label{eq:usual}
			m/N\to\lambda\in (0,1).
		\end{equation} 
		Then, under $\mathrm H_0$ as in~\eqref{eq:twosamden}, we have: 
		\begin{align}\label{eq:rhscnull}
			\rhsc\overset{w}{\longrightarrow}\chi^2_d.
		\end{align}
	\end{theorem}
	
	On account of having a simple limiting null distribution, the result above can be used to calibrate the statistic $\rhsc$ to obtain an asymptotic level $\alpha$ test. To this end, denote by $\chi^2_{d, 1-\alpha}$ the $(1-\alpha)$-th quantile of the $\chi^2_d$ distribution. Then the test which rejects when $\{\rhsc > \chi^2_{d, 1-\alpha}\}$ satisfies~\eqref{eq:uniflevelt} asymptotically. 
	
	The proof of Theorem \ref{theo:nullrankhotelling} is given in~\cref{sec:pfnulldistribution}. It proceeds in two steps: First, we prove a H\'ajek representation type result which shows that  the error incurred in replacing the empirical rank maps in \eqref{eq:rankschotelling} with their population counterparts (that is, $\Rmu(\cdot)$)  is asymptotically negligible under $\mathrm H_0$. The asymptotic distribution in \eqref{eq:rhscnull} then follows from the asymptotic normality of 
	\begin{align}\label{eq:delta_oracle}
		\bm{\Delta}_{m, n}^{\nu,\bJ,\mathrm{or}} := \frac{1}{m}\sum_{i=1}^m \bJ\left(\Rmu(\bX_i)\right)-\frac{1}{n}\sum_{j=1}^n \bJ\left(\Rmu(\bY_j) \right) 
	\end{align}
	and the continuous mapping theorem. Note that $\bm{\Delta}_{m, n}^{\nu,\bJ,\mathrm{or}}$ is the oracle counterpart of $\bm{\Delta}_{m, n}^{\nu,\bJ}$ obtained by replacing the empirical ranks maps in \eqref{eq:delta} with their population analogs. 
	
	\begin{remark}[On assumption~\eqref{eq:nullrhas}]\label{rem:justbound}
		A simple case where assumption~\eqref{eq:nullrhas} can be verified is when $\bh_1^d, \bh_2^d, \ldots, \bh_N^d$ are i.i.d. samples from the reference distribution $\nu$. In that case, equality holds in~\eqref{eq:nullrhas} almost surely by the strong law of large numbers and~\cref{theo:nullrankhotelling} holds conditionally on $\{\bh_1^d,\ldots ,\bh_N^d\}$. There has been some interest in using deterministic choices for $\{\bh_1^d,\ldots ,\bh_N^d\}$ in recent works such as~\cite{Deb19,shi2020distribution} (see also~\cite[Table 12]{Deb19} for some potential benefits of deterministic choices). We will discuss how to verify~\eqref{eq:nullrhas} in some such cases in~\cref{sec:verbd}.
	\end{remark}

	We now proceed to show that the test $\prsc$ is consistent against a large class of alternatives in~\eqref{eq:twosamden}. This is formalized in the following theorem. 
	
	\begin{theorem}[Consistency]\label{theo:rhconsis} 
		Suppose the conditions in \eqref{eq:empgrid},~\eqref{eq:nullrhas} and assumption  {\em\ref{assumption3}} hold. Then, for problem \eqref{eq:twosamden} in the usual asymptotic regime~\eqref{eq:usual}, 
		$\lim_{m,n\to\infty} \E_{\mathrm{H}_1} \left[\prsc\right]=1$
		provided $\E\bJ(\Rl(\bX))$ and $\E\bJ(\Rl(\bY))$ are finite and $\E\bJ(\Rl(\bX))\neq\E\bJ(\Rl(\bY))$.
	\end{theorem}	
	
	The proof of this theorem is given in~\cref{sec:pfconsistency}. It follows from~\cref{theo:nullrankhotelling} which we use to show that $\rhsc$ converges to zero under $H_0$ and a positive number under the alternative, thus, implying consistency.
	\begin{remark}[Connection to Wilcoxon's rank-sum test]\label{rem:wilcox}
		For $d=1$, if $\bJ(\cdot)$ is nondecreasing (as is the case with quantile functions) then it is easy to see that whenever $\bY$ is stochastically larger (respectively, smaller) than $\bX$, we have $\E\bJ(\Rl(\bY))>\E\bJ(\Rl(\bX))$ (respectively, $\E\bJ(\Rl(\bY))<\E\bJ(\Rl(\bX))$). This is the exact condition for the consistency of Wilcoxon's rank sum test (see~\cite{Zimmerman2012} for details).
	\end{remark}
	
	While Theorem \ref{theo:rhconsis} gives the general condition under which the test $\prsc$ is consistent, it is not directly apparent how it applies to the location-shift alternatives where $f_2(\cdot)= f_1(\cdot - \bD)$, for some $\bD \in \R^d\backslash\{\bm 0\}$. In this case, the two-sample testing problem \eqref{eq:twosamden} becomes: 
	\begin{equation}\label{eq:twosamloc_theta}
		\mathrm{H}_0:\bD=\bzr  \qquad \mathrm{versus} \qquad \mathrm{H}_1:\bD\neq \bzr .
	\end{equation} 
	The following proposition shows the consistency $\prsc$ for the above problem.  
	
	\begin{prop}[Consistency under location-shift alternatives]\label{prop:conloc}
		Suppose the conditions in \eqref{eq:empgrid},~\eqref{eq:nullrhas} and assumption {\em\ref{assumption3}} hold with $\bJ(\mx)=\mx$. Recall from~\cref{prop:Mccan} that the population rank map $\Rl(\cdot)$ is the gradient of a convex function, say $\pl(\cdot)$. Assume that $\pl(\cdot)$ is strictly convex on an open set of positive measure with respect to $\mu_1$. Then, for problem \eqref{eq:twosamloc_theta} in the usual asymptotic regime~\eqref{eq:usual}, for any $\bD\in \R^d\setminus \{\bzr\}$,
		$$\lim_{m,n\to\infty} \E_{\bD}\left[\prsc\right]=1.$$
		
	\end{prop}

	The proof of Proposition \ref{prop:conloc} is given in~\cref{sec:pfconsistency}. The main ingredient of the proof is the \textit{cyclical monotonicity property} (see~\cite[Chapter V, page 238]{Rockafellar1970}) of the population multivariate rank map (recall~\cref{def:popquanrank}) which is used to show that the consistency condition $\E\bJ(\Rl(\bX))\neq\E\bJ(\Rl(\bY))$ in Theorem \ref{theo:rhconsis} holds whenever $\bm \Delta \ne \bm 0$. Outside location-shift alternatives, one can also prove consistency under contamination alternatives (see~\cite{Dhar2012,Hodges1956,stepanova2008asymptotic}) which are popular in literature ; see~\cref{prop:conlocontam} in the Appendix for details.
	
	\subsection{Local Asymptotic Power}\label{sec:locpow}
	
	In this subsection we derive the asymptotic power of the test  $\prsc$ against local contiguous alternatives. To quantify the notion of local alternatives, we will adopt the standard smooth parametric model assumptions from the theory of local asymptotic normality (LAN) (see, for example,~\cite[Chapter 7]{Van1998}). To this end, let $\Th\subseteq\mathbb{R}^p$ ($p$ fixed, may or may not be equal to $d$) and $\{\mathcal{P}_{\bt}\}_{\bt\in\Th}$ be a parametric family of distributions in $\mathbb{R}^d$ with density $f(\cdot|\bt)$, with respect to Lebesgue measure, indexed by a $p$-dimensional parameter $\bt\in\Th$. We will assume the following standard regularity conditions on this parametric family:
	\begin{itemize}
		\item The family $\{\mathcal{P}_{\bt}\}_{\bt\in\Th}$ is \textit{quadratic mean differentiable} (QMD) at $\bt=\bt_0\in\Theta$ (see~\cite[Definition 12.2.1]{LR05} for related definitions). It holds for most standard families of
		distributions, including exponential families in natural form.
		
		\item For $\bm X \sim \mathcal{P}_{\bt_0}$, $\mathbb{E}_{\bt_0}\left(\lVert\boldsymbol{\eta}(\bm X, \bt_0)\rVert^2\right) < \infty$, where $\boldsymbol{\eta}(\cdot, \bt) := \frac{\nabla_{\bt} f(\cdot|\bt)}{f(\cdot|\bt)}$ is the \textit{score function}. Also, suppose that the Fisher information exists at $\bt_0$, that is, $I(\bt_0):=\E_{\bt_0}[\boldsymbol{\eta}(\bX, \bt_0)\boldsymbol{\eta}(\bX, \bt_0)^{\top}]$ exists and is invertible.
	\end{itemize}
	Under these regularity assumptions, we will consider the following sequence of hypotheses:
	\begin{equation}\label{eq:twosamsmooth}
		\mathrm{H}_0:f_1(\cdot)=f_2(\cdot)=f(\cdot|\bt_0)  \quad \mathrm{versus} \quad \mathrm{H}_1:f_1=f(\cdot|\bt_0),\ f_2=f\left(\cdot|\bt_0+N^{-\frac{1}{2}}\bh\right)
	\end{equation}
	for some $\bh\neq \bzr$. It is worth noting that under these contiguous alternatives, no test can be consistent, that is, no test can have power converging to $1$ (see~\cref{sec:lowerbound}; cf.~\cite[Theorem 5.4]{shi2020rate}). Therefore it is interesting to look at the power curves of tests along such contiguous alternatives as a way to draw direct non-trivial comparisons between them. 	
	In the following theorem we derive the asymptotic distribution of $\rhsc$ under the local alternatives \eqref{eq:twosamsmooth}. In addition to providing precise expressions for the asymptotic local power, this result will be key in deriving bounds on the ARE of $\rhsc$ with respect to Hotelling $T^2$ test in Section \ref{sec:arehotelling} below.
	\begin{theorem}[Asymptotics under local alternatives]\label{theo:locpowermain1}
		Suppose the condition in \eqref{eq:empgrid},~\eqref{eq:nullrhas} and assumption {\em\ref{assumption3}} hold. Recall that $\boldsymbol{\eta}(\cdot, \bt) = \frac{\nabla_{\bt} f(\cdot|\bt)}{f(\cdot|\bt)}$. Then in the usual asymptotic regime~\eqref{eq:usual}, under $\mathrm{H}_1$ from~\eqref{eq:twosamsmooth} and with $\boldsymbol{G}\sim\mathcal{N}(\bzr,\bm I_d)$, we have: 
		$$\rhsc\overset{w}{\longrightarrow} \Big\lVert \sqrt{\lambda(1-\lambda)}\Serd^{-\frac{1}{2}}\E_{\bt_0}\left[\bJ(\Rmu(\bX))\bh^{\top}\boldsymbol{\eta}(\bX, \bt_0)\right]+\boldsymbol{G}\Big\rVert^2.$$
	\end{theorem}
	The proof of the above result is given in~\cref{sec:pflocalpower}. For this we first derive the joint limiting distribution of $\bm{\Delta}_{m, n}^{\nu,\bJ}$ (recall \eqref{eq:delta}) and the likelihood ratio (for the testing problems~\eqref{eq:twosamsmooth}) under $\mathrm H_0$. The limiting distribution of $\bm{\Delta}_{m, n}^{\nu,\bJ}$ under $\mathrm H_1$ can then be obtained by invoking Le Cam's third lemma~\cite[Theorem 12.3.2]{LR05}, from which the result follows by an application of the continuous mapping theorem.

	\subsection{ARE against Hotelling $T^2$}\label{sec:arehotelling}
	
	We now compute the ARE of $\rhsc$ against the Hotelling $T^2$ test for a variety of settings in location-shift alternatives. As is evident from Theorem \ref{theo:locpowermain1}, the asymptotic efficiency of $\rhsc$ will depend on the choices of $\bJ(\cdot)$ and $\nu$ (collectively on the ERD; recall~\cref{def:erd}) and also on the data generating distribution. This makes it important and interesting to study how the relative efficiencies depend on the aforementioned variables. We begin by describing three popular choices of ERDs below. The specific combinations of $\bJ(\cdot)$ and $\nu$ yielding these ERDs will be provided in the Theorem statements.
	
	\begin{enumerate}
		\item[(1)] \textit{Uniform on the hypercube}: In this case, the ERD is 
		$\mathrm{Unif}[0,1]^d$, the uniform distribution on the hypercube $[0,1]^d$. This choice has appeared in~\cite{Chernozhukov2017,Deb19,ghosal2019multivariate}, and is a natural choice as in the univariate case, the $\mathrm{Unif}[0,1]$ distribution is the most popular reference distribution for rank-based tests in the literature. Note that this ERD has \textit{independent components}.
		
		\item[(2)] \textit{Spherical uniform}: Let $U_1\sim \mathrm{Unif}[0,1]$ and $\boldsymbol{U}_2$ be drawn independently and uniformly from the unit sphere $\{\mx\in\R^d: \lVert \mx\rVert=1\}$. Then the spherical uniform ERD is the distribution of the random variable $U_1\boldsymbol{U}_2$. It is easy to see that the spherical uniform distribution is \textit{spherically symmetric}. The spherical uniform reference distribution has been used in~\cite{delbarrio2019,shi2020distribution,shi2020rate} while defining multivariate ranks.

		\item[(3)] \textit{Standard multivariate normal}: Here, the ERD is $\mathcal N(\bzr, \bm I_d)$, the standard $d$-variate Gaussian distribution with mean vector $\bzr$ and covariance matrix $\bm I_d$. This is the only multivariate distribution which is both spherically symmetric and has independent components~\cite{Dawid1977}. 
		
	\end{enumerate} 
	
	\begin{remark}\label{rem:covmat} We observe that for all $3$ ERDs described above $\Serd=\sigma^2 \bm I_d$, for some $\sigma>0$. In particular, for the uniform distribution on the hypercube $\sigma^2=1/12$, for the spherical uniform distribution 
		$\sigma^2=(3d)^{-1}$, and for the standard $d$-variate normal distribution $\sigma^2=1$.
	\end{remark}


	Following the seminal paper of Hodges and Lehmann~\cite{Hodges1956}, we will compute the ARE of $\rhsc$ against the Hotelling $T^2$ under local perturbations in location-shift models. Recall that this corresponds to $f_2(\cdot)=f_1(\cdot-\bD)$, where $f_1(\cdot)$ satisfies the regularity assumptions described above. In the same vein as~\eqref{eq:twosamloc_theta}, we consider the following testing problem: 
	\begin{equation}\label{eq:twosamloc}
		\mathrm{H}_0:\bD=\bzr  \qquad \mathrm{versus} \qquad \mathrm{H}_1:\bD=N^{-\frac{1}{2}}\bh
	\end{equation}
	for some $\bh\neq \bzr$. Hereafter, we will abbreviate the ARE of $\rhsc$ against the Hotelling $T^2$ test as $\atr$. Recall that the ARE is essentially the inverse ratio of the number of samples required by the respective tests (individually calibrated as asymptotically level $\alpha$ tests) to achieve power $\beta$ when the null shrinks to the alternative at rate $O(1/\sqrt N)$ (refer to \cref{def:asympeff} in the Appendix for the formal definition of the ARE between two tests). As a consequence, the ARE between tests in general depends on the level parameter $\alpha$, the power parameter $\beta$, and also on the sequence $\bD$ converging to $\bzr$. However, interestingly, for tests which have asymptotically non-central $\chi^2$ 
	distributions with the same degrees of freedom under contiguous alternatives (such as in~\eqref{eq:twosamloc}), the ARE is simply the ratio of the corresponding non-centrality parameters (see~\cite[Proposition 5]{Hallin2002}). This applies to $\atr$ which simplifies to
	\begin{equation}\label{eq:atr}
		\atr=\frac{\Big\lVert\Serd^{-\frac{1}{2}}\E_{\mathrm{H}_0}\left[\bJ(\Rmu(\bX))\bh^{\top}\boldsymbol{\eta}(\bm X, \bt_0)\right]\Big\rVert^2}{\Big\lVert(\E_{\mathrm{H}_0}[\bX-\E\bX][\bX-\E\bX]^{\top})^{-\frac{1}{2}} \E_{\mathrm{H}_0}\left[\bX\bh^{\top}\boldsymbol{\eta}(\bm X, \bt_0)\right]\Big\rVert^2}.
	\end{equation}
	
	We provide a detailed discussion on the derivation of \eqref{eq:atr}  around \eqref{eq:AREinterest} in the Appendix. Also, to avoid notational clutter, we have hidden the dependence of $\atr$ on $\bm h$ in the  notation. In fact, all the subsequent results on $\atr$ in this section will be valid for all $\bm h\neq \bzr$. 
	With the above expression in mind, let us compute $\atr$ in the simple case of the Gaussian location family for the different ERDs described above. The proof of the following result is given in~\cref{sec:pfarehotelling}. 
	
	\begin{prop}[Gaussian location problem]\label{prop:Gaussare}  Suppose $\mathcal{P}_{\bt}$ is the $d$-variate normal distribution $\mathcal{N}(\bt,\Sigma)$, for some unknown positive definite covariance matrix $\Sigma$. Then, under condition \eqref{eq:empgrid}, assumption \eqref{eq:nullrhas}, the following conclusions hold for the testing problem~\eqref{eq:twosamloc} in the usual asymptotic regime~\eqref{eq:usual}:
		\begin{enumerate}
			\item[$(1)$] If the ERD is $\mathrm{Unif}[0,1]^d$, obtained by choosing $\nu=\mathcal{N}(\bzr,\bm I_d)$ and $\bJ(\mx)=(\Phi(x_1),\ldots ,\Phi(x_d))$, then  
			$\atr=3/\pi\approx 0.95.$
			
			\item[$(2)$] When the ERD is spherical uniform with $\bJ(\mx)=\mx$ and $\nu\ =$ spherical uniform,  
			\begin{align}\label{eq:atr_spherical}
				\atr= \kappa_d:=\frac{3}{d}\left\{\frac{1}{2^{d-1}}\cdot\frac{\Gamma{(d-0.5)}}{(\Gamma(d/2))^2}+ \omega_d\right\}^2, 
			\end{align} 
			with
			\begin{align}\label{eq:expectation_G}
				\omega_d:= \frac{\sqrt{2\pi} (d-1) }{2^{d/2}\Gamma(d/2)}\E[|Z|^{d-2}]-\frac{\sqrt{2\pi}(2d-2)!!}{2^{d-1}(\Gamma(d/2))^2} {}_{2}F_{1}(d-\tfrac{1}{2}, \tfrac{d}{2}- \tfrac{1}{2}; \tfrac{d}{2}+\tfrac{1}{2}; -1), 
			\end{align}
			where $Z\sim \mathcal{N}(0,1)$, $(2d-2)!!:=1\times 3\times \ldots \times (2d-3)$ and ${}_{2}F_{1}(\cdot,\cdot;\cdot;\cdot)$ is the hypergeometric function (see~{\em\cite{Olde2010}}).  
			\item [$(3)$] When the ERD is $\mathcal{N}(\bzr,\bm I_d)$ with $\bJ(\mx)=\mx$ and $\nu=\mathcal{N}(\bzr,\bm I_d)$, $\atr=1$.
		\end{enumerate}
	\end{prop}

	The result above shows that even in this simple case of the Gaussian location family, the ARE depends crucially on the choice of $\nu$ and $\bJ(\cdot)$. In particular, when $\nu=\mathcal{N}(\bzr,\bm I_d)$, $\bJ(\mx)=(\Phi(x_1),\ldots ,\Phi(x_d))$, or $\bJ(\mx)=\mx$,  $\nu=\mathcal{N}(\bzr,\bm I_d)$, the ARE stays constant over the dimension $d$. However, when $\bJ(\mx)=\mx$ and $\nu\ =$ spherical uniform, the ARE varies with $d$.
	
	\cref{fig:PowerGausscomp1} shows the plots of the efficiencies in~\cref{prop:Gaussare} as a function of the dimension. One of the surprising things that emerge from this plot is that $\kappa_d$ (recall \eqref{eq:atr_spherical}) falls below 0.95 when $d \geq 5$,  that is, {\it in the normal
		family for $d\geq 5$, the rank Hotelling statistic with cubic uniform
		ERD (with $\bJ(\cdot)$ and $\nu$ chosen as in~\cref{prop:Gaussare}, part (1)) is more efficient compared to
		the spherical uniform (in terms of ARE with respect to the Hotelling's $T^2$).} It is also worth noting that although the ARE decreases with increasing dimension when the ERD is the spherical uniform distribution, it stabilizes to a non-zero limit.

	%
		%
		%
	
	\begin{figure}[h]\label{fig:PowerGausscomp1}
		\hspace{0.6in}
		\includegraphics[height=5cm,width=5.5cm]{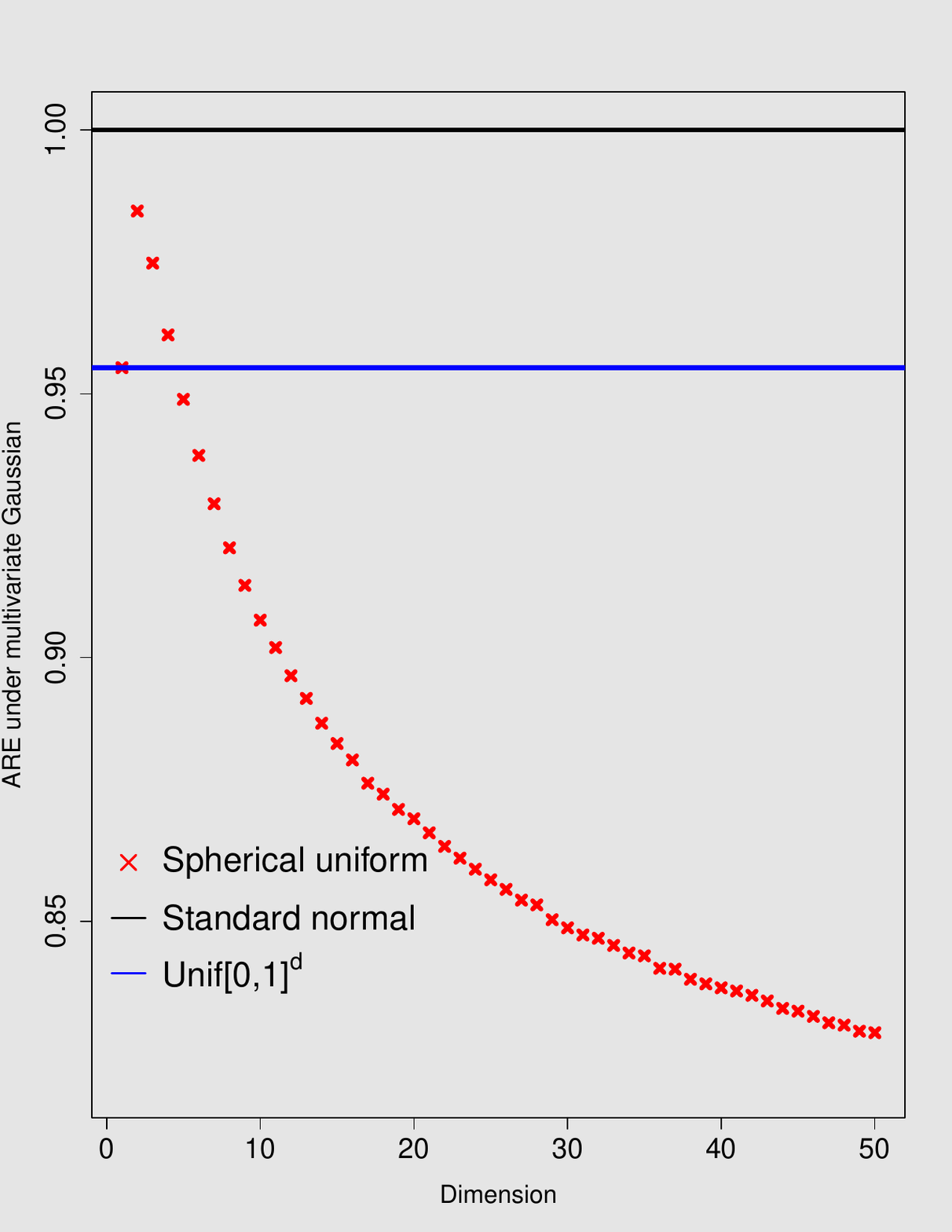}	      
		\includegraphics[height=5cm,width=5.5cm]{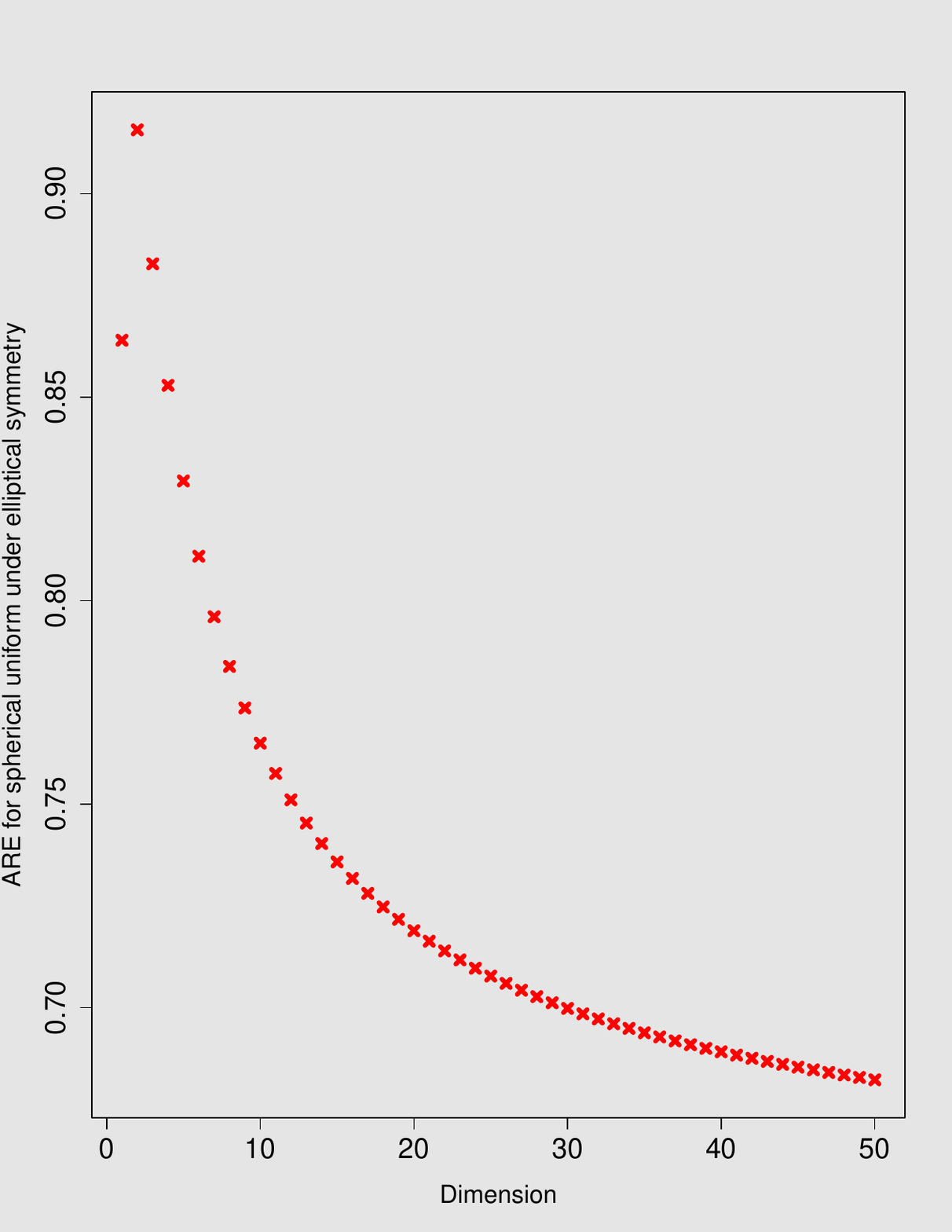}
		\caption{The left panel shows the ARE when $\bX$ follows a multivariate Gaussian distribution for our three candidate ERDs (see~\cref{prop:Gaussare}). The black and blue lines correspond to the AREs for the standard Gaussian and the $\mathrm{Unif}[0,1]^d$ ERD. Note that these lines are constant over dimension at $1$ and $3/\pi$, respectively. For the spherical uniform ERD, the ARE decreases for $d\geq 2$ and is lower than that of the $\mathrm{Unif}[0,1]^d$ ERD for $d\geq 5$. In the right panel, we plot the lower bound for the ARE of $\rhsc$ with respect to the Hotelling $T^2$ test under elliptical symmetry when the spherical uniform ERD is used (see~\cref{prop:areell}). Note that it  decreases for $d\geq 2$.}
	\end{figure}
	
	
	\begin{remark}[Advantages of incorporating score functions]\label{rem:scorefn} 
		In \cref{prop:Gaussare}~(1) we see the theoretical benefits of incorporating score functions. For instance, if we had simply set $\bJ(\mx)=\mx$ and $\nu=\mathrm{Unif}[0,1]^d$, then by~\cref{theo:locpowermain1}, we would have to work with the population rank map from $\mathcal{N}(\bt_0,\Sigma)$ to $\mathrm{Unif}[0,1]^d$, which is harder to obtain in a closed form. On the other hand, when we set $\nu=\mathcal{N}(\bzr,\bm I_d)$, then the population rank map from $\mathcal{N}(\bt_0,\Sigma)$ to $\nu=\mathcal{N}(\bzr,\bm I_d)$ has the simple closed form expression $\Rmu(\bX)=\Sigma^{-\frac{1}{2}}(\bX-\bt_0)$. This implies, $\bJ(\Rmu(\bX))=\Phi(\Sigma^{-\frac{1}{2}}(\bX-\bt_0))$, where the standard normal distribution function $\Phi$  function is applied coordinate-wise, which is again a simple function to analyze. It is interesting to note that using this approach, we are able to recover the efficiency of $3/\pi$ which is the same as the ARE of the Wilcoxon's test against the $t$-test (see~\cite{Hodges1956}).
	\end{remark}
	We now move on to the independent components case. This is a natural choice as product distributions are most natural ways to construct multivariate distributions from univariate ones. To this end, denote by $\fnd$ the class of $d$-dimensional product distributions where the distribution of each component belongs to a location family which is absolutely continuous with respect to the Lebesgue measure. More precisely,  $\{\mathcal{P}_{\bt}\}_{\bt \in \Theta} = \{p_{\bt}(\bz)=\prod_{i=1}^d \tilde{f}_i( z_i - \theta_i)\}_{\bt \in \Theta}$, where $\tilde{f}_1, \tilde{f}_2, \ldots, \tilde{f}_d$ are univariate densities which are absolutely continuous with respect to the Lebesgue measure on $\R$, $\bz=(  z_1,  z_2, \ldots,  z_d)$, and $\bt=(\theta_1,\ldots ,\theta_d) \in \Theta \subseteq \R^d$.

	\begin{theorem}[Independent components case]\label{prop:areind} 
		Suppose condition \eqref{eq:empgrid} and assumption \eqref{eq:nullrhas} hold. Then the following conclusions hold for testing the hypothesis \eqref{eq:twosamloc} in the usual asymptotic regime~\eqref{eq:usual}:
		\begin{enumerate}
			\item[$(1)$] If the ERD is $\mathrm{Unif}[0,1]^d$ obtained by choosing $\nu=\mathcal{N}(\bzr,\bm I_d)$ and $\bJ(\mx)=(\Phi(x_1),\ldots ,\Phi(x_d))$ for $\mx=(x_1,\ldots ,x_d)$, we have $\inf_{\fnd}\atr\geq \frac{108}{125}= 0.864.$
			Further, equality holds if and only if the $i$-th component of $\bX$ has density of the form:
			\begin{equation}\label{eq:areind_density}
				f_i(x)=\frac{3}{20\sqrt{5}\sigma_i^3}\left(5\sigma_i^2-(x-\theta_i)^2\right)\ind\{|x-\theta_i|\leq \sqrt{5}\sigma_i\},
			\end{equation}
			for all $i\in [d]$, where $\sigma_1,\ldots ,\sigma_d>0$. The same conclusion also holds if we use $\nu=\mathrm{Unif}[0,1]^d$ and $\bJ(\mx)=\mx$.
			
			\item[$(2)$] When the ERD is $\mathcal{N}(\bzr,\bm I_d)$ with $\nu=\mathcal{N}(\bzr,\bm I_d)$ and $\bJ(\mx)=\mx$, then 
			$\inf_{\fnd}\atr\geq 1.$
			Further, equality holds in the above display if and only if $\bX\sim\mathcal{N}(\bt, \bm D)$, where $\bm D$ is a diagonal matrix with strictly positive entries.
		\end{enumerate}
	\end{theorem}

	The proof of this result is given in~\cref{prop:areind} (in the Appendix) provides worst case lower bounds for the efficiency of $\prsc$ against Hotelling $T^2$ over the class of multivariate distributions with independent components for the two different choices of ERD with independent components. These can be interpreted as multivariate analogs of the celebrated results of Hodges and Lehmann~\cite{Hodges1956} and Chernoff and Savage~\cite{Chernoff1958}: 
	
	\begin{itemize}
		
		\item {\it Multivariate Hodges-Lehmann phenomenon}: The result in~\cref{prop:areind}-(1) shows that, in the worst case, $\prsc$ needs only around $15\%$ more samples than Hotelling $T^2$ to yield the same power, if $\mathrm{Unif}[0,1]^d$ is used as the ERD, with $\nu=\mathcal{N}(\bzr,\bm I_d)$ and $\bJ(\mx)=(\Phi(x_1),\ldots ,\Phi(x_d))$. Note that the bound in this case does not depend on the dimension $d$ and matches the univariate Hodges-Lehmann bound~\cite{Hodges1956}. 
		
		\item {\it Multivariate Chernoff-Savage phenomenon}:~\cref{prop:areind}-(2) shows that if $\mathcal{N}(\bzr,\bm I_d)$ is used as the ERD, with $\nu=\mathcal{N}(\bzr,\bm I_d)$ and $\bJ(\mx)=\mx$, then $\prsc$ attains at least as large an asymptotic power as Hotelling $T^2$ with as many samples, thereby extending the Chernoff-Savage bound \cite{Chernoff1958} for the univariate Wilcoxon rank sum test. This, in particular, implies the following remarkable fact: {\it If the distribution $\mu_1$ has independent components but is not Gaussian, then $\prsc$ will attain the same power as Hotelling $T^2$ with fewer samples.}
		
	\end{itemize}
	
	Next, we consider the class of multivariate elliptically symmetric distributions. To be precise, $\bX$ has an elliptically symmetric distribution if there exists $\bt\in\R^d$, a positive definite $d\times d$ matrix $\Sigma$, and a function $\underline{f}(\cdot):\R^+\to\R^+$ such that the density $f_1$ of $\bm X$ satisfies:
	\begin{equation}
		\label{eq:citelatter}f_1(\mx)\propto (\mathrm{det}(\Sigma))^{-\frac{1}{2}}\underline{f}\left((\mx-\bt)^{\top}\Sigma^{-1}(\mx-\bt)\right),\quad \mbox{for}\ \mbox{all}\ \mx\in\R^d.
	\end{equation}
	In the following, we will denote by $\fel$ the class of $d$-dimensional elliptically symmetric distributions satisfying some regularity conditions (see~\cref{sec:pfareell} for details) with location parameter $\bt$ and a positive definite matrix $\Sigma$, both unknown. Elliptically symmetric distributions form a rich class of multivariate probability measures which includes the spherical uniform, multivariate Gaussian, $t$, and logistic distributions among others. This class of distributions has attracted a lot of attention in statistical theory (see~\cite{Chmielewski1981,frahm2004generalized} and the references therein for a review) and applications such as in graphical modeling~\cite{Vogel2011}, mathematical finance~\cite{hodgson2002testing}, etc.

	\begin{theorem}[Elliptically symmetric case]\label{prop:areell}
		Suppose condition \eqref{eq:empgrid} and assumption \eqref{eq:nullrhas} hold. Then the following conclusions hold for the hypothesis testing problem \eqref{eq:twosamloc} in the usual asymptotic regime~\eqref{eq:usual}: 
		\begin{enumerate}
			\item[$(1)$] When the ERD is the spherical uniform with $\bJ(\mx)=\mx$ and $\nu\ =$ spherical uniform, 
			$$\inf_{\fel}\atr\geq \frac{81}{500}\cdot \frac{(\sqrt{2d-1}+1)^5}{d^2(\sqrt{2d-1}+5)}\geq 0.648.$$
			The explicit form of the radial density for which the inequality is attained is given in Section \ref{sec:pfareell}, see  equations \eqref{eq:ellip1} and \eqref{eq:ellip2} in the Appendix.
			\item[$(2)$] When the ERD is $\mathcal{N}(\bzr,\bm I_d)$ with $\bJ(\mx)=\mx$ and $\nu=\mathcal{N}(\bzr,\bm I_d)$,  
			$\inf_{\fel}\atr\geq 1.$
			Also, equality holds if and only if $\bX\sim \mathcal{N}(\bt,\Sigma)$ for some positive definite $\Sigma$ and $\bt\in\R^d$.
		\end{enumerate}
	\end{theorem}
	
	The proof of this result is given in~\cref{sec:pfarehotelling}. It should be noted that the above lower bounds match those for tests proposed in~\cite{Hallin2002,Oja2005,paindaveine2004}. However, the construction of the test statistics in~\cite{Hallin2002,Oja2005,paindaveine2004} specifically assumes the knowledge of elliptical symmetry of the underlying distribution. On the other hand, $\rhsc$ assumes no such knowledge on the underlying distribution, and yet successfully attains the same ARE lower bounds. Again we have the following multivariate counterparts of the classical univariate results~\cite{Chernoff1958,Hodges1956} for the class $\fel$.   
	
	\begin{itemize}
		
		\item {\it Multivariate Hodges-Lehmann phenomenon}: For the spherical uniform ERD, with $\bJ(\mx)=\mx$ and $\nu\ =$ spherical uniform, the lower bound on $\atr$ depends on the dimension $d$, which is always bounded below by 0.648. Note that when $d=1$ this lower bound equals $0.864$ as one would expect in light of~\cref{prop:areind}, because for $d=1$ the spherical uniform is simply $\mathrm{Unif}[-1,1]$, a location-scale shift of $\mathrm{Unif}[0,1]$. For $d\geq 2$, the lower bound is decreasing in $d$ and converges to $81/125=0.648$, as $d\to\infty$. The plot of the lower bound as a function of $d$ is shown in~\cref{fig:PowerGausscomp1}.

		\item {\it Multivariate Chernoff-Savage phenomenon}:  For the Gaussian ERD,  with $\bJ(\mx)=\mx$ and $\nu=\mathcal{N}(\bzr,\bm I_d)$, as in the independent components case, $\atr$ is lower bounded by $1$ in the worst case, irrespective of the dimension. Once again this showcases the strength of the proposed statistic in detecting location shifts and the benefits of the Gaussian ERD. 
		
	\end{itemize}
	
	\begin{remark}[Benefits of Gaussian ERD]\label{rem:Gausscore}
		~\cref{prop:Gaussare}, Theorems~\ref{prop:areind},~and~\ref{prop:areell} reveal that $\prsc$, when the ERD is Gaussian (with $\bJ(\mx)=\mx$ and $\nu=\mathcal{N}(\bzr,\bm I_d)$), automatically adapts to the underlying family $\{\mathcal{P}_{\bt}\}_{\theta\in\Theta}$ provided the family is Gaussian with unknown covariance, has independent components, or is elliptically symmetric, respectively. The same conclusion holds under the blind source separation model too (see~\cref{sec:bsep} for details). Therefore,  $\prsc$ with the Gaussian ERD yields a test which is always as efficient as the parametric Hotelling $T^2$ test, for the aforementioned families. The fact that the standard Gaussian is both spherically symmetric (has a density of the form~\eqref{eq:citelatter} with $\bt=\bzr$ and $\Sigma=\bm I_d$) and has independent components, plays a crucial role in the proofs. This shows the benefits (in terms of the ARE) of using the Gaussian ERD when the performance is compared with respect to the Hotelling's $T^2$ test. Interestingly, the advantage of using a  Gaussian reference distribution was also observed in simulations in the related problem of mutual independence testing in the recent paper \cite{shi2020rate}. Our paper corroborates this theme for the two-sample problem using the ARE framework, which provides a theoretical foundation for making an informed choice about the underlying reference distribution. 
		Analogous results for mutual independence testing are presented in~\cref{sec:indtest}.
	\end{remark}
	
	We conclude this section by noting that it is also possible to study AREs beyond location-shift alternatives. Another popular choice is the sequence of contamination alternatives which we discuss in detail in~\cref{sec:contamodel} (see in particular~\cref{prop:arecontam}). 
	
	\section{Two Sample Tests Based on Rank Kernel MMD}\label{sec:rankmmd}

	
	While the rank Hotelling $T^2$ (see~\cref{sec:ranktsq}) has appealing properties in terms of its Pitman efficiency (see~\cref{sec:locpow}), it isn't however consistent against all fixed nonparametric alternatives (see~\cref{theo:rhconsis} and  Propositions~\ref{prop:conloc}). The goal of this section, therefore, is to develop two-sample testing procedures that are exactly distribution-free, conistent against fixed alternatives, and still possess non-trivial Pitman efficiency. We will study a class of distribution-free two-sample tests based on a rank and score transformed version of the celebrated kernel maximum mean discrepancy (MMD); see~\cite{gretton2009fast,Gretton2012}. The main motivation behind choosing kernel MMD here as opposed to Hotelling $T^2$ earlier in the paper is that the former is known to be consistent against fixed alternatives unlike the latter. We will soon see that a rank and score transformation of kernel MMD retains consistency in spite of gaining distribution-freeness. 
	
	More formally, given a reference distribution $\nu$, a symmetric, non-negative definite kernel function $\bKe(\cdot,\cdot) : \R^d \times \R^d \rightarrow \R$ which is continuous Lebesgue a.e. in $\R^d\times\R^d$,  and a score function $\bJ: \R^d \rightarrow \R^d$, the {\it rank-based kernel two-sample statistic} is defined as:     
	\begin{align}\label{eq:kerankscmmd}
		&\;\;\grsc:=\frac{mn}{m+n}\left[ w_{m, n}^{(1)}  + w_{m, n}^{(2)} - b_{m, n}  \right], 
	\end{align}  
	where
	\begin{align}
		w_{m, n}^{(1)} & := \frac{1}{m(m-1)}\sum_{1\leq i\neq j \leq m} \mathsf{K}(\bJ(\hbR_{m,n}(\bX_i)),\bJ(\hbR_{m,n}(\bX_j))), \nonumber \\ 
		w_{m, n}^{(2)}  & := \frac{1}{n(n-1)}\sum_{1\leq i\neq j \leq n} \bKe(\bJ(\hbR_{m,n}(\bY_i)),\bJ(\hbR_{m,n}(\bY_j))), \nonumber \\ 
		b_{m, n} & := \frac{2}{mn}\sum_{1 \leq i \leq m}\sum_{1 \leq j \leq n} \bKe(\bJ(\hbR_{m,n}(\bX_i)),\bJ(\hbR_{m,n}(\bY_j))),
	\end{align}
	and, as before, $\hbR_{m,n}$ is empirical rank map based on the pooled sample $\mX_m \cup \mY_n$.

	\begin{remark}\label{rem:rankker}
		Note that when $\bJ(\mx)=\mx$, $\nu=\textrm{Unif}[0,1]^d$ and
		\begin{equation}\label{eq:ranker}
			\bKe(\mx,\my)=\lVert \mx\rVert + \lVert \my\rVert - \lVert \mx-\my\rVert,
		\end{equation}
		the statistic $\grsc$ is equivalent to the rank energy statistic presented in~{\em \cite{Deb19}}. In fact for $d=1$, it is equivalent to the two-sample Cram\'{e}r-von Mises test~\cite{AndersonCVM1962}, by~\cite[Lemma 4.4]{Deb19}.
	\end{remark}
	

	\begin{prop}[Distribution-freeness]\label{prop:dfreeker}
		Assume that $\mathrm{H}_0$ is true and $\mu_{1}=\mu_{2}\in \mathcal{P}_{\textrm{ac}}(\R^d)$. Then the distribution of $\grsc$ is universal, that is, it is free of $\mu_{1}=\mu_{2}$, for all $m,n\geq 1$.
	\end{prop}
	
	Using the above result we can readily obtain a finite sample distribution-free two-sample test which uniformly controls the Type I error. To this end, fix a level $\alpha\in (0,1)$ and let $c_{m,n}$ denote the upper $\alpha$ quantile of the universal distribution in~\cref{prop:dfreeker}. Consider the test:
	\begin{equation}\label{eq:testrankker}
		\ptsc:=\ind\left(\grsc\geq c_{m,n}\right).
	\end{equation}
	This test is exactly distribution-free for all $m,n\geq 1$ and uniformly level $\alpha$ under $\mathrm{H}_0$, that is, 
	\begin{equation}\label{eq:uniflevelker}
		\sup_{\mu_{1}=\mu_{2}\in \mathcal{P}_{\textrm{ac}}(\R^d)} \E\left[\ptsc\right]\leq\alpha.
	\end{equation}
	
	\subsection{Consistency and Asymptotic Null Distribution}\label{sec:consis}
	
	We now discuss the consistency of the test $\ptsc$. For this, we need the notion of a characteristic kernel, which is defined below.
	\begin{definition}[Characteristic kernel]\label{def:charker}
		A symmetric, non-negative definite kernel $\bKe(\cdot,\cdot)$ will be called characteristic if
		$\E_{\bZ \sim \mu_{1}} \bKe(\bZ,\cdot)=\E_{\bZ \sim \mu_{2}} \bKe(\bZ,\cdot) \quad \text{ if and only if } \quad \mu_{1}=\mu_{2},$
		for all $\mu_{1},\mu_{2}\in\mathcal{P}(\R^d)$ such that $\E_{\bZ \sim \mu_{1}} \bKe(\bZ,\bZ)<\infty$ and $\E_{\bZ \sim \mu_{2}} \bKe(\bZ,\bZ)<\infty$. Characteristic kernels play a central role in nonparametric tests based on reproducing kernel Hilbert spaces (see~\cite{gretton2009fast,Sejdinovic2013,Gretton2012}). Some examples include the kernel in \eqref{eq:ranker}, the {\it Gaussian kernel} where  $\bKe(\mx, \my) :=\exp(-\lVert \mx - \my \rVert^2)$, and the {\it Laplace kernel} where $\bKe(\mx, \my) :=\exp(-\lVert \mx -  \my \rVert_1)$ (here $\lVert \cdot \rVert_1$ denotes the usual $\ell_1$-norm on $\R^d$); see~\cite{sriperumbudur2008injective,sriperumbudur2010} for related results. A general strategy for constructing characteristic kernels from general semimetrics is also given in~\cite{Lyons13}.
	\end{definition}

	The following Theorem shows that $\ptsc$ yields a universally consistent test for problem~\eqref{eq:twosamden} whenever the $\bKe(\cdot,\cdot)$ used in~\eqref{eq:kerankscmmd} satisfies the characteristic property in~\cref{def:charker}. 
	
	\begin{theorem}\label{theo:consis} 
		Suppose assumption~\eqref{eq:empgrid} holds and the kernel $\bKe(\cdot,\cdot)$ is characteristic. Recall that $\bX\sim\mu_1$ and $\bY\sim\mu_2$. Moreover, assume that 
		\begin{equation}\label{eq:kerassn1}
			\E[\bKe(\bJ(\Rl(\bX)),\bJ(\Rl(\bX)))]<\infty, \qquad  \E[\bKe(\bJ(\Rl(\bY)),\bJ(\Rl(\bY)))]<\infty,
		\end{equation} 
		\begin{equation}\label{eq:kerassn2}
			\limsup\limits_{N\to\infty} \frac{1}{N}\sum_{i=1}^N \bKe(\bJ(\bh_i^d),\bJ(\bh_i^d))\leq \int \bKe(\bJ(\bz),\bJ(\bz))\,d\nu(\bz).
		\end{equation}	
		Then for problem \eqref{eq:twosamden} in the usual asymptotic regime~\eqref{eq:usual},  $\lim_{m,n\to\infty} \E\left[\ptsc\right]=1$
		provided $\mu_{1}\neq\mu_{2}$. 
	\end{theorem}
	
	The proof of~\cref{theo:consis} can be found in~\cref{sec:pfrankmmd}. The main idea is to use~\cref{theo:rankmapcon} to show that $\grsc$ converges in probability to $0$ if $\mu_1=\mu_2$ and to a strictly positive number if $\mu_1\neq\mu_2$.
	
	It is useful to compare~\cref{theo:consis} with~\cref{theo:rhconsis}. Note that under the standard assumptions,~\cref{theo:consis} yields consistency of $\ptsc$ for all fixed alternatives, which is a considerably larger class than the alternatives in~\cref{theo:rhconsis} for the test $\prsc$.
	
	Condition~\eqref{eq:kerassn2} can be verified in the same way as condition~\eqref{eq:nullrhas} (see~\cref{sec:verbd}) and so we skip the details for brevity. Note that condition~\eqref{eq:kerassn1} always holds if $\bKe$ is bounded, which in particular, includes the Gaussian and the Laplace kernels discussed above. It also holds for the rank energy test (see~\cref{rem:rankker}) 
	because the rank maps lie in $[0, 1]^d$. Another important case is $\ptsc$ with the \emph{van der Waerden} score:
	\begin{equation}\label{eq:vanscore}
		\bJ(\mx):= F^{-1}_{\chi_d}(\lVert\mx\rVert)\frac{\mx}{\lVert \mx\rVert} \bm{1}(\mx\neq\bzr).
	\end{equation}
	and the spherical uniform reference distribution (for $\nu$). Here $F_{\chi_d}(\cdot)$ is the distribution function of a $\sqrt{\chi_d^2}$ random variable. For ease of exposition, suppose that $\bKe(\cdot,\cdot)$ is the kernel in~\eqref{eq:ranker}. Then by the triangle inequality,
	$$\E[\bKe(\bJ(\Rl(\bX)),\bJ(\Rl(\bX)))]\lesssim \E \lVert \bJ(\Rl(\mathbf{Z}))\rVert = \E \lVert \mathcal{N}(0,I_d) \rVert <\infty,$$ 
	where $\mathbf{Z}\sim \lambda \mu_{1}+(1-\lambda)\mu_{2}$. This establishes the consistency of rank-based kernel two-sample tests which use the van der Waerden score,  provided~\eqref{eq:kerassn2} holds. This was presented as a conjecture (in the context of independence testing) in~\cite{shi2020rate}. We are now able to verify their conjecture using similar techniques as in the proof of~\cref{theo:consis}. We discuss their conjecture in greater detail in~\cref{sec:comparelit}.
	
	\begin{remark}[Moment assumptions and comparisons with kernel MMD]\label{rem:momasn}
		For usual kernel MMD (see~\cite{Gretton2012}), the condition for consistency is $\E \bKe(\bX,\bX)<\infty$ and $\E \bKe(\bY,\bY)<\infty$; c.f.~\eqref{eq:kerassn1} above. For unbounded kernels, such as  the one in~\eqref{eq:ranker}, this imposes additional moment conditions on $\mu_1$ and $\mu_2$. In contrast, if we choose $\bJ\#\nu$ to be a compactly supported absolutely continuous distribution, then~\eqref{eq:kerassn1} will be satisfied for all continuous, albeit unbounded kernels $\bKe(\cdot,\cdot)$, without imposing further moment conditions on $\mu_1$ and $\mu_2$.
	\end{remark}
	
	Next, we move on to the asymptotic null distribution of $\grsc$ (see~\eqref{eq:kerankscmmd}). We have already shown that $\grsc$ is  distribution-free under the null hypothesis $\mathrm{H}_0$ (see~\cref{prop:dfreeker}). In the subsequent theorem, we will further show that under $\mathrm{H}_0$, $\grsc$ is $O_p(1)$ and obtain the limiting distribution explicitly. Towards this end, we first set up some notation. Define  
	$$\tilde{\mathsf{K}}(\mathbf{u},\mathbf{v}):= \bKe(\mathbf{u},\mathbf{v})-\mathbb{E}_{\mathbf{V}}\bKe(\mathbf{u},\mathbf{V})-\mathbb{E}_{\mathbf{U}}\bKe(\mathbf{U},\mathbf{v})+\mathbb{E}_{(\mathbf{U},\mathbf{V})}\bKe(\mathbf{U},\mathbf{V}),$$ 
	where $\mathbf{U},\mathbf{V}$ are drawn independently from $\bJ\#\nu$ (the ERD, see~\cref{def:erd}). Assume that \begin{equation}\label{eq:seckerasn}    
		\limsup\limits_{N\to\infty} \frac{1}{N}\sum_{i=1}^N \bKe^2(\bJ(\bh_i^d),\bJ(\bh_i^d))\leq \int \bKe^2(\bJ(\bz),\bJ(\bz))\,d\nu(\bz)<\infty.
	\end{equation} 
	Then by~\cite[Theorem VI.23]{Reed1980}, there exists a countable collection of eigenvalues $\varpi_1, \varpi_2, \ldots, $ and corresponding orthonormal eigenfunctions $\Psi_1(\cdot), \Psi_2(\cdot), \ldots, $ from $\mathbb{R}^d$ to $\mathbb{R}$ with respect to the $\bJ\#\nu$ measure such that 
	\begin{align}\label{eq:conteigexp}
		\tilde{\mathsf{K}}(\mathbf{u},\mathbf{v})=\sum_{i=1}^{\infty}\varpi_i\Psi_i(\mathbf{u})\Psi_i(\mathbf{v}), \quad \mbox{in}\ L^2\left((\bJ\#\nu)\otimes(\bJ\#\nu)\right).
	\end{align}  
	
	\begin{theorem}[Null distribution]\label{theo:nullker}
		Assume that~\eqref{eq:usual},~\eqref{eq:empgrid}, and~\eqref{eq:seckerasn} hold. Let $\{G_i\}_{i\geq 1}$ be an i.i.d. sequence of $\mathcal{N}(0,1)$ random variables.  
		Then, under $\mathrm H_0$ as in~\eqref{eq:twosamden}: 
		\begin{equation}\label{eq:kernull1}
			\grsc\overset{w}{\longrightarrow}\sum_{i=1}^{\infty}\varpi_i (G_i^2-1).
		\end{equation}
	\end{theorem}
	
	The proof of~\cref{theo:nullker} can be found in~\cref{sec:pfrankmmd}. The proof proceeds in two main steps. Firstly, we show that \begin{equation}\label{eq:basestep1}
		\grsc-\grsr\overset{P}{\longrightarrow} 0,
	\end{equation} 
	where
	\begin{align}\label{eq:kerankscmmdor}
		&\;\;\grsr:=\frac{mn}{m+n}\left[ w_{m, n}^{(1),\mathrm{or}}  + w_{m, n}^{(2),\mathrm{or}} - b_{m, n}^{\mathrm{or}}  \right], 
	\end{align}  
	\begin{align}
		w_{m, n}^{(1),\mathrm{or}} & := \frac{1}{m(m-1)}\sum_{1\leq i\neq j \leq m} \mathsf{K}(\bJ(\Rmu(\bX_i)),\bJ(\Rmu(\bX_j))), \nonumber \\ 
		w_{m, n}^{(2),\mathrm{or}}  & := \frac{1}{n(n-1)}\sum_{1\leq i\neq j \leq n} \bKe(\bJ(\Rmu(\bY_i)),\bJ(\Rmu(\bY_j))), \nonumber \\ 
		b_{m, n}^{\mathrm{or}} & := \frac{2}{mn}\sum_{1 \leq i \leq m}\sum_{1 \leq j \leq n} \bKe(\bJ(\Rmu(\bX_i)),\bJ(\Rmu(\bY_j))).
	\end{align}
	This step crucially uses~\cref{theo:rankmapcon}. The second step is to show that the limiting distriburion of $\grsr$ matches the right hand side of~\eqref{eq:kernull1}. This proceeds using standard theory of \emph{degenerate} $U$-statistics; see e.g.,~\cite{Serfling1980,Reed1980}. On account of degeneracy in $\grsr$, the proof of the first step above requires considerably more work than the proof of a similar step in~\cref{theo:nullrankhotelling} (see the discussion around~\eqref{eq:delta_oracle}). 
	
	\begin{remark}[Comparison with $\rhsc$ and~\cref{theo:nullrankhotelling}]
		The limiting null distribution of $\grsc$ in~\cref{theo:nullker} depends on the ERD (see~\cref{def:erd}) through the $\varpi_i$'s, whereas, in the analogous result~\cref{theo:nullrankhotelling} for $\rhsc$, the limiting null is $\chi_d^2$, irrespective of the ERD.
	\end{remark}
	
	\begin{remark}[Comparison with null distribution of kernel MMD]
		The asymptotic null distribution in~\cref{theo:nullker} is exactly same as that of the usual kernel MMD (without using ranks) as presented in~\cite[Theorem 32]{Sejdinovic2013}, if $\mu_1=\mu_2=\bJ\#\nu$.
	\end{remark}
	
	By virtue of~\cref{theo:nullker}, it is possible to choose a universal cutoff for $\grsc$ under the null, for all large $m$, $n$, based on the limiting distribution in~\eqref{eq:kernull1}, thereby eliminating the need to obtain cutoffs for everey $m,n$.  Admittedly the limiting distribution does not have a simple form. There is extensive literature on approximating limiting distributions of the form~\eqref{eq:kernull1}; see e.g.,~\cite[Theorem 1]{gretton2009fast},~\cite{Gretton2012,bodenham2016comparison}. By virtue of distribution-freeness, this limiting distribution can be approximated using the same techniques verbatim, even before the data is observed.
	
	\subsection{Local Asymptotic Power}\label{sec:nullpow} 
	
	Having established consistency and distribution-freeness of $\ptsc$ in the previous subsections, we will now shift our attention to its local power against contiguous alternatives as in~\eqref{eq:twosamloc}. We will soon see that $\ptsc$ has non-trivial power against such alternatives, in addition to being consistent against all fixed alternatives and exactly distribution-free for all sample sizes. To the best of our knowledge, this combination of properties is not known to be satisfied for any of the other existing two-sample tests in the literature (also see~\cref{sec:compot}).
	
	\begin{theorem}[Asymptotics under contiguous alternatives]~\label{theo:Piteffrank} 
		Suppose that assumptions ~\eqref{eq:empgrid} and \eqref{eq:seckerasn} hold.  
		Then, under $\mathrm{H}_1$ from model~\eqref{eq:twosamsmooth}, the following holds, as $N\to\infty$ in the usual asymptotic regime~\eqref{eq:usual}:  
		$$\grsc\overset{w}{\longrightarrow}\sum_{i=1}^{\infty}\varpi_i\left[\left(G_i+\sqrt{\lambda(1-\lambda)}\mathbb{E}\left[\Psi_i(\bJ(\Rmu(\mathbf{X})))\bh^{\top}\boldsymbol{\eta}(\bX,\bt_0)\right]\right)^2-1\right],$$ 
		where $G_1, G_2, \ldots, $ are i.i.d. standard Gaussian random variables and the eigenvalues $\varpi_1, \varpi_2, \ldots, $ and the eigenfunctions $\Psi_1(\cdot), \Psi_2(\cdot), \ldots$ are as defined in \eqref{eq:conteigexp}. 
	\end{theorem}
	
	The proof of the above result can be found in~\cref{sec:pfrankmmd}. The result above shows that $\ptsc$ has non-trivial asymptotic power against $O(1/\sqrt N)$ alternatives and, as a consequence, has non-trivial Pitman efficiency (see \cref{def:asympeff} in the Appendix for further details), in addition to being distribution-free and computationally feasible (see~\cref{sec:compasgn}). For better understanding, fix $\bh^*:=\frac{\bh}{\lVert \bh\rVert}$ and let $\lVert \bh\rVert\to\infty$. Then, if $\varpi_i>0$ for all $i$, as a consequence of~\cref{theo:Piteffrank}, we have
	\begin{equation}\label{eq:powk}
		\lim\limits_{\lVert \bh \rVert\to\infty}\lim\limits_{N\to\infty}\E_{\mathrm{H}_1}[\ptsc]=1.
	\end{equation}
	While distribution-free testing of two multivariate distributions has a long history which has fostered renewed interest in light of modern applications, to the  best of our knowledge, none of the previously proposed distribution-free tests satisfy the three aforementioned properties simultaneously (see~\cref{sec:compot}).  
	
	\section*{Acknowledgments}
	The authors would like to thank Marc Hallin for numerous insightful comments that greatly improved the quality and the presentation of the paper. We are also grateful to Johan Segers for pointing out an error in an earlier version of the paper.

		\bibliography{biblio}

		\appendix

\begin{center}
\large{\textbf{Appendix}}
\end{center}

	\medskip
	
	In this Appendix section we will present the following: 
	\begin{itemize}
		\item \cref{sec:compare} contains useful implications of our results  in the context of the multivariate independence testing problem. This, in particular, addresses an open problem raised in~\cite{shi2020rate}.
		\item \cref{sec:auxdet} contains auxiliary technical details that we skipped for brevity. This includes (a) analyzing asymptotic properties of our proposed tests under contamination alternatives to augment our results on location alternatives (see~\cref{sec:contamodel}), (b) obtaining a Chernoff-Savage~\cite{Chernoff1958} ARE lower bound for our rank Hotelling $T^2$ test (see~\eqref{eq:testtrank}) under a blind source separation model (see~\cref{sec:bsep}), (c) verifying assumption~\eqref{eq:nullrhas} for deterministic sequences such as quasi-Monte Carlo points (see~\cref{sec:verbd}),  (d) summarizing the computational complexity of our testing procedures from Sections \ref{sec:ranktsq} and \ref{sec:rankmmd} (see~\cref{sec:compasgn}), and (e) comparing our multivariate rank-based tests to existing asymptotically distribution-free tests in the literature (see~\cref{sec:compot}).
		\item \cref{sec:pfmain} contains all the proofs of our results, along with the proofs of additional results from Sections \ref{sec:compare}, \ref{sec:contamodel}, and \ref{sec:bsep}. Further, in~\cref{sec:lowerbound}, we present some asymptotic minimax lower bound results in the context of the testing problem \eqref{eq:twosamsmooth} which shows the rate optimality of our proposed procedures.
		\item \cref{sec:sim} contains detailed simulation studies that support our theoretical results. In particular, we numerically demonstrate (a) the multivariate Hodges-Lehmann \cite{Hodges1956} and Chernoff-Savage \cite{Chernoff1958} phenomena (see~\cref{sec:numill}), (b) the finite sample power comparisons between the rank Hotelling $T^2$ test (see~\eqref{eq:testtrank}) with different ERDs (see~\cref{def:erd}) and the usual Hotelling $T^2$ test (see~\cref{sec:finper}), (c) the consistency of rank Hotelling $T^2$ test (see~\eqref{eq:testtrank}) beyond location alternatives where usual Hotelling $T^2$ fails (see~\cref{sec:beyondloc}), (d) the finite sample power comparisons between the rank MMD test (see~\eqref{eq:testrankker}) with different ERDs (see~\cref{def:erd}) and the usual MMD and energy tests (see~\cref{sec:emmdrank}), and (e) the power comparisons between the same tests as above but now in the high-dimensional regime (see~\cref{sec:highdpower}). 
	\end{itemize}
	
	\section{Broader scope}\label{sec:compare} 
	
	In this section, we will illustrate the broader scope of our techniques to other nonparametric testing problems, by establishing analogous results in the context of mutual independence testing and discussing connections and refinements to related methods in recent literature. In particular, in Section \ref{sec:indtest} we  construct a class of distribution-free nonparametric tests of independence which are natural multivariate analogs of Spearman's rank correlation~\cite{spearman1904proof} and enjoy favorable ARE properties similar to $\rhsc$. In Section \ref{sec:comparelit} we discuss how our techniques provide direct improvements of the results in some related papers such as~\cite{shi2020rate,hallin2020fully,Deb19,hallin2020center,hallin2020rank}, including the resolution of an open question from~\cite{shi2020rate}.
	
	\subsection{Applications to Independence Testing}\label{sec:indtest}
	Suppose $(\bX_1,\bY_1),\ldots ,(\bX_n,\bY_n)$ are i.i.d. observations from  $\mu\in\mathcal{P}(\R^{d_1+d_2})$ with absolutely continuous marginals $\mu_1$ and $\mu_2$ (note that $\mu$ need not be absolutely continuous). We are interested in the following test of independence problem:
	\begin{equation}\label{eq:indep}
		\mathrm{H}_0:\bX_1\!\perp\!\!\!\perp\bY_1 
		\qquad \mathrm{versus} \qquad \mathrm{H}_1:\bX_1 \not\!\perp\!\!\!\perp \bY_1 .
	\end{equation}
	In other words, we want to test the hypothesis $\mathrm{H}_0: \mu=\mu_1\otimes\mu_2$ versus $\mathrm{H}_1: \mu \neq \mu_1\otimes\mu_2$. This is the classical {\it multivariate mutual independence testing problem} which has received enormous attention in the past hundred years (see~\cite[Chapters 1 and 8]{hollander2013nonparametric} and the references therein) with applications in finance~\cite{lu2009financial}, statistical genetics~\cite{liu2010versatile}, survival analysis~\cite{martin2005testing}, etc. 
	
	When $d_1=d_2=1$, the earliest attempt at problem~\eqref{eq:indep} is the Pearson's correlation~\cite{pearson1920notes} which can only detect linear association between two variables. This was soon extended through the classical Spearman's rank correlation coefficient~\cite{spearman1904proof} which can detect any monotonic association between the variables and has the additional benefit of being exactly distribution-free under $\mathrm{H}_0$ for all sample sizes. Since then rank-based distribution-free correlation measures in the univariate case have received a lot of attention (see, for example,~\cite{kendall1938new,hoeffding1948non,blum1961distribution,Bregsma2014} and the references therein).  
	
	In the multivariate setting, the earliest test for problem~\eqref{eq:indep} is probably due to Wilks~\cite{wilks1935independence} which is constructed using the Gaussian likelihood ratio test statistic, and reduces to Pearson's correlation for $d_1=d_2=1$. In this sense, Wilks' test is the natural multivariate analog of Pearson's correlation. Since the advent of Wilks' test, a number of other multivariate tests have been proposed including  the Friedman-Rafsky test based on geometric graphs~\cite{friedman1983} and the celebrated distance covariance test~\cite{Gabor2007} (see~\cite{Josse2016} for a comprehensive survey of other procedures). However, none of these proposals are exactly distribution-free under $\mathrm{H}_0$, the chief hurdle once again, being the lack of a canonical ordering in $\R^d$, for $d\geq 2$. This gap in the literature was recently bridged in a series of works~\cite{Deb19,shi2020distribution,shi2020rate,deb2020measuring} where multivariate ranks based on optimal transport were used (as we did in~\cref{sec:Rank-Hotelling}) to construct multivariate, nonparametric, exactly distribution-free tests for~\eqref{eq:indep}. However, none of these papers provide any explicit expressions for the ARE of their tests against natural counterparts, and consequently do not guarantee any ARE lower bounds. 
	
	The main goal of this section is to overcome the aforementioned gap in the optimal transport based independence testing literature by constructing a multivariate version of Spearman's rank correlation coefficient that is exactly distribution-free and has high ARE compared to the Wilk's test, the natural multivariate analog of Pearson's correlation coefficient.
	
	\subsubsection{Multivariate Spearman's Correlation}\label{sec:msc}
	Before defining our statistic, let us fix some notation. Let $\nu_1\in\mathcal{P}_{\textrm{ac}}(\R^{d_1})$ and $\nu_2\in\mathcal{P}_{\textrm{ac}}(\R^{d_2})$ be two reference distributions. Let $\hbR_1(\bX_1),\ldots ,\hbR_1(\bX_n)$ denote the multivariate ranks of $\bX_1,\ldots ,\bX_n$ constructed as in~\eqref{eq:empopt}~and~\eqref{eq:defemprank} with a fixed grid $\{\bh_{1,i}^{d_1}\}_{i\in [n]}$ satisfying~\eqref{eq:empgrid} with $\nu=\nu_1$. Construct $\hbR_2(\bY_1),\ldots ,\hbR_2(\bY_n)$ using  $\{\bh_{2,i}^{d_2}\}_{i\in [n]}$ analogously, where $\{\bh_{2,i}^{d_2}\}_{i\in [n]}$ is a discretization of $\nu_2$.  Also, given any $d_1\times d_2$ matrix $\bm H$, let $\mbox{\textrm{vec}}(\bm H)$ be the vector obtained by unlisting the entries of $\bm H$ row-wise. For example, 
	$$\bm H=\begin{pmatrix} 1 & 3 & 5 \\ 2 & 4 & 6 \end{pmatrix} \;\;\implies \;\; \mbox{\textrm{vec}}(\bm H)=\begin{pmatrix} 1 & 3 & 5 & 2 & 4 & 6\end{pmatrix}.$$ 
	Next, for any two matrices $\bm H_1$ and $\bm H_2$, we will use $\bm H_1\otimes \bm H_2$ to  denote the standard Kronecker product. Finally, let $\bJ_1:\R^{d_1}\to\R^{d_1}$, $\bJ_2:\R^{d_2}\to\R^{d_2}$ be two injective, continuous score functions and assume that the ERDs, $\bJ_1\#\nu_1$ and $\bJ_2\#\nu_2$ both satisfy~\cref{assumption3} with positive definite covariance matrices $\Serd^{(1)}$ and $\Serd^{(2)}$.
	
	Based on the above notation, setting $\nu:=(\nu_1,\nu_2)$ and $\bJ:=(\bJ_1,\bJ_2)$, our version of multivariate Spearman's correlation is given as:
	\begin{equation}\label{eq:rank_spearman}
		\rhnk:=\left\lVert \left(\sn\otimes \st\right)^{-\frac{1}{2}}\mbox{\textrm{vec}}\left(\frac{1}{\sqrt{n}}\sum_{i=1}^n (\bJ_1(\hbR_1(\bX_i))-\ron)(\bJ_2(\hbR_2(\bY_i))-\rtw)^{\top}\right)\right\rVert^2,
	\end{equation}
	where $\ron:=\frac{1}{n} \sum_{i=1}^n \hbR_1(\bX_i)$ and $\rtw:= \frac{1}{n}\sum_{i=1}^n \hbR_2(\bY_i)$.

	\begin{remark}[Extension of Spearman's rank correlation]\label{rem:conspear}
		When $d_1=d_2=1$ and the fixed grids $\{\bh_{1,i}^{d_1}\}_{i\in [n]}$, $\{\bh_{2,i}^{d_2}\}_{i\in [n]}$ are both chosen as $\{i/n\}_{i\in [n]}$, then $\rhnk$ is the same as the squared Spearman's rank correlation coefficient. Moreover, $\rhnk$ is exactly distribution-free under $\mathrm{H}_0$ as shown in~\cref{prop:dfreeind} below. In this sense, $\rhnk$ is a multivariate extension of the classical Spearman's rank correlation coefficient.
	\end{remark}
	
	Our first result shows that $\rhnk$ is distribution-free under $\mathrm{H}_0$. This follows directly from~\cite[Proposition 2.2]{Deb19} and is formalized in the following proposition:
	\begin{prop}\label{prop:dfreeind}
		Under $\mathrm{H}_0$ as specified in~\eqref{eq:indep}, $\rhnk$ is distribution-free for all $n\geq 1$, that is, its distribution is free of $\mu_1$ and $\mu_2$.
	\end{prop}
	Using the above result we can readily obtain a finite sample distribution-free independence test. Fix a level $\alpha\in (0,1)$ and let $c_n$ denote the upper $\alpha$ quantile of the universal distribution in~\cref{prop:dfreeind}. Consider the test function:
	\begin{equation}\label{eq:testindrank}
		\prnk:=\ind\left(\rhnk\geq c_{n}\right).
	\end{equation}
	This test is exactly distribution-free for all $n\geq 1$ and uniformly level $\alpha$ under $\mathrm{H}_0$, in the sense of~\eqref{eq:uniflevelt}. 
	From~\cref{prop:dfreeind}, it is clear that the asymptotic null distribution of $\rhnk$ should be free of $\mu_1$ and $\mu_2$. In the following theorem, we make this explicit.
	
	\begin{theorem}\label{theo:nullrhnk}
		Suppose the fixed grids $\{\bh_{1,i}^{d_1}\}_{i\in [n]}$, $\{\bh_{2,i}^{d_2}\}_{i\in [n]}$ satisfy~\eqref{eq:nullrhas} with score functions $\bJ_1(\cdot)$, $\bJ_2(\cdot)$ and reference distributions $\nu_1$, $\nu_2$, respectively. Then under $\mathrm{H}_0$ as in~\eqref{eq:indep}, 
		$$\rhnk\overset{w}{\longrightarrow}\chi^2_{d_1 d_2}.$$
	\end{theorem}
	
	The proof of~\cref{theo:nullrhnk} is given in~\cref{sec:pfindtest}. The simple limiting null distribution of $\rhnk$ as presented in~\cref{theo:nullrhnk} can be used to calibrate the statistic $\rhnk$ to obtain an asymptotically level $\alpha$ test.

	\subsubsection{ARE against Wilks' Test~{\em \cite{wilks1935independence}}}\label{sec:arewilks}
	
	In this section, we will derive lower bounds on the ARE of $\rhnk$ against Wilks' statistic over various classes of multivariate distributions. To begin with, we recall the Wilks' test statistic: 
	\begin{small}
		\begin{equation}\label{eq:wilkstest}
			R_n:=n\log\left({\frac{\mbox{det}(\bm{Q}_{1,1})\cdot \mbox{det}(\bm{Q}_{2,2})}{\mbox{det}(\bm Q)}}\right),\ \bm{Q}:=\begin{pmatrix} \bm{Q}_{1,1} & \bm{Q}_{1,2} \\ \bm{Q}_{2,1} & \bm{Q}_{2,2}\end{pmatrix}:=\frac{1}{n}\sum_{i=1}^n \begin{pmatrix} \bX_i-\obx \\ \bY_i-\oby\end{pmatrix}\begin{pmatrix} \bX_i-\obx \\ \bY_i-\oby\end{pmatrix}^{\top},
		\end{equation}
	\end{small}
	$\obx:=\frac{1}{n}\sum_{i=1}^n \bX_i$ and $\oby:= \frac{1}{n}\sum_{i=1}^n \bY_i$. Under $\mathrm{H}_0$ as in~\eqref{eq:indep}, it is well-known that $R_n\overset{w}{\longrightarrow}\chi^2_{d_1d_2}$, which is the same as the limiting distribution obtained in~\cref{theo:nullrhnk}. 
	
	To compare the local power of $\rhnk$ against $R_n$, we need to fix a notion of local alternatives as in~\cref{sec:locpow}. One of the most popular choices of local alternatives in mutual independence testing is the sequence of Konijn alternatives (see~\cite{Konijn1956,Taskinen2004,Hallin2008,gieser1997}) defined below.
	
	\begin{definition}[Konijn alternatives]\label{def:konijn} Suppose $\bX'_1$ and $\bY'_1$ are independent random vectors with Lebesgue absolutely continuous distributions $\mu_1$ and $\mu_2$. Define,
		\begin{equation}\label{eq:trkonijn}\begin{pmatrix} \bX_1 \\ \bY_1\end{pmatrix} := \begin{pmatrix} (1-\delta n^{-\frac{1}{2}})\bm I_{d_1} & \delta n^{-\frac{1}{2}} \bm M_{d_1\times d_2} \\ \delta n^{-\frac{1}{2}} \bm M_{d_1\times d_2}^{\top} & (1-\delta n^{-\frac{1}{2}})\bm I_{d_2}\end{pmatrix} \begin{pmatrix} \bX'_1 \\ \bY'_1 \end{pmatrix},\end{equation}
		where $\delta>0$ and $\bm M_{d_1\times d_2}$ is a $d_1\times d_2$ dimensional matrix. Note that if $\delta=0$, then $\bX_1$ and $\bY_1$ are independent. Therefore, problem~\eqref{eq:indep} can be restated in this framework as:
		\begin{equation}\label{eq:locindep}
			\mathrm{H}_0: \delta=0 \quad \mbox{versus} \quad \mathrm{H}_1: \delta\neq 0.
		\end{equation}
		We will further assume~\cite[Assumption 5.1]{shi2020rate}. This assumption ensures that the probability measures under $\mathrm{H}_0$ and $\mathrm{H}_1$ as in~\eqref{eq:locindep} are contiguous to each other.
	\end{definition}
	
	In the sequel, we will write $\ati$ to denote the ARE of $\rhnk$ with respect to $R_n$ under the Konijn alternatives defined above, that is, when $(\bX_1,\bY_1)$ are generated according to~\eqref{eq:trkonijn}. In order to obtain lower bounds for $\ati$, we can safely assume that both $\bX'_1$ and $\bY'_1$ have finite variances, as otherwise $\ati$ is trivially $\infty$. We also note that problem~\eqref{eq:locindep} is affine invariant, in the following sense: If we replace $\bX'_1,\bY'_1$ by $\bm A(\bX'_1-\bm a)$ and $\bm B(\bY'_1-\bm b)$ in~\eqref{eq:trkonijn}, where $\bm A, \bm B$ are invertible matrices of dimensions $d_1\times d_1$ and $d_2\times d_2$, respectively, and $\bm a\in \R^{d_1}$, $\bm b\in \R^{d_2}$, then $\bX_1$ and $\bY_1$ so obtained are still independent if and only if $\delta=0$. Therefore, without loss of generality we can assume: 
	\begin{assumption}\label{assumption4}
		$\E \bX'_1=\bzr_{d_1}$, $\E \bY'_1=\bzr_{d_2}$ and $\mbox{Var}[\bX'_1]=\bm I_{d_1}$ and $\mbox{Var}[\bY'_1]=\bm I_{d_2}$.
	\end{assumption}
	
	We now present our main theorem of this section in which we provide Chernoff-Savage~\cite{Chernoff1958} type lower bounds for problem~\eqref{eq:locindep}. It shows that, even in the worst case, $\rhnk$ is at least as efficient as the Wilks' statistic $R_n$ in~\eqref{eq:wilkstest}, for all fixed dimensions $d_1$ and $d_2$.
	\begin{theorem}[Lower bounds on $\ati$ with Gaussian ERD]\label{prop:indepeff}
		Suppose the conditions in~\cref{theo:nullrhnk} and Assumption~\ref{assumption4} hold. Then, with $\nu_1=\mathcal{N}(\bzr_{d_1},\bm I_{d_1})$, $\nu_2=\mathcal{N}(\bzr_{d_2},\bm I_{d_2})$, $\bJ_1(\mx)=\mx$, and $\bJ_2(\my)=\my$ (both ERDs are standard Gaussians of appropriate dimensions), the following holds: 
		\begin{equation}\label{eq:indeplb}\inf_{\mu_1,\mu_2\in \fnd\cup\fel}\ati\geq 1,\end{equation}
		where $\fnd$ and $\fel$ are defined as in~\cref{sec:arehotelling}.  Furthermore, equality holds in~\eqref{eq:indeplb} if and only if both $\mu_1$ and $\mu_2$ are standard Gaussians of  appropriate dimensions.
	\end{theorem}
	The proof of~\cref{prop:indepeff} can be found in~\cref{sec:pfindtest}. \cref{prop:indepeff} shows the benefits of using $\rhnk$ with Gaussian ERDs for problem~\eqref{eq:locindep}. Note that the bound in~\eqref{eq:indeplb} is free of the dimensions $d_1$, $d_2$ and also free of the matrix $\bm M$ in~\cref{def:konijn}. The benefits of using a Gaussian ERD were also noted in our analysis of $\rhsc$ (see~\cref{rem:Gausscore}). In fact, we believe that advantages of using a Gaussian ERD are ubiquitous and should extend to other natural multivariate rank-based procedures for testing symmetry, significance of regression coefficients, etc. We would also like to point out that while~\cref{prop:indepeff} provides Chernoff-Savage type lower bounds, it is also possible to have Hodges-Lehmann type lower bounds in this setting. In fact, the Hodges-Lehmann type bounds match those obtained in~\cite[Proposition 2]{Hallin2008}. These bounds have a complicated form and they are \emph{strictly smaller} than $1$ for all fixed dimensions $d_1$ and $d_2$. We  omit those results for brevity.
	
	\subsection{Implications of our results to existing  literature}\label{sec:comparelit}
	In this section, we discuss the implications of our results to other existing papers that deal with nonparametric testing problems using optimal transport. As stated in the Introduction, our goals are quite different from the existing papers, but nevertheless our results have interesting consequences that help resolve open problems in existing literature. We illustrate this by using two examples --- Hallin et al.~\cite{hallin2020fully} and Shi et al.~\cite{shi2020rate}, although similar comments apply to~\cite{Deb19,shi2020distribution,hallin2020center}. \medskip
	
	\noindent \textit{Comparison with Hallin et al.}~\cite{hallin2020fully}: In the recent work~\cite{hallin2020fully}, a basic version of the rank Hotelling $T^2$ statistic $\rhsc$ was presented in passing in~\cite[Page 25]{hallin2020fully} for the special case when the reference distribution is spherical uniform and for a specific choice of the set $\{\bh_1^d,\ldots ,\bh_N^d\}$, such that $\sum_{i=1}^N \bh_i^d=\bzr$ (a slightly stringent requirement that may be hard to satisfy for generic sequences). However, the authors did not study its theoretical properties, such as consistency and asymptotic efficiency. Our results, when applied to the special case proposed in~\cite{hallin2020fully}, imply the consistency of their corresponding test (see~\cref{theo:rhconsis},~Propositions~\ref{prop:conloc}~and~\ref{prop:conlocontam}), and can be used to derive explicit ARE expressions (see~\cref{theo:locpowermain1}) and, most importantly, lower bounds for the AREs (see Theorems~\ref{prop:areind}~and~\ref{prop:areell}). 
	
	
	In this section, we will focus on the ARE results that can be obtained for the special case presented in~\cite{hallin2020fully} and highlight some of the additional benefits to be gained by adopting our general framework. In fact, for the particular case proposed in~\cite{hallin2020fully},  our arguments directly imply the lower bound in~\cref{prop:areell}-(1). The conclusion in~\cref{prop:areell}-(2) also follows if the statistic in~\cite{hallin2020fully} is transformed using the van der Waerden score, given in~\eqref{eq:vanscore}. The corresponding conclusion is formalized in the following proposition. 
	\begin{prop}\label{prop:hallinlb}
		Consider the same assumptions as in~\cref{prop:areell} and recall the definition of $\fel$. Suppose the reference distribution $\nu$ is the spherical uniform and $\bJ(\cdot)$ is the van der Waerden score function  in~\eqref{eq:vanscore} (same combination as discussed  in~\cite{hallin2020fully}). Then,
		$$\inf_{\fel}\atr\geq 1$$
		where equality holds if and only if $\bX\sim\mathcal{N}(\bt,\Sigma)$ for some $\bt\in\R^d$ and positive definite $\Sigma$.
	\end{prop}
	
	On observing that the function $\bJ(\Rmu(\cdot))$ (with $\bJ(\cdot)$ and $\nu$ as above) is in fact the optimal transport map from $\bX$ to a standard normal, the proof of~\cref{prop:hallinlb} follows immediately from the proof of~\cref{prop:areell}-(2). While the statistic in~\cite{hallin2020fully} (that is, $\rhsc$  with spherical uniform reference distribution and a specific choice of $\{\bh_i^d\}_{i\in [n]}$) attains the same ARE lower bound as in~\cref{prop:areell} over the class of elliptically symmetric distributions, as shown above, when it comes to distributions with independent components (as in~\cref{prop:areind}), this special case proposed in~\cite{hallin2020fully} falls short. 
	This is because the optimal transport map from distributions with independent components to the spherical uniform (in the sense of~\cref{prop:Mccan}) is not explicit and consequently not analytically tractable. In fact, we believe that with the spherical uniform reference distribution, the same lower bounds as in~\cref{prop:areind} are no longer true for $d>1$. On the other hand, optimal transport maps from distributions with independent components to the standard Gaussian or $\mathrm{Unif}[0,1]^d$ distributions are tractable which we are able to exploit in our lower bound computations in~\cref{prop:areind}, thanks to our general framework. In fact, it is this flexibility that allows us to show that the standard Gaussian reference distribution is at least as efficient as Hotelling $T^2$ uniformly over both the class of distributions with independent components and those having an elliptically symmetric density. \medskip
	
	\noindent\textit{Comparison with Shi et al.}~\cite{shi2020rate}: We now discuss the implications of our results for the paper~\cite{shi2020rate}, where the authors consider the independence testing problem as introduced in~\eqref{eq:indep}. Recall the notation for multivariate ranks used in the beginning of~\cref{sec:msc} and let the fixed grids $\{\bh_{1,i}^{d_1}\}_{i\in [n]}$ and $\{\bh_{2,i}^{d_2}\}_{i\in [n]}$ be chosen as in~\cite[Page 9]{shi2020rate}. Using this notation, a prototypical example of a test statistic considered in~\cite{shi2020rate} would be the \emph{rank distance covariance} given as:
	\begin{align}\label{eq:rdcovariance}
		\Rdcov_n^2:=\frac{1}{n^2}\sum_{i, j} \Don_{i,j}\Dwo_{i,j} +\frac{1}{n^4} \left(\sum_{i, j} \Don_{i, j} \right) \left(\sum_{i,j}\Dwo_{i,j} \right)-\frac{2}{n^3}\sum_{i,j,k}\Don_{i,j} \Dwo_{i,k},
	\end{align}
	where
	$\Don_{i,j}:=\lVert\bJ_1(\hbR_{1}(\bX_i))-\bJ_1(\hbR_{1}(\bX_j))\rVert$ and $\Dwo_{i,j}:=\lVert \bJ_2(\hbR_{2}(\bY_i))-\bJ_2(\hbR_{2}(\bY_j))\rVert$
	for score functions $\bJ_1(\cdot)$ and $\bJ_2(\cdot)$. In~\cite{shi2020rate}, the authors conjecture that the test based on the above statistic with van der Waerden score function (see~\eqref{eq:vanscore}) is consistent, but a proof was not provided. Here, we answer this question in the affirmative by using~\cref{theo:rankmapcon} and techniques as in the proof of~\cref{theo:rhconsis}. This is formalized in the following proposition (see~\cref{sec:auxpf} for a proof).
	\begin{prop}\label{prop:rdcovcon}
		Assume that $\bR_1(\cdot)$, $\bR_2(\cdot)$ are the optimal transport maps from $\mu_1$ and $\mu_2$ (both are Lebesgue absolutely continuous) to reference distributions $\nu_1$ and $\nu_2$. Also, suppose $\E\lVert \bJ_1(\bR_1(\bX_1))\rVert^2<\infty$ and $\E\lVert \bJ_2(R_2(\bY_1))\rVert^2<\infty$. Then provided both $\{\bh_{1,i}^{d_1}\}_{i\in [n]}$ and $\{\bh_{2,i}^{d_2}\}_{i\in [n]}$ satisfy~\eqref{eq:empgrid},~\eqref{eq:nullrhas} with $\nu_1$ and $\nu_2$, respectively, the following hold as $n \rightarrow \infty$: 
		\begin{equation}\label{eq:rdcon}
			\Rdcov_n^2\overset{P}{\longrightarrow}\E\Big[\Donor_{1,2}\Dwor_{1,2}\Big]+\E\Big[ \Donor_{1,2}\Big]\E\Big[\Dwor_{1,2}\Big]-2\E\Big[\Donor_{1,2}\Dwor_{1,3}\Big],
		\end{equation}
		where $\Delta^{(k),\mathrm{or}}_{i,j}:=\lVert\bJ_k(\bR_{k}(\bX_i))-\bJ_k(\bR_{k}(\bX_j))\rVert$ for $1 \leq i, j \leq n$, $k=1,2$. Moreover, the right hand side of the display above equals $0$ if and only if $\mu=\mu_1\otimes \mu_2$.
	\end{prop}
	
	In fact, the same technique can be used to establish the consistency of the more general class of tests considered in~\cite{shi2020rate}, the details of which we omit for brevity. It is worth emphasizing that~\cref{theo:rankmapcon} is not restricted to any particular statistic, instead it establishes consistency for continuous functions of empirical rank maps. Consequently, it can be applied to establish the consistency of a wide variety of statistics such as those in~\cite{Deb19,hallin2020center,hallin2020fully}. 
	
	\section{Auxiliary Technical Details}\label{sec:auxdet}
	
	\subsection{Contamination model}\label{sec:contamodel}
	
	Another common class of nonparametric alternatives is the {\it contamination model} 
	where 
	\begin{align}\label{eq:fg_delta}
		f_2(\cdot)=(1-\delta)f_{1}(\cdot)+\delta g(\cdot), 
	\end{align}
	where $\delta \in [0,1)$ and $g\ne f_1$ is a probability density function with respect to the Lebesgue measure in $\R^d$. In this case, the testing problem \eqref{eq:twosamden} simplifies to: 
	\begin{equation}\label{eq:twosamloc_contam}
		\mathrm{H}_0:\delta=0  \qquad \mathrm{versus} \qquad \mathrm{H}_1:\delta\neq 0 .
	\end{equation} 
	In the same spirit as Proposition \ref{prop:conloc}, we can obtain a consistency result for $\prsc$ in the contamination model, which simplifies nicely when the contamination density $g(\cdot)$ is itself a location shift of $f_1(\cdot)$ as was studied in the seminal Hodges and Lehmann~\cite{Hodges1956} paper. 
	
	\begin{prop}[Consistency under contamination alternatives]\label{prop:conlocontam}
		Suppose the condition in \eqref{eq:empgrid},~\eqref{eq:nullrhas} and Assumption {\em\ref{assumption3}} hold with $\bJ(\mx)=\mx$. Then, for the testing problem \eqref{eq:twosamloc_contam} in the usual asymptotic regime~\eqref{eq:usual}, 
		\begin{align}\label{eq:consiscontam}
			\lim_{m,n\to\infty} \E_{\delta}\left[\prsc\right]=1, 
		\end{align}
		for any $\delta\in (0,1)$,
		provided $\E \Rl(\bX)\neq \E \Rl(\bW)$, where $\bW$ has density $g(\cdot)$. In particular, \eqref{eq:consiscontam} holds if $g(\cdot)=f_1(\cdot-\bD)$, for some $\bD\neq \bzr$,  provided $\Rl(\cdot)$ satisfies the same assumption as in~\cref{prop:conloc}.
	\end{prop}
	
	The proof of~\cref{prop:conlocontam} follows along the same lines as the proof of~\cref{prop:conloc}. We provide a sketch in~\cref{sec:pfconsistency}. We conclude this section with the following crucial observation. 
	
	\begin{remark}[No moment assumptions required]\label{rem:nomoment} Note that the moment assumptions in~\cref{theo:rhconsis},~\cref{prop:conloc}, and~\cref{prop:conlocontam} are not impositions on $\mu_1$ or $\mu_2$, but on the reference distribution $\nu$ and score function $\bJ(\cdot)$. For instance, if $\bJ(\mx)= \mx$ is the identity function and $\nu=\mathrm{Unif}[0,1]^d$, then the moment assumptions are always satisfied, even if the distributions $\mu_1$, $\mu_2$ do not have finite moments. In contrast, the Hotelling $T^2$ test  requires finite second moments of $\mu_1$ and $\mu_2$ for consistency. 
	\end{remark}
	
	The next natural step after establishing consistency is to study asymptotic relative efficiency under local perturbations of the mixing proportion in the contamination model \eqref{eq:fg_delta}. For this, suppose $f_2(\cdot)=(1-\delta)f_1(\cdot)+\delta g(\cdot)$ as in \eqref{eq:fg_delta} such that the following hold: 
	\begin{itemize}
		\item The support of $g$ is contained in that of $f_1(\cdot)$.
		\item $0<\E_{\mu_1}\left[\big(g(\bX)/f_1(\bX)-1\big)^2\right]<\infty$, for $\bX \sim \mu_1$.
	\end{itemize}
	Under these assumptions, we will consider the following sequence of hypotheses: 
	\begin{equation}\label{eq:twosamcont}
		\mathrm{H}_0:\delta = 0 \qquad \mathrm{versus} \qquad \mathrm{H}_1:\delta = h/\sqrt{N},
	\end{equation}
	for some $h\neq 0$.
	
	Having discussed ARE lower bounds for the testing problem in~\eqref{eq:twosamloc}, we now focus our attention on the contamination model discussed in~\eqref{eq:twosamcont}. Unfortunately, under this model, as in the univariate case  (see~\cite{Hodges1956}), such ARE lower bounds do not exist. To see this we first need the following proposition, which can be proved using the same argument as in the proof of~\cref{prop:Gaussare} (see~\cref{sec:auxpf} for a short proof).  
	\begin{prop}\label{prop:arecontam}
		Suppose $\bX$ has a well-defined positive definite covariance matrix, and $\bW\sim g$ (where $g(\cdot)$ is defined as in~\eqref{eq:fg_delta}) such that $\E\bW$ is finite. Then, under~\cref{assumption3}, the following conclusion holds in the usual asymptotic regime: 
		$$
		\atr=\frac{\Big\lVert\Sigma_{\textsc{ERD}}^{-\frac{1}{2}}\E_{\mathrm{H}_0}\left[\bJ(\Rmu(\bX))\left(\frac{g(\bX)}{f_{1}(\bX)}-1\right)\right]\Big\rVert^2}{\Big\lVert \Sigma^{-\frac{1}{2}}(\E\bW-\E\bX)\Big\rVert^2}.
		$$
	\end{prop}

	To see that the above expression of the ARE cannot have a non-trivial lower bound in general, suppose, for simplicity that the ERD is compactly supported on $[-M,M]^d$ for some $M>0$. Observe that,
	\begin{small}
		\begin{align*}
			\left\lVert\Serd^{-\frac{1}{2}}\E_{\mathrm{H}_0}\left[\bJ(\Rmu(\bX))\left(\frac{g(\bX)}{f_{1}(\bX)}-1\right)\right]\right\rVert^2\leq M^2\bo^{\top}\Serd^{-1}\bo \E_{\mathrm{H}_0}\left|\frac{g(\bX)}{f_{1}(\bX)}-1\right|\leq 2M^2\bo^{\top}\Serd^{-1}\bo.
		\end{align*}
	\end{small}
	On the other hand,
	$$\Big\lVert \Sigma^{-\frac{1}{2}}(\E\bW-\E\bX)\Big\rVert^2\gtrsim \lVert \E\bW-\E\bX\rVert^2,$$
	where the hidden constant above depend only on the maximum eigenvalue of $\Sigma$. Therefore, by making $\lVert \E\bW-\E\bX\rVert^2$ arbitrarily large, there is no hope of getting any lower bounds on $\atr$ in~\cref{prop:arecontam}.
	
	The same phenomenon happens in $d=1$ while comparing the Wilcoxon rank-sum test with Student's $t$, as was duly noted in Hodges and Lehmann~\cite{Hodges1956}. In fact, the authors in~\cite{Hodges1956} point out that the lack of sensitivity of rank-based procedures to extreme contamination is, in many cases, a blessing. This is because, contamination is often due to gross errors in observation resulting in outliers. In such cases, a small proportion of the observed data is expected to come from a very different distribution with potentially larger means than that of the signal distribution $\mu_{1}$. As this difference in mean between the signal and the outlier distributions grow larger and larger, the Hotelling $T^2$ becomes more and more efficient (see the ARE expression above) compared to $\rhsc$, even when the \emph{proportion of contamination} is fixed. This shows that the Hotelling $T^2$ test is more sensitive to outlier distribution than $\rhsc$ even when the proportion of outliers is small. This greater sensitivity to outliers for Hotelling $T^2$ makes it less robust as an inference procedure compared to $\rhsc$. 
	
	We conclude this section by providing the limiting distribution of $\grsc$ (see~\eqref{eq:kerankscmmd}) under local alternatives under model~\eqref{eq:fg_delta}. This result follows from~\cref{theo:Piteffrank} and is presented here primarily for completion. The significance of such limiting distributions has already been discussed in~\cref{sec:rankmmd}.
	
	\begin{prop}\label{prop:Piteffloc}
		Consider the same setting as in~\cref{theo:Piteffrank}. 
		Then, under $\mathrm{H}_1$ from model~\eqref{eq:fg_delta} and \eqref{eq:twosamcont}, the following holds, as $N\to\infty$:  
		$$\grsc\overset{w}{\longrightarrow}\sum_{i=1}^{\infty}\varpi_i\left[\left(G_i+h\sqrt{\lambda(1-\lambda)}\mathbb{E}\left[\Psi_i(\bJ(\Rmu(\mathbf{X})))\left(\frac{g(\bX)}{f_1(\bX)}-1\right)\right]\right)^2-1\right],$$ 
		where $G_1, G_2, \ldots, $ are i.i.d. standard Gaussian random variables and the eigenvalues $\mu_1, \mu_2, \ldots, $ and the eigenfunctions $\Psi_1(\cdot), \Psi_2(\cdot), \ldots$ are as defined in \eqref{eq:conteigexp}. 
	\end{prop}
	
	\subsection{Blind source separation model}\label{sec:bsep}
	Next, we look at the class $\fgen$ --- the class of generative models used in blind source separation~\cite{shlens2014tutorial,cardoso1998blind,comon2010handbook} and in independent component analysis (ICA) ~\cite{Chen2004,Samarov2004,bach2002kernel,hyvarinen2000independent}. To be precise, $\fgen$ includes the class of distributions which can be written as 
	$\bX_{d\times 1} {=} \bm{A}_{d\times d}\bm{W}_{d\times 1}$, where $\bm{A}_{d\times d}$ is an orthogonal matrix (unknown) and  $\bm{W}$ has independent components with Lebesgue density 
	$\tilde{f}(\bm{w})=\prod_{i=1}^d \tilde{f}_{i}(w_i)$, where 
	$\bm{w}=(w_1,\ldots ,w_d)$ and $\tilde{f}_1, \tilde{f}_2, \ldots, \tilde{f}_d$ are univariate densities. It is easy to check that the density of $\bm{X}$ is given by: 
	$$f_1(\mx)=\prod_{i=1}^d \tilde{f}_{i}\left(\sum_{j=1}^d a_{ji}x_j\right),\quad \mbox{for}\ \mbox{all}\ \mx=(x_1,\ldots ,x_d)\in\R^d,$$ 
	where $\bm{A}_{d\times d}=(a_{ij})_{1\leq i,j\leq d}$. In the following, we will obtain a \emph{Chernoff-Savage type result} under the assumption that the family of densities $\{f_1(\mx-\bt)\}_{\bt\in\R^d}$ is a smooth parametric model in the sense of~\cref{sec:locpow}. Note that the matrix  $\bm{A}_{d\times d}$ is unknown.   
	
	\begin{theorem}[Generative model for blind source mixing]\label{prop:areica}
		Suppose condition \eqref{eq:empgrid} and Assumption \eqref{eq:nullrhas} hold. Then the following conclusion holds for the hypothesis testing problem~\eqref{eq:twosamloc} in the usual asymptotic regime~\eqref{eq:usual}:   		
		$$\inf_{\fgen}\atr\geq 1,$$
		when $\bm J(\mx)=\mx$ and $\nu=\mathcal{N}(\bzr,\bm I_d)$. Further, equality holds if and only if $\bX_{d\times 1}{=}\bm{A}_{d\times d}\bm{W}_{d\times 1}$ where $\bm{A}_{d\times d}$ is orthogonal, $\bm{W}\sim \mathcal{N}(\bt,\Sigma)$ for some diagonal matrix with positive diagonal entries.		
	\end{theorem}
	The proof of the above result can be found in~\cref{sec:pfarehotelling}. \cref{prop:areica} can be viewed as a \emph{Chernoff-Savage type result} (see~\cite{Chernoff1958}) for the class of distributions $\fgen$. Here again, $\rhsc$ with the Gaussian reference distribution has an uniformly higher ARE against the Hotelling $T^2$ test.
	
	\subsection{Assumption~\eqref{eq:nullrhas} for Deterministic Sequences}\label{sec:verbd}
	In this section, we discuss how to verify~\eqref{eq:nullrhas} for some popular deterministic sequences and score function combinations. The case when $\bh_i^d$'s are sampled randomly has already been discussed in~\cref{rem:justbound}.
	
	\begin{example}[When ERD is uniform on the unit cube or spherical uniform]\label{ex:erdbounded}
		When the ERD is the $\mathrm{Unif}[0,1]^d$ distribution or the spherical uniform distribution, the natural choice in the literature is to choose $\bJ(\mx)=\mx$ and $\bh_1^d, \bh_2^d, \ldots, \bh_N^d$ in $[0,1]^d$ such that
		$$\frac{1}{N}\sum_{i=1}^N \delta_{\bh_i^d}\overset{w}{\longrightarrow}\nu$$
		where, depending on the case considered, $\nu$ is either $\mathrm{Unif}[0,1]^d$ or the spherical uniform distribution. For the $\mathrm{Unif}[0,1]^d$ distribution, popular choices of the $\{\bh_i^d\}_{i \in [N]}$ include the regular grid or quasi-Monte Carlo sequences such as the Halton sequence (see~\cite[Section D.3]{Deb19} for a discussion). For the spherical uniform distribution, a suitable choice was constructed explicitly in~\cite{delbarrio2019} which has since been used in~\cite{shi2020distribution,shi2020rate}. For these choices, $\{\bJ(\bh_i^d)\}_{i \in [N]}$ and $\{K(\bJ(\bh_i^d),\bJ(\bh_i^d))\}_{i \in [N]}$ are uniformly bounded (assuming $K(\cdot,\cdot)$ is continuous on $[0,1]^d$). Therefore, the assumption~\eqref{eq:nullrhas} follows directly using the dominated convergence theorem.
	\end{example}
	
	\begin{example}[When ERD is standard Gaussian]\label{ex:erdunbounded}
		A natural way to obtain the Gaussian ERD would be to start with $\nu=\mathrm{Unif}[0,1]^d$ or $\nu$ equals to the spherical uniform distribution and choose $\bJ(\cdot)$ appropriately such that $\bJ\#\nu$ is standard Gaussian.
		
		\begin{enumerate}
			\item \emph{When $\nu=\mathrm{Unif}[0,1]^d$}: Suppose we choose $\{\bh_i^d\}_{i \in [N]}$ as the standard Halton sequence (see~\cite{Niederreiter1992} for details on its construction). For $\mx=(x_1,\ldots ,x_d)$, set $\bJ(\mx):=(\Phi^{-1}(x_1),\ldots ,\Phi^{-1}(x_d))$. Also, write $\bh_i^d=(h_{i,1}^d, \ldots ,h_{i,d}^d)$ for $i\in [N]$. Then by condition \eqref{eq:empgrid} and the dominated convergence theorem,~\eqref{eq:nullrhas} follows if we show that:
			\begin{equation}\label{eq:unb1}
				\limsup\limits_{N\to\infty} \frac{1}{N}\sum_{i=1}^N (\Phi^{-1}(h_{i,j}^d))^4<\infty
			\end{equation}
			for all $j\in [d]$. Towards this direction, let $p_j$ denote the $j$-th prime number. Without loss of generality, suppose $h_{1,j}^d<h_{2,j}^d<\ldots <h_{N,j}^d$. Then by construction of Halton sequences, 
			$$h_{i,j}^d-h_{i-1,j}^d\geq (Np_j)^{-1}.$$
			Also, as $\Phi^{-1}(\cdot)$ is increasing, using the above lower bound yields that,
			$$\frac{1}{N}\sum_{i=1}^N (\Phi^{-1}(h_{i,j}^d))^4\leq p_j^{-1}\int_0^1 (\Phi^{-1}(u))^4\,\mathrm du<\infty,$$
			which establishes~\eqref{eq:unb1} and consequently~\eqref{eq:nullrhas}, both. 
			
			The same idea can be used to establish~\eqref{eq:unb1}~and~\eqref{eq:nullrhas}  for other quasi-Monte Carlo sequences or even the uniform $d$-dimensional grid. Moreover, as the arguments above show, it is easy to replace $\Phi^{-1}(\cdot)$ with other quantile functions under appropriate moment assumptions.
			
			\item \emph{When $\nu$ is the spherical uniform}: In this case, the popular choice is to choose $\{\bh_i\}_{i \in [N]}$ as Hallin's discrete grid (see~\cite{delbarrio2019}) and $\bJ(\cdot)$ as the van der Waerden score given in~\eqref{eq:vanscore}. Then condition~\eqref{eq:nullrhas} can be verified in the same way as in~\cite[Proof of Theorem 5.1]{shi2020rate}.
		\end{enumerate}
		
	\end{example}
	
	\subsection{Computational complexity for rank-based tests}\label{sec:compasgn}
	The overall time complexity of the testing procedures described in Sections~\ref{sec:ranktsq} and \ref{sec:rankmmd} can be split up into two additive components as follows:
	\begin{itemize}
		\item[(C1)] Computation of the empirical rank map $(\hat{\bR}_{m,n}(\bZ_1),\ldots ,\hat{\bR}_{m,n}(\bZ_N))$ (see~\eqref{eq:defemprank}) and,
		\item[(C2)] computation of rank Hotelling $T^2$ with a known covariance matrix (see~\cref{sec:ranktsq}) or the computation of rank kernel MMD (see~\cref{sec:rankmmd}), once the empirical ranks are computed.
	\end{itemize}
	The overall time complexity would be the sum total of the complexities in (C1) and (C2). Let us elaborate on the two components individually.
	
	(C1). In order to obtain $(\hat{\bR}_{m,n}(\bZ_1),\ldots ,\hat{\bR}_{m,n}(\bZ_N))$, we need to solve the optimization problem~\eqref{eq:empopt}. This problem is usually referred to as the \emph{assignment} problem. We reproduce it here for completeness:
	\begin{equation}\label{eq:empoptp2}
		\hat{\sigma}:=\argmin_{\sigma=(\sigma_1,\ldots ,\sigma_N)\in S_N} \sum_{i=1}^N \lVert \bZ_i-\bh_{\sigma_i}^d\rVert^2.
	\end{equation}
	Here all the symbols have their usual meanings. A popular deterministic algorithm, namely the \emph{Hungarian algorithm}, can be used to solve~\eqref{eq:empoptp2} in $O(N^3)$ worst case time complexity (see e.g., \cite{Edmonds1970,jonker1987,edmonds1972theoretical,tomizawa1971some,dinitz1969algorithm}). There has been considerable interest in speeding up this algorithm by leveraging parallel and distributed computing methods (see e.g.~\cite[Chapter 4]{bokhari1987assignment}, \cite{wein1991massively,date2016gpu}). We also refer the reader to fast implementations in Google OR-Tools \url{https://developers.google.com/optimization/assignment/assignment_example} for fast solvers in \texttt{Python,Java,C++}, among others. Another popular approach is the \emph{auction algorithm} due to Bertsekas \cite{bertsekas2009auction,bertsekas1989auction}, which has worst case complexity $O(N^3\log{N})$ but also has the advantage of being easy to parallelize (see~\cite{bertsekas1988,bertsekas1979distributed}). A different potential way to speed up computation is to consider $(1+\epsilon)$-approximate solutions to~\eqref{eq:empoptp2} for some $\epsilon>0$. In other words, the goal is to find $\hat{\sigma}^{\epsilon}\in S_N$ such that 
	$$\sum_{i=1}^N \lVert \bZ_i-\bh_{\hat{\sigma}^{\epsilon}_i}^d\rVert^2\le (1+\epsilon) \sum_{i=1}^N \lVert \bZ_i-\bh_{\hat{\sigma}_i}^d\rVert^2.$$
	In~\cite{Agarwal2014}, the authors obtain a deterministic $O(N^{3/2}\epsilon^{-d}\log{(N)})$ time algorithm to get such a $\hat{\sigma}^{\epsilon}$ (also see~\cite{Sharathkumar2012}). Another  algorithm to obtain $\hat{\sigma}^{\epsilon}$ was proposed in \cite{lahn2021n} with time complexity $O(N^{5/4} \mbox{poly}_d(\log{N},1/\epsilon))$, where $\mbox{poly}_d(\log{N},1/\epsilon)$ is a dimension $d$ dependent polynomial function of $\log{N}$ and $1/\epsilon$. This is an improvement to~\cite{Agarwal2014} in terms of sample size but could potentially be worse in terms of the dimension $d$. In a different direction, in \cite{Gabow1989}, the authors show that if $\{\lVert \bZ_i-\bh_j\rVert^2\}_{i,j\in [N]^2}$ are all integers bounded by $M$ (which can be ensured approximately by scaling and rounding), then it is possible to construct a deterministic algorithm to solve~\eqref{eq:empoptp2} exactly, in $O(N^{5/2}\log{(NM)})$ time.
	\par
	
	(C2). Once we have obtained $(\hat{\bR}_{m,n}(\bZ_1),\ldots ,\hat{\bR}_{m,n}(\bZ_N))$, the next step is to compute Hotelling $T^2$ and kernel MMD using these empirical ranks. To begin with, let us consider $\rhsc$ from~\eqref{eq:rankschotelling} which is given by
	$$\rhsc=\frac{mn}{m+n} \big(\bm{\Delta}_{m, n}^{\nu,\bJ}\big)^{\top}\Serd^{-1} \bm{\Delta}_{m, n}^{\nu,\bJ},$$ 
	where 
	$$\bm{\Delta}_{m, n}^{\nu,\bJ}=\frac{1}{m}\sum_{i=1}^m \bJ\left(\hat{\bR}_{m,n}(\bX_i)\right)-\frac{1}{n}\sum_{j=1}^n \bJ\left(\hat{\bR}_{m,n}(\bY_j)\right).$$
	Now, in all our examples $\Serd$ is a diagonal matrix (see~\cref{rem:covmat}), so there is no matrix-vector multiplication necessary. As the computation of $\bm{\Delta}_{m,n}^{\nu,\bJ}$ requires $O(N d)$ time (as we are averaging $N$ numbers for each of $d$ coordinates), it then follows that $\rhsc$ can be computed in $O(Nd+d^2)$ time, where the additional $O(d^2)$ factor arises while computing the norm of $\bm{\Delta}_{m, n}^{\nu,\bJ}$.
	
	We now move on to the computation of $\grsc$ which has the same computational complexity as the computation of the usual rank MMD after the ranks $(\hat{\bR}_{m,n}(\bZ_1),\ldots ,\hat{\bR}_{m,n}(\bZ_N))$ have been computed. It is well known that this complexity is $O(N^2 d)$ (see~\cite{gretton2008,gretton2009fast}). Several methods for speed up have been explored in the literature. One approach is to replace \emph{all} pairwise interactions by only a \emph{sparse set} of pairwise interactions to get a linear $O(Nd)$ time variant (\cite{gretton2012optimal}). A data splitting and aggregation approach that interpolates the linear time variant and the quadratic time variant, was proposed in \cite{zaremba2013b}. In practice, the authors in~\cite{zaremba2013b} suggest split sizes that result in a $O(N^{3/2}d)$ time algorithm. Another method to speed up computation is the FastMMD (\cite{zhao2015fastmmd}), which is based on the \emph{Fastfood} technique from \cite{le2013fastfood}. This is an approximate method that considers basis expansions of particular kernels and is particularly suitable to shift invariant and spherically invariant kernels. If $L$ basis functions are chosen for the approximation, then \cite{zhao2015fastmmd} contains an algorithm where the time complexity is $O(LNd)$ for shift-invariant kernels and $O(LN\log{d})$ for spherically invariant ones.
	
	\subsection{Statistically and computationally feasible distribution-free tests}\label{sec:compot}
	As mentioned before, prior to the recent breakthrough with multivariate ranks defined via optimal transport, most two-sample tests which are asymptotically distribution-free were either based on geometric graphs or depth functions.

	\begin{itemize}
		
		\item {\it Tests based on geometric graphs functions}: This includes the celebrated  Friedman-Rafsky test based on the {\it minimum spanning tree} (MST)\footnote{Given a finite set $S \subset \R^d$, the {\it minimum spanning tree} (MST) of $S$ is a connected graph  with vertex-set $S$ and no cycles, which has the minimum weight, where the weight of a graph is the sum of the distances of its edges.} \cite{friedman1979}, the tests based on nearest-neighbor graphs \cite{henze1988,schilling1986}, and their recent generalizations  \cite{chen2017}, and the cross-match test of Rosenbaum \cite{rosenbaum2005} based on the minimum non-bipartite matching. These tests are asymptotically distribution-free (apart from the cross-match test, which is exactly distribution-free in finite samples), computationally feasible, and universally consistent, but have no power against $O(1/\sqrt N)$  alternatives, that is, they have zero Pitman efficiency \cite{bbb2019}.  More recently, Biswas et al. \cite{munmun2014} proposed another two-sample test based on Hamiltonian cycles, which is also distribution-free in finite samples. However, computing the minimum weight Hamiltonian path is NP-hard, making this test computationally prohibitive beyond small sample sizes, and it is also expected to have zero asymptotic efficiency.

		
		
		\item {\it Tests based on depth functions}: A class of asymptotically distribution-free tests that have non-trivial asymptotic efficiencies are the Liu-Singh rank sum tests \cite{liu1993}. These tests are based on a notion of data-depth and generalize the Mann-Whitney rank test. This include tests based on halfspace depth \cite{tukey} and simplicial depth \cite{Liu1990}, among others (refer to the survey by Oja \cite{ojaR} and the references therein). However,  these tests are only consistent over a restrictive class of alternatives (recall that the Mann-Whitney test is only powerful against alternatives where $\mathbb P(X < Y)  \ne \frac{1}{2}$) and computationally intractable as the dimension increases. 
	\end{itemize}
	In contrast to the methods described above, the class of rank-based kernel two-sample tests proposed in this paper are distribution-free procedures which, in addition to being universally consistent and computationally feasible (see~\cref{sec:compasgn}), are also statistically efficient. This resolves a gap in the literature and justifies the applicability of rank-based methods in modern data applications. 
	
	It should be mentioned that the class of kernel MMD based tests (see~\cite{Gretton2012,gretton2009fast,Sejdinovic2013}) also exhibits consistency against all fixed alternatives and local power against contiguous alternatives; but those are not exactly (or even asymptotically) distribution-free. We will draw more detailed comparisons between this class of tests and $\ptsc$ in~\cref{sec:sim} below.
	
	In another recent line of work, optimal transport-based methods have been used for distribution-free testing in the two-sample problem~\cite{boeckel2018multivariate,Deb19}, independence testing problem~\cite{Deb19,shi2020distribution,shi2020rate}, linear regression~\cite{hallin2020fully,hallin2020center}), etc. However, none of this existing literature establishes any explicit Pitman ARE bounds for their testing procedures. Moreover, the above papers assume that the reference distribution $\nu$ is compactly supported and on some occasions, also assume that the data distributions are compactly supported, to establish consistency guarantees. We, on the other hand, do not require such stringent restrictions for consistency, which in turn allows us to $\nu=\mathcal{N}(\bzr,\bm{I}_d)$ distribution. We prove a general result for convergence of optimal transport-based functionals, namely~\cref{theo:rankmapcon}, under weak assumptions that allow covering all required consistency results in one go. It also helps us address an open problem raised in~\cite{shi2020rate} (see~\cref{sec:compare} above for more details).
	
	\section{Proofs of the Main Results}\label{sec:pfmain}
	
	In this section we present the proofs of our results and the new results introduced in the Appendix itself. The section is organized as follows: We begin with the proof of~\cref{theo:rankmapcon} in~\cref{sec:pfconvergence}. The proofs of the consistency results from Section~\ref{sec:condist} are given in~\cref{sec:pfconsistency}. In~\cref{sec:pfnulldistribution}, we provide the proof of~\cref{theo:nullrankhotelling}. The proof of Theorem~\ref{theo:locpowermain1} is presented in~\cref{sec:pflocalpower}. Further, in~\cref{sec:pfarehotelling} we prove the results from Section~\ref{sec:arehotelling}. The results on rank kernel MMD from~\cref{sec:rankmmd} are proved in~\cref{sec:pfrankmmd}. The proof of the results from~\cref{sec:compare} have been provided in~\cref{sec:pfindtest} and those from~\cref{sec:auxdet} are added in~\cref{sec:auxpf}.
	
	Hereafter, given two positive sequences $\{a_N\}_{N \geq 1}$ and $\{b_N\}_{N \geq 1}$, we will write $a_N \lesssim b_N$ to denote $a_N \leq C b_N$, for some constant $C >0$ not depending on $N$. 
	
	\subsection{Proof of~\cref{theo:rankmapcon}}\label{sec:pfconvergence} 
	Throughout this proof, we will assume $r=p$ for notational simplicity. As will be evident from the arguments below, the proof for other values of $r$ will follow similarly. 
	
	We begin our proof by recalling the following well-known result from convex analysis. 
	
	\begin{lemma}[Alexandroff's Theorem,~\citet{Alexandroff1939}]\label{lem:Alexandroff}
		Let $f:U\to\mathbb{R}$ be a convex function, where $U$ is an open convex subset of $\mathbb{R}^n$. Then $f$ has a second derivative Lebesgue a.e.~in $U$.	
	\end{lemma} 
	
	Hereafter, we assume that all the random variables are defined on the same probability space. Then in the usual asymptotic regime \eqref{eq:usual}, 
	\begin{align}\label{eq:rankmapt1}
		\frac{1}{N}\sum_{i=1}^N \delta_{\mz_i}\overset{w}{\longrightarrow}\lambda\mu_1+(1-\lambda)\mu_2=:\mu\quad a.s.
	\end{align}
	Let $\{(\bm{C}_{N, 1},\bm{D}_{N, 1}), (\bm{C}_{N, 2},\bm{D}_{N, 2}), \ldots, (\bm{C}_{N, p},\bm{D}_{N, p})\}$ denote $p$ i.i.d. draws from the following distribution $$\frac{1}{N}\sum_{i=1}^N \delta_{(\mz_i,\hbR_{m,n}(\mz_i))},$$
	which is the empirical distribution on the (random) set $\{(\mz_j,\hbR_{m,n}(\mz_j))\}_{j\in [N]}$. Also, let $\nu^*$ be the law induced by the random variable $(\mz,\Rl(\mz))$, $\mz\sim \lambda\mu_1+(1-\lambda)\mu_2$. Note that, by~\eqref{eq:empgrid} ~and~\eqref{eq:rankmapt1}, $(\bm{C}_{N,1},\bm{D}_{N, 1})$ is asymptotically tight almost surely as both the variables converge weakly marginally. Consequently, by using the same sequence of steps\footnote{First, by using Prokhorov's Theorem, one can show that given any subsequence, there exists a further subsequence such that $((\bm{C}_{N,1},\bm{D}_{N, 1}),\ldots ,(\bm{C}_{N,p},\bm{D}_{N, p}))$ converges weakly to some distribution on $(\R^{2d})^p$, almost surely. By using~\cite[Corollary 14 and Lemma 2]{Mccann1995}, one can then show that the limiting distribution above is free of the subsequence thereby concluding the proof.} as in~\cite[Theorem 2.1]{Deb19}, we get 
	\begin{equation}\label{eq:rankmapcon1} 
		\left((\bm{C}_{N,1},\bm{D}_{N, 1}),\ldots ,(\bm{C}_{N,p},\bm{D}_{N, p})\right)\overset{w}{\longrightarrow}\nu^*\otimes \ldots \otimes \nu^*, \quad a.s.
	\end{equation}
	Next, observe that $\Rl(\cdot)$ is continuous a.e. in the interior of the support of $\mu$, by Alexandroff's theorem (see~Lemma \ref{lem:Alexandroff}) and the a.e. continuity of $\cF(\cdot)$ and $\bJ(\cdot)$ (by the assumptions in the theorem). Therefore, the map: 
	\begin{equation*}
		g:\left((\by_{1},\bz_{1}),\ldots ,(\by_{p},\bz_{p}) \right) \mapsto \bigg\lVert\cF\big(\bJ\big(\Rl(\by_{1})\big),\ldots ,\bJ\big(\Rl(\by_{p})\big)\big)-\cF\big(\bJ(\bz_{1}),\ldots ,\bJ(\bz_{p})\big)\bigg\rVert
	\end{equation*}
	is continuous a.e. with respect to the $p$-fold product measure $\nu^*\otimes \ldots \otimes \nu^*$. Suppose $(\bZ_1,\Rl(\bZ_1)),\ldots ,(\bZ_p,\Rl(\bZ_p))\sim \nu^*\otimes \ldots \otimes \nu^*$. Then observe that $$g\bigg((\bZ_1,\Rl(\bZ_1)),\ldots ,(\bZ_p,\Rl(\bZ_p))\bigg)=0.$$Therefore, by a direct application of the continuous mapping theorem, 
	$$g\bigg((\bm{C}_{N,1},\bm{D}_{N,1}),\ldots ,(\bm{C}_{N,p},\bm{D}_{N, p})\bigg)\overset{w}{\longrightarrow}0\quad a.s.$$
	Now, since weak convergence to a degenerate measure implies convergence in probability, given any $\varepsilon>0$, the following holds: 
	$$\P\left[g\bigg((\bm{C}_{N,1},\bm{D}_{N,1}),\ldots ,(\bm{C}_{N,p},\bm{D}_{N,p})\bigg)>\varepsilon \Big| \{ \mz_1,\ldots ,\mz_N \} \right]\longrightarrow 0\quad \mbox{a.s.}$$
	and consequently by the bounded convergence theorem, 
	\begin{equation}\label{eq:rankmapcon2}
		\P\left[g\bigg((\bm{C}_{N,1},\bm{D}_{N,1}),\ldots ,(\bm{C}_{N,p},\bm{D}_{N,p})\bigg)>\varepsilon\right]\to 0.
	\end{equation} 
	To complete the proof, define
	\begin{small} 
		\begin{align*}
			& V_{m, n} :=\sum_{(i_1,\ldots ,i_p)\in [N]^p} \bigg\lVert\cF(\bJ(\hbR_{m,n}(\mz_{i_1})),\ldots ,\bJ(\hbR_{m,n}(\mz_{i_p})))-\cF(\bJ(\Rl(\mz_{i_1})),\ldots ,\bJ(\Rl(\mz_{i_p})))\bigg\rVert. 
		\end{align*} 
	\end{small}
	This implies, by recalling the definition of the function $g$, 
	\begin{align}\label{eq:rankmapconinf}
		\P\left( \frac{1}{N^p}V_{m, n}>\varepsilon  \right) &=\P\left(\E\left[g\bigg((\bm{C}_{N,1},\bm{D}_{N,1}),\ldots ,(\bm{C}_{N,p},\bm{D}_{N,p})\bigg) \Big| \mz_1,\ldots ,\mz_N\right]> \varepsilon\right)\nonumber \\ &\leq \varepsilon^{-1}\E\left[g\bigg((\bm{C}_{N,1},\bm{D}_{N,1}),\ldots ,(\bm{C}_{N,p},\bm{D}_{N,p})\bigg)\right],
	\end{align}
	where the last line uses Markov's inequality. We next show that 
	\begin{align}\label{eq:rankmapconinf1}
		\left\{g\bigg((\bm{C}_{N,1},\bm{D}_{N,1}),\ldots ,(\bm{C}_{N,p},\bm{D}_{N,p})\bigg)\right\}_{N\geq 1}\ \mbox{is uniformly integrable}.
	\end{align}
	Note that if we establish~\eqref{eq:rankmapconinf1}, then by~\eqref{eq:rankmapcon2}, we will have $\E\left[g\bigg((\bm{C}_{N,1},\bm{D}_{N,1}),\ldots ,(\bm{C}_{N,p},\bm{D}_{N,p})\bigg)\right]\to 0$ as $N\to\infty$. Combining this observation with~\eqref{eq:rankmapconinf} would then imply $\frac{1}{N^p} V_{m, n}$ convergences in probability to zero, as required.
	
	\emph{Proving~\eqref{eq:rankmapconinf1}}: Observe that
	\begin{align*}
		&\;\;\;\; g\bigg((\bm{C}_{N,1},\bm{D}_{N,1}),\ldots ,(\bm{C}_{N,p},\bm{D}_{N,p})\bigg)\\ &\leq \bigg\lVert\cF\big(\bJ\big(\Rl(\bm{C}_{N,1})\big)\big),\ldots ,\bJ\big(\Rl(\bm{C}_{N,p})\big)\big)\bigg\rVert+\bigg\lVert\cF\big(\bJ(\bm{D}_{N,1}),\ldots ,\bJ(\bm{D}_{N,p})\big)\bigg\rVert  \\ 
		& =: \mathcal{G}_N+\mathcal{H}_N . 
	\end{align*}
	By using the above display,~\eqref{eq:rankmapconinf1} would follow if $\{\mathcal{G}_N\}_{N\geq 1}$ and $\{\mathcal{H}_N\}_{N\geq 1}$ are uniformly  integrable. Observe that $\mathcal{G}_N\overset{w}{\longrightarrow}\cF(\bJ(\tilde{\bZ}_1),\ldots ,\bJ(\tilde{\bZ}_p))$,  where $\tilde{\bZ}_1,\ldots ,\tilde{\bZ}_p\overset{i.i.d.}{\sim} \nu$. The same conclusion also holds for $\mathcal{H}_N$. Now observe that: 
	\begin{align*}
		&\;\;\;\;\limsup_{N\to\infty} \E\left[\mathcal{G}_N\right]\nonumber \\ &= \limsup_{N\to\infty}\frac{1}{N^p}\E\Bigg[\sum_{(i_1,\ldots ,i_p)\in [N]^p} \Bigg(\Big\lVert\cF\big(\bJ(\hbR_{m,n}(\mz_{i_1})),\ldots ,\bJ(\hbR_{m,n}(\mz_{i_p}))\big)\Big\rVert\Bigg)\Bigg]\nonumber \\ & \leq\quad  \int \lVert\cF(\bJ(\bm{z}_1),\ldots ,\bJ(\bm{z}_p))\rVert \,d\nu(\bm{z}_1)\ldots \,d\nu(\bm{z}_p)<\infty.
	\end{align*}
	The same conclusion also holds with $\mathcal{G}_N$ replaced with $\mathcal{H}_N$. Using the above display,~\eqref{eq:rankmapconinf1} then follows by Vitali's convergence theorem (see~\cite[Theorem 5.5]{Shorack2017}).
	
	To establish the almost sure convergence, recall that in this case $\cF(\cdot)$ and $\bJ(\cdot)$ are both assumed to be Lipschitz. Consequently,
	\begin{align*}
		\frac{1}{N^p}V_{m, n} &\lesssim \frac{1}{N}\sum_{i=1}^N \lVert \hbR_{m,n}(\mz_i)-\Rl(\mz_i)\rVert.
	\end{align*} 
	The conclusion then follows in the same manner as the proof of~\cite[Theorem 2.1]{Deb19}, once again by noting that $(\bm{C}_{N,1},\bm{D}_{N,1})$ is asymptotically tight almost surely by~\eqref{eq:empgrid} ~and~\eqref{eq:rankmapt1}.
	
	\subsection{Proofs of~\cref{theo:rhconsis} and \cref{prop:conloc}}\label{sec:pfconsistency}
	In this section we present the proofs of the consistency results of the test $\prsc$. We begin with the proof of ~\cref{theo:rhconsis}. 
	
	\begin{proof}[Proof of~\cref{theo:rhconsis}] Throughout this proof all expectations are taken under $\mathrm{H}_1$ and we will assume $\bm X \sim \mu_1$ and $\bm Y \sim \mu_2$. With this in mind, we first show the following: 
		\begin{align}\label{eq:rhconsis1}
			& \frac{\rhsc}{N \lambda(1-\lambda)} \nonumber \\ 
			& \overset{P}{\longrightarrow}  \left( \E \bJ(\Rl(\bX)) - \E \bJ(\Rl(\bY)) \right)^{\top}\Serd^{-1} \left( \E \bJ(\Rl(\bX)) - \E \bJ(\Rl(\bY))  \right) .
		\end{align} 
		Towards proving this, using~\cref{theo:rankmapcon} , with $p=1$, $r=1$, $q=d$, and $\cF(\mx)=\mx$, note that 
		$$\frac{1}{m}\sum_{i=1}^m \lVert \bJ(\hbR_{m,n}(\bX_i))-\bJ(\Rl(\bX_i))\rVert+\frac{1}{n}\sum_{i=1}^n \lVert \bJ(\hbR_{m,n}(\bY_i))-\bJ(\Rl(\bY_i))\rVert = o_P(1).$$
		This implies, 
		\begin{align*} 
			& \Bigg\lVert\frac{1}{m}\sum_{i=1}^m \bJ(\hbR_{m,n}(\bX_i))-\frac{1}{n}\sum_{i=1}^n  \bJ(\hbR_{m,n}(\bY_i))-\frac{1}{m}\sum_{i=1}^m \bJ(\Rl(\bX_i))+\frac{1}{n}\sum_{i=1}^n  \bJ(\Rl(\bY_i))\Bigg\rVert \nonumber \\ 
			& = o_P(1).  
		\end{align*}
		Next, by using the weak law of large numbers together with Slutsky's theorem gives, 		$$\frac{1}{m}\sum_{i=1}^m \bJ(\hbR_{m,n}(\bX_i))-\frac{1}{n}\sum_{i=1}^n  \bJ(\hbR_{m,n}(\bY_i))\overset{P}{\longrightarrow} \E \bJ(\Rl(\bX)) - \E \bJ(\Rl(\bY)) .$$
		An application of the continuous mapping theorem then completes the proof of~\eqref{eq:rhconsis1}.
		
		Now, to complete the proof of~\cref{theo:rhconsis}, note that whenever $\E\bJ(\Rl(\bX))-\E\bJ(\Rl(\bY))\neq 0$, \eqref{eq:rhconsis1} implies that $\rhsc\overset{P}{\longrightarrow}\infty$. Also, under $\mathrm{H}_0$, $\rhsc$ is $O_p(1)$ (by~\cref{theo:nullrankhotelling}), and consequently $c_{m,n}$ in~\eqref{eq:testtrank} is $O(1)$. Combining these observations immediately yields consistency. 
	\end{proof}
	
	\begin{proof}[Proof of~\cref{prop:conloc}]
		Let $B\subset \R^d$ be an open set such that $\mu_1(B)>0$ and $u_{\mathrm{H}_1}^{\nu}(\cdot)$ is strictly convex on $B$. Fix $\my\in B$. Let $\myt$ be an element in the support of $\mu_1$. We now claim that 
		\begin{equation}\label{eq:conlocmain}
			\langle  \Rl(\myt)-\Rl(\my),\myt-\my\rangle >0.
		\end{equation}
		Note that the LHS above is always $\geq 0$, as $\Rl(\cdot)$ is the gradient of a convex function. To show strict inequality, firstly note that~\eqref{eq:conlocmain} is immediate if $\myt\in B\setminus\{\my\}$ as $u_{\mathrm{H}_1}^{\nu}(\cdot)$ is strictly convex on $B$. Now suppose $\myt\notin B$. As $B$ is open, there exists $\alpha\in (0,1)$ small enough such that the following holds:
		\begin{equation}\label{eq:conloc1}
			\langle  \Rl(\my)- \Rl(\alpha\my+(1-\alpha)\myt),(1-\alpha)(\my-\myt)\rangle > 0,
		\end{equation}
		where we use the fact that $\Rl(\cdot)$ is the gradient of a convex function (see~\cref{prop:Mccan}) and consequently satisfies cyclical monotonicity.

		Next, observe that 
		\begin{align*}
			&\;\;\;\;\langle \Rl(\myt)- \Rl(\my),\myt-\my\rangle\nonumber \\&=\langle \Rl(\myt)- \Rl(\alpha\my+(1-\alpha)\myt), \myt-\my\rangle+\langle \Rl(\my)- \Rl(\alpha\my+(1-\alpha)\myt),\my-\myt\rangle\\ &\geq \langle \Rl(\my)-\Rl(\alpha\my+(1-\alpha)\myt),\my-\myt\rangle >0,
		\end{align*}
		where the last line follows from the fact that $\Rl(\cdot)$ is the gradient of a convex function and~\eqref{eq:conloc1}. 
		
		Now, it suffices to prove that $\bD\neq\bzr$ implies $\E\Rl(\bX)\neq\E\Rl(\bY)$. We will prove this by contradiction. Towards this direction, suppose that $\bD\neq \bzr$. Observe that the condition $\E\Rl(\bX)=\E\Rl(\bY)$ can be written as: $\int (\Rl(\mx+\bD)-\Rl(\mx))\,\mathrm d\mu_1(\mx)=0$, which implies, 
		\begin{align}\label{eq:conloc2}
			\int \langle \Rl(\mx+\bD)-\Rl(\mx),\bD\rangle\,\mathrm d\mu_1(\mx)  = 0. 
		\end{align}	
		As $\langle \Rl(\mx+\bD)-\Rl(\mx),\bD\rangle \geq 0$,~\eqref{eq:conloc2} implies that $\langle \Rl(\mx+\bD)-\Rl(\mx),\bD\rangle = 0$, $\mu_1$-almost everywhere~$\mx$. However, ~\eqref{eq:conlocmain} implies, $\langle \Rl(\mx+\bD)-\Rl(\mx),\bD\rangle > 0$ for all $\mx\in B$, where $\mu_1(B)>0$. This is a contradiction, thereby completing the proof.
	\end{proof}
	
	\subsection{Proof of~\cref{theo:nullrankhotelling}}\label{sec:pfnulldistribution}

	Recall the definition of $\mathbf{\Delta}_{m, n}^{\mathrm{or}}$ from~\eqref{eq:delta_oracle}. Now, define
	$$\rhsr:=\tfrac{mn}{N}\left(\mathbf{\Delta}_{m, n}^{\nu,\bJ,\mathrm{or}}\right)^{\top}\Serd^{-1}\left(\mathbf{\Delta}_{m, n}^{\nu,\bJ,\mathrm{or}}\right)$$
	Throughout we will take limit as $N \rightarrow \infty$ such that \eqref{eq:usual} holds. The main step in the proof of~\cref{theo:nullrankhotelling} is to show that 
	\begin{align}\label{eq:rh_limit}
		\big|\rhsc-\rhsr\big|\overset{P}{\longrightarrow}0
	\end{align}
	under $\mathrm{H}_0$. The proof of \eqref{eq:rh_limit} is deferred. It can be used to complete the proof of~\cref{theo:nullrankhotelling} as follows: Note that by a direct application of the multivariate central limit theorem, 	
	$$\sqrt{\frac{mn}{N}}\mathbf{\Delta}_{m, n}^{\nu,\bJ,\mathrm{or}} \overset{w}{\longrightarrow}\mathcal{N}(\bzr,\Serd),$$
	since $\{\bJ(\Rmu(\bX_1)), \bJ(\Rmu(\bX_2)), \ldots, \bJ(\Rmu(\bX_m)) \}$,  $\{\bJ(\Rmu(\bY_1)), \bJ(\Rmu(\bY_2)), \ldots, $ $\bJ(\Rmu(\bY_n)) \}$ are independent and identically distributed random variables under $\mathrm{H}_0$. This implies, by the continuous mapping theorem, 
	$$\rhsr\overset{w}{\longrightarrow}\chi^2_d,$$
	under $\mathrm{H}_0$. Combining this with \eqref{eq:rh_limit} and the Slutsky's theorem, gives 
	$$\rhsc\overset{w}{\longrightarrow}\chi^2_d$$
	which completes the proof of~\cref{theo:nullrankhotelling}. \medskip

	\noindent {\it Proof of}~\eqref{eq:rh_limit}: 	Note that it suffices to show that:
	\begin{equation}\label{eq:nrh1}
		\lim_{N \rightarrow \infty} \frac{mn}{N}\E\big\lVert \bm \Delta_{m,n}^{\nu,\bJ}-\mathbf{\Delta}_{m, n}^{\nu,\bJ,\mathrm{or}} \big\rVert^2 = 0,
	\end{equation}
	where $\bm \Delta_{m,n}$ is defined as in~\eqref{eq:delta}. For the proof of claim~\eqref{eq:nrh1}, we need the notion of \emph{permutation distributions} as defined below:

	\begin{definition}[Permutation distribution]\label{def:pdist}
		Recall that $\mathcal{Z}_N=\{\mz_1,\mz_2,\ldots ,\mz_N\}$ denotes the pooled sample $\mathcal{X}_m\cup \mathcal{Y}_n$. For $1\leq i\leq N$, define 
		$$L_i = 
		\left\{
		\begin{array}{ccc}
			1 & \text{ if }   \mz_i\in\mathcal{X}_m,   \\
			2  &  \text{ if }   \mz_i\in\mathcal{Y}_n.    
		\end{array}
		\right.$$ 
		Note that if $\mu_1=\mu_2$, then 
		\begin{align}\label{eq:L12}
			\mathbb{P}(L_i=1|\mZ_N)=\tfrac{m}{N}=1-\mathbb{P}(L_i=2|\mZ_N). 
		\end{align}
		(Observe that $L_1, L_2, \ldots, L_N$ are identically distributed, but they are not independent.) In particular, the distribution of $(\rhsc,\rhsr)$ is completely determined by the joint distribution of $(L_1,\ldots ,L_N)$ conditional on $\{\mz_1,\ldots ,\mz_N\}$. We will refer to this distribution as the \emph{permutation distribution}.
	\end{definition}

	Let $\E_{\mZ_N}$ denote the conditional expectation with respect to $\{\bZ_1,\ldots ,\bZ_N\}$. Define $$\bhg_i:=\bJ(\hbR_{m,n}(\bZ_i)),\quad \pg_i:=\bJ(\Rmu(\bZ_i)).$$ Observe that the left hand side of~\eqref{eq:nrh1} can be written as:
	\begin{align}\label{eq:nrh2}
		\E & \big\lVert \bm \Delta_{m,n}^{\nu,\bJ}-\mathbf{\Delta}_{m, n}^{\nu,\bJ,\mathrm{or}} \big\rVert^2 \nonumber \\ 
		& =  \frac{mn}{N}\E\Bigg\lVert \sum_{i=1}^N \bhg_i\left\{\frac{\ind\{L_i=1\}}{m}-\frac{\ind\{L_i=2\}}{n}\right\}-\sum_{i=1}^N \pg_i\left\{\frac{\ind\{L_i=1\}}{m}-\frac{\ind\{L_i=2\}}{n}\right\} \Bigg\rVert^2\nonumber \\ 
		& = \frac{mn}{N}\E\Bigg\lVert \sum_{i=1}^N \bhg_i \ell_i - \sum_{i=1}^N \pg_i \ell_i \Bigg\rVert^2, 
	\end{align}
	where $\ell_i:=\frac{\ind\{L_i=1\}}{m}-\frac{\ind\{L_i=2\}}{n}$, for $i \in [N]$. Now,  \eqref{eq:nrh2} can be written as 
	\begin{align}\label{eq:nrh_T}
		\E  \big\lVert \bm \Delta_{m,n}^{\nu,\bJ}-\mathbf{\Delta}_{m, n}^{\nu,\bJ,\mathrm{or}} \big\rVert^2 & = T_1 + T_2 -  T_3, 
	\end{align} 
	where 
	\begin{align}\label{eq:nrh2_T}
		T_1 = \frac{mn}{N}\E\Bigg\lVert \sum_{i=1}^N \bhg_i \ell_i  \Bigg\rVert^2, ~ T_2  = \frac{mn}{N}\E\Bigg\lVert \sum_{i=1}^N \pg_i \ell_i \Bigg\rVert^2, ~ T_3 = \frac{2mn}{N} \E\left[\sum_{1 \leq i, j \leq N} \bhg_i^{\top}\pg_j \ell_i \ell_j \right] .  
	\end{align}

	We will now show that each of the three terms in~\eqref{eq:nrh2} converges to the same limit as $N\to\infty$. We begin with $T_1$. Note that $\bhg_{1}, \bhg_{2}, \ldots, \bhg_{N}$ are measurable with respect to the sigma field induced by $\Z_N$. Therefore the conditional expectation $\E_{\Z_N}$ only operates on the indicator variables above to yield the corresponding probabilities. In particular, recall~\eqref{eq:L12} and note that, for $i\neq j$, 
	$$\P(L_i=1,L_j=1|\mZ_N)=\frac{m(m-1)}{N(N-1)},\quad \P(L_i=2,L_j=2|\mZ_N)=\frac{n(n-1)}{N(N-1)},$$
	and, 
	$$\P(L_i=1,L_j=2|\mZ_N)=\frac{mn}{N(N-1)}.$$
	The above identities imply $\E_{\mZ_N}\ell_i=0$, for all $i \in [N]$, and, for $i \ne j$,  
	\begin{equation}\label{eq:nrh3}
		\frac{mn}{N}\E_{\mZ_N}\left[ \ell_i \ell_j \right] = - \frac{1}{N(N-1)} .
	\end{equation}	
	Now, recalling the definition of $T_1$ from~\eqref{eq:nrh2_T} and taking iterated expectation, first with respect to the permutation distribution (conditional on $\mZ_N$) and then with respect to the randomness of $\mZ_N$, gives, 
	\begin{align}\label{eq:nrh9}
		T_1 = & \frac{mn}{N}\E\E_{\mZ_N}\Bigg\lVert \sum_{i=1}^N \bhg_i \ell_i \Bigg\rVert^2 \nonumber \\ 
		&=\frac{mn}{N}\E\left[\sum_{i=1}^N \lVert \bhg_i\rVert^2\left\{\frac{\P(L_i=1|\mZ_N)}{m^2}+\frac{\P(L_i=2|\mZ_N)}{n^2} \right\} + \sum_{1 \leq i\neq j \leq N}\bhg_i^{\top}\bhg_j\E_{\mZ_N}[\ell_i \ell_j]\right] \nonumber \\ 
		& =\frac{1}{N}\E\left[\sum_{i=1}^N \lVert \bhg_i\rVert^2\right]-\frac{1}{N(N-1)}\E\left[\sum_{ 1 \leq i\neq j \leq N }\bhg_i^{\top}\bhg_j\right], 
	\end{align} 
	where the last step uses \eqref{eq:L12} and~\eqref{eq:nrh3}. Now, invoking~\cref{theo:rankmapcon}, with $p=q=r=1$, $\cF(\mx)=\lVert \mx\rVert^2$ and ~\cref{theo:rankmapcon}, gives 
	\begin{equation}\label{eq:nrh5}
		\frac{1}{N}\sum_{i=1}^N \big|\lVert\bhg_i\rVert^2-\lVert \pg_i\rVert^2\big|\overset{P}{\longrightarrow}0.
	\end{equation} 
	By using assumption~\eqref{eq:nullrhas} coupled with Vitali's convergence theorem (see~\cite[Theorem 5.5]{Shorack2017}), the above convergence happens in $L_1$. Similarly, using \cref{theo:rankmapcon}, with $p=2$, $r=2$, $q=1$, $\cF(\mx,\my)=\mx^{\top}\my$, together with Vitali's theorem gives, 
	\begin{equation}\label{eq:nrh6}
		\frac{1}{N(N-1)}\E\left[\sum_{1 \leq i\neq j \leq N}\big|\bhg_i^{\top}\bhg_j-\pg_i^{\top}\pg_j\big|\right] {\longrightarrow} 0.
	\end{equation}
	Finally, using the weak law of large numbers in~\eqref{eq:nrh5}~and~\eqref{eq:nrh6}, it follows that 
	\begin{equation}\label{eq:nrh7}
		\lim_{N \rightarrow \infty}\E\left[\frac{1}{N}\sum_{i=1}^N \lVert \bhg_i\rVert^2\right] = \int \lVert \bJ(\mx)\rVert^2\,\mathrm d\nu(\mx),
	\end{equation}
	and 	
	\begin{equation}\label{eq:nrh8}
		\lim_{N \rightarrow \infty}	\frac{1}{N(N-1)}\E\left[\sum_{1 \leq i\neq j \leq N }\bhg_i^{\top}\bhg_j\right] = \int \bJ(\mx)^{\top}\bJ(\my)\,\mathrm d\nu(\mx)\,\mathrm d\nu(\my).
	\end{equation}
	Combining~\eqref{eq:nrh9},~\eqref{eq:nrh7}~and~\eqref{eq:nrh8}, gives
	\begin{align*}
		\lim_{N \rightarrow \infty} T_1 = \eta_{\bJ}:= \int \lVert \bJ(\mx)\rVert^2\,\mathrm d\nu(\mx)-\int \bJ(\mx)^{\top}\bJ(\my)\,\mathrm d\nu(\mx)\,\mathrm d\nu(\my).
	\end{align*}
	
	Similar arguments can be applied to the other two terms in \eqref{eq:nrh3} to show that $\lim_{N \rightarrow \infty} T_2 = \eta_{\bJ}$ and $\lim_{N \rightarrow \infty} T_3 = 2 \eta_{\bJ}$. This proves~\eqref{eq:nrh1}.
	
	\subsection{Proof of~\cref{theo:locpowermain1}}\label{sec:pflocalpower}
	Recall the setup of~\eqref{eq:twosamsmooth} and its corresponding assumptions from Section~\ref{sec:locpow}, both. Define the likelihood ratio statistic as:
	$$V_N:=\sum_{j=1}^n \log\frac{ f\left(\bY_j|\bt_0+N^{-1/2}\bh\right)}{f(\bY_j|\bt_0)}.$$
	By using the local asymptotic normality under~\eqref{eq:twosamsmooth}, $V_N$ can be written as:
	
	\begin{equation}\label{eq:likscore} V_N=\dot{V}_N-\frac{(1-\lambda)}{2}\bh^{\top}\bm{I}(\bt_0)\bh+o_{P}(1), \text{~where~}\dot{V}_N:=\frac{h}{\sqrt{N}}\sum_{j=1}^n \frac{\bo^{\top}\nabla_{\bt} f(\bY_j|\bt)|_{\bt_0}}{f(\bY_j|\bt_0)}. 
	\end{equation}
	Denote 
	\begin{align}\label{eq:rmu_T12}
		\bm T_1 &:=\frac{1}{m} \sum_{i=1}^m \left\{\bJ(\Rmu(\bX_i))-\E_{\mathrm{H}_0} \bJ(\Rmu(\bX_1)) \right\},  \nonumber \\ 
		\bm T_2 &:=\frac{1}{n} \sum_{j=1}^n \left\{\bJ(\Rmu(\bY_j))-\E_{\mathrm{H}_0} \bJ(\Rmu(\bY_1)) \right\} . 
	\end{align} 
	Consequently, by the multivariate central limit theorem, the following result holds:
	\begin{equation}\label{eq:3dmct}
		\begin{pmatrix} \sqrt{\frac{mn}{N}} \bm T_1 \\ \sqrt{\frac{mn}{N}} \bm T_2 \\ V_N \end{pmatrix} \overset{w}{\to}\mathcal{N}\left(\begin{pmatrix} \bzr_{d \times 1} \\ \bzr_{d \times 1} \\ -\frac{c_1}{2} \end{pmatrix},\begin{pmatrix} (1-\lambda)\Serd & \bzr_{d \times d} & \bzr_{d \times 1} \\ \bzr_{d \times d} & \lambda\Serd & \bsg_1  \\ \bzr_{1 \times d}^{\top} & \bsg_1^{\top} &  c_1\end{pmatrix}\right) 
	\end{equation} 
	under $\mathrm{H}_0$, where $\bzr_{d \times 1}$ denotes the vector of zeros of length $d$, $\bzr_{d \times d}$ the $d\times d$ matrix of zeros, $c_1:=(1-\lambda)\bh^{\top}\bm{I}(\bt_0)\bh$, and 
	$$\bsg_1:=\sqrt{\lambda(1-\lambda)}\E_{\mathrm{H}_0}\left[\bJ(\Rmu(\bY))\frac{\bh^{\top}\nabla_{\bt} f(\bY|\bt)|_{\bt_0}}{f(\bY|\bt_0)}\right].$$ 
	Note that under $\mathrm{H_0}$, $\bm{\Delta}_{m, n}^{\nu,\bJ,\mathrm{or}}=\bm T_1-\bm T_2$ (recall~\eqref{eq:delta_oracle}). Therefore, by~\eqref{eq:nrh1} and~\eqref{eq:3dmct}, under $\mathrm{H}_0$, 
	\begin{equation*}
		\begin{pmatrix} \sqrt{\frac{mn}{N}} \bm{\Delta}_{m, n} \\ V_N \end{pmatrix} \overset{w}{\to}\mathcal{N}\left(\begin{pmatrix} \bzr\\ -\frac{c_1}{2} \end{pmatrix},\begin{pmatrix} \Serd & \bsg_1  \\  \bsg_1^{\top}  & c_1\end{pmatrix}\right).
	\end{equation*}	
	Then by Le Cam's third lemma~\cite[Corollary 12.3.2]{LR05} and the continuous mapping theorem, we have:
	\begin{equation*}
		\rhsr\overset{w}{\longrightarrow}\Big\lVert -\Serd^{-\frac{1}{2}}\bsg_1+\bm{G}\Big\rVert_2^2,
	\end{equation*}
	under $\mathrm{H}_1$, where $\bm{G}\sim\mathcal{N}(\bzr,\bm I_d)$. Next note that in~\eqref{eq:rh_limit}, we showed that $|\rhsc-\rhsr|\overset{P}{\longrightarrow}0$ under $\mathrm{H}_0$. By contiguity, the same conclusion also holds under $\mathrm{H}_1$. Therefore, by an application of Slutsky's theorem, we get:
	\begin{equation}\label{eq:ARErankt}
		\rhsc\overset{w}{\longrightarrow}\Big\lVert -\Serd^{-\frac{1}{2}}\bsg_1+\bm{G}\Big\rVert_2^2.
	\end{equation}
	This completes the proof.
	
	\subsection{Proofs from Section~\ref{sec:arehotelling}}\label{sec:pfarehotelling}
	In this section we will present the proofs of the results from Section~\ref{sec:arehotelling}. We begin with the formal definition of the asymptotic (Pitman) relative efficiency of two tests (see~\cite{pitman1948lecture,nikitin2011asymptotic,Van1998}). 
	
	\begin{definition}[Asymptotic (Pitman) relative efficiency]\label{def:asympeff}
		Consider the sequence of testing problems 
		\begin{align}\label{eq:theta01}
			\mathrm{H}_0:\theta=\theta_0 \quad  \text{versus} \quad \mathrm{H}_1:\theta=\theta_{1, K},
		\end{align} 
		for $\theta_0, \theta_{1, K} \in\R$ and all $K \geq 1$, where $\theta_{1, K} \to\theta_0$, as $K \rightarrow \infty$. Suppose $\{T_{1, K}\}_{K\geq 1}$ and $\{T_{2, K}\}_{K\geq 1}$ are two sequences of level $\alpha\in (0,1)$ tests for the problem~\eqref{eq:theta01}, with associated test functions $\{\phi_{1, K}\}_{K\geq 1}$ and $\{\phi_{2, K}\}_{K \geq 1}$. Fixing a power level $\beta\in [\alpha,1)$, define
		$$N_{1}(\alpha,\beta,\theta_1):=\min\{K \geq 1: \E_{\theta_1}\phi_{1, K'}\geq \beta,\ \text{ for all } K'\geq K\},$$
		and $N_2(\alpha,\beta,\theta_1)$ similarly as above with $\phi_{1,K'}$ replaced by $\phi_{2,K'}$. Then the asymptotic (Pitman) relative efficiency (ARE) of the sequence of tests $\phi_{1,K}$ with respect to $\phi_{2,K}$, along a sequence $\theta_{1, K} \to\theta_0$, as $K \rightarrow \infty$, is given by
		$$\mathrm{ARE}(\{T_{1,K}\}_{K\geq 1},\{T_{2,K}\}_{K\geq 1})=\lim\limits_{\theta_{1, K} \to\theta_0}\frac{N_2(\alpha,\beta,\theta_{1, K})}{N_1(\alpha,\beta,\theta_{1, K})}$$
		provided the limit exists. Compared to other notions of asymptotic efficiency, such as the Bahadur efficiency~\cite{Bahadur67}, the asymptotic (Pitman) relative efficiency is reputed to be a fairly good approximation for moderate sample sizes in many testing problems (see~\cite{GO81}; also see~\cref{sec:sim}).  
	\end{definition}

	As is evident from the above definition, the ARE of two tests will in general depend on $\alpha$ and $\beta$. However, when the tests have asymptotically non-central chi-squared distributions with the same degrees of freedom (as is the case with the Hotelling $T^2$ and the $\rhsc$ statistics), the ARE is simply the ratio of the corresponding non-centrality parameters (see \cite[Proposition 5]{Hallin2002} and \cite[Theorem 14.19]{Van1998}). Equipped with this fact, we begin the proofs of our results from Section~\ref{sec:arehotelling}.

	\subsubsection{Proof of~\cref{prop:Gaussare}}
	
	Recall the definition of the Hotelling $T^2$ statistic $T_{m,n}$ from~\eqref{eq:hotelling}. Note that if $\mbox{Var}_{\mathrm{H}_0} (\bX)$ exists then under the regularity assumptions of~Theorem \ref{theo:locpowermain1}, it is easy to check that:
	\begin{align}\label{eq:AREt}
		T_{m,n}\overset{w}{\longrightarrow} \Big\lVert \sqrt{\lambda(1-\lambda)}\tilde{\Sigma}^{-\frac{1}{2}} \E_{\mathrm{H}_0}\left[\bX\bh^{\top}\boldsymbol{\eta}(\bm X, \bt_0)\right]+\bm{G}\Big\rVert^2,
	\end{align} 
	where $\bm{G}\sim \mathcal{N}(\bzr,\bm I_d)$ and 
	\begin{equation}\label{eq:tsig}
		\tilde{\Sigma}:=\E[\bX-\E\bX][\bX-\E\bX]^{\top}.
	\end{equation}
	Recall from~\cref{sec:arehotelling}   that $\atr$ denotes the ARE of $\rhsc$ with respect to $T_{m, n}$. Then, using the non-centrality parameters of $\rhsc$ and $T_{m, n}$ from ~\eqref{eq:ARErankt} and~\eqref{eq:AREt}, respectively, and invoking~\cite[Theorem 14.19]{Van1998} gives, 
	\begin{equation}\label{eq:AREinterest}
		\atr=\frac{\Big\lVert\Serd^{-\frac{1}{2}}\E_{\mathrm{H}_0}\left[\bJ(\Rmu(\bX))\bh^{\top}\boldsymbol{\eta}(\bm X, \bt_0)\right]\Big\rVert^2}{\Big\lVert\tilde{\Sigma}^{-\frac{1}{2}} \E_{\mathrm{H}_0}\left[\bX\bh^{\top}\boldsymbol{\eta}(\bm X, \bt_0)\right]\Big\rVert^2}.
	\end{equation}
	
	\noindent{\it Proof of~\cref{prop:Gaussare}}~(1): Let us define $\bm A:=\Sigma^{-\frac{1}{2}}:= ((a_{ij}))_{1 \leq i, j  \leq d}$, where $a_{ij}$ denotes the $(i, j)$-th element of the matrix $\bm A$. In this case, $\tilde{\Sigma}=\Sigma$. Also note that $\bm A(\bX-\bt_0)\overset{w}{=}\bm{G}\sim\mathcal{N}(\bzr,\bm I_d)$. Using 
	$$\boldsymbol{\eta}(\bm X, \bt_0)=\frac{\nabla_{\bt} f(\bX|\bt)|_{\bt_0}}{f(\bX|\bt_0)}= \Sigma^{-1} (\bX-\bt_0) = \bm A \bm{G} ,$$ 
	and $\E_{\mathrm{H}_0}\left[\boldsymbol{\eta}(\bm X, \bt_0)\right]=0$, observe that 
	\begin{align}\label{eq:Gaussare1}
		\left\lVert\Sigma^{-\frac{1}{2}} \E_{\mathrm{H}_0}\left[\bX\bh^{\top}\boldsymbol{\eta}(\bm X, \bt_0)\right]\right\rVert^2 =\left\lVert \E_{\mathrm{H}_0}\left[\bm{G} \cdot \bh^{\top}\bm A \bG\right]\right \rVert^2  
		=\sum_{j=1}^d \left(\sum_{i=1}^d h_i a_{ij}\right)^2  =\bh^{\top}\Sigma^{-1}\bh.
	\end{align}
	To see the above equality, observe
	$$\E_{\mathrm{H}_0}\left[(\bm{G} \cdot \bh^{\top}\bm A \bG)_j\right]=\sum_{1 \leq i,k \leq d} h_i a_{ik}\E_{\mathrm{H}_0}[G_kG_j]=\sum_{i=1}^d h_i a_{ij}.$$
	Now, set $\Rmu(\bX):=\bm A(\bX-\bt_0)$. Note that, with this definition, $\Rmu(\bX)\sim\mathcal{N}(\bzr,\bm I_d)$. Moreover, $\Rmu(\bX)$ is the gradient of the convex function $\frac{1}{2}(\bX-\bt_0)^{\top}A(\bX-\bt_0)$. Therefore, $\Rmu(\cdot)$ is the required optimal transport map in this case (see~\cref{prop:Mccan}). Next note that, 
	$$\bJ(\Rmu(\bX))=\left(\Phi(\be_1^{\top} \bm A (\bX-\bt_0)),\ldots ,\Phi(\be_d^{\top} \bm A (\bX-\bt_0))\right),$$
	where $\be_i$ is the $i$-\textrm{th} vector of the canonical basis in $\R^d$. Combining the above display with~\cref{rem:covmat}  and the integration by parts formula gives, 
	\begin{align}\label{eq:Gaussare2}
		& \left\lVert\Serd^{-\frac{1}{2}}\E_{\mathrm{H}_0}\left[\bJ(\Rmu(\bX))\bh^{\top}\boldsymbol{\eta}(\bm X, \bt_0)\right] \right\rVert^2\nonumber \\ 
		&=12\sum_{j=1}^d \left\{ \left(\sum_{i=1}^d h_i\E\left[\frac{\partial}{\partial X_i} \Phi\left(\be_j^{\top} \bm A (\bX-\bt_0) \right)\right]\right)^2 \right \} \nonumber \\ 
		& =12\sum_{j=1}^d \left\{\left(\sum_{i=1}^d h_i a_{ij}\right)^2\left(\E \phi\left(\be_j^{\top} \bm A (\bX-\bt_0))\right)\right)^2 \right\} , 
	\end{align}
	where $\bm X= (X_1, \ldots, X_d)^{\top}$. Now, for $j \in [d]$, define $\bd_j :=A\be_j$ and note that by the standard block determinant formula,
	$$\textrm{det}(\Sigma^{-1}+\bd_j\bd_j^{\top})=\textrm{det}(\Sigma^{-1})(1+\bd_j^{\top}\Sigma\bd_j)=2(\textrm{det}(\Sigma))^{-1}.$$
	Using the above display, we get the following chain of equalities, for $j \in [d]$:
	\begin{align}\label{eq:Gaussare3}
		\E \phi\left(\be_j^{\top} \bm A (\bX-\bt_0) \right) & = \E \phi\left(\bd_j^{\top} (\bX-\bt_0) \right) \nonumber \\ 
		&=\frac{\sqrt{\textrm{det}(\Sigma^{-1})}}{(\sqrt{2\pi})^{d+1}}\int\exp\left(-\frac{1}{2}(\mx-\bt_0)^{\top}\left(\bd_j\bd_j^{\top}+\Sigma^{-1}\right)(\mx-\bt_0)\right)\,\mathrm d\mx\nonumber \\ &=\left(\frac{ \textrm{det}(\Sigma^{-1}) }{ 2\pi \cdot \textrm{det}(\Sigma^{-1}+\bd_j\bd_j^{\top})} \right)^{\frac{1}{2}} \nonumber \\ 
		& =\frac{1}{2\sqrt{\pi}}.
	\end{align}
	Using~\eqref{eq:Gaussare2} in~\eqref{eq:Gaussare3} gives, 
	$$\left\lVert\Serd^{-\frac{1}{2}}\E_{\mathrm{H}_0}\left[\bJ(\Rmu(\bX))\bh^{\top}\boldsymbol{\eta}(\bm X, \bt_0)\right] \right\rVert^2 =  \frac{3}{\pi}  \sum_{j=1}^d \left(\sum_{i=1}^d h_i a_{ij}\right)^2 = \frac{3}{\pi} \cdot \bh^{\top}\Sigma^{-1}\bh.$$ 
	This implies, by~\eqref{eq:AREinterest} and~\eqref{eq:Gaussare1}, $\atr=\frac{3}{\pi}$, which completes the proof of part (1). \medskip

	\noindent{\it Proof of~\cref{prop:Gaussare}}~(2): Let $H_d(\cdot)$ be the cumulative distribution function of a $\sqrt{\chi^2_d}$ distribution and $h_d(\cdot)$ be the associated probability density. It is easy to check that 
	\begin{align}\label{eq:density_r}
		h_d(r)=\frac{1}{2^{d/2-1} \Gamma(d/2) } r^{d-1}e^{-\frac{r^2}{2}},
	\end{align}	
	for $r \geq 0$. Next, define
	\begin{equation}\label{eq:Gaussare7}
		\Rmu(\bX):=\frac{\bm A(\bX-\bt_0)}{\sqrt{(\bX-\bt_0)^{\top}\bm A^2(\bX-\bt_0)}}\cdot H_d\left(\sqrt{(\bX-\bt_0)^{\top}\bm A^2(\bX-\bt_0)}\right).
	\end{equation}
	We will first show that $\Rmu(\bX)$ defined above is the optimal transport map from $\bX$ to the spherical uniform distribution. Towards this direction, consider the following standard lemma which we state and prove for completeness. 
	\begin{lemma}\label{lem:ellipot}
		Suppose $\bX'_1$ has an elliptically symmetric distribution as in~\eqref{eq:elldefn} with parameters $\bt$, $\Sigma$ and $\bX'_2$ has an elliptically symmetric distribution with parameters $\bzr_d$ and $\bm I_d$ (that is, the distribution of $\bX'_2$ is spherically symmetric). Let $H_1(\cdot)$ be the distribution function of $\lVert\Sigma^{-\frac{1}{2}}(\bX'_1-\bt)\rVert$ and $H_2(\cdot)$ be the distribution function of $\lVert\bX'_2\rVert$. Then the optimal transport map from the distribution of $\bX'_1$ to that of $\bX'_2$ is given by:
		\begin{equation}\label{eq:ellipot1}
			\bR(\bX'_1):=\frac{\Sigma^{-\frac{1}{2}}(\bX'_1-\bt)}{\lVert\Sigma^{-\frac{1}{2}}(\bX'_1-\bt)\rVert}H_2^{-1}\left(H_1\left(\lVert\Sigma^{-\frac{1}{2}}(\bX'_1-\bt)\rVert\right)\right).
		\end{equation} 
	\end{lemma}
	\begin{proof}
		$\bR(\cdot)$ as defined in~\eqref{eq:ellipot1} clearly satisfies $\bR(\bX'_1)\overset{d}{=}\bX'_2$. Further note that $\bR(\cdot)$ is the gradient of the function:
		$$\int_0^{\lVert \Sigma^{-\frac{1}{2}} (\bX'_1-\bt)\rVert} H_2^{-1}(H_1(r))\,\mathrm dr,$$
		which is a convex function due to the monotonicity of $H_2^{-1}\circ H_1(\cdot)$. Therefore, by~\cref{prop:Mccan},  $\bR(\cdot)$ as defined in~\eqref{eq:ellipot1} is the optimal transport map in this case.
	\end{proof}
	
	It follows from~\cref{lem:ellipot} with $\bX'_1\overset{D}{=}\bX_1$ and $\bX'_2$ as the spherical uniform distribution, that $\Rmu(\cdot)$ as defined in~\eqref{eq:Gaussare7} above is the optimal transport map in this case. 
	
	Next, let us write: 
	$$\Rmu(\bX)=( r_1(\bX), \ldots  r_d(\bX)),$$
	where $\bm r_j(\bX)$ denotes the $j$-th coordinate of the rank vector $\Rmu(\bX)$, for $j \in [d]$. Now, write $\bm{G}=\bm A(\bX-\bt_0)$ and note $\bm{G}\overset{w}{=}\mathcal{N}(\bzr,\bm I_d)$. For any $i,j \in [d]$, observe that:
	\begin{equation}\label{eq:Gaussare4}
		\frac{\partial}{\partial X_i}  r_j(\bX)=\frac{a_{ij} H_d(\lVert \bm{G}\rVert)}{\lVert \bm{G}\rVert}+\frac{(\be_j^{\top}\bG) h_d(\lVert \bG\rVert)}{\lVert \bG\rVert^2}\cdot \be_i^{\top}A\bG-\frac{(\be_j^{\top}\bG)H_d(\lVert \bG\rVert)}{\lVert \bG\rVert^{3}}\cdot \be_i^{\top}A\bG.
	\end{equation}
	By using the spherical symmetry of $\bG$, we further get:
	\begin{align} 
		\E\left[\frac{(\be_j^{\top}\bG) h(\lVert \bG\rVert)}{\lVert \bG\rVert^2}\cdot \be_i^{\top}A\bG\right]&=\E\left[\frac{(a_{ij} G_j^2)h_d(\lVert \bG\rVert)}{\lVert \bG\rVert^2}\right]\nonumber \\
		& = \frac{a_{ij}}{d} \E\left[ h_d(\lVert \bG\rVert)  \right] \label{eq:Gaussare51}  \\ 
		& = \frac{a_{ij}}{d}\cdot\frac{1}{2^{d-2}(\Gamma(d/2))^2}\int_0^{\infty} e^{-r^2}r^{2d-2}\,\mathrm dr \label{eq:Gaussare52}
		\\ & = \frac{a_{ij}}{d}\cdot\frac{1}{2^{d-1}}\cdot \frac{\Gamma(d-0.5)}{(\Gamma(d/2))^2}. \label{eq:Gaussare53}
	\end{align}
	Here, \eqref{eq:Gaussare51} follows by using that conditional on $\lVert \bG\rVert$, $G_1,\ldots ,G_d$ have the same marginal distribution,~\eqref{eq:Gaussare52} uses \eqref{eq:density_r}, and \eqref{eq:Gaussare53} is a simple integration exercise using the properties of the Gamma integral.
	
	Similarly, 
	\begin{align}
		&\;\;\;\E\left[\frac{(\be_j^{\top}\bG) H_d(\lVert \bG\rVert)}{\lVert \bG\rVert^{3}}\cdot \be_i^{\top}A\bG\right]\nonumber \\&=\E\left[\frac{(a_{ij} G_j^2)H_d(\lVert \bG\rVert)}{\lVert \bG\rVert^{3}}\right]\nonumber \\ &=\frac{a_{ij}}{d}\E\left[\frac{H_d(\lVert \bG\rVert)}{\lVert \bG\rVert}\right]\label{eq:Gaussare63}\\ &=\frac{a_{ij}}{d}\cdot \frac{1}{2^{d-2}(\Gamma(d/2)^2)}\int_0^{\infty}\int_0^t e^{-\frac{x^2+t^2}{2}}x^{d-1} t^{d-2} \,\mathrm dt\,\mathrm dx\nonumber \\ &=\frac{a_{ij}}{d\cdot 2^{d/2}\Gamma(d/2)}\E|G_1|^{d-2}-\frac{a_{ij}}{d\cdot 2^{d/2-1}(\Gamma(d/2))^2}\int_0^{\infty} \Gamma(d/2,t^2/2)e^{-\frac{t^2}{2}}t^{d-2}\,\mathrm dt\label{eq:Gaussare6}
	\end{align}
	where $$\Gamma(a,b):=\int_b^{\infty}  \exp(-x)x^{a-1}\,\mathrm dx$$
	for $a,b>0$ is popularly called the upper incomplete Gamma function. By~\cite{Kummer1837}, the incomplete Gamma function in~\eqref{eq:Gaussare6} can alternatively be written as:
	$$\Gamma(d/2,t^2/2)=e^{-\frac{t^2}{2}}\frac{t^d}{2^{d/2}}\int_0^{\infty} e^{-\frac{ut^2}{2}}(1+u)^{d/2-1}\,\mathrm du.$$
	Using the above identity gives, 
	\begin{align}\label{eq:Gaussare62}
		&\int_0^{\infty} \Gamma(d/2,t^2/2)e^{-\frac{t^2}{2}}t^{d-2}\,\mathrm dt \nonumber \\ 
		&=\frac{\sqrt{2\pi}(2d-2)!!}{2^{d/2+1}}\int_1^{\infty} \frac{u^{d/2-1}}{(1+u)^{d-0.5}}\,\mathrm du \nonumber \\ 
		& =\frac{\sqrt{2\pi}(2d-2)!!}{2^{d/2}}{}_{2}F_{1}(d-0.5,d/2-0.5;d/2+0.5;-1) := \mathcal{C}_d,
	\end{align}
	where the last line follows from Kummer's identity (see~\cite[Section 2.3]{bailey1935generalized}).
	
	Next, by using the integration by parts formula, we have:
	\begin{align}
		&\;\;\;\left\lVert\Serd^{-\frac{1}{2}}\E_{\mathrm{H}_0}\left[\bJ(\Rmu(\bX))\bh^{\top}\boldsymbol{\eta}(\bm X, \bt_0)\right]\right\rVert^2\nonumber \\
		&=3d\sum_{j=1}^d \left(\sum_{i=1}^d h_i \E\left[\frac{\partial}{\partial X_i} r_j(\bX)\right]\right)^2\label{eq:Gaussare64} \\ &=3\bh^{\top}\Sigma^{-1}\bh\cdot \frac{1}{d}\bigg[\frac{1}{2^{d-1}}\cdot\frac{\Gamma{(d-0.5)}}{(\Gamma(d/2))^2}+\frac{\sqrt{2\pi}(d-1)}{2^{d/2}\Gamma(d/2)}\E[|G_1|^{d-2}]- \mathcal{C}_d \bigg]^2.\label{eq:Gaussare61}
	\end{align}
	Here~\eqref{eq:Gaussare61} follows by plugging the expressions obtained in~\eqref{eq:Gaussare53},~\eqref{eq:Gaussare6},~and~\eqref{eq:Gaussare62} in~\eqref{eq:Gaussare4}. 
	
	Using~\eqref{eq:Gaussare61} with~\eqref{eq:Gaussare1}~and~\eqref{eq:AREinterest} completes the proof of \cref{prop:Gaussare}~(2).

	\medskip 	
	
	\noindent{\it Proof of~\cref{prop:Gaussare}}~(3): Note that in this case the optimal transport map is the same as the function $\Rmu(\cdot)$ defined in part (1). Once again using~\cref{rem:covmat} gives, 
	\begin{align*}
		\left\lVert\Serd^{-\frac{1}{2}}\E_{\mathrm{H}_0}\left[\bJ(\Rmu(\bX))\bh^{\top}\boldsymbol{\eta}(\bm X, \bt_0)\right]\right\rVert^2  	& =\left\lVert \E_{\mathrm{H}_0}\left[\Sigma^{-\frac{1}{2}} (\bX-\bt_0) \bh^{\top}\bm A\Sigma^{-\frac{1}{2}} (\bX-\bt_0)\right]\right\rVert^2, 
	\end{align*}
	which is exactly the same as~\eqref{eq:Gaussare1}. Hence, $\atr=1$. \qed

	\subsubsection{Proof of~\cref{prop:areind}} \hfill \medskip 
	
	\noindent {\it Proof of~\cref{prop:areind}}~(1):
	It is easy to see that $\atr=\infty$ if $\mbox{Var}[X_i]=\infty$, for some $1 \leq i \leq d$. Therefore, we will assume here that $\mbox{Var}[\bX]=\mbox{diag}(\sigma_1^2,\ldots ,\sigma_d^2)$, where $0<\sigma_i^2<\infty$, for $i\in [d]$. 
	Write $\bt_0:=(\theta_{0,1},\ldots ,\theta_{0,d})$. Then for $f (\cdot|\bt_0) \in \fnd$, 
	\begin{align}\label{eq:grad_independent}
		\bh^{\top}\boldsymbol{\eta}(\bm X, \bt_0)=\frac{\bh^{\top}\nabla_{\bt} f(\mx|\bt)|_{\bt_0} }{f(\mx|\bt_0)}= - \sum_{i=1}^d h_i \frac{\frac{\mathrm d}{\mathrm d x_i}f_i(x_i-\theta_{0,i})}{f_i(x_i-\theta_{0,i})},
	\end{align}
	since $f(\mx|\bt_0)=\prod_{i=1}^d f_i( x_i - \theta_{0, i})$.

	To begin with we consider the non-centrality parameter of the Hotelling $T^2$ statistic $T_
	{m, n}$ under the contiguous alternative. For this, recalling \eqref{eq:AREt} and using \eqref{eq:grad_independent}, note that,
	\begin{align}\label{eq:areindT2}
		\left\lVert\Sigma^{-\frac{1}{2}} \E_{\mathrm{H}_0}\left[\bX\bh^{\top}\boldsymbol{\eta}(\bm X, \bt_0)\right]\right \rVert^2 & =\sum_{i=1}^d \frac{h_i^2}{\sigma_i^2}\left(\int x_i\frac{\mathrm d}{\mathrm d x_i}f_i(x_i-\theta_{0,i})\,\mathrm dx_i\right)^2 \nonumber \\ 
		& =\sum_{i=1}^d \frac{h_i^2}{\sigma_i^2}.
	\end{align}
	Let $F_i(\cdot-\theta_{0,i})$ be the cumulative distribution function associated with $f_i(\cdot-\theta_{0,i})$ and $\Phi(\cdot)$ be the cumulative distribution function of a Gaussian random variable. Now, define
	$$\Rmu(\bX):=(\Phi^{-1}\circ F_1(X_1-\theta_{0,1}),\ldots ,\Phi^{-1}\circ F_d(X_d-\theta_{0,d})).$$ 
	To see that this is indeed the optimal transport map in this case, note that $\Rmu(\bX)\sim\mathcal{N}(\bzr,\bm I_d)$ and $\Rmu(\bX)$ as defined is the gradient of the following function:
	$$\sum_{i=1}^d \int_{-\infty}^{X_i} \Phi^{-1}\circ F_i(t_i-\theta_i)\,\mathrm dt_i,$$
	which due to the monotonicity of $\Phi^{-1}\circ F_i(\cdot-\theta_i)$ is a convex function. Therefore, by applying~\cref{prop:Mccan}, $\Rmu(\bX)$ is the required optimal transport map in this case. This implies that $$\bJ(\Rmu(\bX))=(F_1(X_1-\theta_{0,1}),\ldots ,F_d(X_d-\theta_{0,d})),$$
	as $\bJ(\mx)=(\Phi(x_1),\ldots ,\Phi(x_d))$ where $\mx=(x_1,\ldots ,x_d)$.
	
	Next, we consider the non-centrality parameter of the statistic $\rhsc$ under the contiguous alternative. Here, Remark~\ref{rem:covmat}, and~\eqref{eq:grad_independent} gives, 
	\begin{align}\label{eq:areindrank}
		\left\lVert\Serd^{-\frac{1}{2}}\E_{\mathrm{H}_0}\left[\bJ(\Rmu(\bX))\bh^{\top}\boldsymbol{\eta}(\bm X, \bt_0)\right]\right\rVert^2  	& = 12 \sum_{i=1}^d h_i^2 \left(\int F_i(x_i-\theta_{0, i})\frac{\mathrm d}{\mathrm d x_i}f_i(x_i-\theta_{0,i})\,\mathrm dx_i\right)^2 \nonumber \\ 
		& =12\sum_{i=1}^d h_i^2 \left(\int f_i^2(x_i-\theta_{0,i})\,\mathrm dx_i\right)^2, 
	\end{align}
	where the last step uses the integration by parts formula. 
	
	Now, combining~\eqref{eq:AREinterest},~\eqref{eq:areindT2}, and~\eqref{eq:areindrank} gives, 
	\begin{align}\label{eq:areindT2rank}
		\inf_{\fnd}\atr=\inf_{\fnd}\left(\sum_{i=1}^d \frac{h_i^2}{\sigma_i^2}\right)^{-1} \left\{12 \sum_{i=1}^d h_i^2 \left(\int f_i^2(x_i-\theta_{0,i})\,\mathrm dx_i\right)^2 \right\}.
	\end{align}
	Next, consider the following optimization problem: 
	\begin{align}\label{eq:ind_optimization}
		\inf_{f_i} \sigma_i^2\left(\int f_i^2(x_i-\theta_{0,i})\,\mathrm dx_i\right)^2, 
	\end{align}
	such that $\int f_i(x_i-\theta_{0,i})\mathrm dx_i =1$. Here $\int (x_i-\theta_{0, i})^2 f_i(x_i-\theta_{0,i}) \mathrm dx_i =\sigma_i^2$. 
	This is precisely the optimization problem that arises in the 1-dimensional case while minimizing the ARE of the Wilcoxon's test with respect to Student's $t$-test over location families. In particular,~\cite[Theorem 1]{Hodges1956} shows that the infimum in \eqref{eq:ind_optimization} is attained by the class of densities in \eqref{eq:areind_density}  and the minimum value is $9/125$. Plugging this in \eqref{eq:areindT2rank}, we get:
	$$\inf_{\fnd} \atr\geq \inf_{\fnd} \frac{108}{125}\left(\sum_{i=1}^d \frac{h_i^2}{\sigma_i^2}\right)^{-1}\left(\sum_{i=1}^d \frac{h_i^2}{\sigma_i^2}\right)=0.864.$$
	This proves the result in part (1) for the vase $\nu=\mathcal{N}(\bzr,\bm{I}_d)$ and $\bJ(\mx)=(\Phi(x_1),\ldots ,\Phi_(x_d))$. A similar sequence of arguments also lead to the conclusion when $\nu=\mathrm{Unif}[0,1]^d$ and $\bJ(\mx)=\mx$. We omit the details for brevity. \medskip 
	
	\noindent {\it Proof of~\cref{prop:areind}}~(2): Recall that $\Phi(\cdot)$ denotes the standard normal cumulative distribution function and $\phi(\cdot)$ denotes the standard normal density. Similar to part (1), in this case it can be checked that 
	$$\Rmu(\bX)=(\Phi^{-1}(F_1(X_1-\theta_{0,1})),\ldots ,\Phi^{-1}(F_d(X_d-\theta_{0,d})))$$ is the optimal transport map. This implies, 
	\begin{align}\label{eq:areind_score}
		&\;\;\;\;\left\lVert\Serd^{-\frac{1}{2}}\E_{\mathrm{H}_0}\left[\bJ(\Rmu(\bX))\bh^{\top}\boldsymbol{\eta}(\bm X, \bt_0)\right] \right\rVert^2 \nonumber \\ 
		&= \sum_{i=1}^d h_i^2 \left(\int \Phi^{-1}(F_i(x_i-\theta_{0, i})) \frac{\mathrm d}{\mathrm d x_i}f_i(x_i-\theta_{0,i})\,\mathrm dx_i\right)^2 \nonumber \\ 
		&=\sum_{i=1}^d h_i^2 \left(\bigg(\phi\big(\Phi^{-1}(F_i(x_i-\theta_{0,i}))\big)\bigg)^{-1} f_i^2(x_i-\theta_{0,i})\,\mathrm dx_i\right)^2,
	\end{align}
	where the last step uses integration by parts. Now, as in part~(1) combining~\eqref{eq:AREinterest},~\eqref{eq:areindT2}, and~\eqref{eq:areind_score} gives, 
	\begin{align} \label{eq:areind5}
		& \inf_{\fnd}\atr \nonumber \\ 
		& =\inf_{\fnd}\left(\sum_{i=1}^d \frac{h_i^2}{\sigma_i^2}\right)^{-1} \left\{ \sum_{i=1}^d h_i^2 \left(\int(\phi(\Phi^{-1}(F_i(x_i-\theta_{0,i}))))^{-1} f_i^2(x_i-\theta_{0,i})\,\mathrm dx_i\right)^2 \right\} \nonumber \\ 
		& =1, 
	\end{align}
	Here~\eqref{eq:areind5} follows by considering the following optimization problem: 
	\begin{align}\label{eq:indgauss_optimization}
		\inf_{f_i} \sigma_i^2\left((\phi(\Phi^{-1}(F_i(x_i-\theta_{0,i}))))^{-1}\int f_i^2(x_i-\theta_{0,i})\,\mathrm dx_i\right)^2, 
	\end{align}
	under the same constraints as in part (1). 
	This is precisely the optimization problem that arises in the 1-dimensional case~\cite[Theorem 2.1]{Gastwirth1968}, where the minimum value is $1$ and the minimizing density is that of $\mathcal{N}(0,\sigma^2)$ for any $\sigma>0$. This completes the proof.
	\qed

	\subsubsection{Proof of~\cref{prop:areell}} \label{sec:pfareell}
	
	We begin by formally defining the class $\fel$. Recall that $\bX$ is said to have an elliptically symmetric distribution if there exists $\bt\in\R^d$, a positive definite $d\times d$ matrix $\Sigma$, and a radial density function $\underline{f}(\cdot):\R^+\to\R^+$ such that the density $f_1$ of $\bm X$ satisfies:
	\begin{equation}\label{eq:elldefn}
		f_1(\mx)\propto (\mathrm{det}(\Sigma))^{-\frac{1}{2}}\underline{f}\left((\mx-\bt)^{\top}\Sigma^{-1}(\mx-\bt)\right).
	\end{equation}
	We denote by $\fel$ the class of $d$-dimensional elliptically symmetric distributions satisfying the following standard regularity conditions on the function $\underline{f}$ (see, for example,~\cite{hallin2002optimal}):
	\begin{itemize}	
		\item  $\int_{\R^+} r^{d+1}\underline{f}(r)\,\mathrm dr<\infty,$
		
		\item $\sqrt{\underline{f}}(\cdot)$ admits a weak derivative, which is denoted by $\left(\sqrt{\underline{f}}\right)'(\cdot)$. This means that
		$$\int_{\R^+} \sqrt{\underline{f}(r)}\psi'(r)\,\mathrm dr=-\int_{\R^+} \left(\sqrt{\underline{f}}\right)'(r)\psi(r)\,\mathrm dr,$$
		for all $\psi:\R^+\to\R^+$ which are compactly supported and infinitely differentiable. \item $\int_{\R^+} r^{d-1}\left[\left(\sqrt{\underline{f}}\right)'(r)\right]^2\,\mathrm dr<\infty$. 
	\end{itemize} 
	\medskip

	\noindent{\it Proof of~\cref{prop:areell}}~(1): Define $\overline{\bX}:=\Sigma^{-\frac{1}{2}}(\bX-\bt_0)$. It is easy to check that, in this case,
	$$\tilde{\Sigma}=\E(\bX-\bt_0)(\bX-\bt_0)^{\top}=\left(\frac{1}{d}\E\lVert \overline{\bX}\rVert^2\right)\Sigma.$$
	Therefore, using the same computation as in~\eqref{eq:Gaussare1}, we get that:
	\begin{equation}\label{eq:areell50}
		\atr=\frac{\E\lVert\overline{\bX}\rVert^2}{d}\cdot\frac{\Big\lVert\Serd^{-\frac{1}{2}}\E_{\mathrm{H}_0}\left[\bJ(\Rmu(\bX))\bh^{\top}\boldsymbol{\eta}(\bm X, \bt_0)\right]\Big\rVert^2}{\bh^{\top}\Sigma^{-1}\bh}.
	\end{equation}
	First note that~\eqref{eq:Gaussare1} holds for any elliptically symmetric distribution. Therefore, to prove the result it suffices to compute the non-centrality parameter of the limiting distribution of $\rhsc$ under contiguous alternatives and optimize it over $\fel$.  
	
	Towards this direction, let $\overline{H}(\cdot)$ be the distribution function of the random variable $\lVert \Sigma^{-\frac{1}{2}}(\bX-\bt_0)\rVert$ and $\overline{h}(\cdot)$ be the corresponding density function. By using~\cref{lem:ellipot}, the required optimal transport map in this case is given by:
	$$\Rmu(\bX):=\frac{\Sigma^{-\frac{1}{2}}(\bX-\bt_0)}{\lVert \Sigma^{-\frac{1}{2}}(\bX-\bt_0)\rVert}\overline{H}\left(\lVert \Sigma^{-\frac{1}{2}}(\bX-\bt_0)\rVert\right).$$
	As before, we write $\bm A=\Sigma^{-\frac{1}{2}}$ and $\Rmu(\bX)=(r_1(\bX),\ldots ,r_d(\bX))$. Note that under this notation,~\eqref{eq:Gaussare4},~\eqref{eq:Gaussare51}~and~\eqref{eq:Gaussare63} continue to hold with $H_d(\cdot)$ and $h_d(\cdot)$ replaced by $\overline{H}(\cdot)$ and $\overline{h}(\cdot)$, respectively. Therefore, by using the integration by parts formula as in~\eqref{eq:Gaussare64} gives, 		\begin{align}
		&\;\;\;\left\lVert\Serd^{-\frac{1}{2}}\E_{\mathrm{H}_0}\left[\bJ(\Rmu(\bX))\bh^{\top}\boldsymbol{\eta}(\bm X, \bt_0)\right]\right\rVert^2\nonumber \\
		& =3d\sum_{j=1}^d\left(\sum_{i=1}^n h_i \left\{a_{ij}\left(1- \frac{1}{d} \right)\E\left[\frac{\overline{H}(\lVert\overline{\bX}\rVert)}{\lVert\overline{\bX}\rVert}\right]+\frac{a_{ij}}{d}\E[\overline{h}(\lVert \overline{\bX}\rVert)]\right\}\right)^2\label{eq:computation}\\ 
		&=\frac{3}{d}\cdot \bh^{\top}\Sigma^{-1}\bh \left(\E \overline{h}(\lVert \overline{\bX}\rVert)+(d-1)\E\left[\frac{\overline{H}(\lVert  \overline{\bX}\rVert)}{\lVert \overline{\bX}\rVert}\right]\right)^2, \nonumber 
	\end{align}
	where \eqref{eq:computation} follows by plugging~\eqref{eq:Gaussare4} in~\eqref{eq:Gaussare64}.
	
	Using the above display coupled with~\eqref{eq:areell50} gives, 
	\begin{align}\label{eq:opt_ell}
		\inf_{\fel} \atr=\frac{3}{d^2}\inf_{\fel} \left\{ \E\lVert \overline{\bX}\rVert^2\cdot \left(\E \overline{h}(\lVert \overline{\bX}\rVert)+(d-1)\E\left[\frac{\overline{H}(\lVert \overline{\bX}\rVert)}{\lVert \overline{\bX}\rVert}\right]\right)^2 \right\} . 
	\end{align}
	To solve this optimization problem, we will now proceed in the same way as in the proof of 
	~\cite[Proposition 7]{Hallin2002}. To begin with note that the optimization problem only depends on the distribution of $\overline{\bX}$. Moreover, by replacing the density $h_d(\cdot)$ using the transformation $\overline{h}(r)\mapsto \overline{h}(r/\sigma)/\sigma$, for any $\sigma>0$, it is easy to check that the RHS of \eqref{eq:opt_ell} above does not change. Therefore, we can assume without loss of generality that $\int_0^{\infty} r^2 \overline{h}(r)\,\mathrm dr=1$. Write $v\equiv v(r):=\overline{H}(r)$ and $\dot{v}\equiv \dot{v}(r)=\overline{h}(r)$ (the derivative of $\overline{H}(\cdot)$). Note that for solving \eqref{eq:opt_ell}, it is enough to optimize the following integral form:
	\begin{equation}\label{eq:areell51}
		\inf_{v}\int \left(\dot{v}^2+\frac{d-1}{r} v\dot{v}\right)\, \mathrm dr,\; \mathrm{subject}\ \mathrm{to}\ \int \dot{v}\,\mathrm dr=1 \text{ and } \int r^2\dot{v}\, \mathrm dr=1.
	\end{equation} 
	Let $\lambda_1$ and $\lambda_2$ be the Lagrange multipliers associated with the two constraints in~\eqref{eq:areell51}. Define,
	$$F(r,v,\dot{v}):=\left(\dot{v}+\frac{(p-1)v}{r}+\lambda_1+\lambda_2r^2\right)\dot{v}.$$
	By the Lagrange formula, the optimal solution is an element $v$ that satisfies the constraints:
	$$\frac{\partial F}{\partial v}-\frac{\partial }{\partial r}\frac{\partial F}{\partial \dot{v}}=0,\quad \int_0^{\infty} \dot{v}\, \mathrm dr=1,\quad \text{and} \quad \int_0^{\infty} r^2\dot{v}\, \mathrm dr=1.$$
	It is easy to check that $\frac{\partial F}{\partial v}=\frac{(p-1)\dot{v}}{r}$ and
	\begin{align*}
		\frac{\partial }{\partial r}\frac{\partial F}{\partial \dot{v}}=2\ddot{v}+\frac{(p-1)\dot{v}}{r}-\frac{(p-1)v}{r^2}+2\lambda_2r.
	\end{align*}
	Therefore the Euler-Lagrange equation becomes:
	$$r^2 \ddot{v}-\frac{(p-1)v}{2}=-\lambda_2 r^3.$$
	This is the standard non-homogeneous Cauchy-Euler differential equation. In fact, precisely the same equation appears in~\cite[Equation 21]{hallin2002optimal}. Consequently, using the arguments in~\cite[Page 29-30]{hallin2002optimal},  the lower bound in~\cref{prop:areell}~(1) follows, with the minimizing radial density function given as:
	\begin{align}\label{eq:ellip1}
		\underline{f}(r)&=\frac{1}{\sigma}\Bigg(\frac{9\sqrt{3}}{5\sqrt{5}}\cdot \frac{(\sqrt{2d-1}+1)^{5/2}}{(\sqrt{2d-1}-5)(\sqrt{2d-1}+5)^{3/2}}\cdot \left(\frac{r}{\sigma}\right)^2\nonumber \\ &~~~~~~~~~-\frac{3(\sqrt{2d-1}+1)}{\sqrt{2d-1}-5}\cdot \left(\frac{3(\sqrt{2d-1}+1)}{5(\sqrt{2d-1}+5)}\right)^{(\sqrt{2d-1}+1)/4}\cdot \left(\frac{r}{\sigma}\right)^{(\sqrt{2d-1}-1)/2}\Bigg)\nonumber \\ & ~~~~~~~~~ \ind\left[0<r<\sigma\left(\frac{5(\sqrt{2d-1}+5)}{3(\sqrt{2d-1}+1)}\right)^{1/2}\right]
	\end{align}
	for $d\neq 13$, and
	\begin{equation}\label{eq:ellip2}\underline{f}(r)=\frac{243}{125\sigma^3}\left(\ln\frac{5}{3}-\ln\frac{r}{\sigma}\right)\cdot r^2\cdot \ind\left(0<r<\frac{5\sigma}{3}\right)\end{equation}
	for $d=13$, and some $\sigma>0$.
	\medskip

	\noindent{\it Proof of~\cref{prop:areell}}~(2): By using~\cref{lem:ellipot}, it follows that:
	$$\Rmu(\bX)=\frac{\Sigma^{-\frac{1}{2}} (\bX-\bt_0)}{\lVert \Sigma^{-\frac{1}{2}} (\bX-\bt_0)\rVert}\cdot H_d^{-1}\circ\overline{H}\left(\lVert \Sigma^{-\frac{1}{2}} (\bX-\bt_0)\rVert\right)$$
	is the required optimal transport map. Recall that $H_d(\cdot)$ is the distribution function of a $\sqrt{\chi^2_d}$ distribution and $\overline{H}(\cdot)$ is the distribution function of $\lVert \Sigma^{-\frac{1}{2}}(\bX-\bt_0) \rVert$. Once again, we write $\Rmu(\bX)=(r_1(\bX),\ldots ,r_2(\bX))$ and set $\overline{\bX}=\Sigma^{-\frac{1}{2}}(\bX-\bt_0)$. Note that~\eqref{eq:Gaussare4} holds with $H_d(\cdot)$ replaced with $H_d^{-1}\circ \overline{H}(\cdot)$ and $h_d(\cdot)$ replaced with,
	$$\frac{d}{dr} H_d^{-1}(\overline{H}(r))=\frac{\overline{h}(r)}{h_d(H_d^{-1}(\overline{H}(r)))}.$$
	Using the above observation in~\eqref{eq:Gaussare4}~and~\eqref{eq:Gaussare64}, we get:
	\begin{align*}
		\left\lVert\Serd^{-\frac{1}{2}}\E_{\mathrm{H}_0}\left[\bJ(\Rmu(\bX))\bh^{\top}\boldsymbol{\eta}(\bm X, \bt_0)\right]\right\rVert^2\nonumber  =\frac{\bh^{\top}\Sigma^{-1}\bh}{d^2}  \left(\E \left[ \overline{h}(\lVert \overline{\bX}\rVert) \right] +(d-1)\E\left[\frac{\overline{H}(\lVert  \overline{\bX}\rVert)}{\lVert \overline{\bX}\rVert}\right]\right)^2. 
	\end{align*}
	Plugging the above observation in~\eqref{eq:areell50} gives, 
	\begin{align}\label{eq:areell52}
		\inf_{\fel} & \atr \nonumber \\ 
		& = \inf_{\fel} \frac{\E\lVert \overline{\bX}\rVert^2}{d^3} \left\{ \E \left[ \frac{\overline{h}(\lVert \overline{\bX}\rVert)}{h_d(H_d^{-1}(\overline{H}(\lVert \overline{\bX}\rVert)))} \right] +(d-1)\E\left[\frac{H_d^{-1}\circ\overline{H}(\lVert \overline{\bX}\rVert)}{\lVert \overline{\bX}\rVert}\right]^2 \right\}.
	\end{align}
	
	Now, by~\cite[Theorem 1]{paindaveine2004} (also see~\cite[Lemma 1]{Hallin2008}), the following holds:
	\begin{align}\label{eq:areell53}
		& \left(\E \left[ \frac{\overline{h}(\lVert \overline{\bX}\rVert)}{h_d(H_d^{-1}(\overline{H}(\lVert \overline{\bX}\rVert)))} \right] +(d-1)\E\left[\frac{H_d^{-1}\circ\overline{H}(\lVert \overline{\bX}\rVert)}{\lVert \overline{\bX}\rVert}\right]\right)^2 \nonumber \\ 
		& ~~~~~~~~~~~~~~~~~~~~~~~ \geq d^4\left(\E\left[\lVert \overline{\bX}\rVert H_d^{-1}(\overline{H}(\lVert \overline{\bX}\rVert))\right]\right)^{-2}.
	\end{align}
	By the Cauchy Schwartz inequality, $\left(\E\left[\lVert \overline{\bX}\rVert H_d^{-1}(\overline{H}(\lVert \overline{\bX}\rVert))\right]\right)^{2}\leq d \E\lVert\overline{\bX}\rVert^2$. Using this observation in~\eqref{eq:areell53} yields:
	$$\left(\E \frac{\overline{h}(\lVert \overline{\bX}\rVert)}{h_d(H_d^{-1}(\overline{H}(\lVert \overline{\bX}\rVert)))}+(d-1)\E\left[\frac{H_d^{-1}\circ\overline{H}(\lVert \overline{\bX}\rVert)}{\lVert \overline{\bX}\rVert}\right]\right)^2\geq \frac{d^4}{d\E\lVert\overline{\bX}\rVert^2}=\frac{d^3}{\E\lVert\overline{\bX}\rVert^2}.$$
	Plugging this observation in~\eqref{eq:areell52} completes the proof of~\cref{prop:areell}~(2). \qed

	\subsection{Proofs from~\cref{sec:rankmmd}}\label{sec:pfrankmmd}
	\begin{proof}[Proof of~\cref{theo:consis}]
		First we define an oracle version of~\eqref{eq:kerankscmmd} as follows: 
		\begin{align} 
			\;\;\trsr:=\frac{mn}{m+n}\left[ w_{m, n, \ell}^{(1),\mathrm{or}}  + w_{m, n, \ell}^{(2),\mathrm{or}} - b_{m, n, \ell}^{\mathrm{or}}  \right], 
			\label{eq:kerankorH01} 
		\end{align}  
		where
		\begin{align*}
			w_{m, n, \ell}^{(1),\mathrm{or}} & := \frac{1}{m(m-1)}\sum_{1\leq i\neq j \leq m} \mathsf{K}(\bJ(\Rb(\bX_i)),\bJ(\Rb(\bX_j))), \nonumber \\ 
		w_{m, n, \ell}^{(2),\mathrm{or}}  & := \frac{1}{n(n-1)}\sum_{1\leq i\neq j \leq n} \bKe(\bJ(\Rb(\bY_i)),\bJ(\Rb(\bY_j))), \nonumber \\ 
			b_{m, n, \ell}^{\mathrm{or}} & := \frac{2}{mn}\sum_{1 \leq i \leq m}\sum_{1 \leq j \leq n} \bKe(\bJ(\Rb(\bX_i)),\bJ(\Rb(\bY_j))) , 
		\end{align*}
		%
		for $\ell=0,1$. Note that, $\gamma_{m,n,0}^{\nu,\bJ,\mathrm{or}}$ coincides with~\eqref{eq:kerankscmmdor}. We now claim the following, which we shall prove later. 
		\begin{equation}\label{eq:consis1}
			\frac{1}{N}(\grsc-\trsr)\overset{P}{\longrightarrow}0,\qquad \mbox{under}\ \mathrm{H}_{\ell}.
		\end{equation}

		First we complete the proof by assuming claim~\eqref{eq:consis1}. Note that the strong law of large numbers for $U$-statistics implies, 
		\begin{align}
			w_{m, n, \ell}^{(1),\mathrm{or}} & \overset{a.s.}{\longrightarrow}  \E_{\mathrm{H}_{\ell}} \bKe(\bJ(\Rb(\bX_1))  ,\bJ(\Rb(\bX_2))) := w_\ell^{(1)}(\mu_1) \nonumber \\ 
			w_{m, n, \ell}^{(2),\mathrm{or}}  & \overset{a.s.}{\longrightarrow} \E_{\mathrm{H}_{\ell}} \bKe(\bJ(\Rb(\bY_1)),\bJ(\Rb(\bY_2)))  :=  w_\ell^{(2)}(\mu_2), \nonumber \\ 
			b_{m, n, \ell}^{\mathrm{or}} & \overset{a.s.}{\longrightarrow} 2\E_{\mathrm{H}_{\ell}} \bKe(\bJ(\Rb(\bX_1)),\bJ(\Rb(\bY_1))) := b_\ell(\mu_1, \mu_2), \nonumber 
		\end{align} 
		under $\mathrm{H}_{\ell}$. 
		Hence, 
		\begin{align*}
			\frac{1}{N} \trsr \overset{a.s.}{\longrightarrow} & \lambda(1-\lambda) ( w_\ell^{(1)}(\mu_1) + w_\ell^{(2)}(\mu_2) - b_\ell(\mu_1, \mu_2) ) := \gamma_{\ell}^2(\mu_1, \mu_2),	
		\end{align*} 
		under $\mathrm{H}_{\ell}$.  
		Consequently by~\eqref{eq:consis1},
		\begin{equation}\label{eq:store2}
			\frac{1}{N}\grsc\overset{P}{\longrightarrow} \gamma_{\ell}^2(\mu_1, \mu_2),\quad \mbox{under}\ \mathrm{H}_{\ell}.
		\end{equation}
		As $\bJ(\cdot)$ is injective and $\Rl(\cdot)$ is invertible in the sense of~\cref{prop:Mccan}, by standard considerations of kernel maximum mean discrepancy (see~\cite[Theorem 5 and Lemma 6]{gretton2008}), we get that $\gamma_{\ell}^2(\mu_1, \mu_2)>0$ whenever $\mu_1\neq \mu_2$  and $\gamma_{\ell}^2(\mu_1, \mu_2)=0$ if $\mu_1=\mu_2$. Therefore, it follows from~\eqref{eq:testrankker}~and~\eqref{eq:uniflevelker} that 
		\begin{equation}\label{eq:store1}
			c_{m,n}=o(N),
		\end{equation}
		in the usual asymptotic regime~\eqref{eq:usual} . Consequently,
		$$\E_{\mathrm{H}_1}\left[\ptsc\right]=\P_{\mathrm{H}_1}\left[\frac{1}{N} \grsc\geq \frac{1}{N} c_{m,n}\right]\to 1,$$
		where the last limit follows by combining~\eqref{eq:store2}~and~\eqref{eq:store1}. This completes the proof.
		
		Now we move on to the proof of~\eqref{eq:consis1}. Observe that:
		\begin{align}\label{eq:consis2}
			&\;\;\;\frac{1}{N} |\grsc-\trsr|\lesssim T_1 + T_2+T_3 , 
		\end{align} 
		where 
		\begin{align} 
			T_1 & := \frac{1}{m^2}\sum_{1 \leq i\neq j \leq m} \big|\bKe(\bJ(\hbR_{m,n}(\bX_i)),\bJ(\hbR_{m,n}(\bX_j)))-\bKe(\bJ(\Rb(\bX_i)),\bJ(\Rb(\bX_j)))\big| , \nonumber \\ T_2 & := \frac{1}{mn}\sum_{i=1}^m\sum_{j=1}^n \big|\bKe(\bJ(\hbR_{m,n}(\bX_i)),\bJ(\hbR_{m,n}(\bY_j)))-\bKe(\bJ(\Rb(\bX_i)),\bJ(\Rb(\bY_j)))\big| , \nonumber \\ 
			T_3 &:= \frac{1}{n^2}\sum_{1 \leq i\neq j \leq n}\big|\bKe(\bJ(\hbR_{m,n}(\bY_i)),\bJ(\hbR_{m,n}(\bY_j)))-\bKe(\bJ(\Rb(\bY_i)),\bJ(\Rb(\bY_j)))\big|. \nonumber 
		\end{align}
		Observe that $T_1$ converges to $0$ in probability under $\mathrm{H}_{\ell}$ by using~\cref{theo:rankmapcon} with $p=2$, $q=1$ and $\cF(\cdot,\cdot)=\bKe(\cdot,\cdot)$. The other terms in~\eqref{eq:consis2} can be handled similarly, thereby establishing the claim in~\eqref{eq:consis1}. 
	\end{proof}
	
	\begin{proof}[Proof of~\cref{theo:nullker}]
		The main step in this proof is to establish~\eqref{eq:basestep1}. This requires us to prove a H\'{a}jek projection result under the null, which is of independent interest, and hence stated as a theorem below. Its proof is deferred to the end of this section.
		\begin{theorem}[Asymptotic multivariate H\'ajek representation]\label{theo:hajek} 
			Consider the same set of assumptions as in~\cref{theo:nullker} and recall \eqref{eq:kerankscmmdor}. Then the following conclusion holds:
			\begin{align}\label{eq:gamma_difference}
				\E (\grsc-\grsr)^2=o(1), 
			\end{align}
			which in turn implies $|\grsc-\grsr|=o_{P}(1)$.
		\end{theorem}
		The proof of~\cref{theo:hajek} is provided at the end of this section. Based on~\cref{theo:hajek} however, the proof of~\cref{theo:nullker} is immediate. Note that by~\cite[Theorem 12]{Gretton2012}, we have:
		$$\grsr\overset{w}{\longrightarrow} \sum_{i=1}^{\infty}\varpi_i(G_i^2-1)$$
		where $\varpi_i$'s and $G_i$'s are taken from the statement~\cref{theo:nullker}. Combining the above observation with~\cref{theo:hajek} and Slutsky's Theorem, completes the proof.
	\end{proof}
	
	\begin{proof}[Proof of~\cref{theo:Piteffrank}]
		We will first prove the result for the model in~\eqref{eq:twosamsmooth}. Note that $\mathrm{H}_{1}$ and $\mathrm{H}_0$ are mutually contiguous (see~\cite[Corollary 12.3.1]{LR05}). By using~\cref{theo:hajek} we  observe that, given any $\epsilon>0$, we have:
		\begin{align}\label{eq:contig1}
			\zph(|\grsc-\grsr|\geq \epsilon)\to 0 \; \implies \; \pho(|\grsc-\grsr|\geq \epsilon)\to 0
		\end{align}
		by contiguity. This allows us to reduce the analysis of $\grsc$ to that of $\grsr$ under both $\pho$ and $\zph$. Once we have reduced the problem to $\grsr$, the argument is similar to other similar results obtained in~\cite{Chikkagoudar2014,kim2018robust,Gregory1977}. 
		
		\noindent Towards this direction, define $\bH:=\bJ\circ \Rmu$ and set,
		\begin{align*} 
			\tilde{\mathsf{K}}(\bH(\mx),\bH(\my)) &  :=\bKe(\bH(\mx),\bH(\my))-\ezph[\bH(\mx),\bH(\bY)] \nonumber \\ 
			& -\ezph[\bH(\bX),\bH(\my)]+\ezph[\bH(\bX),\bH(\bY)] , 
		\end{align*}
		where $\bX,\bY$ are i.i.d. $\mu_1=\mu_2$. Note that there is a slight abuse of notation here with $\tilde{\mathsf{K}}$ as defined in the main paper. They are actually equivalent up to a change of variable. Nevertheless we work with the above representation in this proof.
		
		Define $\tilde{\mathcal{L}}^2:=L^2(\R^d,\mu_1)$ and $\mathcal{L}^2:=L^2(\R^d\times\R^d,\mu_1\otimes \mu_1)$. According to~\cite[Theorem VI.23]{Reed1980}, there exists eigenvalues $\varpi_i$ and eigenfunctions $\Psi_i(\cdot)$'s in~$\tilde{\mathcal{L}}^2$ for $\tilde{\mathsf{K}}(\cdot,\cdot)$, such that the following conclusions hold:
		\begin{align}\label{eq:eigen1}
			\tilde{\mathsf{K}}(\bH(\mx),\bH(\my))=\sum_{i=1}^{\infty}\varpi_i\Psi_i(\bH(\mx))\Psi_i(\bH(\my)), 	\end{align} 
		where the convergence is in $\mathcal{L}^2$, 
		\begin{align}\label{eq:KH}
			\int \tilde{\mathsf{K}}(\bH(\mx),\bH(\my)) \Psi(\bH(\my))\,\mathrm d\mu_1 &= \varpi_i\Psi_i(\bH(\mx)), \nonumber \\ 
			\int \Psi_i(\bH(\mx))\Psi_j(\bH(\mx))\,\mathrm d\mu_1 & =\mathbf{1}(i=j), 
		\end{align} 
		for $i, j \geq 1$, 
		and 
		\begin{align}\label{eq:eigen2}
			\varpi_i\ezph[\Psi_i(\bH(\bX))]=0\;\; \mathrm{and}\;\; \sum_{i=1}^{\infty}\varpi_i^2<\infty.
		\end{align} 
		
		It is sufficient to prove that:
		\begin{equation}\label{eq:claimun}
			\grsr\overset{w}{\longrightarrow}\sum_{i=1}^{\infty}\varpi_i\left[\left(G_i+\sqrt{\lambda(1-\lambda)}\mathbb{E}\left[\Psi_i(\bH(\bX))\bh^{\top}\boldsymbol{\eta}(\bX,\bt_0)\right]\right)^2-1\right].
		\end{equation}
		It is easy to check that $\grsr$ can be rewritten as follows: 
		$$\grsr = \frac{mn}{m+n}\left[ \tilde{w}_{m, n}^{(1),\mathrm{or}, u}  + \tilde{w}_{m, n}^{(2),\mathrm{or}, u} - \tilde{b}_{m, n}^{\mathrm{or}, u}  \right] .$$ 
		where 
		\begin{align}
			\tilde{w}_{m, n}^{(1),\mathrm{or}, u} & := \frac{1}{m(m-1)}\sum_{1\leq i \ne j \leq m} \tilde{\mathsf{K}}(\bJ(\Rmu(\bX_i)),\bJ(\Rmu(\bX_j))), \nonumber \\ 
			\tilde{w}_{m, n}^{(2),\mathrm{or}, u}  & := \frac{1}{n(n-1)}\sum_{1\leq i \ne j \leq n} \tilde{\bKe}(\bJ(\Rmu(\bY_i)),\bJ(\Rmu(\bY_j))), \nonumber \\ 
			\tilde{b}_{m, n}^{\mathrm{or}, u} & := \frac{2}{mn}\sum_{1 \leq i \leq m}\sum_{1 \leq j \leq n} \tilde{\bKe}(\bJ(\Rmu(\bX_i)),\bJ(\Rmu(\bY_j))) . \nonumber 
		\end{align}

		Next, the main idea here is to approximate $\tilde{\mathsf{K}}(\cdot,\cdot)$ which has an infinite expansion (in $\mathcal{L}^2$), with a truncated expansion. Accordingly, for any $L\geq 1$, we define, 
		\begin{align}\label{eq:kernelmnL}
			\grln = \frac{mn}{m+n}\left[ \tilde{w}_{m, n, L}^{(1),\mathrm{or}, u}  + \tilde{w}_{m, n, L}^{(2),\mathrm{or}, u} - \tilde{b}_{m, n, L}^{\mathrm{or}, u}  \right] . 
		\end{align} 
		where 
		\begin{align*}
			\tilde{w}_{m, n, L}^{(1),\mathrm{or}, u} &:=\frac{1}{m(m-1)}\sum_{1\leq i\neq j \leq m} \sum_{\ell=1}^L  \varpi_{\ell} \Psi_\ell(\bH(\bX_i))\Psi_\ell(\bH(\bX_j)) \nonumber \\ 
			\tilde{w}_{m, n, L}^{(2),\mathrm{or}, u} & := \frac{1}{n(n-1)}\sum_{1 \leq i\neq j \leq n} \sum_{\ell=1}^L  \varpi_{\ell} \Psi_\ell(\bH(\bY_i))\Psi_\ell(\bH(\bY_j)) \nonumber \\ 
			\tilde{b}_{m, n, L}^{\mathrm{or}, u} & := \frac{2}{mn}\sum_{1 \leq i \leq m} \sum_{1 \leq j \leq n} \sum_{\ell=1}^L  \varpi_{\ell} \Psi_\ell(\bH(\bX_i))\Psi_\ell(\bH(\bY_j)) . 
		\end{align*}
		Also, define 
		\begin{align*} 
			\Theta_L&:=\sum_{\ell=1}^L  \varpi_{\ell} \left[\left(G_i+\sqrt{\lambda(1-\lambda)}\E_{{\mu_1}}\Psi_\ell(\bH(\bX))\bh^{\top}\boldsymbol{\eta}(\bX,\bt_0)\right)^2-1\right].
		\end{align*}
		where $(G_1,G_2,\ldots ,)$ is an infinite sequence of i.i.d. standard Gaussian random variables. We claim that the following three conclusions hold:
		\begin{enumerate}
			\item Given any $\epsilon>0$ and any sequence $L_N\to\infty$, $\pho(|\grnn-\grun|\geq \epsilon)\to 0$.
			\item For any $\epsilon>0$, we have  $\lim_{L,\tilde{L}\to\infty}\P(|\Theta_L-\Theta_{\tilde{L}}|\geq \epsilon)=0$.
			\item For any fixed $L\geq 1$, $\grln\overset{w}{\longrightarrow}\Theta_L$ under $\pho$.
		\end{enumerate}
		Combining the above claims with~\cite[Lemma 2.5]{Sen2010}
		will complete the proof of~\eqref{eq:claimun} on using~\eqref{eq:contig1}.
		
		First we prove step 1 above. The argument here is similar to that in~\cite[Page 39-40]{Gretton2012}, where the fact that $\lim_{L \rightarrow \infty}\sum_{\ell=L+1}^{\infty} \varpi_\ell^2 = 0$ is used along with~\cite[Section 5.5.2]{Serfling1980}~and~\cite[Theorem 9.8.2]{Dudley2002}. In particular, the authors show in~\cite[Page 39-40]{Gretton2012} that $\zph(|\grnn-\grun|\geq \epsilon)\to 0$. By using contiguity, the same conclusion holds under $\pho$. We omit further details for brevity.
		
		We then move onto step 3. For $L<\tilde{L}$, a simple second moment computation shows:
		$$\E[\Theta_L-\Theta_{\tilde{L}}]^2\lesssim \sum_{\ell=L}^{\tilde{L}}  \varpi_{\ell} ^2\E[ G_{\ell}^2-1 ]^2+\sum_{\ell=L}^{\tilde{L}}  \varpi_{\ell} ^2c_{\ell}^2+\left(\sum_{\ell=L}^{\tilde{L}}  \varpi_{\ell} c_{\ell}^2\right)^2 , $$
		where $$c_{\ell}:=\sqrt{\lambda(1-\lambda)}\E_{\mathrm{H}_0}\Psi_\ell(\bH(\bX))\bh^{\top}\boldsymbol{\eta}(\bX,\bt_0).$$ 
		Note that $\sum_{\ell=1}^{\infty}  \varpi_{\ell} ^2<\infty$ by \eqref{eq:eigen2}. Moreover, by the Cauchy-Schwarz inequality 
		\begin{align}
			c_{\ell}^2 & \leq \lambda(1-\lambda) \E_{\mathrm{H}_0} \left[ \Psi_\ell(\bH(\bX))^2 \right]  \E_{\mathrm{H}_0} \left[ \left( \bh^{\top} \boldsymbol{\eta}(\bX,\bt_0) \right)^2 \right]  \nonumber \\ 
			& = h^2 \lambda(1-\lambda)  \E_{\mathrm{H}_0} \left[ \left( \bh^{\top} \boldsymbol{\eta}(\bX,\bt_0) \right)^2 \right] < \infty , 
		\end{align}
		using \eqref{eq:KH} and the assumption that $ \E_{\bm \theta_0} \left[ \| \boldsymbol{\eta}(\bX,\bt_0) \|^2 \right] < \infty$. Hence, $\limsup_{L\to\infty}\sum_{\ell=1}^L \varpi_{\ell}^2  c_{\ell}^2 < \infty$.  Now, we show that $\limsup_{L\to\infty} \sum_{\ell=1}^L \varpi_{\ell}  c_{\ell}^2<\infty$. 
		This follows by observing that:
		\begin{align*}
			&\;\;\;\;\;\limsup\limits_{L\to\infty}\sum_{l=1}^L \varpi_l c_l^2\\ &=\lambda(1-\lambda)\limsup\limits_{L\to\infty} \int \sum_{l=1}^L \varpi_l \Psi_l(\bH(\mx))\Psi_l(\bH(\my))\bh^{\top}\boldsymbol{\eta}(\mx,\bt_0)\boldsymbol{\eta}(\my,\bt_0)^{\top}\bh \, \mathrm d\mu_1(\mx)\, \mathrm d\mu_1(\my)\\ &\leq \lambda(1-\lambda) \sqrt{\limsup\limits_{L\to\infty}\E \left[\left(\sum_{l=1}^L \varpi_l\Psi_l(\bH(\bX))\Psi_l(\bH(\bY))\right)^2 \right] }\E\left(\bh^{\top}\boldsymbol{\eta}(\bX,\bt_0)\right)^2<\infty.
		\end{align*}
		Here we have used the Cauchy-Schwartz inequality in the final line, with $\bX,\bY$ i.i.d. $\mu_1$. The finitenesss claim above follows from~\eqref{eq:eigen1} and the assumption $\E_{\bt_0}\lVert \boldsymbol{\eta}(\bX,\bt_0)\rVert^2<\infty$.
		
		We now move on to step 2. The proof for this part is based on Le Cam's Third Lemma (see~\cite[Theorem 6.6]{Van1998}). Recalling the definition of $\tilde{w}_{m, n, L}^{(1),\mathrm{or}, u}$ from \eqref{eq:kernelmnL} note that

		Note that the first term of $\grln$ can be written as,
		\begin{align*} 
			\frac{mn}{m+n} \tilde{w}_{m, n, L}^{(1),\mathrm{or}, u} & = \frac{mn}{Nm(m-1)}\sum_{\ell=1}^L\sum_{i\neq j}  \varpi_{\ell} \Psi_\ell(\bH(\bX_i))\Psi_\ell(\bH(\bX_j))\\ 
			&=\frac{mn}{N(m-1)}\sum_{\ell=1}^L \left[\left(\frac{\sqrt{ \varpi_{\ell} }\sum_{i=1}^m \Psi_\ell(\bH(\bX_i))}{\sqrt{m}}\right)^2-\frac{ \varpi_{\ell} \sum_{i=1}^m \Psi_\ell ^2(\bH(\bX_i))}{m}\right].
		\end{align*}
		A similar expression can be written for $\frac{mn}{m+n} \tilde{w}_{m, n, L}^{(1),\mathrm{or}, u}$ and $\frac{mn}{m+n} \tilde{b}_{m, n, L}^{\mathrm{or}, u}$ as well. Also, note that the likelihood ratio for testing $\mathrm{H}_0$ versus $\mathrm{H}_1$ is given as follows:
		$$V_N:=\sum_{i=1}^N \log\left\{\frac{f(\bX_j|\bt_0+N^{-1/2}\bh)}{f(\bX_j|\bt_0)}\right\}.$$
		In order to apply Le Cam's third lemma, we therefore need to study the limiting joint distribution of $((\mathcal{Z}_N^L)^{\top}, (\mathcal{W}_N^L)^{\top},V_N)$ under $\zph$, where $\mathcal{Z}_N^L:=(Z_{N,1},\ldots ,Z_{N,L})^{\top}$, $\mathcal{W}_N^L:=(W_{N,1},\ldots ,W_{N,L})^{\top}$ with 
		$$Z_{N, \ell}:= \frac{1}{\sqrt{m}} \sum_{i=1}^m \Psi_\ell(\bH(\bX_i)) \quad \text{ and 
		} \quad W_{N, \ell} := \frac{1}{\sqrt{n}} \sum_{j=1}^n \Psi_\ell(\bH(\bY_j)).$$ 
		Note that by~\cite[Theorem 12.2.3]{LR05}, it suffices to analyze the limiting joint distribution of $(\mathcal{Z}_N^L,\mathcal{W}_N^L,\dot{V}_N)$ under $\zph$, where recall that
		\begin{align*}
			\dot{V}_N:=\dot{V}_N-\frac{(1-\lambda)}{2}\bh^{\top}I(\bt_0)\bh+o_{P}(1),\quad \dot{V}_N:=\frac{1}{\sqrt{N}}\sum_{j=1}^n \bh^{\top}\boldsymbol{\eta}(\bY_j,\bt_0).
		\end{align*}
		and $I(\cdot)$ is the Fisher Information matrix. 
		By the multivariate central limit theorem and~\eqref{eq:eigen2}, we have, under $\mathrm{H}_0$:
		\begin{align*}
			\begin{pmatrix} \mathcal{Z}_N^L\\ \mathcal{W}_N^L\\ V_N \end{pmatrix} \overset{w}{\to}\mathcal{N}\left(\begin{pmatrix} \mathbf{0}_{L\times 1} \\ \mathbf{0}_{L\times 1}\\ -((1-\lambda)/2)\bh^{\top}I(\bt_0)\bh\end{pmatrix},\begin{pmatrix} \mathrm{Id}_{L\times L} & \mathbf{0}_{L\times L} & \mathbf{0}_{L\times 1}\\ \mathbf{0}_{L\times L} & \mathrm{Id}_{L\times L} & \bm p \\ \mathbf{0}_{1\times L} & \bm p^\top & (1-\lambda)\bh^{\top}I(\bt_0)\bh\end{pmatrix}\right) , 
		\end{align*}
		where
		$$\bm p:=\sqrt{1-\lambda} \begin{pmatrix}
			\ezph\left[\Psi_1(\bH(\bX))\bh^{\top}\boldsymbol{\eta}(\bX,\bt_0)\right] \\ 
			\ezph\left[\Psi_2(\bH(\bX))\bh^{\top}\boldsymbol{\eta}(\bX,\bt_0)\right] \\
			\vdots \\ 
			\ezph\left[\Psi_L(\bH(\bX))\bh^{\top}\boldsymbol{\eta}(\bX,\bt_0)\right]  
		\end{pmatrix} . $$
		By appealing to Le Cam's Third Lemma, we then have under $\mathrm{H}_1$,
		\begin{align*}
			\begin{pmatrix} \mathcal{Z}_N^L\\ \mathcal{W}_N^L\end{pmatrix} \overset{w}{\to}\mathcal{N}\left(\begin{pmatrix} \mathbf{0}_{L\times 1} \\ \bm p\end{pmatrix},\begin{pmatrix} \mathrm{Id}_{L\times L} & \mathbf{0}_{L\times L} \\ \mathbf{0}_{L\times L} & \mathrm{Id}_{L\times L} \end{pmatrix}\right).
		\end{align*}
		Let $(Z_1,\ldots ,Z_L,W_1,\ldots ,W_L)$ be a sequence of i.i.d. standard Gaussian random variables. An application of the continuous mapping theorem then yields the following (under $\mathrm{H}_1$):
		\begin{align*}
			& \grln \\ 
			& \overset{w}{\to}\lambda\sum_{\ell=1}^L  \varpi_{\ell} (Z_{\ell}^2-1)+(1-\lambda)\sum_{\ell=1}^L \left\{ \varpi_{\ell} ((W_{\ell}+p_{\ell})^2-1)-2(\lambda(1-\lambda)) \varpi_{\ell} Z_{\ell}(W_{\ell}+p_{\ell}) \right\} \\ &\overset{w}{=}\sum_{\ell=1}^L  \varpi_{\ell} \left[\left(\sqrt{\lambda}(W_{\ell}+p_{\ell})-\sqrt{1-\lambda}Z_{\ell}\right)^2-1\right]\\ &\overset{w}{=}\sum_{\ell=1}^L  \varpi_{\ell} \left[\left(G_{\ell}+\sqrt{\lambda(1-\lambda)}\E_{\mathrm{H}_0}\Psi_\ell(\bH(\bX))\bh^{\top}\boldsymbol{\eta}(\bX,\bt_0)\right)^2-1\right].
		\end{align*}
		This completes the proof of \cref{theo:Piteffrank}.
	\end{proof}
	
	\begin{remark}
		By the above arguments, it is evident that the rate of convergence of $|\grsc-\grsr|$ is governed by the rate of convergence of $\frac{1}{N^2} \sum_{i\neq j} \E[(\hG_{ij}-\nhg_{ij})^2]$. Under stronger assumptions, it is possible to quantify this rate using~\cite[Theorem 2.2]{deb2021rates}. 
	\end{remark} 
	
	\begin{proof}[Proof of~\cref{theo:hajek}]
		
		We will first write the standard decomposition,
		\begin{equation}\label{eq:baseq1}
			\E[(\grsc-\grsr)^2]=\E [(\grsc)^2] -2 \E [\grsc \grsr ]+ \E [(\grsr)^2].
		\end{equation}
		Next we will simplify each term in~\eqref{eq:baseq1}. In order to do this, we will use the notion of the permutation distribution from~\cref{def:pdist}.
		
		We begin with the first term from the right hand side of~\eqref{eq:baseq1}. In the subsequent discussion, unless otherwise stated, in all summation signs, the indices will vary from $1$ to $N$. Let $\E_{\Z}$ denote as the conditional expectation given $\mathcal{Z}_N$ (that is, the permutation distribution in~\cref{def:pdist}). Also, set 
		$$\hG_{ij}:=\bKe(\bJ(\hbR_{m,n}(\bZ_i)),\bJ(\hbR_{m,n}(\bZ_j))),$$
		and $N_0:=mn/N$. Then we get 
		\begin{align}\label{eq:firsterm}
			\frac{1}{N_0^{2}}\E[(\grsc)^2]=\E\E_{\Z}\left[ \mathcal{W}_1 + \mathcal{W}_2 - \mathcal{B} \right]^2 , 
		\end{align} 
		where 
		\begin{align*}
			\mathcal{W}_1 & = \frac{2}{m(m-1)} \sum_{1 \leq i<j \leq N} \hG_{ij} \ind(L_i=L_j=1) , \nonumber \\ 
			\mathcal{W}_2 & = \frac{2}{n(n-1)} \sum_{1 \leq i<j \leq N} \hG_{ij} \ind(L_i=L_j=2) , \nonumber \\ 
			\mathcal{B} & = \frac{2}{mn}  \sum_{1 \leq i,j \leq N}  \hG_{ij} \ind(L_i=1,L_j=2) . 
		\end{align*} 
		
		Clearly, when we expand the squares in~\eqref{eq:firsterm}, each resulting term will involve a summation over a quadruple of indices, say $\{i,j,s,t\}$. Now $\{\hG_{ij}\}_{1 \leq i, j \leq N}$ are measurable with respect to the sigma field induced by $\Z$. Therefore, the conditional expectation $\E_{\Z}$ only operates on the indicator variables above to yield the corresponding probabilities. These probabilities are governed by the number of distinct indices in $\{i,j,s,t\}$. Towards this direction, given a collection of indices $S$, let $|S|$ denote the number of distinct indices in $S$. Suppose that $L_S:=\{L_i:i\in S\}$. Given two positive integers $a$ and $b$, $a\geq b$, let $p(a,b)=a!/(a-b)!$. Then a simple combinatorial argument shows that
		\begin{align}
			\alpha_{|S|}^{(1)}:=\P(L_S=1|\Z)=\frac{p(m,|S|)}{p(N,|S|)},\quad \alpha_{|S|}^{(2)}:=\P(L_S=2|\Z)=\frac{p(n,|S|)}{p(N,|S|)},\; \label{eq:genp1} 
		\end{align} 
		and 
		\begin{align} 
			\alpha_{|S_1|,|S_2|}:=\P(L_{S_1}=1, L_{S_2}=2|\Z)=\frac{p(m,|S_1|)p(n,|S_2|)}{p(N,|S_1|+|S_2|)} . \label{eq:genp2} 
		\end{align} 
		(Clearly, $\alpha_{|S|, 0} = \alpha_{|S|}^{(1)}$ and $\alpha_{0, |S|} = \alpha_{|S|}^{(2)}$.) Using~\eqref{eq:genp1}~and~\eqref{eq:genp2}, we will simplify the first term in the right hand side of~\eqref{eq:firsterm}. The crucial idea here is to split the summation according to the number of distinct indices.
		\begin{align}\label{eq:confirstermW1}
			\E\left[ \mathcal{W}_1^2 \right] & = \frac{1}{m^2(m-1)^2} \left( 2T_1 + 4 T_2 + T_3 \right) , 
		\end{align} 
		where 
		\begin{align*}	
			T_1 & :=\E\E_{\Z} \left[ \sum_{1 \leq i\neq j \leq N}\hG_{ij}^2 
			\ind(L_i=L_j=1) \right] =  \alpha_2^{(1)} \E \left[\sum_{1 \leq i\neq j \leq N}\hG_{ij}^2 \right] , \nonumber \\ 
			T_2& := \E\E_{\Z} \left[ \sum_{1 \leq i\neq j\neq s \leq N}\hG_{ij}\hG_{si}\ind(L_i=L_j=L_s=1) \right] =  \alpha_3^{(1)} \E \left[ \sum_{1 \leq i\neq j\neq s \leq N}\hG_{ij}\hG_{si}  \right] , \nonumber \\ 
			T_3 & := \E\E_{\Z} \left[\sum_{1 \leq i\neq j\neq s\neq t \leq N}\hG_{ij}\hG_{st}\ind(L_i=L_j=L_s=L_t=1) \right]  = \alpha_4^{(1)} \E \left[\sum_{1 \leq i\neq j\neq s\neq t \leq N}\hG_{ij}\hG_{st} \right] . 
		\end{align*} 	
		Hence, using \eqref{eq:confirstermW1}, 
		\begin{align}\label{eq:confirsterm1}
			\E\left[ \mathcal{W}_1^2 \right] & = \frac{1}{m^2(m-1)^2} \E \left[ 2 \alpha_2^{(1)} \mathcal A_{\hG} + 4 \alpha_3^{(1)}  \mathcal B_{\hG} + \alpha_4^{(1)}  \mathcal C_{\hG} \right] .
		\end{align} 
		where $\mathcal A_{\hG}:=\sum_{1 \leq i\neq j \leq N}\hG_{ij}^2$,  $\mathcal B_{\hG} := \sum_{1 \leq i\neq j\neq s \leq N}\hG_{ij}\hG_{si} $, and $\mathcal C_{\hG} := \sum_{1 \leq i\neq j\neq s\neq t \leq N}\hG_{ij}\hG_{st}$.  
		Similarly, 
		\begin{align}\label{eq:firsterm2}
			\E\left[ \mathcal{W}_2^2 \right] & = \frac{1}{n^2(n-1)^2} \E \left[ 2 \alpha_2^{(2)} \mathcal A_{\hG} + 4 \alpha_3^{(2)}  \mathcal B_{\hG} + \alpha_4^{(2)}  \mathcal C_{\hG} \right] .
		\end{align}

		We can carry out similar computations for the other terms from the right hand side of~\eqref{eq:baseq1}. The idea for all the terms is the same as above and we present the final expressions for all the terms arising out of expanding the squares in~\eqref{eq:baseq1} below:
		\begin{align}\label{eq:firsterm3}
			\E\left[ \mathcal{B}^2 \right] &=\frac{4}{m^2n^2}\E\left[\alpha_{1,1} \mathcal A_{\hG} + ( \alpha_{1,2} +  \alpha_{2,1} ) \mathcal B_{\hG}  + \alpha_{2,2} \mathcal C_{\hG} \right] , 
		\end{align} 
		\begin{align}\label{eq:firsterm4} 
			\E\left[ \mathcal W_1 \mathcal B \right]  & =\frac{2}{m^2 (m-1) n}\E\left[ 2 \alpha_{2,1} \mathcal B_{\hG}  +\alpha_{3,1}  \mathcal C_{\hG} \right] , 
		\end{align} 
		\begin{align}\label{eq:firsterm5}
			\E\left[ \mathcal W_2 \mathcal B \right] & =\frac{2}{mn^2 (n-1)}\E\left[ 2 \alpha_{1,2} \mathcal B_{\hG} + \alpha_{1,3} \mathcal C_{\hG} \right] , 
		\end{align} 
		and 
		\begin{align}\label{eq:firsterm6}
			\E\left[ \mathcal W_1 \mathcal W_2 \right] &=\frac{1}{m(m-1)n(n-1)}\E\left[ \alpha_{2,2} \mathcal C_{\hG} \right] . 
		\end{align} 
		Plugging in the expressions from~\eqref{eq:confirsterm1},~\eqref{eq:firsterm2},~\eqref{eq:firsterm3},~\eqref{eq:firsterm4},~\eqref{eq:firsterm5},~\eqref{eq:firsterm6} into~\eqref{eq:firsterm}, we get:
		\begin{align}\label{eq:gammaabc}
			\frac{1}{N_0} \E[(\grsc)^2] & =\E\Bigg[ a_{m, n} \mathcal A_{\hG} + b_{m, n} \mathcal B_{\hG}  +c_{m, n}  \mathcal C_{\hG}  \Bigg] , 
		\end{align}
		where  
		\begin{align*}
			a_{m, n} & := \frac{2\alpha_{2}^{(1)}}{m^2(m-1)^2}+\frac{2\alpha_{2}^{(2)}}{n^2(n-1)^2}+\frac{4\alpha_{1,1}}{m^2n^2} , \nonumber \\ 
			b_{m, n} & := \frac{4\alpha_{3}^{(1)}}{m^2(m-1)^2}+\frac{4\alpha_{3}^{(2)}}{n^2(n-1)^2}+\frac{4\alpha_{1,2}}{m^2n^2}+\frac{4\alpha_{2,1}}{m^2n^2}-\frac{8\alpha_{2,1}}{m^2 (m-1) n} -\frac{8\alpha_{1,2}}{mn^2 (n-1) } , \nonumber \\ 
			c_{m, n} & :=  \frac{\alpha_{4}^{(1)}}{m^2(m-1)^2}+\frac{\alpha_{4}^{(2)}}{n^2 (n-1)^2}+\frac{4\alpha_{2,2}}{m^2n^2}-\frac{4\alpha_{3,1}}{m^2(m-1)n}-\frac{4\alpha_{1,3}}{mn^2(n-1) } \nonumber \\ 
			& \hspace{2.0in} + \frac{2\alpha_{2,2}}{m(m-1)n(n-1)} . 
		\end{align*} 
		%
		%
		
		Next, define $\nhg_{ij}:=\bKe(\bJ(\Rmu(\mz_i)),\bJ(\Rmu(\mz_j)))$. It is easy to see that $\E[(\grsr)^2]$ can be dealt with in the exact same fashion as $\E[(\grsc)^2]$ with $\hG_{ij}$ is replaced by $\nhg_{ij}$. To avoid clutter and for future reference, we only present the final expression for $\E[(\grsr)^2]$ which will be useful for us. 
		\begin{align}\label{eq:orabc}
			\frac{1}{N_0}  \E[(\grsr)^2] & =\E\Bigg[ a_{m, n} \mathcal A_{\nhg} + b_{m, n} \mathcal B_{\nhg}  +c_{m, n}  \mathcal C_{\nhg}  \Bigg] , 
		\end{align} 
		where $\mathcal A_{\nhg}:=\sum_{1 \leq i\neq j \leq N}\nhg_{ij}^2$,  $\mathcal B_{\nhg} := \sum_{1 \leq i\neq j\neq s \leq N}\nhg_{ij}\nhg_{si} $, and $\mathcal C_{\nhg} := \sum_{1 \leq i\neq j\neq s\neq t \leq N}\hG_{ij}\hG_{st}$.  
		
		%
		%
		
		%
		%
		A similar argument as above also yields:
		\begin{align}\label{eq:orabc2}
			\frac{1}{N_0^2} \E[ \grsc \grsr ] =\E\Bigg[ a_{m, n} \mathcal A_{\hG, \nhg}   + b_{m, n}  \mathcal B_{\hG, \nhg}   + c_{m, n} \mathcal C_{\hG, \nhg}     \Bigg] , 
		\end{align} 
		where $\mathcal A_{\hG, \nhg} := \sum_{1 \leq i\neq j \leq N}\hG_{ij}\nhg_{ij}$,  $\mathcal B_{\hG, \nhg} := \sum_{1 \leq i\neq j \neq t \leq N}\hG_{ij}\nhg_{it}$, 
		and $\mathcal C_{\hG, \nhg} := \sum_{1 \leq i\neq j \neq s \neq t \leq N}\hG_{ij}\nhg_{st}$. 
		By plugging in the expressions from~\eqref{eq:gammaabc},~\eqref{eq:orabc}~and~\eqref{eq:orabc2} in to~\eqref{eq:baseq1}, we get:
		\begin{align}\label{eq:hajek1}
			\E\left[\grsc-\grsr\right]^2 & = S_1 + S_2 + S_3 , 
		\end{align} 
		where 
		\begin{align*}
			S_1 & := N_0^{2} a_{m, n} \sum_{1 \leq i\neq j \leq N} \E [ ( \hG_{ij}-\nhg_{ij})^2 ] , \nonumber \\ 
			S_2 & := 4 N_0^2 b_{m, n} \sum_{1 \leq i\neq j\neq s \leq N} \E [ \nhg_{ij}\nhg_{it} + \hG_{ij}\hG_{it} - 2 \hG_{ij}\nhg_{it} ] , \nonumber \\ 
			S_3 & := N_0^2 c_{m, n} \sum_{1 \leq i\neq j\neq s\neq t \leq N} \E [ \hG_{ij}\hG_{st}+\nhg_{ij}\nhg_{st}-2\hG_{ij}\nhg_{st} ] .
		\end{align*}

		Next we show that each of the terms in~\eqref{eq:hajek1} above converges to $0$ . In this regard, assumptions~\eqref{eq:kerassn1}~and~\eqref{eq:kerassn2} will play a crucial role. Let us start with $S_1$. 
		Recalling the definition of $a_{m, n}$ from \eqref{eq:gammaabc} and by   \eqref{eq:genp1} and \eqref{eq:genp2} we get: 
		\begin{align*}
			N_0^2 a_{m, n} & = \frac{2\alpha_{2}^{(1)} N_0^2}{m^2(m-1)^2}+\frac{2\alpha_{2}^{(2)} N_0^2}{n^2(n-1)^2}+\frac{4\alpha_{1,1} N_0^2}{m^2n^2} \nonumber \\ 
			& =  \frac{N_0^2}{N(N-1)} \left( \frac{2 }{m(m-1)}+\frac{2}{n(n-1)}+\frac{4}{m n} \right) = O\left(\frac{1}{N^2}\right) . 
		\end{align*}
		Hence, 
		\begin{align}\label{eq:s1}
			S_1 & \lesssim \frac{1}{N(N-1)}  \sum_{1 \leq i\neq j \leq N} \E[(\hG_{ij}-\nhg_{ij})^2]. 
		\end{align}
		To bound this term, we will need a Cauchy-Schwartz type
		inequality. Towards this direction, note that by using standard properties of reproducing kernels, see e.g.,~\cite[Theorem 1]{gretton2008}, there exists a Hilbert  space $\mathcal{H}$ of functions with an inner product $\langle \cdot,\cdot\rangle_{\mathcal{H}}$ such that $K(x,y)=\langle \bKe(x,\cdot),\bKe(y,\cdot)\rangle_{\mathcal{H}}$. This implies that
		\begin{align*}
			\hG_{ij}^2&=\bKe(\bJ(\hbR_{m,n}(\bZ_i)),\bJ(\hbR_{m,n}(\bZ_j)))^2 \\ &=\langle \bKe(\bJ(\hbR_{m,n}(\bZ_i)),\cdot), \bKe(\bJ(\hbR_{m,n}(\bZ_j)),\cdot) \rangle_{\mathcal{H}}^2\\ &\leq \langle \bKe(\bJ(\hbR_{m,n}(\bZ_i)),\cdot),\bKe(\bJ(\hbR_{m,n}(\bZ_i)),\cdot)\rangle_{\mathcal{H}}\, \langle\bKe(\bJ(\hbR_{m,n}(\bZ_j)),\cdot),\bKe(\bJ(\hbR_{m,n}(\bZ_j)),\cdot)\rangle_{\mathcal{H}}\\ &=\hG_{ii}\hG_{jj}
		\end{align*}		
		By using the above, we get:
		$$\frac{1}{N(N-1)}\sum_{i\neq j} \hG_{ij}^2\lesssim \left(\frac{1}{N}\sum_{i=1}^N \hG_{ii}\right)^2\leq \frac{1}{N}\sum_{i=1}^N \hG_{ii}^2.$$ 
		Similarly, we can show that 
		$$\frac{1}{N(N-1)}\sum_{i\neq j} \nhg_{ij}^2\lesssim \left(\frac{1}{N}\sum_{i=1}^N \nhg_{ii}\right)^2\leq \frac{1}{N}\sum_{i=1}^N \nhg_{ii}^2.$$
		Using the above chain of inequalities with~\eqref{eq:seckerasn} and Vitali's Theorem, we get that the sequence of random variables
		$$\frac{1}{N(N-1)}\sum_{i\neq j} (\hG_{ij}-\nhg_{ij})^2$$
		is uniformly integrable. Therefore, by~\eqref{eq:s1}, to show $S_1$ converges to $0$, it suffices to show convergence of the display in probability to $0$. This follows from using~\cref{theo:rankmapcon} with $p=2,r=2,q=1$ with $\cF(\mx,\my)=\bKe(\mx,\my)$.
		
		Let us now show that $S_2$ converges to $0$. First note that 
		\begin{align*}
			b_{m,n} 
			&=\frac{4mn(n-1)+4mn(m-1)-4(m+n)(m-1)(n-1)}{mn(m-1)(n-1)N(N-1)(N-2)}+O\left(\frac{1}{N^5}\right)
		\end{align*}  
		which implies
		$$N_0^2b_{m,n}=O\left(\frac{1}{N^3}\right).$$
		Therefore,
		\begin{align*}
			|S_2|&\lesssim \frac{1}{N^3}\sum_{i\neq j\neq t}\E[|\hG_{ij}-\nhg_{ij}||\nhg_{it}|]+\frac{1}{N^3}\sum_{i\neq j\neq t}\E[|\hG_{it}-\nhg_{it}||\hG_{ij}|]\\ &\leq \sqrt{\frac{1}{N^2}\sum_{i\neq j}\E[(\hG_{ij}-\nhg_{ij})^2]}\left(\sqrt{\frac{1}{n^2}\sum_{i\neq t} \E\nhg_{it}^2}+\sqrt{\frac{1}{n^2}\sum_{i\neq t} \hG_{it}^2}\right)
		\end{align*}  
		It is then sufficient to show that 
		$$\frac{1}{N^2}\sum_{i\neq j} \E[(\hG_{ij}-\nhg_{ij})^2]\to 0,$$
		which we have already proved above.
		
		The final step is to prove that $S_3\to 0$. Towards this direction, note that 
		\begin{align*}
			c_{m,n} 
			&=\frac{-4(m^2-m+n^2-n)}{N(N-1)(N-2)(N-3)m(m-1)n(n-1)}+O\left(\frac{1}{N^6}\right).
		\end{align*}
		Therefore, $$N_0^2 c_{m,n}=O\left(\frac{1}{N^4}\right).$$
		Consequently,
		\begin{align*}
			|S_3|&\lesssim \frac{1}{N^4}\sum_{1\leq i\neq j\neq s\neq t\leq N}\E[|\hG_{ij}-\nhg_{ij}||\nhg_{st}|]+\frac{1}{N^4}\sum_{1\leq i\neq j\neq s\neq t\leq N}\E[|\hG_{st}-\nhg_{st}| |\hG_{ij}|]\\ &\leq \sqrt{\frac{1}{N^2}\sum_{i\neq j}\E[(\hG_{ij}-\nhg_{ij})^2]}\left(\sqrt{\frac{1}{n^2}\sum_{i\neq t} \E\nhg_{it}^2}+\sqrt{\frac{1}{n^2}\sum_{i\neq t} \hG_{it}^2}\right)
		\end{align*}
		which converges to $0$ as argued above.

		This completes the proof of~\cref{theo:hajek} with a further application of Markov's inequality. 	
	\end{proof} 
	
	\subsection{Proofs from~\cref{sec:compare}}\label{sec:pfindtest}
	\begin{proof}[Proof of~\cref{theo:nullrhnk}]
		
		Let $\bR_1(\cdot)$ and $\bR_2(\cdot)$ be the optimal transport maps from the distribution of $\bX_1$ and $\bY_1$ to $\nu_1$ and $\nu_2$, respectively. Define $\jro:=\frac{1}{n} \sum_{i=1}^n \bJ_1(\bR_1(\bX_i))$, $\jrt:= \frac{1}{n}\sum_{i=1}^n \bJ_2(\bR_2(\bY_i))$ and
		\begin{equation}\label{eq:pfindtest1}
			\ornk:=\left\lVert \left(\sn\otimes \st\right)^{-\frac{1}{2}}\mbox{\textrm{vec}}\left(\frac{1}{\sqrt{n}}\sum_{i=1}^n (\bJ_1(\bR_1(\bX_i))-\jro)(\bJ_2(\hbR_2(\bY_i))-\jrt)^{\top}\right)\right\rVert^2.
		\end{equation}
		Note that $\ornk$ is obtained by replacing the empirical rank maps $\hbR_1(\cdot)$ and $\hbR_2(\cdot)$ in \eqref{eq:rank_spearman}  with their population counterparts $\bR_1(\cdot)$ and $\bR_2(\cdot)$. The proof of~\cref{theo:nullrhnk} now proceeds in two steps:
		\begin{itemize}	
			
			\item In the first step we will show that, under $\mathrm{H}_0$, 
			\begin{align}\label{eq:pfrank}
				|\rhnk-\ornk|\overset{P}{\longrightarrow}0. 
			\end{align}
			
			\item Next, we will show that, under $\mathrm{H}_0$, 
			\begin{align}\label{eq:pfrank_distribution}
				\ornk\overset{w}{\longrightarrow}\chi^2_{d_1 d_2}.
			\end{align}
			
		\end{itemize}
		Combining~\eqref{eq:pfrank} and~\eqref{eq:pfrank_distribution} with Slutsky's theorem completes the proof of~\cref{theo:nullrhnk}.

		We begin with the proof of~\eqref{eq:pfrank_distribution}. For this, let $\bm{m}_1:=\E\bJ_1(\bR_1(\bX_1))$, $\bm{m}_2:=\E\bJ_2(\bR_2(\bY_1))$ and observe that,
		\begin{align} 
			\ornk&=\left\lVert \left(\sn\otimes \st\right)^{-\frac{1}{2}}\mbox{\textrm{vec}}\left(\frac{1}{\sqrt{n}}\sum_{i=1}^n (\bJ_1(\bR_1(\bX_i))-\bm{m}_1)(\bJ_2(\hbR_2(\bY_i))-\bm{m}_2)^{\top}\right)\right\rVert^2 \nonumber \\ 
			& \hspace{3.25in} +o_P(1). 
		\end{align} 
		Note that the first term in the RHS of the above display converges weakly to $\chi^2_{d_1 d_2}$ by combining the multivariate central limit theorem with the continuous mapping theorem. This completes the proof of~\eqref{eq:pfrank_distribution}.
		
		To prove \eqref{eq:pfrank}, note that it suffices to prove that:
		\begin{equation}\label{eq:step1req} 
			\E\left\lVert \mbox{\textrm{vec}}\left(\sum_{i=1}^n \hbG_{\bm X_i}\hbG_{\bm Y_i}^{\top}-\sum_{i=1}^n \bGa_{\bm X_i}\bGa_{\bm Y_i}^{\top}\right)\right\rVert^2=o(n),\end{equation}
		where, for $i \in [n]$, 
		$$\hbG_{\bm X_i}:=\bJ_1(\hbR_1(\bX_i))-\bar{\bJ_1}, \quad \hbG_{\bm Y_i}:=\bJ_2(\hbR_2(\bY_i))-\bar{\bJ_2},$$ 
		and
		$$\bGa_{\bm X_i}:=\bJ_1(\bR_1(\bX_i))-\jro, \quad \bGa_{\bm Y_i}:=\bJ_2(\bR_2(\bY_i))-\jrt.$$

		To prove~\eqref{eq:step1req}, let $S_n$ denote the set of all permutations of the set $\{1,2,\ldots ,n\}$. Also, suppose $\sigma$ is a random permutation sampled uniformly over $S_n$ and independently of $(\bX_1,\bY_1), \ldots , (\bX_n,\bY_n)$. It is easy to see that, under $\mathrm{H}_0$, we have:
		$$(\bX_1,\bY_1),\ldots ,(\bX_n,\bY_n)\overset{D}{=}(\bX_1,\bY_{\sigma(1)}),\ldots ,(\bX_n,\bY_{\sigma(n)}).$$
		Let $\mathcal{X}_n:=\{\bX_1,\ldots ,\bX_n\}$ and $\mathcal{Y}_n:=\{\bY_1,\ldots ,\bY_n\}$ be the \emph{unordered} sets of observations. Denote by $\Ezn$ the expectation conditional on $\mZ_n:=(\mathcal{X}_n,\mathcal{Y}_n)$. Based on this notation, the LHS of~\eqref{eq:step1req} can be written as:
		\begin{align}\label{eq:step1req1}
			\E\left\lVert \mbox{\textrm{vec}}\left(\sum_{i=1}^n \hbG_{\bm X_i}\hbG_{\bm Y_i}^{\top}-\sum_{i=1}^n \bGa_{\bm X_i}\bGa_{\bm Y_i}^{\top}\right)\right\rVert^2 = T_1 + T_2 - 2 T_3, 
		\end{align}
		where 
		$$T_1:=\E\mbox{Tr}\left(\Ezn\left[\sum_{1 \leq i,j \leq n}\hbG_{\bm X_i}\hbG_{\bm Y_i}^{\top}\hbG_{\bm Y_j}\hbG_{\bm X_j}^{\top}\right]\right),$$
		$$T_2:=\E\mbox{Tr}\left(\Ezn\left[\sum_{1 \leq i,j \leq n}\bGa_{\bm X_i}\bGa_{\bm Y_i}^{\top}\bGa_{\bm Y_j}\bGa_{\bm X_j}^{\top}\right] \right),$$
		$$T_3:=\E\mbox{Tr}\left(\Ezn\left[\sum_{1 \leq i, j \leq n}\hbG_{\bm X_i}\hbG_{\bm Y_i}^{\top}\bGa_{\bm Y_j}\bGa_{\bm X_j}^{\top}\right] \right) , 
		$$	
		with $\mbox{Tr}(\cdot)$ denoting the trace of a matrix. 
		
		We will focus on $T_3$. The analysis for the other two terms will follow along similar lines. Towards this direction, note that:
		\begin{align*}
			\Ezn & \Bigg[\sum_{1 \leq i,j \leq n}\hbG_{\bm X_i}\hbG_{\bm Y_i}^{\top} \bGa_{\bm Y_j}\bGa_{\bm X_j}^{\top}\Bigg] \nonumber \\ 
			& =\Ezn\left[\sum_{1 \leq i,j \leq n} \hbG_{\bm X_i}\hbG_{\bm Y_{\sigma(i)}}^{\top}\bGa_{\bm Y_{\sigma(j)} }\bGa_{\bm X_j}^{\top}\right]\\ 
			&=\frac{1}{n}\left( \sum_{i=1}^n \hbG_{\bm Y_i}^{\top}\bGa_{\bm Y_i}\right)\left(\sum_{i=1}^n \hbG_{\bm X_i}\bGa_{\bm X_i}^{\top}\right)+  \frac{1}{n(n-1)} \left(\sum_{1 \leq i\neq j \leq n} \hbG_{\bm Y_i}^{\top}\bGa_{\bm Y_j}\right)\left(\sum_{1 \leq i\neq j \leq n} \hbG_{\bm X_i}\bGa_{\bm X_j}^{\top}\right)\\ 
			&=\frac{1}{n-1}\left(\sum_{i=1}^n \hbG_{\bm Y_i}^{\top}\bGa_{\bm Y_i}\right)\left(\sum_{i=1}^n \hbG_{\bm X_i}\bGa_{\bm X_i}^{\top}\right),
		\end{align*}
		where the last line uses $\sum_{i=1}^n \hbG_{\bm X_i}=\bzr_{d_1}$ and $\sum_{i=1}^n \hbG_{\bm Y_i}=\bzr_{d_2}$. Using the above display, under $\mathrm{H_0}$ we get:
		\begin{align*}
			\frac{T_3}{n} & = \frac{1}{n}  \E\mbox{Tr} \left( \Ezn\Bigg[\sum_{1\leq i,j \leq n}\hbG_{\bm X_i}\hbG_{\bm Y_i}^{\top}  \bGa_{\bm Y_j}\bGa_{\bm X_j}^{\top}  \Bigg] \right) \\ 
			& =\E\left[\frac{1}{n}\sum_{i=1}^n \hbG_{\bm Y_i}^{\top}\bGa_{\bm Y_i}\right] \E\left[\frac{1}{n-1}\sum_{i=1}^n \mbox{Tr}\left(\hbG_{\bm X_i}\bGa_{\bm X_i}^{\top}\right)\right] \\ 
			&=\E\left[\frac{1}{n}\sum_{i=1}^n \hbG_{\bm Y_i}^{\top}\bGa_{\bm Y_i}\right] \E\left[\frac{1}{n-1}\sum_{i=1}^n \hbG_{\bm X_i}^{\top}\bGa_{\bm X_i}\right] \\ 
			& \longrightarrow \E\lVert \bJ_1(\bR_1(\bX_1))-\bm{m}_1\rVert^2\E\lVert \bJ_2(\bR_2(\bY_1))-\bm{m}_2\rVert^2, 
		\end{align*}
		as $n\to\infty$. Here, the last line follows by applying~\cref{theo:rankmapcon} in the same way as in the proof of~\cref{theo:nullrankhotelling} (both theorems are). We skip the details for brevity.
		
		In the same way, it can be shown that $T_1/n$ and $T_2/n$ also converge to the same limit as in the display above. This implies, $T_1+T_2-2T_3= o(n)$, that is, the LHS of~\eqref{eq:step1req1}  is $o(n)$, which completes the proof of~\eqref{eq:step1req1} and, hence,~\eqref{eq:step1req}. \medskip
	\end{proof}
	
	\begin{proof}[Proof of~\cref{prop:indepeff}] Consider the testing problem~\eqref{eq:locindep} under the Konijn alternatives as in~\cref{def:konijn} (both). Also, let the Lebesgue densities associated with $\mu_1$ and $\mu_2$ be $f_1(\cdot)$ and $f_2(\cdot)$, respectively. In this setting, by~\cite[Lemma 3.2.1]{gieser1993}, the sequence of joint distributions of $(\bX_1,\bY_1), \ldots ,(\bX_n,\bY_n)$ under $\mathrm{H}_1$ and $\mathrm{H}_0$ are contiguous to each other. In fact, by defining the joint density of $(\bX_1,\bY_1)$ under $\mathrm{H}_1$ as $f_{n,\delta}(\cdot,\cdot)$, it follows from the same lemma that,
		\begin{align}\label{eq:locindep1}
			L_{n, \delta} := \sum_{i=1}^n \log\frac{f_{n,\delta}(\bX_i,\bY_i)}{f_{n,0}(\bX_i,\bY_i)}=\frac{\delta}{\sqrt{n}}\sum_{i=1}^{n} \ell(\bX_i,\bY_i)-\frac{\delta^2}{2}\mbox{Var}[\ell(\bX_1,\bY_1)]+o_P(1),
		\end{align}
		where
		$$\ell(\mx,\my):=d_1+d_2-\frac{\mx^{\top}\nabla f_1(\mx)}{f_1(\mx)}-\frac{\my^{\top}\nabla f_2(\my)}{f_2(\my)}+\frac{\mx^{\top}\bm{M}\nabla  f_2(\my)}{f_2(\my)}+\frac{\my^{\top}\bm{M}\nabla  f_1(\mx)}{f_1(\mx)}.$$
		Let $\bR_1(\cdot)$ and $\bR_2(\cdot)$ denote the optimal transport maps (in the sense of~\cref{prop:Mccan}) from $\mu_1$ and $\mu_2$ (the marginal distributions of $\bX'_1$ and $\bY'_1$ to the reference distributions $\nu_1$ and $\nu_2$). Then, by a standard application of the multivariate central limit theorem, we have: under $\mathrm{H}_0$, 
		\begin{align}\label{eq:locindep2}
			\begin{pmatrix} \mbox{\textrm{vec}}\left(\frac{1}{\sqrt{n}}\sum_{i=1}^n (\bJ_1(\bR_1(\bX_i))-\jro)(\bJ_2(\hbR_2(\bY_i))-\jrt)^{\top}\right) \\ L_{n, \delta} \end{pmatrix}  &\overset{w}{\longrightarrow}\mathcal{N}\left(\bm \kappa, \Gamma \right),
		\end{align}
		where 
		$$\bm \kappa := \begin{pmatrix} \bzr_{d_1d_2} \\ \frac{\delta^2}{2}\mbox{Var}[\ell(\bX_1,\bY_1)] \end{pmatrix} \quad \text{and} \quad \Gamma := \begin{pmatrix} \Serd^{(1)}\otimes \Serd^{(2)} & \delta\bm{\gamma}_{d_1d_2} \\ \delta\bm{\gamma}_{d_1d_2}^{\top} & \delta^2\mbox{Var}[\ell(\bX_1,\bY_1)] \end{pmatrix}, $$ 
		\begin{align}\label{eq:g_d1d2}
			\bm{\gamma}_{d_1d_2}:=g(\bX'_1, \bY'_1) \cdot \mbox{\textrm{vec}}\left(\E\left(\bJ_1(\bR_1(\bX'_1))-\bm{m}_1\right)\left(\bJ_2(\bR_2(\bY'_1))-\bm{m}_2\right)^{\top}  \right), 
		\end{align}
		with $$g(\bX'_1, \bY'_1):= \frac{(\bX'_1)^{\top}\bm{M}\nabla  f_2(\bY'_1)}{f_2(\bY'_1)}+\frac{(\bY'_1)^{\top}\bm{M}\nabla  f_1(\bX'_1)}{f_1(\bX'_1)} ,$$
		$\bm{m}_1:=\E[\bJ_1(\bR_1(\bX'_1))]$, and $\bm{m}_2:=\E[\bJ_2(\bR_2(\bY'_1))]$. Now, recall the definition of $\ornk$ from~\eqref{eq:pfindtest1}. Then, by using~\eqref{eq:locindep2} with Le Cam's third lemma~\cite{LR05}, under $\mathrm{H}_1$ as in~\eqref{eq:locindep}, the following holds: 
		\begin{equation}\label{eq:locindep3}
			\ornk\overset{w}{\longrightarrow}\chi^2_{d_1d_2}\left(\delta^2 \bm{\gamma}_{d_1d_2}^{\top}\left(\Serd^{(1)}\otimes \Serd^{(2)}\right)^{-1}\bm{\gamma}_{d_1d_2}\right). 
		\end{equation}
		Note that RHS of~\eqref{eq:locindep3} denotes the $\chi^2$ distribution with $d_1d_2$ degrees of freedom and non-centrality parameter $$\delta^2 \bm{\gamma}_{d_1d_2}^{\top}\left(\Serd^{(1)}\otimes \Serd^{(2)}\right)^{-1}\bm{\gamma}_{d_1d_2}.$$ 
		This implies, using~\eqref{eq:pfrank},~\eqref{eq:locindep3}~and contiguity, that 
		\begin{align}\label{eq:rh_noncentral}
			\rhnk\overset{w}{\longrightarrow}\chi^2_{d_1d_2}\left(\delta^2 \bm{\gamma}_{d_1d_2}^{\top}\left(\Serd^{(1)}\otimes \Serd^{(2)}\right)^{-1}\bm{\gamma}_{d_1d_2}\right), 
		\end{align}
		under $\mathrm{H}_1$.
		
		Next, recall the Wilks' test defined in~\eqref{eq:wilkstest}  using the variable $R_n$. From~\cite{Taskinen2004}, we have,
		\begin{align}\label{eq:Rn_noncentral}
			R_n\overset{w}{\longrightarrow}\chi^2_{d_1d_2}\left(4\delta^2\lVert\mbox{\textrm{vec}}(\bm M)\rVert^2\right), 
		\end{align}
		under $\mathrm{H}_1$. Therefore, using~\eqref{eq:rh_noncentral} and~\eqref{eq:Rn_noncentral}, \begin{equation}\label{eq:locindep4}\ati=\frac{\bm{\gamma}_{d_1d_2}^{\top}\left(\Serd^{(1)}\otimes \Serd^{(2)}\right)^{-1}\bm{\gamma}_{d_1d_2}}{4\lVert\mbox{\textrm{vec}}(\bm M)\rVert^2}.
		\end{equation}
		
		Note that the expression of $\ati$ in~\eqref{eq:locindep4} holds in general without restricting to $\mu_1,\mu_2\in\fnd$ or $\fel$. We now provide lower bounds to $\ati$ obtained above. Recall that the theorem specifies both our ERDs to be standard Gaussian. Therefore, $\Serd^{(1)}=\bm I_{d_1}$, $\Serd^{(2)}=\bm I_{d_2}$ which implies $\Serd^{(1)}\otimes \Serd^{(2)} = \bm I_{d_1d_2}$, $\bm{m}_1=\bzr_{d_1}$, and $\bm{m}_2=\bzr_{d_2}$. By symmetry, it suffices to consider the following three cases: (1) $\mu_1,\mu_2\in\fnd$, (2) $\mu_1,\mu_2\in\fel$, and (3) $\mu_1\in\fnd$, $\mu_2\in\fel$. \medskip
		
		\noindent \textit{Case} (1): Denote by marginal distribution functions of $\bX'_1$ and $\bY'_1$ by $F(\cdot)$ and $G(\cdot)$, respectively. As $\mu_1$ has independent components, $F(\cdot)$ has the following form:
		$$F(x_1,\ldots ,x_{d_1})=\prod_{i=1}^{d_1} F_i(x_i),$$
		where $F_1(\cdot), F_1(\cdot), \ldots, F_d(\cdot)$ are univariate cumulative distribution functions. Let $f_i(\cdot)$ denote the probability density function associated with $F_i(\cdot)$, for $i \in [d]$. Similarly, let $G_1(\cdot), G_2(\cdot), \ldots, G_d(\cdot)$ and $g_1(\cdot), g_2(\cdot), \ldots, g_d(\cdot)$ be the distribution and density functions associated with the components of  $\bY'_1$, respectively. (Note that  $\bY'_1$ also has independent components by assumption.) Write $\bX'_1=(X'_{1,1},\ldots ,X'_{1,d_1})$ and $\bY'_1=(Y'_{1,1},\ldots ,Y'_{1,d_1})$. Using this notation and the same argument as in the proof of~\cref{prop:areind}, shows that 
		\begin{align}\label{eq:locindep5} 
			\bR_1(\bX'_1)&=(\Phi^{-1}(F_1(X'_{1,1})),\ldots ,\Phi^{-1}(F_{d_1}(X'_{1,d_1}))), \nonumber \\ 
			\bR_2(\bY'_1)&=(\Phi^{-1}(G_1(Y'_{1,1})),\ldots ,\Phi^{-1}(G_{d_2}(Y'_{1,d_2}))), 
		\end{align}
		are the required optimal transport maps, where $\Phi(\cdot)$ is the standard Gaussian cumulative distribution function. Next, recalling \eqref{eq:g_d1d2} and noting that in this case $\bm{m}_1=\bzr_{d_1}$, $\bm{m}_2=\bzr_{d_2}$ gives, 
		\begin{equation}\label{eq:locindep6}\bm{\gamma}_{d_1d_2}=\mbox{\textrm{vec}}\left(\E\left(\bR_1(\bX'_1)\bR_2(\bY'_1)\right)^{\top}\left(\frac{(\bX'_1)^{\top}\bm{M}\nabla  f_2(\bY'_1)}{f_2(\bY'_1)}+\frac{(\bY'_1)^{\top}\bm{M}\nabla  f_1(\bX'_1)}{f_1(\bX'_1)}\right)\right).\end{equation}
		Now, for $k \in [d_1]$, note that by the integration by parts formula, 
		\begin{equation}\label{eq:locindep7}\E\left[\frac{\nabla f_1(\bX'_1)}{f_1(\bX'_1)}\Phi^{-1}(F_k(X'_{1,k}))\right]=\bm{e}_k\int (\Phi^{-1})'(F_k(x))f_k^2(x)\,\mathrm dx. \end{equation}
		Define,
		\begin{equation}\label{eq:locindep8}
			C_{1,k}:=\E\left(\Phi^{-1}(F_{k}(X'_{1,k}))(X'_{1,k})\right), \quad D_{1,k}:=\bm{e}_k\int (\Phi^{-1})'(F_k(x))f_k^2(x)\,\mathrm dx, 
		\end{equation}
		and, similarly, $C_{2,\ell}$, $D_{2,\ell}$, for $\ell \in [d_2]$ with the $k$-th entry of $\bX'_1$ replaced by the $\ell$-th entry of $\bY_1'$.  Note that under~\cref{assumption4}, for $k \in [d_1]$ and $\ell \in [d_2]$, 
		\begin{align}\label{eq:FG}
			\E\left(\Phi^{-1}(F_{k}(X'_{1,k}))(\bX'_1)^{\top}\right)=C_{1,k} \bm{e}_k^{\top} \quad \text{and} \quad \E\left(\Phi^{-1}(G_{\ell}(Y'_{1,\ell}))(\bY'_1)^{\top}\right)=D_{1,\ell}\bm{e}_\ell^{\top}.
		\end{align}  
		Denoting  the matrix $\bm M:=((m_{k\ell}))_{k \in [d_1], \ell \in [d_2]}$ and combining~\eqref{eq:locindep5}, \eqref{eq:locindep6}, \eqref{eq:locindep7}, \eqref{eq:locindep8}, and \eqref{eq:FG} gives, 
		\begin{align}\label{eq:locindep9}
			\lVert \bm{\gamma}_{d_1d_2}\rVert^2&=\sum_{k=1}^{d_1} \sum_{\ell=1}^{d_2} m_{k\ell}^2\Big(C_{1,k}D_{2,\ell}+D_{1,k}C_{2,\ell}\Big)^2\nonumber\\ 
			& \geq 4 \sum_{k=1}^{d_1} \sum_{\ell=1}^{d_2} m_{k\ell}^2\left(C_{1,k}D_{1,k}\right)\left(C_{2,\ell}D_{2,\ell}\right), 
		\end{align}
		where the last step uses using the elementary inequality $(a+b)^2\geq 4ab$, for $a, b \in \R$.
		
		Observe that the lower bound in~\eqref{eq:locindep9} has effectively \emph{decoupled} as the product of two quantities, the first of which only depends on the distribution of the $k$-th component of $\bX'_1$ while the second one only depends on the distribution of the $\ell$-th component of $\bY'_1$. Now, following the proof of~\cite[Theorem 2.1]{Gastwirth1968} shows, $C_{1,k}D_{1,k}\geq 1$, for any distribution $F_k$ (the cumulative distribution function of $X'_{1,k}$), and similarly, $C_{2,\ell}D_{2,\ell}\geq 1$ for any distribution $G_{\ell}$ (the cumulative distribution function of $Y'_{1,\ell}$). Moreover,  equality holds if and only if both $X'_{1,k}$ and $Y'_{1,\ell}$ have standard normal distribution. This observation together with~\eqref{eq:locindep9} in~\eqref{eq:locindep4}, shows that
		$$\ati=\frac{\lVert \bm{\gamma}_{d_1d_2}\rVert^2}{4\lVert \mbox{\textrm{vec}}(\bm M)\rVert^2}\geq \frac{4 \sum_{k=1}^{d_1} \sum_{\ell=1}^{d_2} m_{k\ell}^2 }{4\lVert \mbox{\textrm{vec}}(\bm M)\rVert^2}=1.$$
		This completes the proof for case (1). \medskip
		
		\noindent \textit{Case} 2: Recall that $f_1$ and $f_2$ denote the probability density functions of $\bX'_1$ and $\bY_1'$, respectively. Since, $\mu_1,\mu_2\in\fel$ in this case and~\cref{assumption4}  is satisfied, we can assume without loss of generality that $f_1$ is proportional to
		$\underline{f}_1(\lVert \mx\rVert^2)$ and $f_2$ is proportional to $\underline{f}_2(\lVert \mx\rVert^2)$, for some radial density functions $\underline{f}_1$ and $\underline{f}_2$. Also, recall that $H_{d_1}(\cdot)$ and $H_{d_2}(\cdot)$ are the cumulative distribution functions of $\sqrt{\chi^2_{d_1}}$ and $\sqrt{\chi^2_{d_2}}$ distributions, respectively (as defined in the proof of~\cref{prop:Gaussare}).

		Now, using the same argument as in~\cref{lem:ellipot} we get,  
		$$\bR_1(\bX'_1)=\frac{\bX'_1}{\lVert \bX'_1\rVert} H_{d_1}^{-1}(\tP_1(\lVert \bX'_1\rVert)) \quad \text{and} \quad \bR_2(\bY'_1)=\frac{\bY'_1}{\lVert \bY'_1\rVert} H_{d_2}^{-1}(\tP_1(\lVert \bY'_1\rVert)),$$ 
		where $\tP_1(\cdot)$ and $\tP_2(\cdot)$ are the distribution functions of $\lVert \bX'_1\rVert$ and $\lVert \bY'_1\rVert$, respectively. For $k \in [d_1]$,  define 
		$$\bm C_{1,k}:=\E[(\bR_1(\bX'_1))_k(\bX'_1)]=\underbrace{\E\left[\lVert \bX'_1\rVert H_{d_1}^{-1}(\tP_1(\lVert \bX'_1\rVert))\right]}_{C_1}\frac{\bm{e}_k}{d_1}$$ 
		and 	
		$$\bm D_{1,k}:=\E\left[(\bR_1(\bX'_1))_k\frac{\nabla f_1(\bX'_1)}{f_1(\bX'_1)}\right]=\underbrace{\E\left[\frac{\underline{f}_1'}{\underline{f}_1}(\lVert \bX'_1\rVert) H_{d_1}^{-1}(\tP_1(\lVert \bX'_1\rVert))\right]}_{D_1}\frac{\bm{e}_k}{d_2}.$$
		Similarly, for $\ell \in [d_2]$, define $\bm C_{2,\ell}$, $\bm D_{2,\ell}$, $C_2$ and $D_2$ with the $k$-the element of $\bX'_1$ replaced by the $\ell$-th element of $\bY'_1$. Now, using the same steps as in~\eqref{eq:locindep9}, we get:
		\begin{align}\label{eq:locindep10}
			\lVert \bm{\gamma}_{d_1d_2}\rVert^2&= \sum_{k,\ell} \left(\bm C_{1,k}^{\top} M \bm D_{2,\ell}+\bm D_{1,k}M\bm C_{2,\ell}\right)^2\nonumber \\ &\geq \frac{4}{d_1^2d_2^2} (C_1D_1) (C_2D_2) \sum_{k=1}^{d_1} \sum_{\ell=1}^{d_2}  m_{k \ell}^2.
		\end{align}
		Note that the once again the lower bound in~\eqref{eq:locindep10} has \emph{decoupled} into two separate problems, one involving the distribution of $\bX'_1$ and the other involving the distribution of $\bY'_1$. Now, by~\cite[Theorem 1]{paindaveine2004}, $C_1D_1\geq d_1^2$ and $C_2D_2\geq d_2^2$,  where equality holds if and only if both $\bX'_1$ and $\bY'_1$ have standard normal distributions of appropriate dimensions. Using this observation and~\eqref{eq:locindep10} in~\eqref{eq:locindep4} gives, $$\ati=\frac{\lVert \bm{\gamma}_{d_1d_2}\rVert^2}{4\lVert \mbox{\textrm{vec}}(\bm M)\rVert^2}\geq \frac{4d_1^2d_2^2\sum_{k=1}^{d_1} \sum_{\ell=1}^{d_2}  m_{k \ell}^2}{4d_1^2d_2^2\lVert \mbox{\textrm{vec}}(\bm M)\rVert^2}=1,$$
		which completes the proof for case (2). \medskip
		
		\noindent \textit{Case} (3): When $\mu_1\in\fnd$ and $\mu_2\in\fel$, the proof proceeds exactly similar to the above two cases. In particular, we can get a similar lower bound to those obtained in~\eqref{eq:locindep9}~and~\eqref{eq:locindep10}, which will again decouple into two separate problems, one involving $\mu_1$ and the other involving $\mu_2$. We can then separately optimize over $\mu_1\in\fnd$ as we did in case (1) and $\mu_2\in\fel$ as we did in case (2), to complete the proof. The details are omitted.
	\end{proof}
	
	\begin{proof}[Proof of~\cref{prop:rdcovcon}]
		
		Note that, an application of the triangle inequality followed by the Cauchy-Schwarz inequality gives, 
		\begin{align}\label{eq:deltaor}
			\frac{1}{n^2}\bigg|\sum_{i,j} \left( \Don_{i,j}\Dwo_{i,j}- \Donor_{i,j}\Dwor_{i,j} \right) \bigg|\leq \sqrt{Q_1 T_2} + \sqrt{Q_2 T_1}
		\end{align}
		where $$Q_1:=\frac{1}{n^2}\sum_{i,j}(\Don_{i,j}-\Donor_{i,j})^2, \quad Q_2:=\frac{1}{n^2}\sum_{i,j}(\Dwo_{i,j}-\Dwor_{i,j})^2,$$ 
		and $T_1:=\frac{1}{n^2}\sum_{i,j} (\Donor_{i,j})^2$, $T_2 := \frac{1}{n^2}\sum_{i,j} (\Dwo_{i,j})^2$. Note that $T_1=O_P(1)$ and  $T_2=O_P(1)$ by~\cref{assumption3}. Moreover, by using~\cref{theo:rankmapcon} with $p=r=2$, $q=1$, $\cF(\mx_1,\mx_2)=\lVert \mx_1-\mx_2\rVert$, we get $Q_1\overset{P}{\longrightarrow}0$ and $Q_2\overset{P}{\longrightarrow}0$. This implies, the LHS of \eqref{eq:deltaor} converges to $0$ in probability. Consequently, the weak law of large numbers for V-statistics gives, 
		$$\frac{1}{n^2}\sum_{i,j} \Don_{i,j}\Dwo_{i,j}\overset{P}{\longrightarrow} \E\Big[ \Donor_{1,2}\Dwor_{1,2}\Big].$$
		Using similar computations we can find the weak limits of the other two terms in the definition of $\Rdcov_n^2$ in \eqref{eq:rdcovariance} to establish~\eqref{eq:rdcon} (both equations are).
		
		By using~\cite[Theorem 3(i)]{Gabor2007}, the right hand side of~\eqref{eq:rdcon} converges to $0$ in probability if and only if $\bJ_1(\bR_1(\bX_1))$ and $\bJ_2(\bR_2(\bY_1))$ are independent. Further, as $\bJ_1(\cdot)$, $\bJ_2(\cdot)$, $\bR_1(\cdot)$, and $\bR_2(\cdot)$ are injective (in the sense of~\cref{prop:Mccan}), $\bJ_1(\bR_1(\bX_1))$ and $\bJ_2(\bR_2(\bY_1))$ are independent if and only if $\bX_1$ and $\bY_1$ are independent. 
	\end{proof}
	
	\subsection{Proofs from~\cref{sec:auxdet}}\label{sec:auxpf}
	\begin{proof}[Proof of~\cref{prop:conlocontam}]
		Recall that $\bW$ has density $g(\cdot)$. Observe that $$\E\Rl(\bY)=(1-\delta)\E\Rl(\bX)+\delta \E\Rl(\bW)$$ and consequently, $\E\Rl(\bX)\neq \E\Rl(\bY)$ if and only if $\E\Rl(\bX)\neq \E\Rl(\bW)$. This observation combined with~\cref{theo:rhconsis} proves the first part of the proposition. The second part then follows from~\cref{prop:conloc}.
	\end{proof}

	\begin{proof}[Proof of~\cref{prop:arecontam}] 
		Recall the setup of~\eqref{eq:twosamcont} and its corresponding assumptions from~\cref{sec:contamodel}. In this case the likelihood ratio $V_N$ is defined by:
		$$V_N:=\sum_{j=1}^N \log\left[ \frac{\left(1-\frac{h}{\sqrt{N}}\right)f_1(\bY_j)+\frac{h}{\sqrt{N}}g(\bY_j)}{f_1(\bY_j)}\right].$$
		Once again, by using local asymptotic normality, $V_N$ can be written as:
		
		$$V_N=\dot{V}_N-\frac{1-\lambda}{2}h^2\underbrace{\int \left(\frac{g(\mx)}{f_1(\mx)}-1\right)^2 f_1(\mx)\,\mathrm d\mx}_{A}+o_{P}(1), $$
		where 
		$$\dot{V}_N=\frac{h}{\sqrt{N}}\sum_{j=1}^n \left(\frac{g(\bY_j)}{f_1(\bY_j)}-1\right).$$ Now, recalling~\eqref{eq:rmu_T12} and the multivariate central limit theorem, the following result holds:
		\begin{equation}\label{eq:3dmctt}
			\begin{pmatrix} \sqrt{\frac{mn}{N}} \bm T_1 \\ \sqrt{\frac{mn}{N}} \bm T_2 \\ V_N \end{pmatrix} \overset{w}{\to}\mathcal{N}\left(\begin{pmatrix} \bzr \\ \bzr\\ -\frac{c_2}{2} \end{pmatrix},\begin{pmatrix} (1-\lambda)\Serd & \bzr_{d \times d} & \bzr \\ \bzr_{d \times d} & \lambda\Serd & \bsg_2  \\ \bzr^{\top} & \bsg_2^{\top} &  c_2\end{pmatrix}\right)
		\end{equation}
		under $\mathrm{H}_0$, where $c_2:=(1-\lambda)h^2A$ and 
		$$\bsg_2:=h\sqrt{\lambda(1-\lambda)}\E_{\mathrm{H}_0}\left[\bJ(\Rmu(\bY))\left(\frac{g(\bY)}{f_1(\bY)}-1\right)\right].$$ The conclusion in part (2) then follows from Le Cam's third lemma~\cite[Corollary 12.3.2]{LR05} and the continuous mapping theorem.		
	\end{proof}
	
	\begin{proof}[Proof of~\cref{prop:Piteffloc}]
		The argument for model~\eqref{eq:twosamcont} is similar to that used in the proof of~\cref{theo:Piteffrank} . Recall the definitions of $\mathcal{Z}_N^L$ and $\mathcal{W}_N^L$ from the proof of~\cref{theo:Piteffrank}. Set $V_N$ as the likelihood ratio and note that under the regularity assumptions (2a), (2b) and (2c), we have: 
		\begin{equation}\label{eq:Pcontam1}
			V_N=\dot{V}_N-\frac{1-\lambda}{2}h^2\underbrace{\int \left(\frac{f_C(\mx)}{f_{\bX}(\mx)}-1\right)^2 f_{\bX}(\mx)\, \mathrm d\mx}_{A}+o_{P}(1), 	\end{equation} 
		where $$\dot{V}_N=\frac{h}{\sqrt{N}}\sum_{j=1}^n \left(\frac{f_C(\bY_j)}{f_{\bX}(\bY_j)}-1\right).$$ Set $g(\cdot):=f_C(\cdot)-f_{\bX}(\cdot)$. Note that $\int g(\mx)\, \mathrm d\mx=0$.
		Note that $\int g(\mx)\,d\mx=0$. We only need to show the following:
		\begin{align}\label{eq:Pcontam2}
			\begin{pmatrix} \mathcal{Z}_N^L\\ \mathcal{W}_N^L\\ V_N \end{pmatrix} \overset{w}{\to}\mathcal{N}\left(\begin{pmatrix} \mathbf{0}_{L\times 1} \\ \mathbf{0}_{L\times 1}\\ -((1-\lambda)/2)h^2A\end{pmatrix},\begin{pmatrix} \mathrm{Id}_{L\times L} & \mathbf{0}_{L\times L} & \mathbf{0}_{L\times 1}\\ \mathbf{0}_{L\times L} & \mathrm{Id}_{L\times L} & \bm p \\ \mathbf{0}_{1\times L} & \bm p^\top & (1-\lambda)h^2 A\end{pmatrix}\right)
		\end{align}
		where
		$$\bm p:=h \sqrt{1-\lambda} \begin{pmatrix} 
			\ezph\left[\Psi_1(\bH(\bX))\frac{g(\bX)}{f_{\bX}(\bX)}\right] \\ 
			\ezph\left[\Psi_2(\bH(\bX))\frac{g(\bX)}{f_{\bX}(\bX)}\right] \\ 
			\vdots \\ 
			\ezph\left[\Psi_L(\bH(\bX))\frac{g(\bX)}{f_{\bX}(\bX)}\right] 
		\end{pmatrix} . $$
		The proof could then be completed using Le Cam's third lemma and the continuous mapping theorem in the same way as in~\cref{theo:Piteffrank}. In order to establish~\eqref{eq:Pcontam2}, after using the multivariate central limit theorem and~\eqref{eq:Pcontam1}, we only need to show:
		\begin{equation}\label{eq:Pcontam3}
			\lim_{n \rightarrow \infty} \ezph [\mathcal{Z}_N^L\dot{V}_N] = \mathbf{0}_{L\times L}
		\end{equation} 
		and 
		\begin{equation}\label{eq:Pcontam4}
			\lim_{n \rightarrow \infty} \ezph [\mathcal{W}_N^L\dot{V}_N] = \bm p . 
		\end{equation}
		Note that~\eqref{eq:Pcontam3} is immediate as $(\bX_1,\ldots ,\bX_m)$ and $(\bY_1,\ldots ,\bY_n)$ are independent. Next, observe that,
		\begin{align*}
			\ezph\left[ \mathcal{W}_N^L\dot{V}_N \right] &=\frac{h}{\sqrt{nN}}
			\begin{pmatrix}
				\ezph\left[\left(\sum_{j=1}^n \Psi_1(\bH(\bY_j))\right)\left(\sum_{j=1}^n \frac{g(\bY_j)}{f_{\bX}(\bY_j)}\right)\right] \\ 
				\ezph\left[\left(\sum_{j=1}^n \Psi_2(\bH(\bY_j))\right)\left(\sum_{j=1}^n \frac{g(\bY_j)}{f_{\bX}(\bY_j)}\right)\right] \\ 
				\vdots \\ 
				\ezph\left[\left(\sum_{j=1}^n \Psi_L(\bH(\bY_j))\right)\left(\sum_{j=1}^n \frac{g(\bY_j)}{f_{\bX}(\bY_j)}\right)\right] 
			\end{pmatrix} \\ 
			&=\frac{ h \sqrt{n}}{\sqrt{N}} \begin{pmatrix} 
				\frac{1}{n}\sum_{j=1}^n \ezph\left[\Psi_1(\bH(\bY_j))\frac{g(\bY_j)}{f_{\bX}(\bY_j)}\right] \\ 
				\frac{1}{n}\sum_{j=1}^n \ezph\left[\Psi_2(\bH(\bY_j))\frac{g(\bY_j)}{f_{\bX}(\bY_j)}\right] \\ 
				\vdots \\ 
				\frac{1}{n}\sum_{j=1}^n \ezph\left[\Psi_L(\bH(\bY_j))\frac{g(\bY_j)}{f_{\bX}(\bY_j)}\right] 
			\end{pmatrix}  \rightarrow \bm p , 
		\end{align*} 
		as $n \rightarrow \infty$. This completes the proof of~\eqref{eq:Pcontam4}. 
	\end{proof}
	
	\begin{proof}[Proof of~\cref{prop:areica}]
		Let $\E[(\bm{W}-\E\bm{W})(\bm{W}-\E\bm{W})^{\top}]=\Sigma$. Note that $\Sigma$ is a diagonal matrix, since $\bm W$ has independent components. We write $\Sigma=\mbox{diag}(\sigma_1^2,\ldots ,\sigma_d^2)$. Recall the definition of $\tilde{\Sigma}$ from~\eqref{eq:tsig}. We can assume without loss of generality that $\bt_0=\bzr$. Note that
		\begin{equation}\label{eq:areica1}
			\tilde{\Sigma}=\bm{A}\E[(\bm{W}-\E\bm{W})(\bm{W}-\E\bm{W})^{\top}]\bm{A}^{\top}=\bm{A}\Sigma\bm{A}^{\top}.
		\end{equation}
		Define $R_i:=\sum_{j=1}^d h_j a_{j,i}$. Observe that
		$$f(\mx|\bt):=f_1(\mx-\bt)=\prod_{i=1}^d \tilde{f}_{i}\left(\sum_{j=1}^d a_{j,i}(x_j-\theta_j)\right),\quad \mbox{for}\ \mbox{all}\ \mx=(x_1,\ldots ,x_d)\in\R^d,$$
		By using the same computation as in~\eqref{eq:Gaussare1} coupled with~\eqref{eq:areica1}, we get that:
		\begin{equation}\label{eq:areica2}
			\atr=\frac{\Big\lVert\Serd^{-\frac{1}{2}}\E_{\mathrm{H}_0}\left[\bJ(\Rmu(\bX))\mathbf{1}^{\top}\boldsymbol{\eta}(\bm X, \bzr)\right]\Big\rVert^2}{\sum_{i=1}^d R_i^2/\sigma_i^2}.
		\end{equation}
		With $\bJ(\mx)=\mx$, $\nu$ as the standard Gaussian distribution, it is easy to see that $\Serd=\bm I_d$. Next observe that
		\begin{align}\label{eq:areica3}
			\bh^{\top}\boldsymbol{\eta}(\bm X, \bzr)=\frac{\bh^{\top}\nabla_{\bt} f(\bX|\bt)|_{\bt=\bzr}}{f(\bX|\bt=\bzr)}=\sum_{i=1}^d R_i\frac{\tilde{f}_i'(\sum_{j=1}^d a_{ji}x_j)}{\tilde{f}_i(\sum_{j=1}^d a_{ji}x_j)}.
		\end{align}
		We now compute $\Rmu(\bX)$. Let $\tilde{F}_i(\cdot)$ be the distribution function of $W_i$, for $1 \leq i \leq d$. As $(W_1,\ldots ,W_d)$ has independent components, $\bm{A}$ is orthogonal, we can use~\cite[Lemma A.8]{ghosal2019multivariate} to get:
		\begin{align}\label{eq:areica4}
			\Rmu(\bX)=\bm{A}\ \bm{v}(\bX), 
		\end{align} 
		where  $\bm{v}(\bX):=(v_1(\bX),\ldots ,v_d(\bX))$, with $v_j(\bX):=\Phi^{-1}\circ \tilde{F}_j\left(\sum_{k=1}^d a_{kj}X_j\right)$ and $\Phi(\cdot)$ the standard normal cumulative distribution function.
		
		Next, by using~\eqref{eq:areica3}~and~\eqref{eq:areica4}, we have:
		\begin{align*}\label{eq:areica5}
			&\;\;\;\;\;\;\Big\lVert\Serd^{-\frac{1}{2}}\E_{\mathrm{H}_0}\left[\bJ(\Rmu(\bX))\bh^{\top}\boldsymbol{\eta}(\bm X, \bzr)\right]\Big\rVert^2\nonumber \\&=\sum_{i=1}^d \left(\E\left[\left\{\sum_{j=1}^d R_j\frac{\tilde{f}_j'(W_j)}{\tilde{f}_j(W_j)}\right\}\cdot \left\{\sum_{j=1}^d a_{ij}\Phi^{-1}\circ \tilde{F}_j(W_j)\right\}\right]\right)^2\nonumber \\ &=\sum_{i=1}^d \left(\sum_{j=1}^d a_{ij}R_j\Phi^{-1}\circ \tilde{F}_j(W_j)\frac{\tilde{f}_j'(W_j)}{\tilde{f}_j(W_j)}\right)^2\nonumber \\ &=\sum_{j=1}^d R_j^2 \left(\E\left[\Phi^{-1}\circ \tilde{F}_j(W_j)\frac{\tilde{f}_j'(W_j)}{\tilde{f}_j(W_j)}\right]\right)^2.
		\end{align*}
		By using the same argument as in~\eqref{eq:areind5} and right after~\eqref{eq:indgauss_optimization}, we further have:
		$$\sum_{j=1}^d R_j^2 \left(\E\left[\Phi^{-1}\circ \tilde{F}_j(W_j)\frac{\tilde{f}_j'(W_j)}{\tilde{f}_j(W_j)}\right]\right)^2\geq \sum_{j=1}^d \frac{R_j^2}{\sigma_j^2}.$$
		Plugging the above observation into~\eqref{eq:areica2} completes the proof.
	\end{proof}
	
	\subsection{Lower Bounds}\label{sec:lowerbound}
	This section is devoted to proving lower bounds for the testing problems \eqref{eq:twosamsmooth} and \eqref{eq:twosamcont} described in~\cref{sec:locpow} (all). In other words, we show that for both these problems, the power of any level $\alpha$ test function is upper bounded by $\alpha+\varepsilon$, for any given $\varepsilon>0$ and for all large enough $m, n$. This is formalized in the following proposition for the hypothesis~\eqref{eq:twosamsmooth}. The proof for~\eqref{eq:twosamcont} is similar.

	\begin{prop}[Lower bound in testing]\label{prop:lowerbound}
		Fix any $\varepsilon>0$ and let $\mathcal{T}_{\alpha}^{m,n}$ be the set of all level $\alpha$ test functions based on $\mathcal{Z}_N$. 
		Then, provided $\alpha+\varepsilon<1$, there exists $h_{\varepsilon}>0$ such that, for all $m, n$ large enough, 
		$$\inf_{T_{\alpha}\in\mathcal{T}_{\alpha}^{m,n}} \sup_{|h|\geq h_{\varepsilon}}\P_{\mathrm{H}_1}(T_{\alpha}=0)\geq 1-\alpha-\varepsilon,$$
		where $\mathrm{H}_1$ is specified as in~\eqref{eq:twosamsmooth}, with $\bh= h_{\varepsilon} \bo$. 
	\end{prop} 
	
	\begin{remark}[Rate-optimality of $\rhsc$]\label{rem:rateop}
		Recall that in~\cref{theo:locpowermain1}  we show, for the testing problems in~\eqref{eq:twosamsmooth}~and~\eqref{eq:twosamcont}, $\rhsc$ has a non-trivial power $\in (\alpha,1)$ under $\mathrm{H}_1$. This combined with~\cref{prop:lowerbound} above, shows the rate-optimality of the test based on $\rhsc$.
	\end{remark}
	
	\begin{proof}[Proof of~\cref{prop:lowerbound}]
		The proof is a standard application of the connection between minimax lower bounds and total variation distance between probability measures (see, for example,~\cite[Chapter 6]{Groeneboom2014}). Towards this, set $\P^{(N)}:=\mu_1^{(N)}$ and $\mathbb{Q}^{(N)}:=\mu_1^{(m)}\otimes \mu_2^{(n)}$, where $\mu_2$ is as specified under $\mathrm{H}_1$ in~\eqref{eq:twosamloc} . Let $\tv$ denote the total variation distance between $\P^{(N)}$ and $\mathbb{Q}^{(N)}$, and $\hd$ denote the Hellinger distance between $\P^{(N)}$ and $\mathbb{Q}^{(N)}$. It suffices to show that given any $\varepsilon>0$, $\tv\leq\varepsilon$, for all large enough $m,n$. In fact, since~\cite[Equation 2.20]{Tsybakov2009}, $\tv\leq \hd$, it suffices to show $\hd \leq \varepsilon$, for all large enough $m,n$. To this end, following~\cite[Page 83]{Tsybakov2009}, we have
		\begin{equation}\label{eq:lb1}
			1-\frac{\hds}{2}= \E_{\bt_0}\left[\prod_{i=1}^n\frac{\sqrt{f(\bX_i|\bt_0+h\bo/\sqrt{N})}}{\sqrt{f(\bX_i|\bt_0)}}\right].
		\end{equation}
		Now, using~\eqref{eq:likscore} gives, 
		$$\prod_{i=1}^n\frac{\sqrt{f(\bX_i|\bt_0+h\bo/\sqrt{N})}}{\sqrt{f(\bX_i|\bt_0)}}\overset{w}{\longrightarrow} \exp\left(-\frac{h^2(1-\lambda)}{4}\cdot \bo^{\top} \bm{I}(\bt_0)\bo +\frac{h\sqrt{(1-\lambda) \bo^{\top}\bm{I}(\bt_0)\bo}}{2}G\right),$$
		where $G\sim\mathcal{N}(0,1)$. Observe that
		$$\E_{\bt_0}\left[\left(\prod_{i=1}^n\frac{\sqrt{f(\bX_i|\bt_0+h\bo/\sqrt{N})}}{\sqrt{f(\bX_i|\bt_0)}}\right)^2\right]=1.$$
		Therefore, using uniform integrability and~\eqref{eq:lb1} gives, 
		$$1-\frac{\hds}{2}\to \exp\left(-\frac{h^2(1-\lambda)\bo^{\top}\bm{I}(\bt_0)\bo}{8}\right).$$
		Hence, one can choose $h_{\varepsilon}>0$ such that when $h \geq h_{\varepsilon}$, then $\hd \leq \varepsilon$, for all large enough $m,n$. 
	\end{proof}
	
	\section{Simulations}\label{sec:sim}
	In this section, we will illustrate our theoretical findings through numerical experiments. 
	The section is organized as follows: In~\cref{sec:numill} we use numerical experiments to demonstrate the multivariate Hodges-Lehmann phenomenon and the multivariate Chernoff-Savage phenomenon which we discussed after~\cref{prop:areind}. In~\cref{sec:finper} we compare the finite sample power of $\prsc$ (see~\eqref{eq:testtrank}) to Hotelling $T^2$. In~\cref{sec:beyondloc}, we show that $\prsc$ has high power even beyond location-shift alternatives unlike the usual Hotelling $T^2$. We next move on to the rank MMD test $\ptsc$ (see~\eqref{eq:testrankker}). In~\cref{sec:emmdrank}, we compare the finite sample power of $\ptsc$ with the usual energy and MMD test in moderate dimensions. Similar comparisons in the high-dimensional setting are presented in~\cref{sec:highdpower}.
	
	\subsection{Numerical illustration of Hodges-Lehmann and Chernoff-Savage type results}\label{sec:numill}
	\begin{figure}[h]
		\begin{center}
			\includegraphics[height=7.5cm,width=7cm]{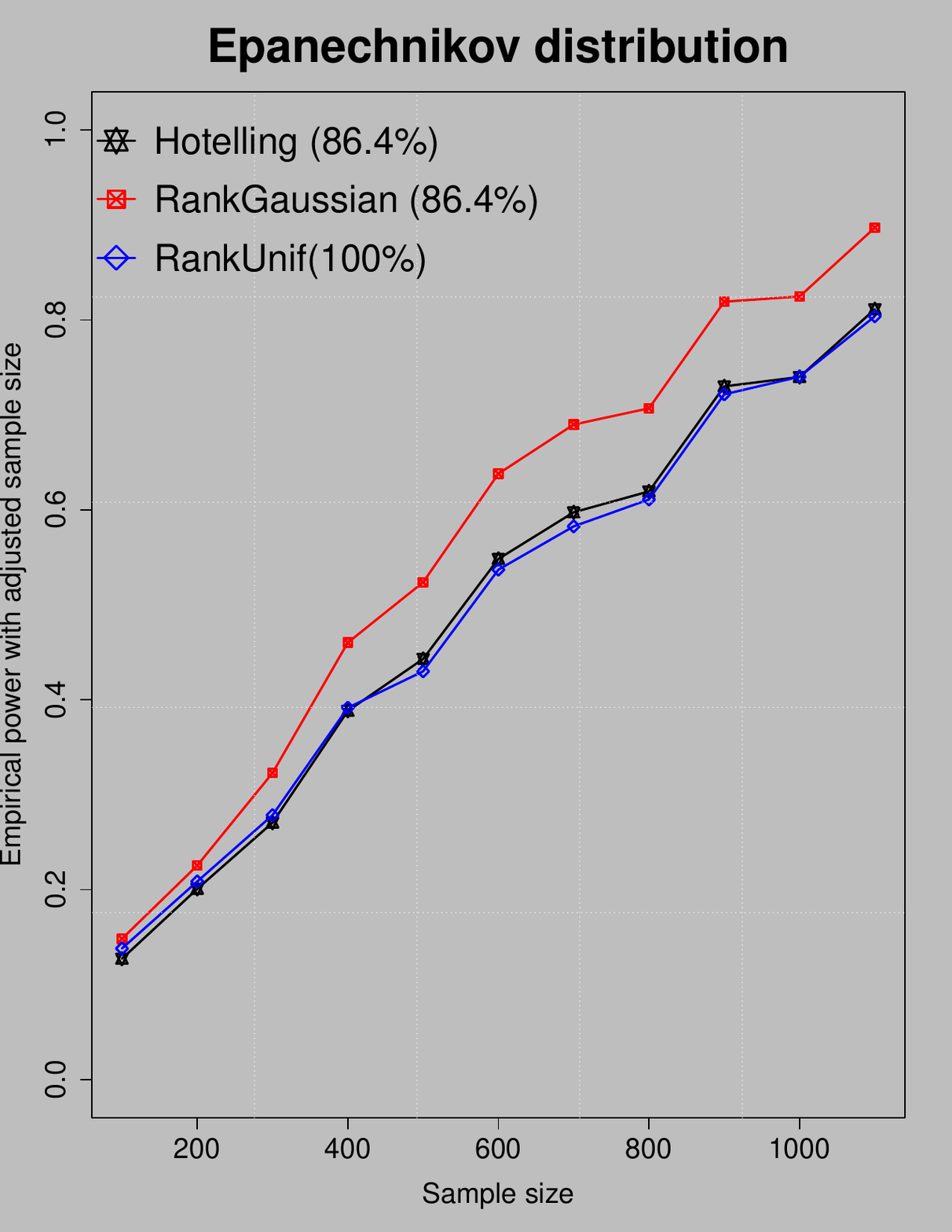}	
			\includegraphics[height=7.5cm,width=7cm]{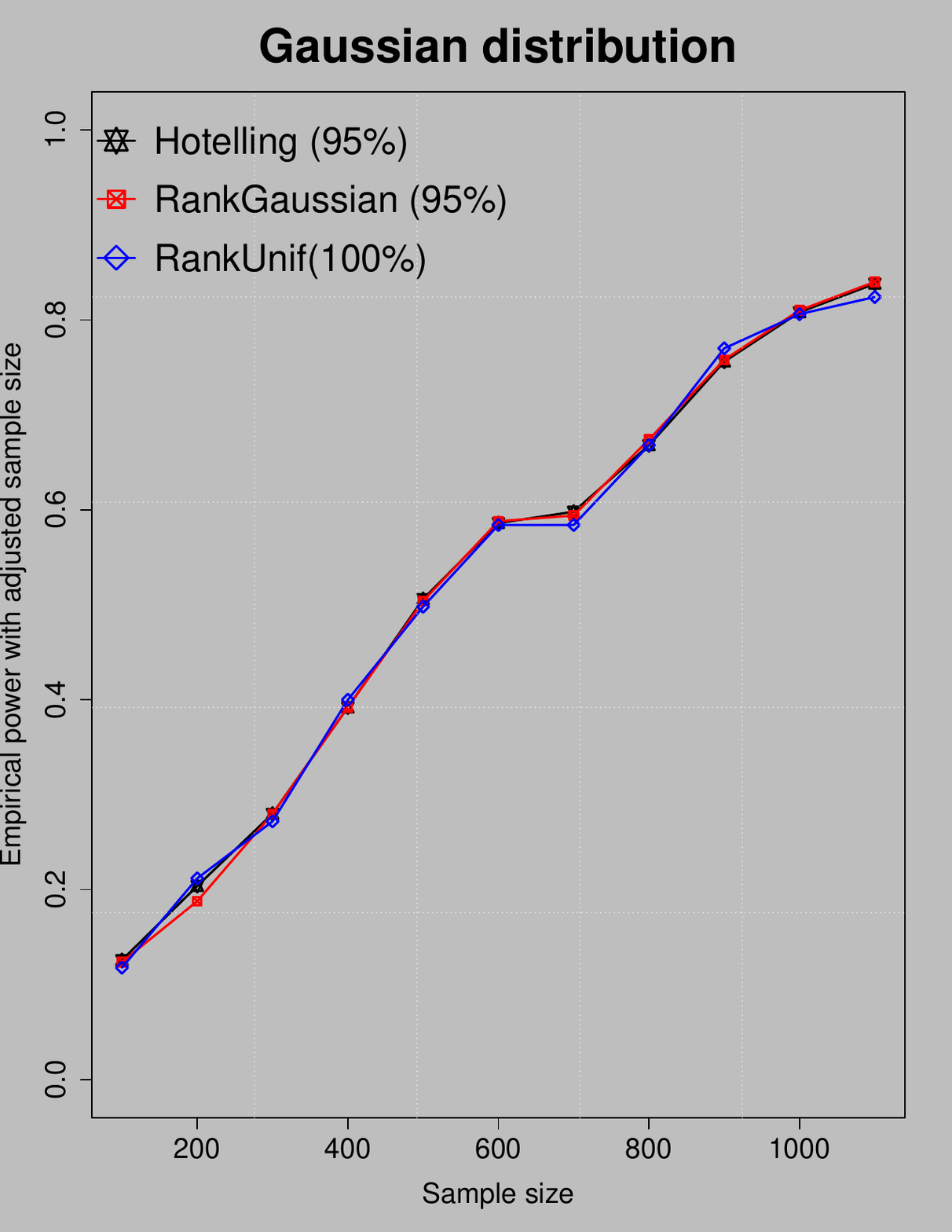}
		\end{center}	
		\caption{In the left panel, we sample $\bX_1$ and $\bY_1$ according to setting (H1). We plot the power curve for $\rhsc$ with $\nu=\nu_U$ with sample sizes $m=n$ varying in $[100,1100]$ (in blue), the power curve for Hotelling $T^2$ with the number of samples $m,n$ replaced by $\approx 0.864m$, $0.864n$ (in black), and the power curve for $\rhsc$ with $\nu=\nu_G$ with the same number of samples as used for Hotelling $T^2$ (in red). In the right panel, we do the same but with $\bX_1$, $\bY_1$ sampled according to setting (H2).}
		\label{fig:settingA5}
	\end{figure}
	In this section, we use numerical experiments to demonstrate the Hodges-Lehmann and  Chernoff-Savage type behavior theoretically observed in~\cref{prop:areind}. Towards this consider the following simulation settings:
	\begin{enumerate}
		\item[(H1)] $\bX_1$ follows a bivariate Epanechnikov distribution (see~\cref{prop:areind}) with independent components and location parameter $\bzr$, and $\bY_1$ has the same distribution with location parameter $0.1 \cdot \mathbf{1}_2$. 
		\item[(H2)] $\bX_1$ follows a bivariate standard normal distribution with location parameter $\bzr$, and $\bY_1$ has the same distribution with location parameter $0.1 \cdot \mathbf{1}_2$.
	\end{enumerate}
	For each of these settings, we compare the power curves of the Hotelling $T^2$ test versus tests based on $\rhsc$, where we use $\nu\equiv\nu_U\equiv\mbox{Unif}[0,1]^d$ and $\bJ(\mx)=\mx$, or $\nu\equiv\nu_G\equiv\mathcal{N}_d(\bzr,\bm I_d)$ and $\bJ(\mx)=\mx$. These will be referred to as the {\tt RankUniform} and {\tt RankGaussian} versions of the test based on $\rhsc$, respectively. In the sequel, we will not repeat the choice of the score function as it is set to the identity map in both cases. Each of the tests are carried out at level $0.05$ and the power curves are plotted as the sample size $m=n$ varies in the interval $[100,1100]$. To obtain the power curves, we used $500$ independent replications.
	
	To understand how~\cref{fig:settingA5} should be interpreted in the aforementioned settings, it is instructive to recall the informal/intuitive understanding of ARE which we presented in the Introduction, namely:
	\begin{changemargin}{0.5cm}{0.5cm} 
		\textit{The ARE of $T_1$ relative to $T_2$ is the ratio of the number of samples needed to attain the same power when using the test $T_2$ compared to the same for test $T_1$.}
	\end{changemargin} 
	\noindent In other words, if the ARE of $T_1$ with respect to $T_2$ is $0.9$, intuitively it means that the power of $T_1$ with $n$ samples and that of $T_2$  with $\approx 0.9n$ should be similar, at least for large $n$. 
	
	We now look at~\cref{fig:settingA5} in light of the above discussion. In the left panel of~\cref{fig:settingA5}, we focus on setting (H1). Note that, by~\cref{prop:areind}, the ARE of $\rhsc$ with $\nu=\nu_U$ against Hotelling $T^2$ is $0.864$ under setting (H1). Therefore we plot the power curve for $\rhsc$ with $\nu=\nu_U$ with $m=n\in [100,1100]$ samples and the power curve for Hotelling $T^2$ with the number of samples $m,n$ replaced by $\approx 0.864m$, $0.864n$. As per the aforementioned intuitive understanding of ARE, these two power curves should be fairly close, specially for large $n$. This is exactly what we observe in the left panel of~\cref{fig:settingA5}. The black and blue lines (for Hotelling $T^2$ and  $T_{m,n}^{\nu_U,\bJ}$ respectively) are very close for the entire spectrum of sample sizes considered in the left panel of~\cref{fig:settingA5}. In the same plot, we also show the power curve of $\rhsc$ with $\nu=\nu_G$ where the sample sizes $m,n$ in this case were chosen in the same way as that for Hotelling $T^2$ test. By~\cref{prop:areind}, the ARE against Hotelling $T^2$ in this case is \emph{larger} than $1$. Therefore one would expect $\rhsc$ with $\nu=\nu_G$ to have better power in this setting which is exactly what we observe in the left panel of~\cref{fig:settingA5}. The power curve of $\rhsc$ with $\nu=\nu_G$ is significantly higher than that of Hotelling $T^2$ or $\rhsc$ with $\nu=\nu_U$. 
	
	A similar observation is made in the right panel of~\cref{fig:settingA5} under setting (H2) - the standard Gaussian location model. In this case, by~\cref{prop:Gaussare} , we have $\mbox{ARE}(T^{\nu_G,\bJ},T)=1$ and $\mbox{ARE}(T^{\nu_U,\bJ},T)=0.95$. We therefore allot $m,n\in [100,1100]$ samples for computing the power curve for $\rhsc$ with $\nu=\nu_U$ and sample sizes of $\approx 0.95m,0.95n$ while computing power curves for both Hotelling $T^2$ and $\rhsc$ with $\nu=\nu_G$. As the theory predicts, all the power curves are virtually identical in the right panel of~\cref{fig:settingA5}. The resemblance can be seen through the entire spectrum of sample sizes considered.

	\subsection{Power comparisons}\label{sec:finper}
	In this subsection, we consider the following $4$ settings. In each of these settings, we choose the dimension $d=\{2,4\}$ and the sample sizes as $m=n=300$.
	\begin{itemize}
		\item[(A1)] $\bX_1\sim\mathcal{N}(\bzr_d,\bm I_d)$ and $Y_1\sim\mathcal{N}(\theta \cdot \mathbf{1}_d,\bm I_d)$ where $\theta\in\R$ varies in the interval $[0.01,0.20]$.
		\item[(A2)] $\bX_1$ has a \emph{logistic} distribution with location parameter $\bzr_d$ and scale parameter $\bm I_d$, and $\bY_1$ has a logistic distribution with the same scale parameter but with location parameter $\theta \cdot \mathbf{1}_d$. Once again, we choose $\theta\in\R$ with $\theta\in [0.01,0.20]$.
		\item[(A3)] $\bX_1$ has a \emph{Laplace} distribution with location parameter $\bzr_d$ and scale parameter $0.5\times \bm I_d+0.5\times \mathbf{1}_d\mathbf{1}_d^{\top}$, and $\bY_1$ has a Laplace distribution with the same scale parameter but with location parameter $\theta \cdot \mathbf{1}$, $\theta\in [0.01,0.5]$.
		\item[(A4)] $\bX_1$, $\bY_1$ both belong to a log-normal family of distributions. In particular, $\log{\bX_1}\sim\mathcal{N}(\bzr_d,\bm I_d)$ and $\log{\bY_1}\sim\mathcal{N}(\theta \cdot \mathbf{1}_d,\bm I_d)$ where $\theta\in\R$ varies in the interval $[-0.25,-0.01]$.
	\end{itemize}
	As before, for each of the settings (A1)-(A4), we compare the power curves of the Hotelling $T^2$ test versus tests based on $\rhsc$, where we use $\nu\equiv\nu_U\equiv\mbox{Unif}[0,1]^d$ and $\bJ(\mx)=\mx$, or $\nu\equiv\nu_G\equiv\mathcal{N}_d(\bzr,\bm I_d)$ and $\bJ(\mx)=\mx$. Each of the tests are carried out at level $0.05$ and the power curves are plotted as the parameter $\theta$ varies over the aforementioned ranges in settings (A1)-(A4). To obtain the power curves, we used $500$ independent replications. The plots can be found in Figures~\ref{fig:settingA1}-\ref{fig:settingA4} respectively. We now discuss our principle findings from these simulations.
	
	Let us begin with the Gaussian setting (A1) (see~\cref{fig:settingA1}). In this case, for $d=2$ (left panel of~\cref{fig:settingA1}), the performance of Hotelling $T^2$ and $\rhsc$ with $\nu=\nu_G$, $\bJ(\mx)=\mx$ are almost identical, whereas the performance of $\rhsc$ with $\nu=\nu_U$, $\bJ(\mx)=\mx$ is slightly worse. This is in alignment with~\cref{prop:Gaussare} where we show that the ARE of $\rhsc$ based on $\nu=\nu_G$ with respect to Hotelling $T^2$ is $1$; and the same for $\rhsc$ based on $\nu=\nu_U$ is $0.95$. Therefore, we would expect the performances of $\rhsc$ with $\nu=\nu_G$ and Hotelling $T^2$ to be close, and that of $\nu=\nu_U$ to be slightly worse, asymptotically. It is interesting to see a similar behavior manifest itself for a moderate sample size of $m=n=300$. For $d=4$ (right panel of~\cref{fig:settingA1}), the performance of Hotelling $T^2$ is slightly better than the performance of $\rhsc$ with $\nu=\nu_G$, whose performance is in turn slightly better than $\rhsc$ with $\nu=\nu_U$. As expected, the agreement with the aforementioned AREs is slightly weaker in the $d=4$ case than in the $d=2$ case. However, 
	the power curves are still reasonably close so as to justify the theoretical AREs of $1$ and $0.95$ as mentioned above.
	
	Note that, when working with $\rhsc$, $\nu=\nu_G$, the standard Gaussian example in the preceding paragraph is in a way, the ``worst case" distribution as per~\cref{prop:Gaussare}, and Theorems~\ref{prop:areind},~\ref{prop:areell},~and~\ref{prop:areica} (all). The aforementioned results show that $\rhsc$ with $\nu=\nu_G$ has ARE $1$ against Hotelling $T^2$ and an ARE \emph{larger} than $1$ for any distribution belonging to the families of distributions covered in Theorems~\ref{prop:areind}-\ref{prop:areica}. We illustrate this using settings (A2) and (A3), featuring a standard Laplace distribution (independent components, see~\cref{prop:areind}) and a correlated logistic distribution (elliptically symmetric, see~\cref{prop:areell}). In both cases, the power curves in Figures~\ref{fig:settingA2}~and~\ref{fig:settingA3} show that $\rhsc$ with $\nu=\nu_G$ has higher power than Hotelling $T^2$. This difference in the power curves is quite pronounced in~\cref{fig:settingA2} (standard Laplace with $d=2,4$)  and the left hand panel of~\cref{fig:settingA3} (correlated logistic with $d=2$), while this difference is only marginal in the right hand panel of~\cref{fig:settingA3} (correlated logistic with $d=4$). Further $\rhsc$ with $\nu=\nu_U$ has the highest power in~\cref{fig:settingA2} (standard Laplace with $d=2,4$). This too, is justified by our  theoretical results, in particular~\cref{theo:locpowermain1}, using which it is easy to check that the following approximations hold (up to errors due to numerical integration):
	$$\mbox{ARE}(T^{\nu_U,\bJ},T)=\frac{3}{2}>\frac{32}{25}\approx \mbox{ARE}(T^{\nu_G,\bJ},T).$$
	For the correlated logistic setting (A3), see~\cref{fig:settingA3}, the power curve of $\rhsc$ with $\nu=\nu_U$ is very similar to $\rhsc$ with $\nu=\nu_G$ and none of these two power curves seem to be uniformly better than the other. Note that, particularly for the $d=2$ case, $\rhsc$ with $\nu=\nu_U$ performs significantly better than Hotelling $T^2$. For $d=4$, it has similar performance compared to Hotelling $T^2$; once again none of the power curves uniformly dominate the other.
	
	In the next simulation setting, we explore the performance of the $3$ candidate tests in a heavy-tailed log-normal setting (A4). The log-normal distribution is heavy tailed in the sense that it has all moments finite but has an infinite exponential moment. Traditionally rank-based procedures perform better in such cases than their non rank-based counterparts. Both our proposed tests $\rhsc$ with $\nu=\nu_G$ and $\nu=\nu_U$ manifest this behavior by significantly outperforming the Hotelling $T^2$ test both for $d=2$ and $4$ (see~\cref{fig:settingA4}). The difference in the power curves between the rank-based procedures and the Hotelling $T^2$ test is the largest in this setting compared to the other settings ((A1)-(A3)) considered in this section. A similar observation was also made in~\cite{Deb19} where the authors show that other (optimal transport based) multivariate rank tests for the two-sample problem outperform their non rank-based counterparts in different heavy-tailed settings. We believe that theoretically investigating the robustness properties of these multivariate rank tests would be an interesting future research direction.
	\begin{figure}[h]
		\begin{center}
			\includegraphics[height=7.5cm,width=7cm]{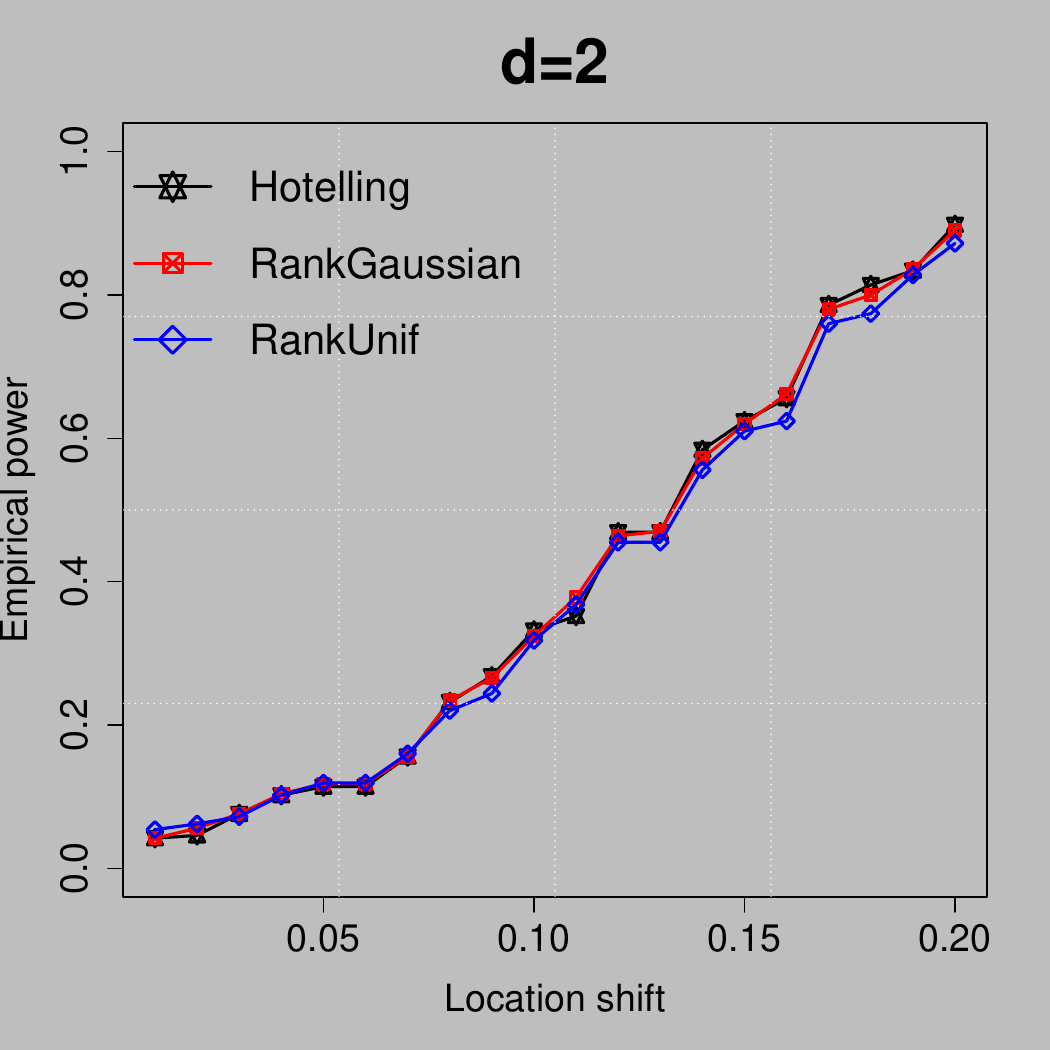}  
			\includegraphics[height=7.5cm,width=7cm]{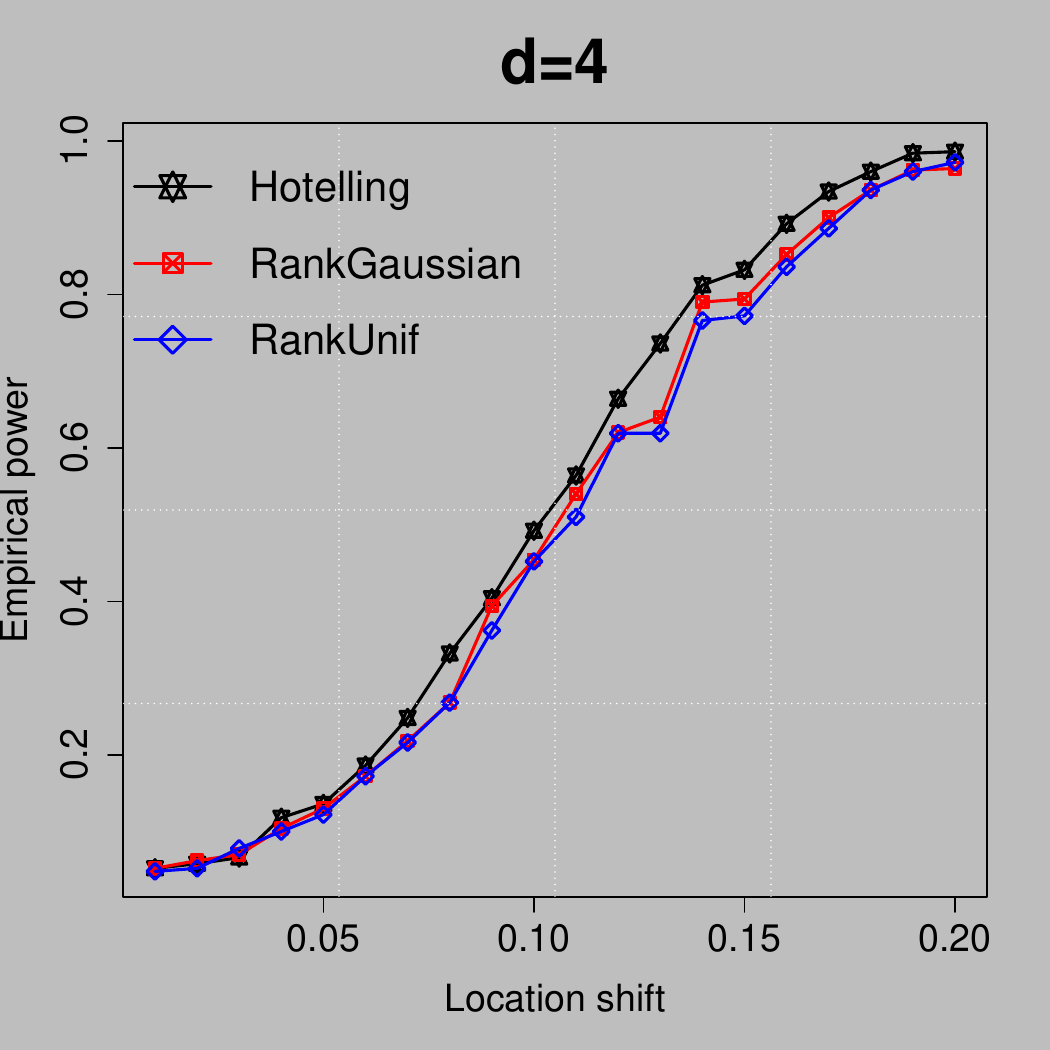} 
		\end{center}	
		\caption{In the left panel, we sample $\bX_1$, $\bY_1$ according to setting (A1) with $d=2$. The power curves are plotted as $\theta\in [0.01,0.2]$. The red line represents the power curve for $\rhsc$ with $\nu=\nu_G$, the blue line represents the same for $\rhsc$ with $\nu=\nu_U$, and the black line represents the same for Hotelling $T^2$. In the right panel, we plot the same with $d=4$. }
		\label{fig:settingA1}
	\end{figure}
	\begin{figure}[h]
		\begin{center}
			\includegraphics[height=7.5cm,width=7cm]{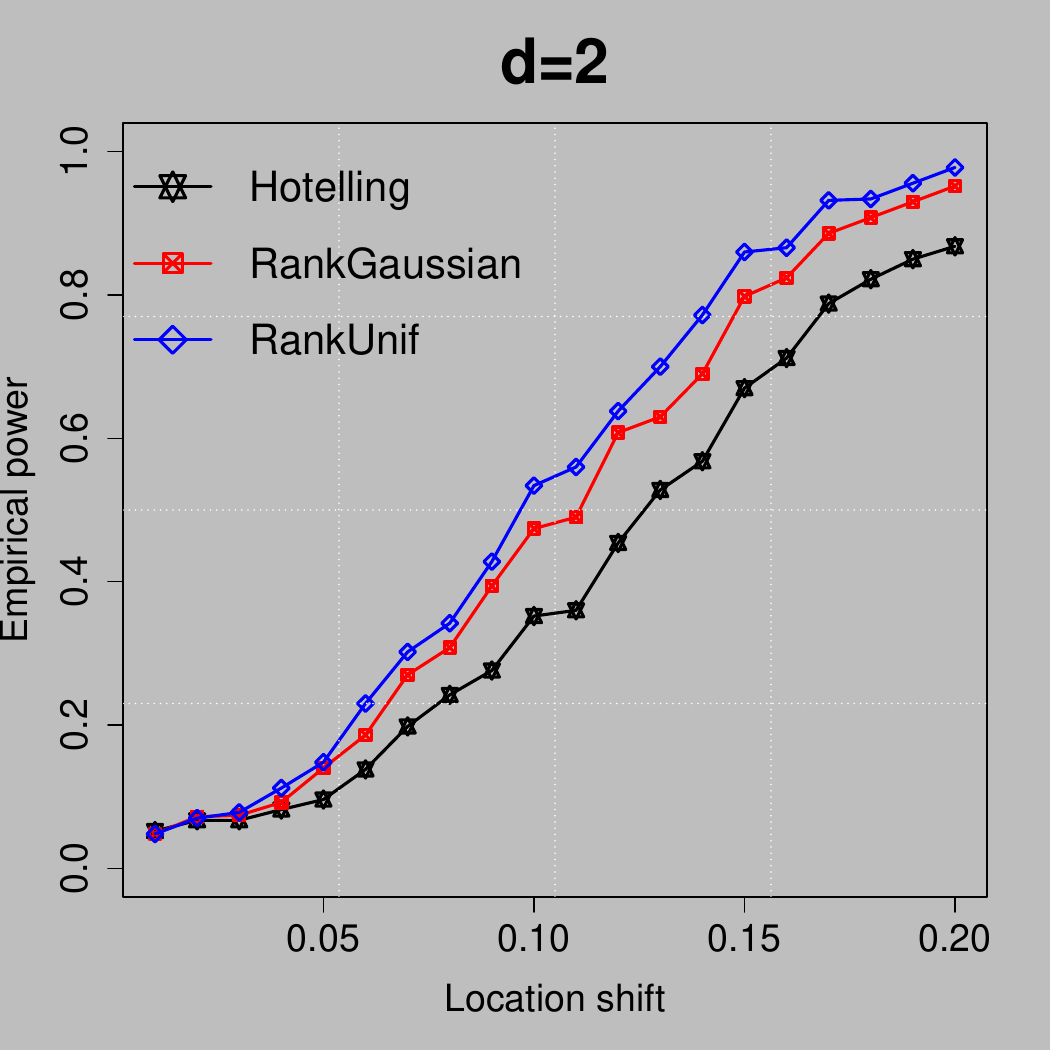}	
			\includegraphics[height=7.5cm,width=7cm]{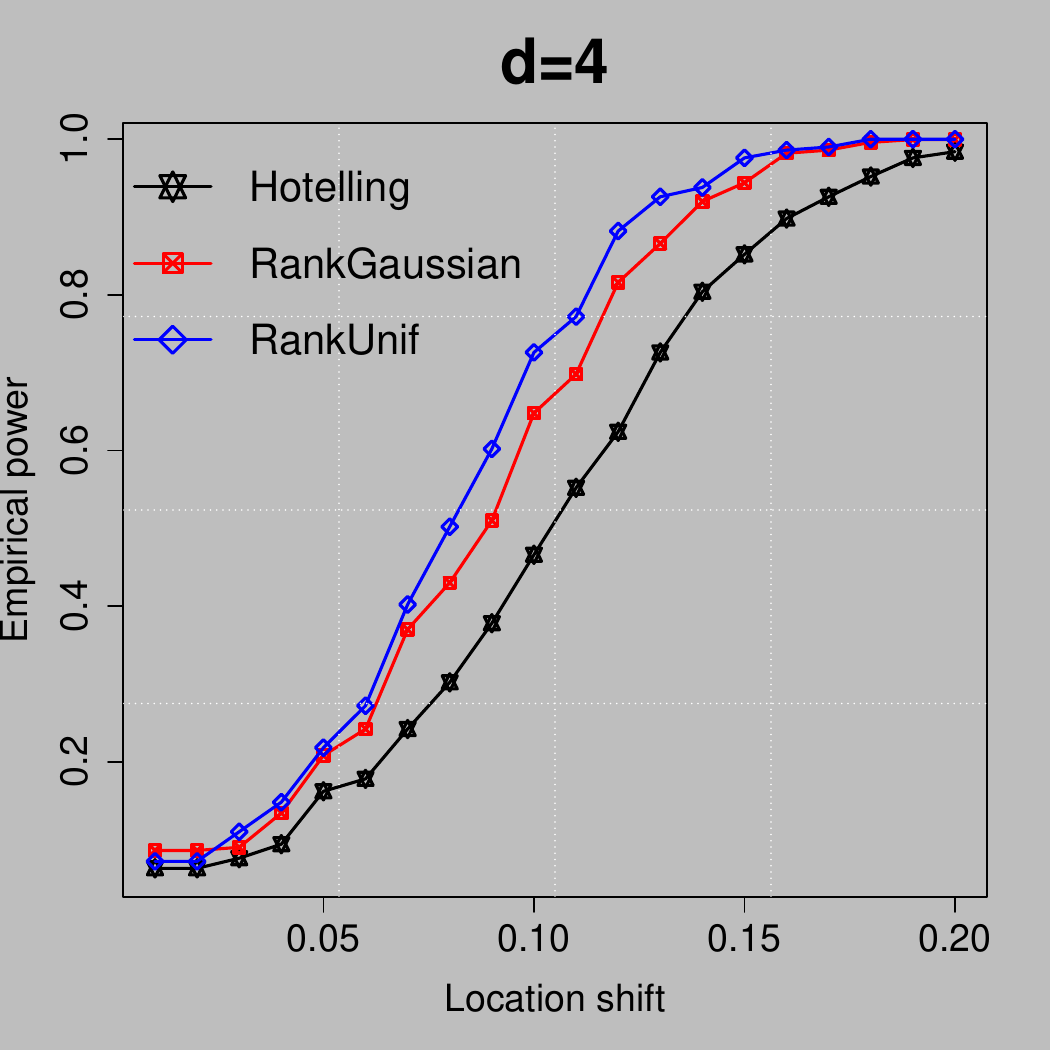}
		\end{center}	
		\caption{Power curves under setting (A2) with $d=2$ (left panel) and $d=4$ (right panel). The rest of the description of this figure is similar to that of~\cref{fig:settingA1}.}
		\label{fig:settingA2}
	\end{figure}
	\begin{figure}[h]
		\begin{center}
			\includegraphics[height=7.5cm,width=7cm]{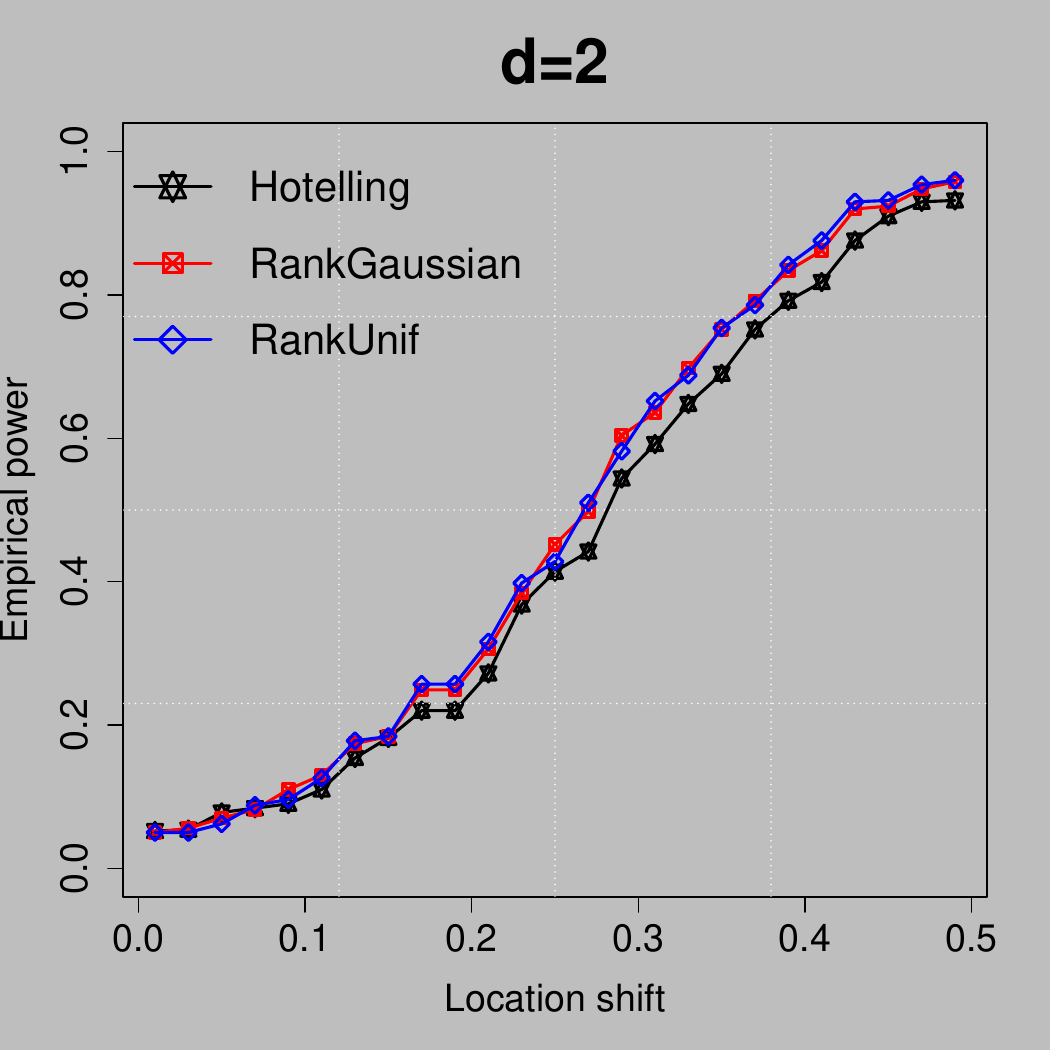}	
			\includegraphics[height=7.5cm,width=7cm]{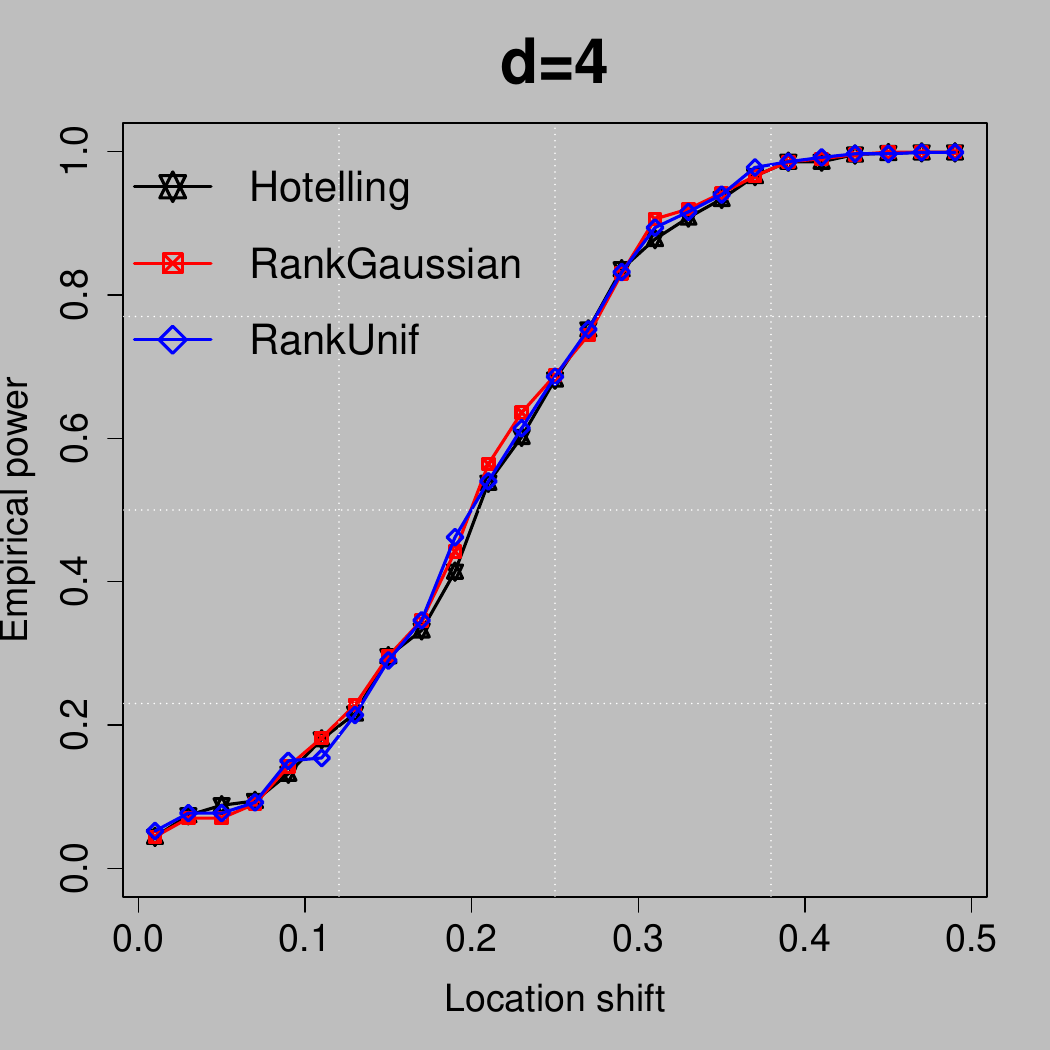}
		\end{center}	
		\caption{Power curves under setting (A3) with $d=2$ (left panel) and $d=4$ (right panel). The rest of the description of this figure is similar to that of~\cref{fig:settingA1}.}
		\label{fig:settingA3}
	\end{figure}
	
	\begin{figure}[h]
		\begin{center}
			\includegraphics[height=7.5cm,width=7cm]{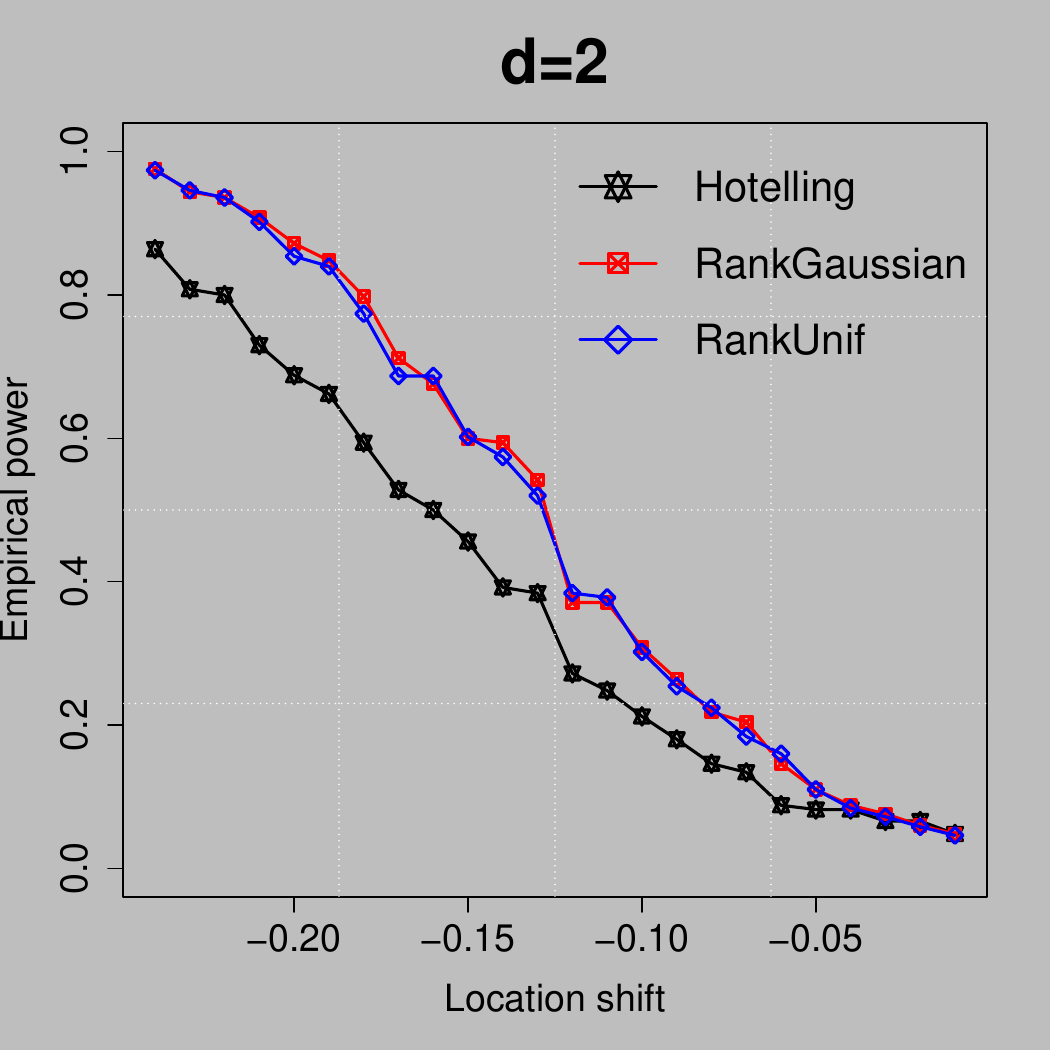}	
			\includegraphics[height=7.5cm,width=7cm]{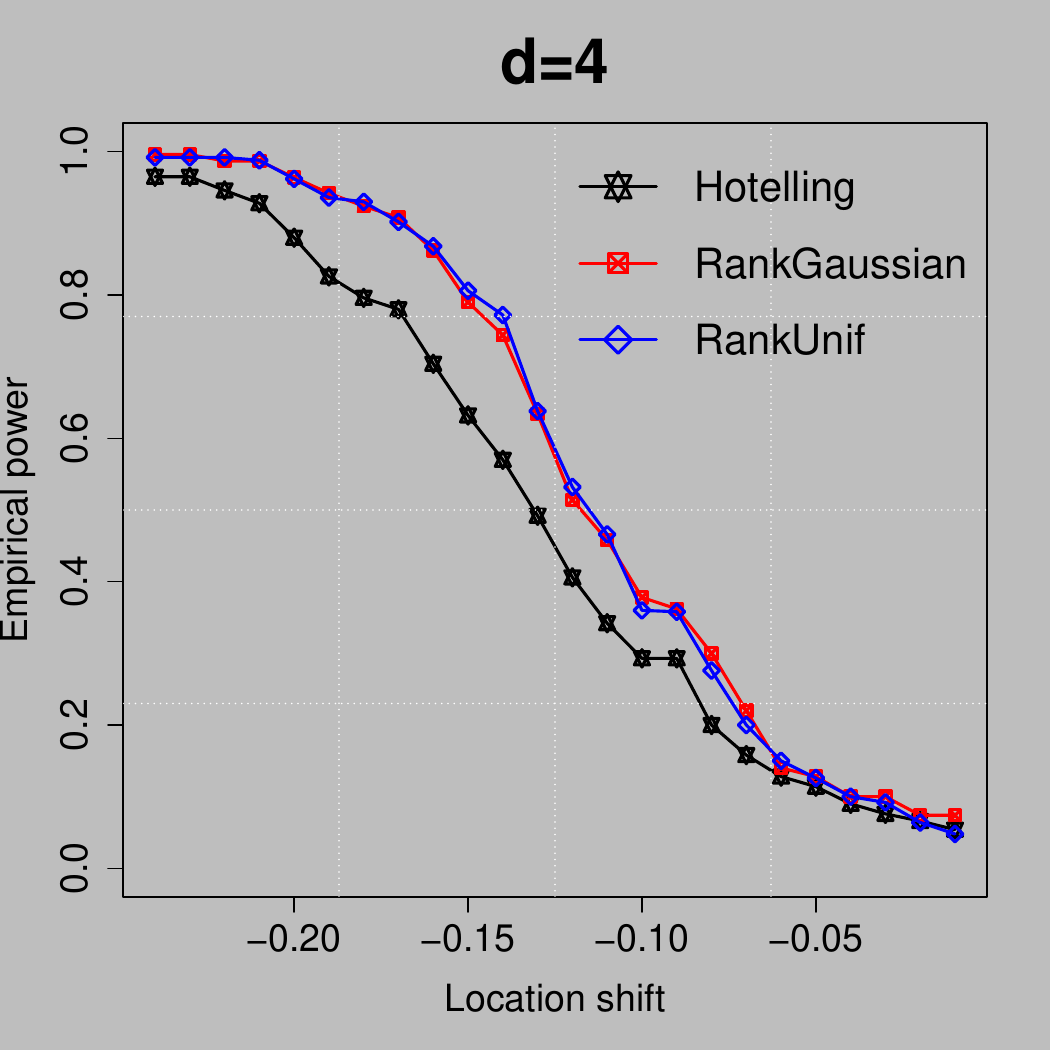}
		\end{center}	
		\caption{Power curves under setting (A4) with $d=2$ (left panel) and $d=4$ (right panel). The rest of the description of this figure is similar to that of~\cref{fig:settingA1}.}
		\label{fig:settingA4}
	\end{figure}
	
	\subsection{Beyond location/mean alternatives}\label{sec:beyondloc}
	One of the main reasons behind the popularity of Wilcoxon type tests (for $d=1$) over the $t$-test is the fact that Wilcoxon type methods can detect a more general class of alternatives, beyond simply mean shift. In this section, we use two examples to show that the rank Hotelling $T^2$ test based on $\rhsc$ can also detect beyond location alternatives which Hotelling $T^2$ is unable to do. For the two settings described below, we have chosen $d=3$ and $m=n=200$. The tests are carried out at level $0.05$ and the power curves are plotted based on $1000$ independent replicates.
	\begin{enumerate}
		\item[(A5)] $\bX=(X_1,X_2,X_3)$ and $\bY=(Y_1,Y_2,Y_3)$ where $X_1, X_2, X_3$ are i.i.d. $(\mathcal{N}(0,1))^2-1$ and $Y_1, Y_2, Y_3$ are i.i.d. 
		$(\mathcal{N}(\theta,1))^2-\E[(\mathcal{N}(\theta,1))^2]$, $\theta\in [0,0.2]$. By construction therefore, both $\bX$ and $\bY$ have mean $\bzr$. However, various other aspects of the two distributions are different, for instance, their variances.
		\item[(A6)] $\bX=(X_1,X_2,X_3)$ and $\bY=(Y_1,Y_2,Y_3)$ where $X_1, X_2, X_3$ are i.i.d. $(\mathcal{N}(0,1))^4-3$ and $Y_1,Y_2,Y_3$ are i.i.d. $(\mathcal{N}(\theta,1))^4-\E[(\mathcal{N}(\theta,1))^4]$, $\theta\in [0,0.2]$. Once again, both $\bX$ and $\bY$ have mean $\bzr$, but different variances.
	\end{enumerate}
	\begin{figure}[h]
		\begin{center}
			\includegraphics[height=7.5cm,width=7cm]{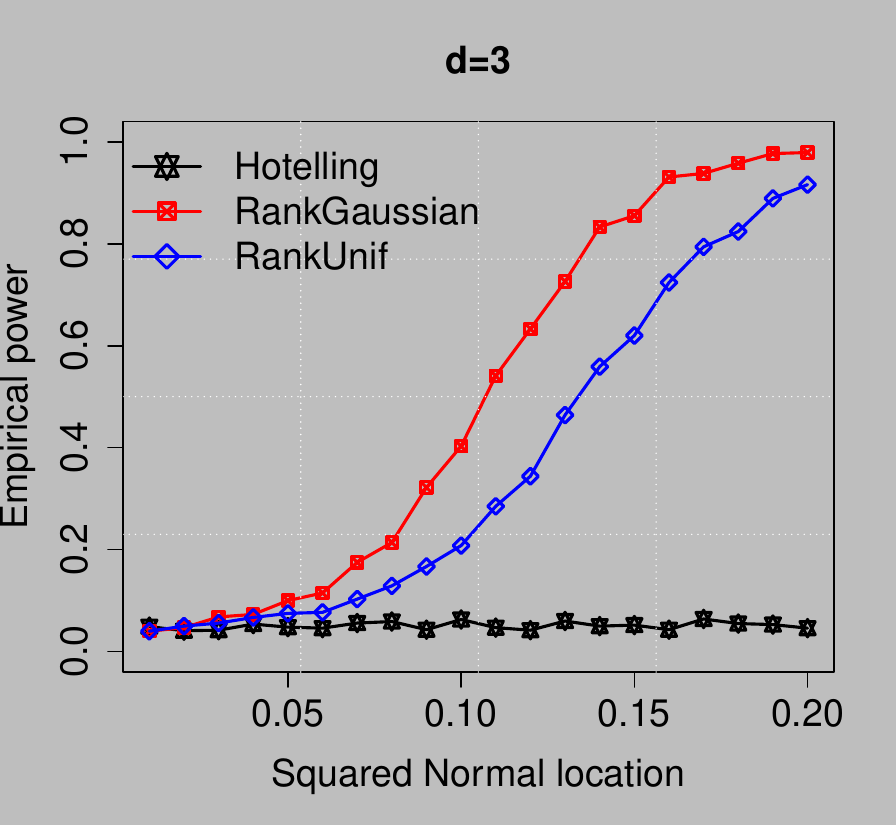}	
			\includegraphics[height=7.5cm,width=7cm]{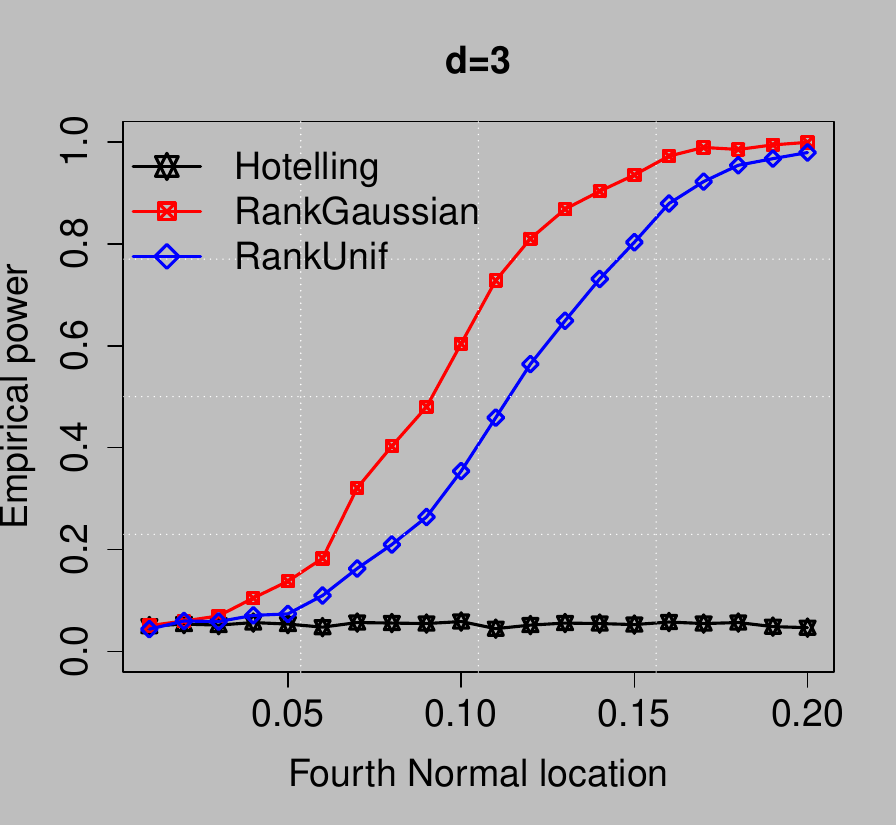}
		\end{center}	
		\caption{In the left panel, we sample $\bX$, $\bY$ according to setting (A5). The power curves are plotted as $\theta\in [0,0.2]$. The red line represents the power curve for $\rhsc$ with $\nu=\nu_G$, the blue line represents the same for $\rhsc$ with $\nu=\nu_U$, and the black line represents the same for Hotelling $T^2$. In the right panel, we plot the same for setting (A6).}
		\label{fig:settingA6}
	\end{figure}
	
	The settings chosen above are rather natural and formed by taking simple non-linear transformations of independent standard normals. We now compare the performances of Hotelling $T^2$ and $\rhsc$ with $\nu=\nu_G$ (the Gaussian reference distribution) and $\nu=\nu_U$ (the $\mathrm{Unif}[0,1]^d$ reference distribution), where $\bJ(\mx)=\mx$, in~\cref{fig:settingA6}. 
	
	As expected, the Hotelling $T^2$ is powerless in both settings as it is unable to detect differences beyond changes in mean. Also, in both cases, the $\mathrm{Unif}[0,1]^d$ reference distribution ($\nu_U$) is significantly outperformed by the Gaussian reference distribution ($\nu_G$). This is very much in line with the observations (see Figures~\ref{fig:settingA5}--\ref{fig:settingA4}) and the results proved in the rest of the paper. The plots also make it clear that both our proposals, namely, $\rhsc$ with $\nu=\nu_G$ and $\nu=\nu_U$ are able to detect differences beyond the mean shift. Both the above tests have power curves that approach $1$ as the $\theta$ parameter moves closer and closer to $0.2$. In fact, we tried several other non-linear transforms (such as an exponential function) of standard normals (appropriately centered) compared to the settings (A5) and (A6), and we observed similar behavior as in~\cref{fig:settingA6}.
	
	Overall~\cref{fig:settingA6} makes it clear that $\rhsc$ is indeed able to detect differences beyond the mean shift, unlike Hotelling $T^2$, in the multivariate setting. We believe this also makes the conclusion of~\cref{theo:rhconsis}  worthy of detailed investigation. While we have shown in~\cref{prop:conloc}  that the consistency condition in~\cref{theo:rhconsis} encompasses location alternatives, it is now apparent that the criteria covers more general alternatives.
	
	\subsection{Comparing Energy, MMD with their multivariate rank versions}\label{sec:emmdrank}
	
	In the previous simulation sections, we have focused on $\rhsc$, its efficiency and advantages over usual Hotelling $T^2$ in the multivariate setting. We will now move on to the other category of tests we described in~\cref{sec:rankmmd} , that is, the tests based on $\grsc$. Recall that these tests can be viewed as multivariate rank versions of the celebrated energy test (see~\cite{Szekely2013}) and the kernel MMD (see~\cite{Gretton2012}). In the sequel, we will draw two sets of comparisons described below:
	\begin{enumerate}
		\item \emph{Energy vs rank energy}: Consider the energy test for the two-sample testing problem. We compare it with the rank energy test as described in~\cref{rem:rankker}, where $\bJ(\mx)=\mx$ and $\nu$ is taken to be either $\nu_U=\mathrm{Unif}[0,1]^d$ or $\nu_G=\mathcal{N}(\bzr,{\bm I}_d)$.
		\item \emph{MMD vs rank MMD}: Consider the MMD test for the two-sample testing problem with the Gaussian kernel, that is, $\bKe(\mx,\my)=\exp(-\lVert \mx-\my\rVert^2/h)$, where the $h$ (bandwidth) parameter is chosen using the `sigest' mechanism (see~\cite{caputo2002appearance,Karatzoglou2004}). This mechanism can be viewed as an extension of the `median heurstic' and is based on finding an appropriate bandwidth based on the quantiles of $\lVert \mx-\my\rVert^2$ as $\mx$ and $\my$ vary over the observed data points. This is also convenient to implement via the \texttt{kernlab} package (see~\cite{kernlab2002}) in \texttt{R}. We compare it with the rank MMD test as described in~\eqref{eq:kerankscmmd}, where $\bJ(\mx)=\mx$ and $\nu$ is taken to be either $\nu_U=\mathrm{Unif}[0,1]^d$ or $\nu_G=\mathcal{N}(\bzr,{\bm I}_d)$.
	\end{enumerate}
	As in the previous section, we will work with $d=3$, $m=n=200$ and carry out the tests at level $0.05$ and obtain power curves with $1000$ independent replicates. The following settings have been used.
	\begin{enumerate}
		\item[(A7)] $\bX\sim \mathcal{N}(\bzr,\boldsymbol{\Sigma})$ and $\bY\sim\mathcal{N}(\bzr,\boldsymbol{\Sigma}_{\theta})$ where $\boldsymbol{\Sigma}(i,j)=0.3^{|i-j|}$ and $\boldsymbol{\Sigma}_{\theta}=\theta^{|i-j|}$ where the `scale' $\theta$ varies in the interval $(0.3,0.7)$.
		\item[(A8)] Construct $\tilde{\bX}_1\sim\mathcal{N}(\bzr,\boldsymbol{\Sigma})$ and $\tilde{\bY}_1\sim\mathcal{N}(\bzr,\boldsymbol{\Sigma}_{\theta})$ as in (A7). Then set $\bX=\exp(\tilde{\bX}_1)$ and $\bY=\exp(\tilde{\bY}_1)$.
		\item[(A9)] $\bX\sim \mathcal{N}(\bzr,\boldsymbol{\Sigma})$ and $\bY\sim\mathcal{N}(\bzr,\boldsymbol{\Sigma}_{\theta})$ where $\boldsymbol{\Sigma}(i,j)=0.3$ if $i\neq j$, $\boldsymbol{\Sigma}(i,i)=1$, and set  $\boldsymbol{\Sigma}_{\theta}(i,j)=\theta$ if $i\neq j$, $\boldsymbol{\Sigma}_{\theta}(i,i)=1$, where the `scale' $\theta$ varies in the interval $(0.3,0.7)$.
		\item[(A10)] Construct $\tilde{\bX}_1\sim\mathcal{N}(\bzr,\boldsymbol{\Sigma})$ and $\tilde{\bY}_1\sim\mathcal{N}(\bzr,\boldsymbol{\Sigma}_{\theta})$ as in (A9). Then set $\bX=\exp(\tilde{\bX}_1)$ and $\bY=\exp(\tilde{\bY}_1)$.
	\end{enumerate}
	Note that both the settings (A7) and (A9) feature the centered Gaussian distribution with commonly studied covariance matrices. For instance, (A7) is the autocorrelation matrix up to order $2$ of an autoregressive model of order $1$. (A9) is the popular equicorrelation matrix which has been studied extensively in the hypothesis testing literature (see~\cite{Steiger1980testing,gupta1982tests}). The settings (A8) and (A10) are formed by taking a lognormal transform of the settings (A7) and (A9) respectively. This is a popular way to make the data somewhat heavy-tailed. It was recently observed in~\cite{Deb19} that this modification can sometimes adversely impact the performance of energy statistic, probably because it is not very robust against outliers. 
	
	The power plots for the above simulation settings are given in~\cref{fig:settingA7}. The broad common trend in all the $4$ plots is that our proposed distribution-free multivariate rank versions of the energy test and the MMD clearly dominate the standard, existing versions (see~\cite{Szekely2013,Gretton2012}). Moreover, in all cases, the Gaussian reference distribution based implementation of $\grsc$ dominates the corresponding $\mathrm{Unif}[0,1]^d$ reference distribution based implementation. This echoes the general message of our paper that using the Gaussian reference distribution can have beneficial power properties in a variety of models. Note that this also matches a similar observation made in~\cref{fig:settingA5,fig:settingA1,fig:settingA4} when comparing $\rhsc$ with reference distributions $\nu_G$ and $\nu_U$, with the standard Hotelling $T^2$ test. 
	
	On a finer level, let us compare the plots in the left hand column (with normal distribution) of~\cref{fig:settingA7} to the right hand column (with lognormal distribution). Observe that the performance of the energy test tapers off in the lognormal distribution power plots when compared to the corresponding normal distribution power plots, particularly for larger values of the scale parameter. This is perhaps not surprising given that it works directly with inter-point distances and may be adversely affected by a few ``large'' observations, which typically arise when working with the lognormal distribution. It is surprising that the MMD too exhibits a similar phenomena, although it uses a \emph{bounded} Gaussian kernel. Having said that, the kernel MMD outperforms the energy test in all cases, which we believe is due to the carefully selected bandwidth. Interestingly, all the rank versions proposed in this paper have a consistent performance in the normal and corresponding lognormal power plots. Further, we see that for larger values of the scale $\theta$ in the lognormal plots, the power of $\grsc$ with the energy kernel (see~\eqref{eq:ranker}) and $\nu=\nu_G$ dominates the power of the MMD with Gaussian kernel despite the careful bandwidth selection process in the latter. This is testimony to the added robustness provided by our rank-based procedures when compared to the energy test or the MMD.
	
	\begin{figure}[h]
		\begin{center}		
			\subcaptionbox{Setting (A7)}{\includegraphics[height=7.5cm,width=7cm]{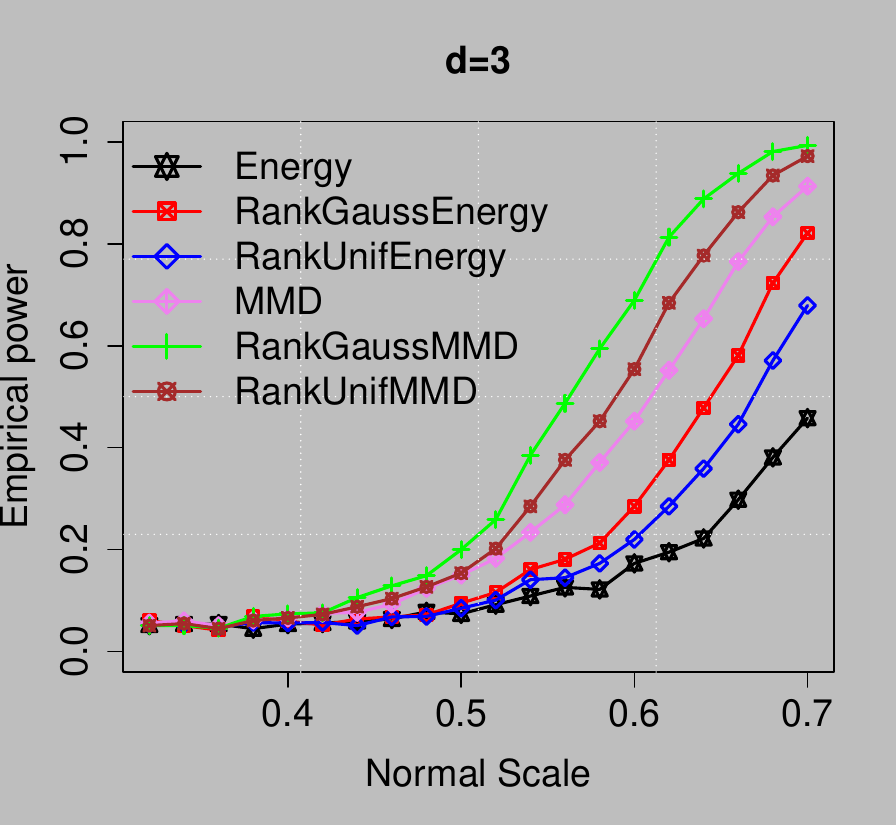}}
			\subcaptionbox{Setting (A8)}{\includegraphics[height=7.5cm,width=7cm]{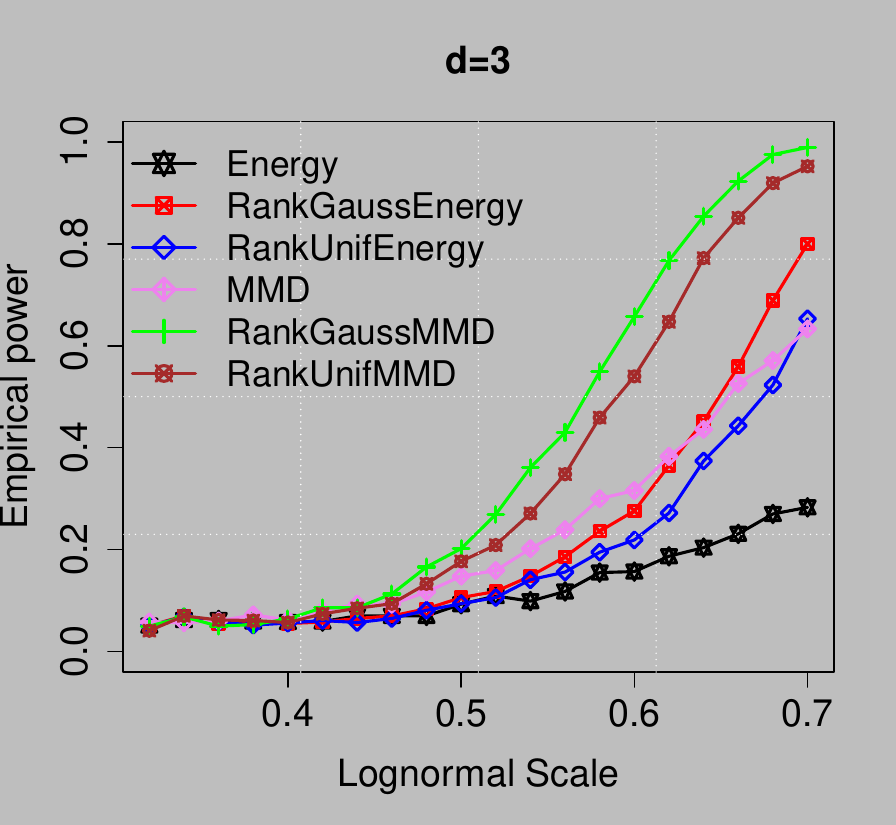}}
			\subcaptionbox{Setting (A9)}{\includegraphics[height=7.5cm,width=7cm]{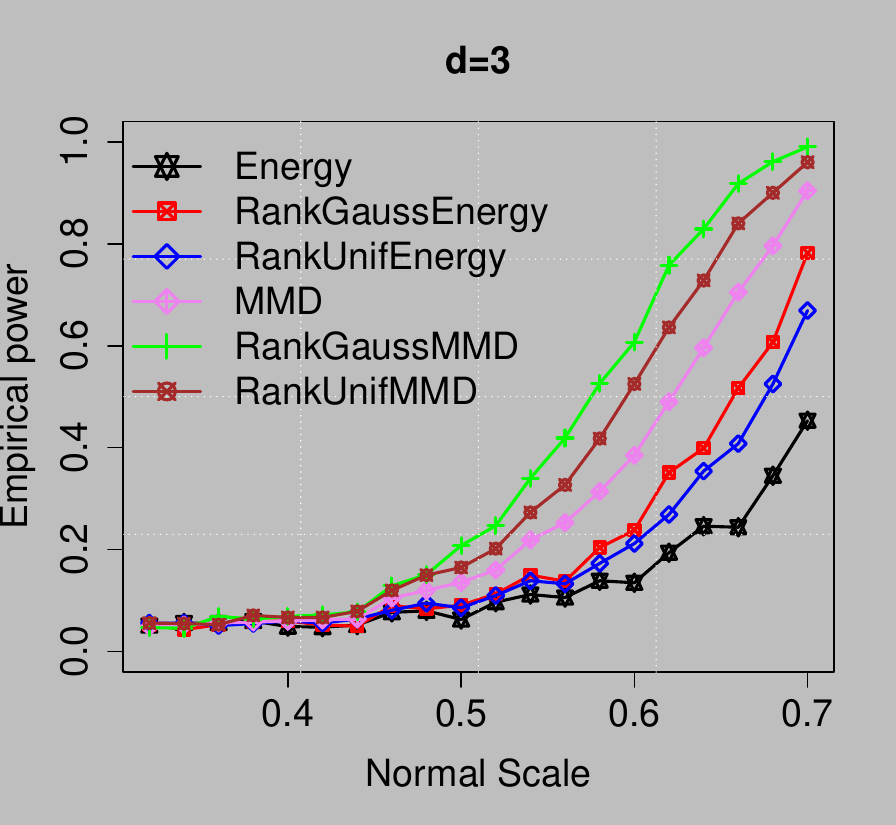}}	
			\subcaptionbox{Setting (A10)}{\includegraphics[height=7.5cm,width=7cm]{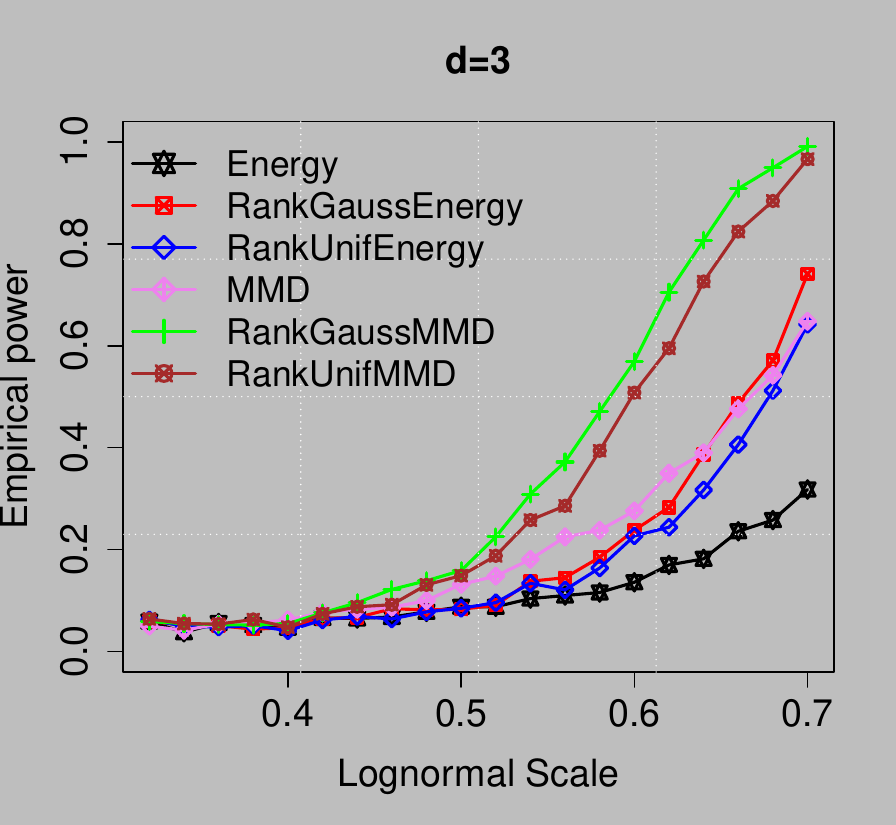}}
		\end{center}	
		\caption{The top left and right panels correspond to settings (A7) and (A8), whereas the bottom left and right panels correspond to settings (A9) and (A10) respectively. In all the figures, ``RankGaussEnergy'' and ``RankUnifEnergy'' are used to denote $\grsc$ with $\nu=\nu_G$ and $\nu=\nu_U$ respectively, with $\bJ(\mx)=\mx$ and the kernel from~\cref{rem:rankker}  in both cases. Similarly, ``RankGaussMMD'' and ``RankUnifMMD'' are used to denote $\grsc$ with $\nu=\nu_G$ and $\nu=\nu_U$ respectively, with $\bJ(\mx)=\mx$ and the Gaussian kernel with appropriate bandwidth (see the description in~\cref{sec:emmdrank}) in both cases.} 
	\label{fig:settingA7}
\end{figure}

\subsection{Behavior in higher dimension}\label{sec:highdpower}
In the previous section, we have focused on the \emph{low dimensional setting} and compared the tests based on $\grsc$ (as described in~\cref{sec:rankmmd}) with the celebrated energy test (see~\cite{Szekely2013}) and the kernel MMD (see~\cite{Gretton2012}). In contrast the focus of this section is on \emph{significantly higher dimensional} regime. We have used the same choices of $\nu$, $\bJ(\cdot)$, and bandwidths for kernel MMD as in the previous section.

Specifically, we consider $d=75$, $m=n=200$ and carry out the tests at level $0.05$ and obtain power curves with $1000$ independent replicates. The following natural settings have been considered: 
\begin{enumerate}
	\item[(A11)] \emph{Gaussian location shift}: Same as the setting (A1) above, with $d=75$. 
	\item[(A12)] \emph{Lognormal location shift}: Same as the setting (A4) above, with $d=75$.
	\item[(A13)] \emph{Gaussian correlation decay}: Same as setting (A7) above with $d=75$.
	\item[(A14)] \emph{Gaussian equicorrelated design}: Same as setting (A9) above with $d=75$
\end{enumerate}

The power plots for the above simulation settings are given in~\cref{fig:settingA8}. Generally speaking, the plots provide strong evidence about the competitive nature of the test based on $\grsc$ compared to those based on energy and kernel MMD, even in such high dimensional problems. Only in setting (A14), do we see some significant gains of using the usual kernel MMD as the (equi)correlation parameter under the alternative crosses $0.5$. In the settings (A11), (A12) and (A13), the rank-based methods perform as well and often better than the non rank-based competitors.

The broad common trend in all the $4$ plots is that the closer the alternative is to the null, the better is the performance of our rank-based procedures. This is supported by the fact that near the left hand bottom corner, the power curves for the rank-based procedures are in general above those of the energy and kernel MMD based tests. This is quite noticeable in the upper left hand panel of~\cref{fig:settingA8} which corresponds to the Gaussian location problem. We observe here that up to the point where the location parameter under the alternative is $\leq 0.6$, both the curves corresponding to the energy and kernel MMD are lower than those of our rank-based procedures. The curve corresponding to the energy test marginally crosses that of our tests after crossing $0.06$ on the $x$-axis. For the power curve of the kernel MMD test, this takes longer (around $0.08$ in the $x$-axis). For the upper right hand panel of~\cref{fig:settingA8}, corresponding to the lognormal case, the power curves are largerly indistinguishable. While the kernel MMD still has lower power near the bottom left, the difference with the other curves is quite small and hence could just be an effect of pure noise. In the bottom left panel of~\cref{fig:settingA8}, we observe that there is quite a difference between say the power curve of the energy distance based test compared to the rank energy test and the rank MMD test, both using the Gaussian reference distribution. In fact, when the correlation parameter (on the $x$-axis) is near its highest (around $0.7$), the power of energy test is around $0.4$ whereas those of the Gaussian reference based rank tests are around $0.7$. The power curve of the test based on MMD also lies below the rank tests up to around $0.6$ in the $x$-axis. The bottom right panel of~\cref{fig:settingA8} provides an example where the kernel MMD test works significantly better than all the other tests when the (equi)correlation parameter crosses $0.5$. In this case as well, for smaller values of the (equi)correlation parameter, the rank-based tests work better.

We also see that the benefits of using the Gaussian reference distribution continue to show up. This is quite evident in the bottom left panel of~\cref{fig:settingA8}, where both the rank-based tests (that is, rank energy and rank MMD) with Gaussian reference distribution perform noticeably better than the ones with the Uniform reference distribution. In the bottom right panel of~\cref{fig:settingA8}, once again the power curve of the rank-based energy test with the Gaussian reference distribution is noticeably above that of rank-based energy test with the uniform reference distribution. 

Overall, we find the simulation results in the high dimensional settings quite encouraging. In particular, it seems that even in the high dimensional regime, the Gaussian reference distribution continues to have some attractive power properties, the likes of which we have theoretically exhibited in the low dimensional setting. We leave the theoretical understanding of rank-based tests in the high dimensional regime for further research.

\begin{figure}[h]
	\begin{center}
		\subcaptionbox{Setting (A11)}{\includegraphics[height=7.5cm,width=7cm]{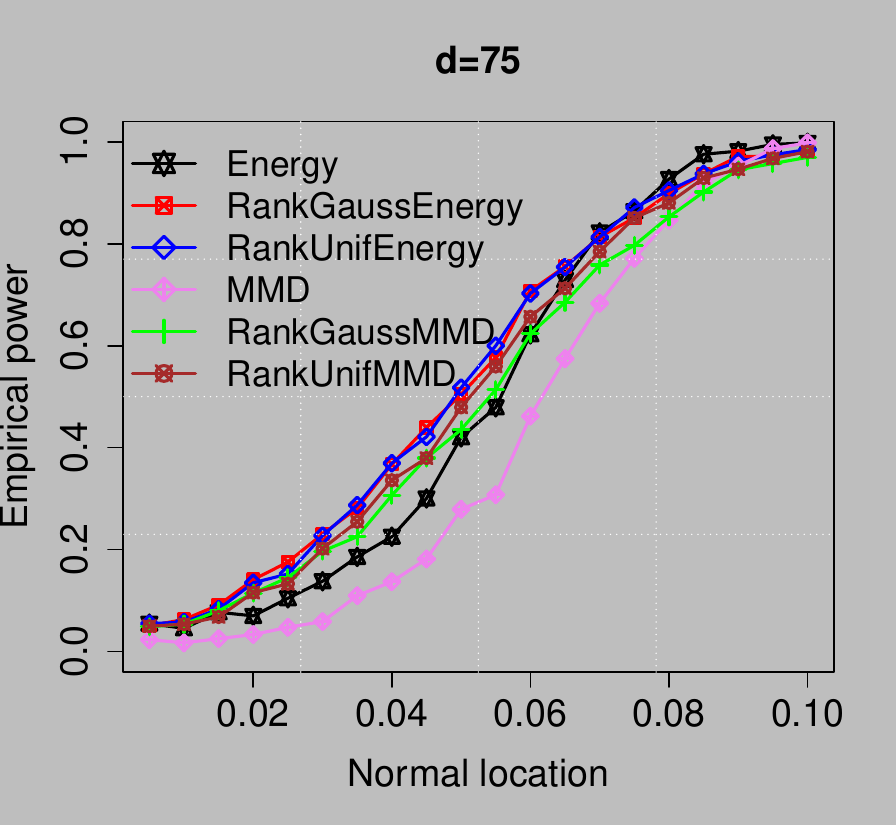}}	
		\subcaptionbox{Setting (A12)}{\includegraphics[height=7.5cm,width=7cm]{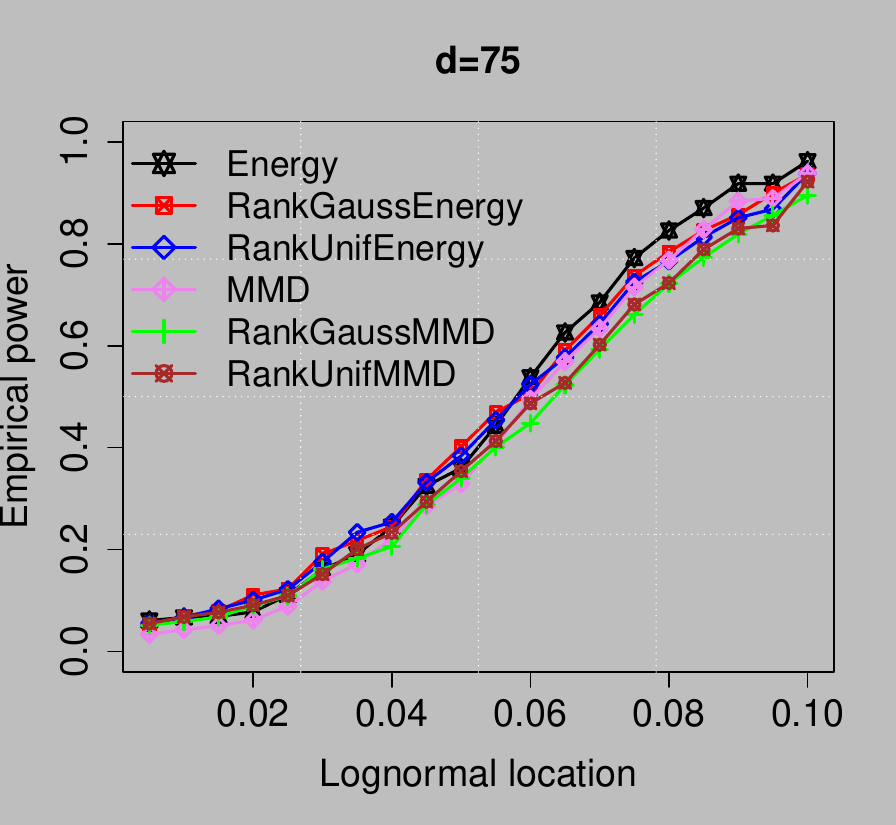}}
		\subcaptionbox{Setting (A13)}{\includegraphics[height=7.5cm,width=7cm]{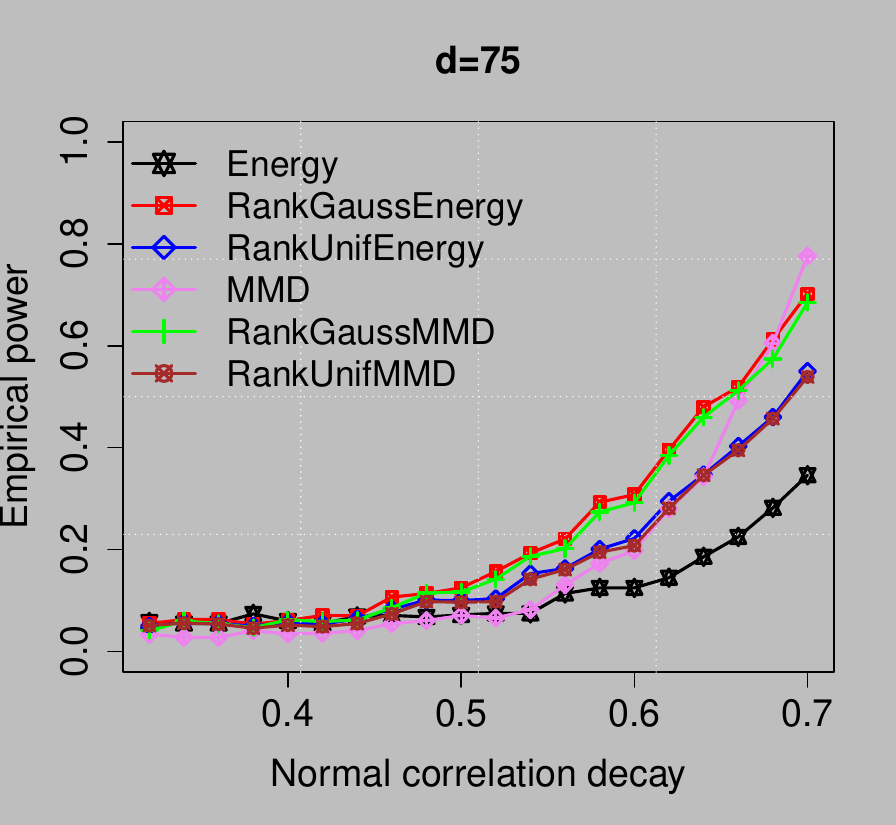}}
		\subcaptionbox{Setting (A14)}{\includegraphics[height=7.5cm,width=7cm]{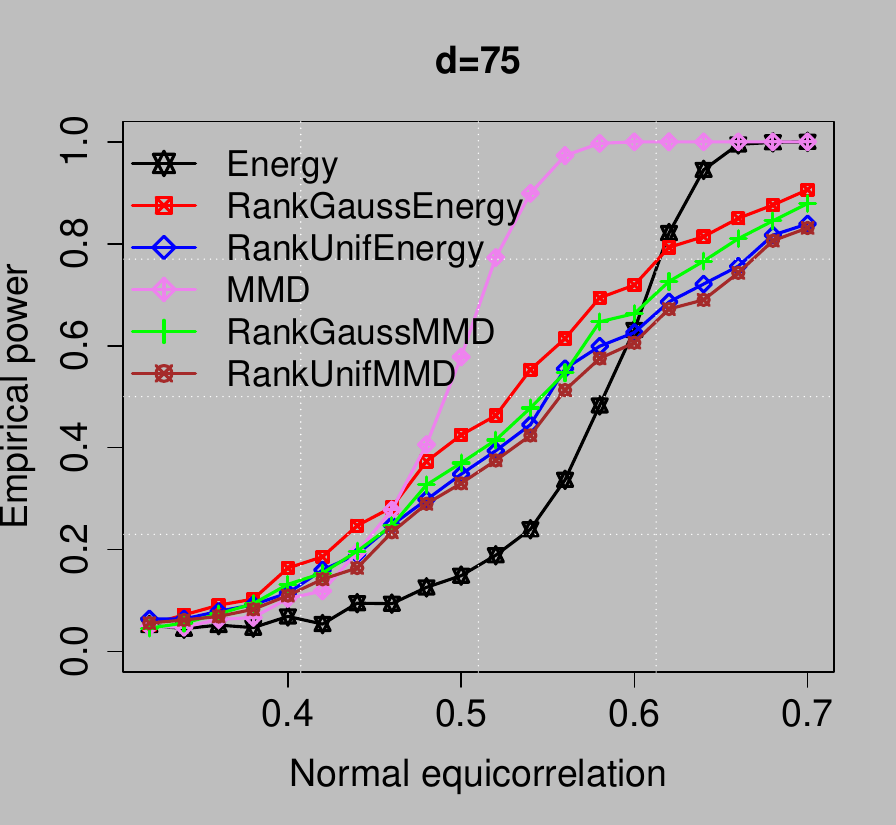}}
	\end{center}	
	\caption{The top left and right panels correspond to settings (A11) and (A12), whereas the bottom left and right panels correspond to settings (A13) and (A14) respectively. In all the figures, ``RankGaussEnergy'' and ``RankUnifEnergy'' are used to denote $\grsc$ with $\nu=\nu_G$ and $\nu=\nu_U$ respectively, with $\bJ(\mx)=\mx$ and the kernel from~\cref{rem:rankker}  in both cases. Similarly, ``RankGaussMMD'' and ``RankUnifMMD'' are used to denote $\grsc$ with $\nu=\nu_G$ and $\nu=\nu_U$ respectively, with $\bJ(\mx)=\mx$ and the Gaussian kernel with appropriate bandwidth (see the description in~\cref{sec:emmdrank}) in both cases.}
	\label{fig:settingA8}
\end{figure}

\end{document}